\documentclass[aos]{imsart}

\RequirePackage{amsthm,amsmath,amsfonts,amssymb}
\RequirePackage[colorlinks,citecolor=blue,urlcolor=blue]{hyperref}
\RequirePackage{graphicx}

\usepackage{bm,stackengine} 
\usepackage{epsfig}
\usepackage{psfrag}
\usepackage{enumerate} 
\usepackage{float} 
\usepackage[dvipsnames]{xcolor}
\usepackage{array}
\usepackage{hyperref}
\usepackage{subcaption}
\usepackage[normalem]{ulem}

\def\namedlabel#1#2{\begingroup
    #2%
    \def\@currentlabel{#2}%
    \phantomsection\label{#1}\endgroup
}
\usepackage{enumitem}

\usepackage{mathtools}
\mathtoolsset{showonlyrefs}

\newtheorem{theorem}{Theorem}[section]
\newtheorem{lemma}[theorem]{Lemma}
\theoremstyle{remark}

\theoremstyle{plain}

\newtheorem{theo}{Theorem}[section] 

\newtheorem{lem}[theo]{Lemma}

\theoremstyle{definition}

\newtheorem{de}{Definition}[section]

\newtheorem{remark}{Remark}[section]

\theoremstyle{plain}

\newtheorem{theos}{Theorem}

\newtheorem{lems}[theos]{Lemma}
\newtheorem{cors}[theos]{Corollary}

\newtheorem{proposition}[theos]{Proposition}
\newtheorem{corollary}[theos]{Corollary}

\theoremstyle{definition}

\newcommand{\Risk}{\mathcal{R}}
\newcommand{\Loss}{\mathcal{L}}

\newcommand{\lone}[1]{\|#1\|_1}
\newcommand{\ltwo}[1]{\|#1\|_2}

\newcommand{\zetamax}{\zeta_{\max}}

\newcommand{\CLoss}{\mathcal{C}}

\newcommand{\zetastar}{\zeta^*}
\newcommand{\taustar}{\tau^*}
\newcommand{\betastar}{\beta^*}

\newcommand{\normzero}[1]{\|#1\|_0}

\def\bu{\boldsymbol{u}}
\def\bh{\boldsymbol{h}}

\def\bX{\boldsymbol{X}}

\def\bthetahat{\widehat{\boldsymbol{\theta}}}

\def\bz{\boldsymbol{z}}

\def\by{\boldsymbol{y}}
\def\thetahatloo{ {\widehat {\boldsymbol \theta}_{\mathrm{loo}}} }
\def\bX{\boldsymbol{X}}
\def\be{\boldsymbol{e}}
\def\bw{\boldsymbol{w}}

\def\bwhat{\widehat{\boldsymbol{w}}}

\def\Var{\mathrm{Var}}

\def\bK{\boldsymbol{K}}
\def\hmu{\hat{\mu}}

\newcommand{\Aevent}{\mathcal{A}}

\newcommand{\lambdamin}{\lambda_{\min}}
\newcommand{\lambdamax}{\lambda_{\max}}

\newcommand{\betamax}{\beta_{\max}}
\def\betamin{\beta_{\mathrm{min}}}

\newcommand{\xper}{\breve{\boldsymbol{x}}^\perp}

\newcommand{\kappamin}{\kappa_{\min}}
\newcommand{\kappamax}{\kappa_{\max}}

\newcommand{\thetastarloo}{\bm{\theta}^*_{\mathrm{loo}}}

\newcommand{\mytheta}{\bm{\theta}}
\newcommand{\mySigma}{\bm{\Sigma}}

\newcommand{\numobs}{\ensuremath{n}}
\newcommand{\usedim}{\ensuremath{p}}

\newcommand{\thetastar}{\ensuremath{\bm{\theta}^*}}

\newcommand{\thetahat}{{\widehat{\boldsymbol{\theta}}}}

\newcommand{\mprob}{\ensuremath{\mathbb{P}}}

\newcommand{\real}{\ensuremath{\mathbb{R}}}
\newcommand{\defn}{\ensuremath{: \, =}}
\newcommand{\argmin}{\arg\!\min}
\newcommand{\argmax}{\arg\!\max}

\newcommand{\inprod}[2]{\ensuremath{\langle #1 , \, #2 \rangle}}

\newcommand{\Exs}{\ensuremath{\mathbb{E}}}

\newcommand{\NORMAL}{\mathcal{N}}

\newcommand{\sign}{\ensuremath{\mbox{\rm sign}}}

\long\def\comment#1{}

\newcommand{\HACKPROOF}{\begin{proof}}
\newcommand{\HACKENDPROOF}{\end{proof}}

\newcommand{\Ball}{\ensuremath{\mathbb{B}}}

\newcommand{\Ind}{\ensuremath{\textbf{I}}}

\newcommand{\diag}{\ensuremath{\mbox{diag}}}

\newcommand{\trace}{\ensuremath{\mbox{Tr}}}

\newlength{\widebarargwidth}
\newlength{\widebarargheight}
\newlength{\widebarargdepth}

\makeatletter
\long\def\@makecaption#1#2{
        \vskip 0.8ex
        \setbox\@tempboxa\hbox{\small {\bf #1:} #2}
        \parindent 1.5em  %% How can we use the global value of this???
        \dimen0=\hsize
        \advance\dimen0 by -3em
        \ifdim \wd\@tempboxa >\dimen0
                \hbox to \hsize{
                        \parindent 0em
                        \hfil 
                        \parbox{\dimen0}{\def\baselinestretch{0.96}\small
                                {\bf #1.} #2
                                } 
                        \hfil}
        \else \hbox to \hsize{\hfil \box\@tempboxa \hfil}
        \fi
        }
\makeatother

\def\bg{\bm{g}}
\def\bv{\bm{v}}
\def\bx{\bm{x}}

\def\cB{{\mathcal{B}}}
\newcommand{\bvhat}{\widehat{\bv}}

\def\taumax{\tau_{\mathrm{max}}}

\def\reals{{\mathbb R}}
\def\f{\bm{f}}
\def\normal{{\sf N}}
\def\<{\langle}
\def\>{\rangle}
\def\E{{\mathbb E}}
\def\de{{\rm d}}

\def\bzero{\bm{0}}

\def\eps{\varepsilon}
\def\supp{{\rm supp}}
\newcommand{\indic}[1]{\mathbf{1}_{#1}}

\usepackage[mathscr]{euscript}
\DeclareSymbolFont{rsfs}{U}{rsfs}{m}{n}
\DeclareSymbolFontAlphabet{\mathscrsfs}{rsfs}

\def\cuE{\mathscrsfs{E}}
\def\tcuE{\tilde{\mathscrsfs{E}}}
\def\cuF{\mathscrsfs{F}}
\def\cuG{\mathcal{G}}
\def\softthreshold{\eta_{\mathrm{soft}}}

\def\xistar{\xi^*}
\def\kappastar{\kappa^*}

\def\cuL{\mathcal{L}}
\def\cuD{\mathcal{D}}
\def\kappacond{\kappa_{\mathrm{cond}}}
\def\barthetastar{{\bar\mytheta^*}}
\def\omegastar{{\omega^*}}
\def\bX{\boldsymbol{X}}

\def\bthetahat{\widehat{\boldsymbol{\theta}}}

\def\div{\mathrm{div}\,}
\def\df{\mathsf{df}}
\def\risk{\mathsf{R}}

\def\bt{\boldsymbol{t}}
\def\sMM{\mbox{\tiny\rm MM}}
\def\sBP{\mbox{\tiny\rm BP}}
\def\sRE{\mbox{\tiny\rm RE}}

\def\sigmamin{\sigma_{\mathrm{min}}}
\def\sigmamax{\sigma_{\mathrm{max}}}
\def\numin{\nu_{\mathrm{min}}}

\def\zetamin{\zeta_{\mathrm{min}}}
\def\zetamax{\zeta_{\mathrm{max}}}

\def\bz{\boldsymbol{z}}
\def\yhat{\widehat{\boldsymbol{y}}}

\newcommand{\util}{\widetilde{\bm{u}}}
\newcommand{\subg}{\widehat{\bm{t}}}
\newcommand{\bmx}{\bm{X}}

\def\uhat{\widehat{\boldsymbol{u}}}
\def\thetahatd{\widehat{\theta}^{\mathrm{d}}}
\def\bthetahatd{\widehat{\boldsymbol{\theta}}^{\mathrm{d}}}
\def\bthat{\widehat{\boldsymbol{t}}}

\def\hlambda{\hat{\lambda}}

\def\RE{{\rm RE}}

\def\brevebx{\breve{\boldsymbol{x}}}
\def\cuT{\mathcal{T}}
\def\bA{\boldsymbol{A}}
\def\bB{\boldsymbol{B}}

\def\cuC{\mathcal{C}}
\def\cuA{\mathcal{A}}
\def\cuV{\mathcal{V}}
\def\bM{\boldsymbol{M}}

\def\sfM{\mathsf{M}}
\def\bb{\boldsymbol{b}}

\def\hbw{\widehat{\boldsymbol{w}}}
\def\hbv{\widehat{\boldsymbol{v}}}

\def\Lstar{L^*}
\def\cuZ{\mathcal{Z}}
\def\alphamax{\alpha_{\mathrm{max}}}
\def\Deltamin{\Delta_{\mathrm{min}}}
\def\Deltamax{\Delta_{\mathrm{max}}}
\def\breveby{\breve{\boldsymbol{y}}}
\def\brevey{\breve{y}}

\def\cuPmodel{\mathcal{P}_{\mathrm{model}}}

\def\bsigma{\boldsymbol{\sigma}}
\def\FCP{\mathsf{FCP}}
\def\bG{\boldsymbol{G}}

\def\htau{\hat{\tau}}
\def\htauloo{\widehat{\tau}_{\mathrm{loo}}}
\def\cuK{{\mathcal{K}}}
\newcommand{\indep}{\perp \!\!\! \perp}

\def\bC{{\boldsymbol C}}
\def\cT{{\mathcal{T}}}
\def\S{{\mathbb S}}
\def\Ball{{\sf B}}
\def\bdelta{\boldsymbol{\delta}}

\endlocaldefs

\begin{document}

\begin{frontmatter}
\title{The Lasso with general Gaussian designs with applications to hypothesis testing}
%\title{A sample article title with some additional note\thanksref{t1}}
\runtitle{Lasso with general Gaussian designs}
%\thankstext{T1}{A sample additional note to the title.}

\begin{aug}
%%%%%%%%%%%%%%%%%%%%%%%%%%%%%%%%%%%%%%%%%%%%%%
%%Only one address is permitted per author. %%
%%Only division, organization and e-mail is %%
%%included in the address.                  %%
%%Additional information can be included in %%
%%the Acknowledgments section if necessary. %%
%%%%%%%%%%%%%%%%%%%%%%%%%%%%%%%%%%%%%%%%%%%%%%
\author[A]{\fnms{Michael} \snm{Celentano}\thanksref{mc,am}},
\author[A,B]{\fnms{Andrea} \snm{Montanari}\thanksref{am}}
\and
\author[C]{\fnms{Yuting} \snm{Wei}\thanksref{yw}}
% \author[A]{\fnms{Michael} \snm{Celentano}\thanks{Partially supported by the NSF grants CCF-1714305, IIS-1741162, and by the ONR grant N00014-18-1-2729.}\thanks{Partially supported by the National Science Foundation Graduate Research Fellowship under grant DGE-1656518.}},
% \author[A,B]{\fnms{Andrea} \snm{Montanari}\footnotemark[1]}

\thankstext{am}{Partially supported by the NSF grants CCF-1714305, IIS-1741162, and by the ONR grant N00014-18-1-2729. We thank the anonymous reviewers for their valuable reviews.}
\thankstext{mc}{Partially supported by the National Science Foundation Graduate Research Fellowship under grant DGE-1656518.}
\thankstext{yw}{Partially supported  by the NSF grants DMS-2015447/2147546, CAREER award DMS-2143215 and the Google Research Scholar Award. }

%%%%%%%%%%%%%%%%%%%%%%%%%%%%%%%%%%%%%%%%%%%%%%
%% Addresses                                %%
%%%%%%%%%%%%%%%%%%%%%%%%%%%%%%%%%%%%%%%%%%%%%%
\address[A]{Department of Statistics, University of California at Berkeley}

\address[B]{Department of Electrical Engineering,
Stanford University}

\address[C]{Department of Statistics and Data Science,
University of Pennsylvania}
\end{aug}

\begin{abstract}
	The Lasso  is  a method for  high-dimensional regression,
	which is now  commonly used when the number of covariates $p$ is of the same order or larger
	than the number of observations $n$. Classical asymptotic normality
	theory does not apply to this model due to two fundamental reasons: $(1)$ The regularized risk is non-smooth;
	$(2)$ The distance between the estimator $ \widehat{\bm\theta}$ and the true parameters vector $\bm \theta^*$ cannot be
	neglected. As a consequence, standard perturbative arguments that are the traditional basis for asymptotic
	normality fail. 

	On the other hand, the Lasso estimator can be precisely characterized in the regime in which both $n$ and $p$
	are large and $n/p$ is of order one. This characterization was first obtained in the case of Gaussian
	designs with i.i.d. covariates: here we generalize it to Gaussian correlated designs with non-singular covariance structure.
	This is expressed in terms of a simpler ``fixed-design'' model.
	We establish non-asymptotic bounds on the distance between the distribution of various quantities in the two models,
	which hold uniformly over signals $\bm \theta^*$ in a suitable sparsity class and 
	over values of the regularization
	parameter.

	As an application, we study the distribution of the debiased Lasso and show that a degrees-of-freedom correction
	is necessary for computing valid confidence intervals.
\end{abstract}



\end{frontmatter}
%%%%%%%%%%%%%%%%%%%%%%%%%%%%%%%%%%%%%%%%%%%%%%
%% Please use \tableofcontents for articles %%
%% with 50 pages and more                   %%
%%%%%%%%%%%%%%%%%%%%%%%%%%%%%%%%%%%%%%%%%%%%%%
% \setcounter{tocdepth}{2}
% \tableofcontents

% \yutingcomment{equation numbers got messed up...}

\section{Introduction}
\label{sec:intro}

Questions of statistical inference and decision theory are often addressed by characterizing
 the distribution of the estimator of interest
under a variety of assumptions on the data distribution. 
A central role is played by normal theory which guarantees
that broad classes of estimators are asymptotically normal with prescribed covariance structure \cite{fisher1922mathematical,le2012asymptotic}.
Normality theory can serve as the basis for inference, facilitate the comparison of estimators, and justify claims of efficiency.

In high dimensions, 
the distributional theory available for many estimators of interest is more limited. 
Frequently we have access to upper and lower bounds on important quantites like the estimation or prediction error or the size of a selected model.
These may have the correct dependence on sample size, dimensionality, and certain structural parameters, but are usually loose in their leading constants.  
Asymptotic normality often breaks down in high dimensions, even when considering
low-dimensional projections of the coefficients vector \cite{bayati2011lasso,javanmard2014hypothesis,zhang2014confidence,sur2019modern}.
There has been substantial progress in recovering normality in special cases by resorting to careful constructions designed to remove bias and target normality \cite{bayati2011lasso,javanmard2014hypothesis,zhang2014confidence,bellec2019biasing,chen2019inference}.
It is of substantial interest to identify precisely the conditions under which such constructions succeed and fail.
This challenge is compounded by the fact that resampling methods also
fail in this context \cite{el2018can}.

The Lasso is arguably the prototypical method in high-dimensional statistics.
Given data $\{(y_i,\bx_i)\}_{i\le n}$, with $y_i\in\reals$, $\bx_i\in\reals^p$, it performs linear regression of the
$y_i$'s on the $\bx_i$'s by solving the optimization problem
\begin{align}
\label{EqnOrgRisk}
	\thetahat 
		&\defn 
		\argmin_{\mytheta \in \reals^{\usedim}} \Risk(\mytheta) 
		:= 
		\argmin_{\mytheta \in \reals^{\usedim}} \left\{\frac{1}{2\numobs} \ltwo{\by - \bX \mytheta }^2 + \frac{\lambda}{\sqrt{\numobs}} \lone{\mytheta}\right\}\,.
\end{align}
Here $\by\in\reals^n$ is the vector with $i$-th entry equal to $y_i$, and $\bX\in\reals^{n\times p}$ is the matrix with
$i$-th row given by $\bx_i^{\top}$.
Throughout the paper we will assume the model to be well-specified. Namely, there exist $\thetastar\in\reals^p$
such that
\begin{align}
\label{EqnLM}
	\by = \bX \thetastar + \sigma \bz\,,
\end{align}
where $\bz \sim \normal(\bzero, \Ind_{\numobs})$ is a Gaussian 
noise vector.\footnote{The assumption of Gaussian noise is not necessary for our results, but is made throughout to simplify our exposition and proofs. See Remark \ref{rmk:gaussian-error}.}
In the informal discussion below, we will assume $\thetastar$ to be $s$-sparse (i.e. to have at most
$s$ non-zero entries), although our theorems apply more generally to coefficient vectors that are only 
approximately sparse.

% \yutingcomment{these paragraphs repeat what has been said below?}
% Further, we assume the covariates are distributed $\bx_i \sim \normal(0,\mySigma)$.
% Under most conditions, 
% the Lasso estimator is not normally distributed, even approximately.
% A large body of work studies a correction to the Lasso estimator --- the so-called \emph{debiased Lasso} --- and identifies conditions under which it provides approximately unbiased and normal estimates of low dimensional projections of $\thetastar$.

% The present paper provides a precise distributional characterization of the Lasso and debiased lasso estimators with correlated normal covariates.
% Our results hold as soon as the sampling rate $n/p$ exceeds an important threshold---the so-called Donoho-Tanner phase transition \cite{donohoTanner2009}.
% Depending on the setting,
% our results provide a more precise characterization of the behavior of the Lasso or debiased Lasso estimate, or they establish control under weaker requirements on the sample-size and regularization parameter. 

\paragraph*{Distribution theory for the Lasso}

A substantial body of theoretical work studies the Lasso with fixed (non-random) designs $\bX$ 
in the regime
$s\log (p/s) / n = O(1)$ \cite{bickel2009simultaneous,buhlmann2011statistics,negahban2012,bellec2018} 
by providing estimation error bounds that are rate optimal.
These results have two types of limitations.
First, they usually require that $\lambda$ be chosen larger than the approximate minimax choice
$\lambda_{\sMM} = c\sigma \sqrt{\log(p/s)}$  (with $c$ a constant which cannot be 
taken arbitrarily small). 
In practice, however, $\lambda$ is chosen by cross-validation and is often 
significantly smaller than $\lambda_{\sMM}$ because the coefficient $\thetastar$ is
not the  least favorable one \cite{Chetverikov2016,miolane2018distribution}.
Second, these require restricted eigenvalue or similar compatibility
conditions on the design matrix $\bX$. These conditions only hold for sample sizes that 
are strictly larger than what is necessary for accurate estimation when $\bX$ is random.

A more recent line of research attempts to address these limitations by
characterizing the distribution of $\thetahat$ with Gaussian design matrices
\cite{bayati2011lasso,javanmard2018debiasing,thrampoulidis2015regularized,miolane2018distribution}. 
For example, \cite{bayati2011lasso} proved in the case of iid Gaussian designs
an exact characterization
of the distribution of $\thetahat$ which is simple enough to be described in words. 
Imagine, instead of observing
$\by$ according to the linear model \eqref{EqnLM}, we are given $\by^f = \thetastar + \tau \bg$
where $\bg\sim  \normal(\bzero, \Ind_{p})$, and $\tau>\sigma$ is the original noise level inflated
by the effect of undersampling. Then $\thetahat$ is approximately distributed as $\eta(\by^f;\zeta)$
where $\eta(x;\zeta) := (|x|-\lambda/\zeta)_+\sign(x)$ is the soft thresholding function (applied to vectors entrywise)
and $\zeta$ controls the threshold value. 
The values of $\tau,\zeta$ are determined by a system
of two nonlinear equations (see below). 
This analysis, as well as that in \cite{thrampoulidis2015regularized,miolane2018distribution}, 
assumes $n,p$ and the number of non-zero coefficients $s$ to be
large and of the same order. It further applies to any $\lambda$ scaling as  $c \sigma \sqrt{\log(p/s)}$.
In particular, unlike the Lasso results in \cite{bickel2009simultaneous,buhlmann2011statistics,negahban2012,bellec2018}, 
the constant $c$ here can be taken arbitrarily small, though non-vanishing asymptotically, 
which covers the typical values of the regularization selected by standard procedures 
such as cross-validation \cite{Chetverikov2016,miolane2018distribution}. 

Of course the case of i.i.d. Gaussian covariates is highly idealized and one can think of
two directions in which the results of 
\cite{bayati2011lasso,thrampoulidis2015regularized,miolane2018distribution} 
could be brought closer to reality:
\begin{enumerate}
\item Non-Gaussian but still independent and ---say--- sub-Gaussian covariates.
Both numerical simulations and universality arguments suggest that the same characterization 
that was proven for Gaussian covariates also applies to this case.
Rigorous universality results were proven in 
\cite{bayati2015universality,oymak2018universality,montanari2017universality} in closely 
related settings. 
Hence, while mathematically interesting, this generalization yields limited 
new statistical insight.
\item Gaussian but correlated designs. As we will see, in this case the asymptotic 
characterization is different and depends on the covariance $\mySigma=\E\{\bx_i\bx_i^\top\}$.
The covariance $\mySigma$ (or an estimate of $\mySigma$) plays a key role in statistically 
important tasks such as debiasing and hypothesis testing. This will be the focus of the present paper.
\end{enumerate}
By analogy with the uncorrelated designs, we expect our results for correlated
Gaussian designs to apply also to correlated non-Gaussian designs. A set of results proved
after a first appearance of this manuscript work supports this expectation \cite{hu2020universality,montanari2022universality,han2022universality}.

Throughout the paper,
 we assume that the covariates (each row of $\bX$) have distribution 
\begin{align}
 	\bx_i \sim \normal(0,\mySigma)
 \end{align} 
for some well-conditioned and known covariance matrix $\mySigma.$
As in the i.i.d. case, our results present two advantages with respect 
to fixed-design theory. First,
 they allow for any $\lambda$ of the order $c\sigma \sqrt{\log(p/s)}$, 
with $c$ an arbitrarily small (non-zero) constant. 
Second, they provide guarantees for sample sizes $n$ at which the restricted 
eigenvalue condition does not hold.

In fact, we provide guarantees for all sample sizes above the Gaussian dimension 
of the relevant descent cone. This critical sample size marks a sharp transition in the 
ability of $\ell_1$-based methods to achieve noiseless and stable sparse recovery 
in compressed sensing \cite{chandrasekaran2010,tropp2015convex}.
We will refer to this as the \emph{Donoho-Tanner} phase transition (although the original
work of \cite{donoho2005neighborliness,donoho2009counting} was limited to i.i.d. designs).
More details can be found in our Section~\ref{Sec:preliminary}.

In the case of correlated designs, \cite{javanmard2018debiasing} proved a similar
characterization in the regime $s \log(p)/n = o(1)$ assuming a bound on
$\|\mySigma^{-1}\be_j\|_1$.
The regime studied \cite{javanmard2018debiasing} is substantially simpler than the one studied
here. In particular, the characterization proved here simplifies in that regime, in that 
one can take $\tau = \sigma$ and $\zeta = 1$.

An important consequence of our theory is the asymptotic optimality of a hyperparameter 
tuning method based on the following
 degrees-of-freedom adjusted residuals
 \begin{align}
 \label{eq:tau-hat}
   \widehat{\tau}(\lambda)^2 := \frac{\|\by-\bX\thetahat\|_2^2}{n(1-\|\thetahat\|_0/n)^2}\, .
 \end{align}
 It was already observed in \cite{miolane2018distribution} that minimizing $\widehat{\tau}(\lambda)$ 
 over $\lambda$ provides
 a good selection procedure for the regularization parameter. Our results provide theoretical 
 support for this approach under general Gaussian designs.
 Recently (and after this paper was originally posted),
 this criterion has been generalized to a wider class of losses and penalties \cite{bellecShen2022}.
 
\paragraph*{Distribution theory for the debiased Lasso}

The debiased Lasso is a recently popularized approach for performing hypothesis testing 
and computing confidence
regions for low-dimensional projections of $\thetastar$. Most constructions take the form:
\begin{equation*}
	\bthetahatd = \thetahat + \frac1n\bM \bX^\top(\by - \bX \thetahat)\,,
\end{equation*}
for an appropriate and possibly data-dependent choice of the matrix $\bM$. Under appropriate
choices of $\bM$, low-dimensional projections of $\bthetahatd$ are approximately normal 
with mean $\thetastar$.

The first constructions for the debiased Lasso took $\bM$ to be suitable estimators 
of the precision matrix
$\mySigma^{-1}$ and proved approximate normality when $\|\thetastar\|_0 =:s = o(\sqrt{n} / \log p)$ 
\cite{zhang2014confidence,van2014asymptotically,javanmard2014hypothesis,javanmard2014confidence,javanmard2018debiasing}.
Later work considered the case of Gaussian covariates with known covariance, and set 
$\bM=\mySigma^{-1}$.
In this idealized setting, the sparsity condition  was relaxed  to 
$s = o(n/(\log p)^2)$ under an $\ell_1$-constraint on $\mySigma^{-1}\be_j$  
\cite{javanmard2018debiasing}, and to $s = o( n^{2/3} / \log(p/s)^{1/3})$ for general 
$\mySigma$ \cite{bellec2019biasing}.
\begin{figure}[h!]
\centerline{\includegraphics[width=.78\textwidth]{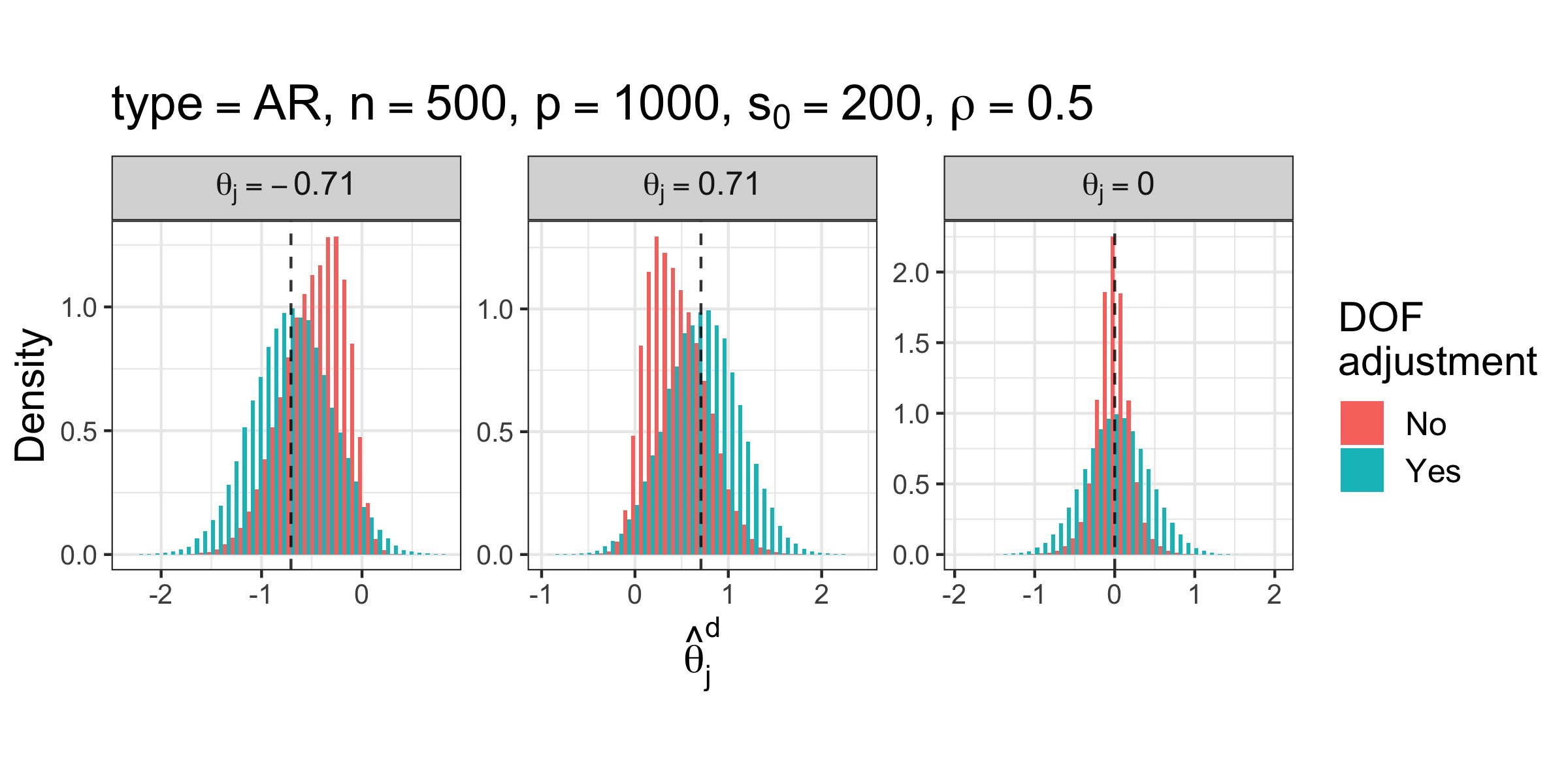}}
\centerline{\includegraphics[width=.8\textwidth]{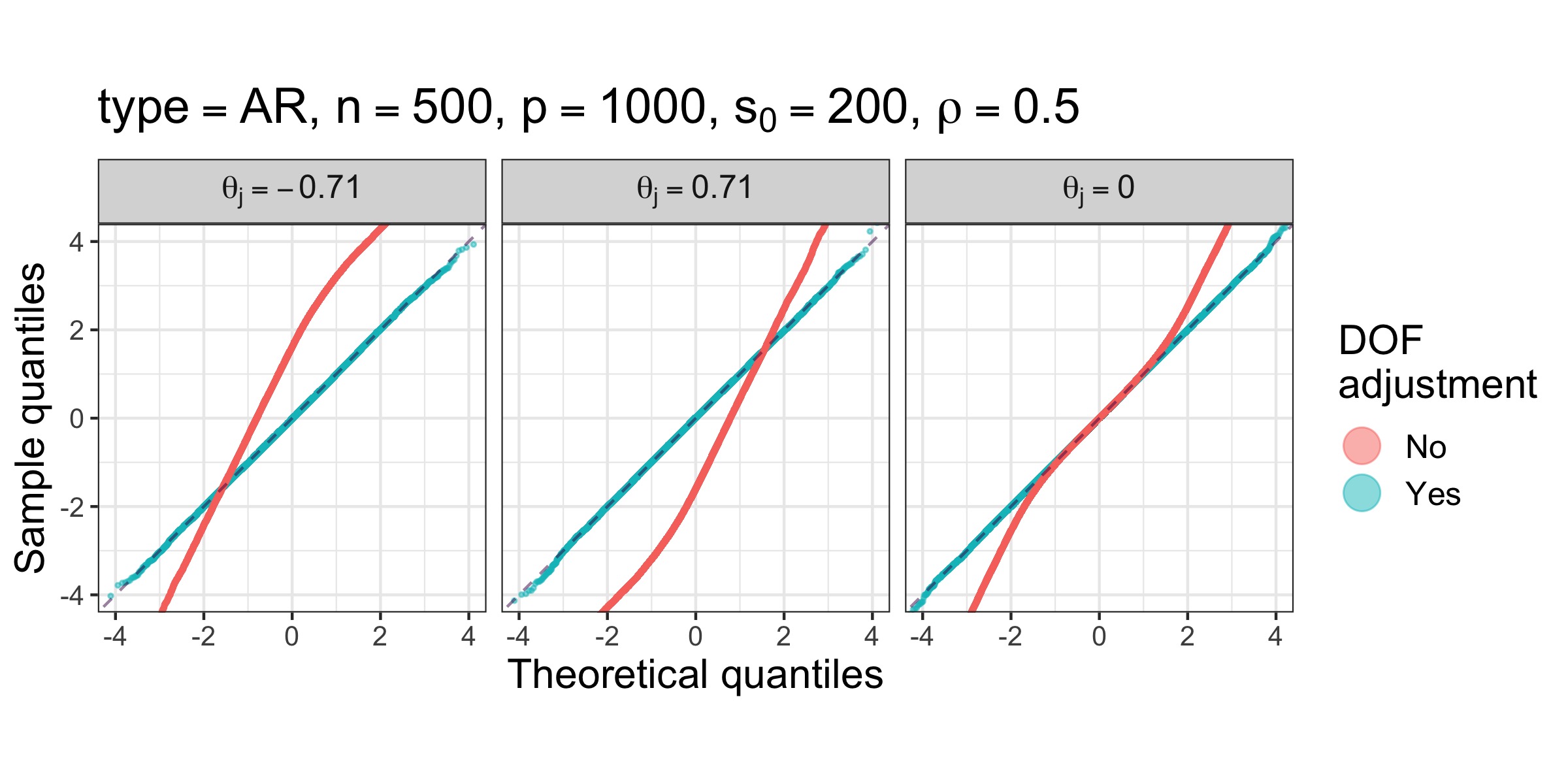}}
\caption{The debiased Lasso with and without degrees-of-freedom (DOF) adjustment.
  Here $p = 1000$, $n = 500$, $s = 200$, $\Sigma_{ij} = \rho^{|i-j|} = 0.5^{|i-j|}$, $\lambda = 4/\sqrt{n} = .18$, $\sigma = 1$. The coefficients vector $\thetastar$ contains $100$ entries $\theta^*_i=+.707$, and $100$ entries $\theta^*_i=-.707$. 
  The histogram plots the raw values of $\widehat \theta_j^{\mathrm{d}}$ without standardization, with the true value of $\theta_j^*$ drawn as the vertical dashed line.
  The qqplot is made with theoretical quantiles from the standard normal distribution.}
\label{FigQQplots_and_histograms}
\end{figure}
The latter conditions turn out to be tight for $\bM= \mySigma^{-1}$.

For larger values of $s$, it is necessary to adjust the previous construction for the 
degrees of freedom by setting\footnote{More precisely,
\cite{javanmard2014hypothesis,miolane2018distribution} showed that the degrees-of-freedom 
correction is needed for uncorrelated designs 
with $s = \Theta(n)$,
\cite{bellec2019biasing} showed that it is needed for correlated designs with 
$n \gg s \gg n^{2/3}/\log(p/s)^{1/3}$,
and \cite{Bellec2019SecondOP} studied it for correlated designs with $s = \Theta(n)$, but under stronger conditions on the sample size and regularization parameter than considered here.}\label{ftnt:dof}
 $\bM=\mySigma^{-1}/(1-\|\thetahat\|_0/n)$:
\begin{equation}\label{EqnDBlasso}
	\bthetahatd = \thetahat +\frac{1}{n-\|\thetahat\|_0} \mySigma^{-1} \bX^\top(\by - \bX \thetahat)\,.
\end{equation}
Figure \ref{FigQQplots_and_histograms} illustrates the difference
between the
debiased estimator with and without degrees-of-freedom correction. It is clear that debiasing
without degrees-of-freedom correction can lead to invalid inference.

% \yutingcomment{ \cite{bellec2019biasing}  is the updated version of \cite{Bellec2019SecondOP}}

Recently, Bellec and Zhang \cite{Bellec2019SecondOP,bellec2019biasing} 
established asymptotic normality and unbiasedness of the coordinates 
$\widehat\theta^{\mathrm{d}}_j$ of the debiased estimator of Eq.~\eqref{EqnDBlasso}.
As in the present work, they assumed  correlated 
Gaussian designs in the proportional regime $s\asymp n\asymp p$.
Our results on debiasing are not directly comparable with the ones of \cite{Bellec2019SecondOP}:
on the one hand, we assume weaker condition on the regularizations and the sample size;
on the other hand, we establish normality in a weaker sense. See Section \ref{sec:DB-lasso} for
further discussion.

Our results on the debiased Lasso do not imply that a fixed 
coordinate of $\bthetahatd$ is approximately unbiased and normally distributed.
Indeed, without additional assumptions, there can be a small subset of coordinates 
for which normality does not hold \cite{Bellec2019SecondOP}.
Instead, we present an alternative \emph{leave-one-out} method to construct confidence 
intervals for which we  prove asymptotic validity via a direct argument. 
An advantage of the leave-one out
method is that it produces p-values for single coordinates that are exact (not just 
asymptotically valid for large $n$, $p$).
Empirically, the leave-one-out intervals almost exactly agree with the  debiased intervals
in several settings. On the other hand, we demonstrate that ---for certain carefully designed 
$(\thetastar,\mySigma)$---
 the leave-one-out intervals can  be smaller than the debiased intervals.

\paragraph*{Notation} We generally use lowercase for scalars (e.g. $x,y,z,\dots$), boldface 
lowercase for vectors (e.g. $\bu,\bv,\bw,\dots$)
and boldface uppercase for matrices (e.g. $\bA, \bB, \bC,\dots$).
We denote the support of vector $\bx$ as $\supp(\bx) := \{i \mid \bx_i \neq 0\}$.
In addition, the $\ell_q$ norm of a vector $\bx\in\reals^n$ is $\|\bx\|_q^q\equiv\sum_{i=1}^n|x_i|^q$. 
For $r \geq 0$ and $q\in (0,\infty)$, we use $\Ball_q(\bv;r)$ to represent the corresponding 
$\ell_q$-ball of radius $r$ and center $\bv$, namely, 
\begin{equation}
	\Ball_q( \bv;r) := \left\{\bx \in \reals^p \bigm| \|\bx-\bv\|_q \leq r\right\}\;\; \text{for}\;\; q > 0,\;\; \text{and}\;\; \Ball_0( s ) := \left\{\mytheta \in \reals^p \bigm| \|\mytheta\|_0 \leq s \right\}\,.
      \end{equation}
If the center is omitted, it should be understood that the ball is centered at $\bm{0}$.
A function $\phi:\real^p \times \real^p \rightarrow \reals$ is $L$-Lipschitz if for every 
$\bm{x}, \bm{y} \in \real^p \times \real^p$, it satisfies 
$|\phi(\bm{x}) - \phi(\bm{y})| \leq L\ltwo{\bm{x}-\bm{y}}$.
The notation $\S_{\ge 0}^n$ is used to denote the set of $n\times n$ positive semidefinite matrices.
We reserve $n$ for the sample size, $p$ for the dimension of the unknown parameter 
$\mytheta^*$, and always define $\delta := n/p$.

\section{A glimpse of our results}
\label{Sec:preview}

Our main result establishes an approximate
%distributional characterization of the Lasso estimator derives from an asymptotic 
equivalence between the undersampled linear model of Eq.~\eqref{EqnLM} and a related statistical model:  
\begin{equation}
	\by^f = \mySigma^{1/2} \thetastar + \frac{\tau}{\sqrt{n}} \bg\,. \label{eq:FixedDef}
\end{equation}
%
%(This generalizes the case $\mySigma = \Ind_p$ which is 
%analyzed in \cite{miolane2018distribution}.)
Here $\bg \sim \normal(\bzero,\Ind_p)$ and $\tau \geq 0$.
We may take any square-root of the matrix $\mySigma$.
For simplicity, we always assume that we take a symmetric square-root.
The reader should have in mind a setting in which the singular values of 
$\mySigma$ and the noise parameter $\tau$ are of order 1.

We call Eq.~\eqref{eq:FixedDef} the \emph{fixed-design model} (hence the superscript $f$) 
and call model \eqref{EqnLM} the \emph{random-design model}.
The Lasso estimator in the fixed-design model can be written as 
\begin{gather}
\label{EqnSTH}
	\thetahat^f:=
	\eta(\by^f,\zeta) \defn \argmin_{\mytheta \in \reals^{\usedim}}
	\left\{\frac{\zeta}{2}\ltwo{\by^f - \mySigma^{1/2}\mytheta}^2 + \frac{\lambda}{\sqrt{n}} \lone{\mytheta} \right\}
	\,,
 \end{gather}
 with predictions given by $\yhat(\by^f,\zeta) \defn \mySigma^{1/2} \eta(\by^f,\zeta)$.
 We define the debiased Lasso in the fixed-design model as
 \begin{gather}
 \label{EqnDBf}
 	\thetahat^{f,\de}	
 		:= 
 		\thetahat^f + \mySigma^{-1/2}(\by^f - \mySigma^{1/2}\thetahat^f) 
 		=
 		\mySigma^{-1/2}\by^f
 		=
 		\thetastar + \frac{\tau}{\sqrt{n}} \mySigma^{-1/2} \bg.
 \end{gather}
 %
% The results of this paper take the form of certain distributional equivalences between the Lasso and debiased Lasso estimators in the random-design and fixed-design models.
 %This equivalence holds for particular choices of $\tau$ and $\zeta$,
 The approximate equivalence between the random design and fixed design models 
 holds for particular choices of $\tau$ and $\zeta$,
 which we denote $\tau^*$ and $\zeta^*$. 
 % and which are determined by a system of non-linear equations. 
 Such an equivalence is relatively straightforward in the low dimensional regime:
 in that case, it is sufficient to take $\tilde\by^f = n^{-1}\mySigma^{-1/2}\bX^{\top}\by$,
 and check that for $n\gg p$,  this is approximately distributed as $\by^f$ of 
 Eq.~\eqref{eq:FixedDef} with $\tau=\sigma$. 
 This equivalence was extended by \cite[Theorem 5.1]{javanmard2018debiasing} 
 to $n\gg s\log(p)/n$, assuming $\max_j \| \mySigma^{-1}\be_j \|_1 = O(1)$.
 As long as these conditions are met, we can keep $\tau= \sigma$ and $\zeta = 1$.
 
 Here we consider the more interesting case $s \log(p/s) / n = \Theta(1)$ without an
  $\ell_1$-restriction on the rows of $\mySigma^{-1}$.
 In this regime, the equivalence only holds if we properly select
  $\tau^* > \sigma$ and $\zeta^* < 1$.

To specify these choices of $\tau$ and $\zeta$, let the in-sample prediction risk 
and degrees-of-freedom of the Lasso estimator in the fixed-design model be 
\begin{subequations}
	\begin{align}
	\label{EqnE1}
		\risk(\tau^2,\zeta) & \defn 
		\E\left[\Big\|\yhat\Big(\mySigma^{1/2} \thetastar + \frac{\tau}{\sqrt{n}} \bg, \zeta\Big) - \mySigma^{1/2}\thetastar\Big\|_2^2\right]\,,\\
		\label{EqnE2}	
		\df(\tau^2,\zeta) &\defn 
		\frac{\sqrt{n}}{\tau} \E\left[\Big\<\yhat\Big(\mySigma^{1/2}\thetastar + \frac{\tau}{\sqrt{n}} \bg,\zeta\Big), \bg\Big\> \right] \\
               \nonumber &~= \E\left[ \Big\|\eta\Big(\mySigma^{1/2}\thetastar+\frac{\tau}{\sqrt{n}}\bg , \zeta\Big)\Big\|_0\right]\,,
	\end{align}	
\end{subequations}
where the expectation is taken over $\bg\sim\normal(0,\Ind_p)$.
Here, for notational simplicity, we leave the dependence of $\risk(\tau^2,\zeta)$ and $\df(\tau^2,\zeta)$ on $\thetastar$, $\mySigma$, $n$, $p$ and $\lambda$ implicit.
The notion of ``degrees-of-freedom'' is standard to quantify the model complexity of statistical procedures 
(see, e.g.~\cite{hastie2017generalized,efron2004least,efron1997improvements} and references therein),
and its equivalence to the expected sparsity of the Lasso estimate  holds, for example, by 
\cite[Theorem 1]{zou2007}.
The parameters $\taustar, \zetastar$ are chosen as solutions
to the system of equations
\begin{subequations}
	\begin{align}
	\label{EqnEqn1}
		\tau^2 &= \sigma^2 + \risk(\tau^2, \zeta)\,,\\
	\label{EqnEqn2}
		\zeta &= 1 - \frac{\df(\tau^2, \zeta)}{n}\,.
	\end{align}	
\end{subequations}
We refer to these equations as the \emph{fixed-point equations}.
As asserted in Section~\ref{sec:solution-property}, there exists a unique pair of solution to the above fixed-point equations under weak conditions.

\paragraph*{Role of fixed-point equations}
%As we have already stated, the results of this paper make certain distributional connections between the Lasso and debiased Lasso estimators in the random-design and fixed-design models respectively.
Before presenting our assumptions and results formally, 
it is useful to discuss the interpretation of $\tau^* $ and $\zeta^*$.
%we find it helpful to first provide a preview of our results,
%which further assists in interpreting the quantities appearing in the fixed-point equations. Concrete statements and details can be found in Section~\ref{sec:main}. 
In what follows, we take $\thetahat^f$ and $\thetahat^{f,\de}$ to be computed according to 
Eq.~\eqref{EqnDBf} in the fixed-design model with parameters $\tau=\tau^*$,
$\zeta=\zeta^*$ which solve the fixed-point equations \eqref{EqnEqn1} and \eqref{EqnEqn2}.

\begin{itemize}

	\item \textbf{Prediction and estimation error of the Lasso.} 
	We can interpret $\tau^{*2}$ as a theoretical prediction for the test error 
	$\E[(y_{\mathrm{test}} - \bx_{\mathrm{test}}^\top\thetahat )^2]$ on an independent test sample 
	$(\bx_{\mathrm{test}},y_{\mathrm{test}})$.
	Indeed, we obviously have 
	$\E[(y_{\mathrm{test}} - \bx_{\mathrm{test}}^\top\thetahat )^2] = 
	\sigma^2 + \| \thetahat - \thetastar \|_{\mySigma}^2 $.
	We will prove that the prediction risk $\| \thetahat - \thetastar \|_{\mySigma}^2$ concentrates on 
	the prediction risk of the fixed design model 
	$\risk(\tau^{*2},\zeta^*) = \E[\| \thetahat^f - \thetastar \|_{\mySigma}^2]$, cf. Eq.~\eqref{EqnE1}.
	Similarly, we will prove that $\| \thetahat - \thetastar \|_2^2$ concentrates on 
	$\E[\| \thetahat^f - \thetastar \|_2^2]$. 
	We conclude that
	$\E[(y_{\mathrm{test}} - \bx_{\mathrm{test}}^\top\thetahat )^2]$ concentrates on 
	$\tau^{*2}$ by Eq.~\eqref{EqnEqn1}.
	
	\item \textbf{Model size of the Lasso.}
	$\zeta^*$ is interpreted as (a theoretical prediction for) the
	 fraction of coordinates \emph{not} selected by the Lasso.
	Indeed, we will prove that the model size in the random design model $\| \thetahat \|_0$ concentrates around 
	$\df(\tau^{*2},\zeta^*)$, that is the expected model size in the fixed-design model,
	cf. Eq.~\eqref{EqnE2}.
	The interpretation follows by the second fixed point equation \eqref{EqnEqn2}.
%	We establish conditions under which $\zeta^*$ is bounded away from 0, 
%	% \yutingcomment{check back later on this}
%	which is essential in controlling the debiased Lasso with degrees-of-freedom adjustment (see Lemma~\ref{LemTauZetaBounds}).
	By Eq.~\eqref{EqnSTH},
	we can also interpret $\zeta^*$ as an inverse effective regularization parameter. 
	Thus, the larger the size of the selected model, the smaller the effective regularization.

	\item \textbf{False discovery proportion (FDP) of the debiased Lasso.}
	Consider the task of constructing confidence intervals for coordinates of $\thetastar$.
	For each $j \in [p]$, define the interval
	\begin{align}
	\label{EqnDebiasedCI}
		\mathsf{CI}^\mathrm{d}_j 
		&:= 
           \left[ \thetahatd_j - \Sigma_{j|-j}^{-1/2} \widehat \tau  z_{1-q/2} / \sqrt{n} ,~~ \thetahatd_j + \Sigma_{j|-j}^{-1/2}\widehat\tau \,
           z_{1-q/2} / \sqrt{n} \right]\,,
	\end{align}
	where $z_{1-q/2}$ is the $(1-q/2)$-quantile of the standard normal distribution, 
	$\widehat\tau$ is an empirical estimate of $\tau^*$ (defined formally in \eqref{eq:tau-hat}), and 
	\begin{equation*}
	% \label{eqn:conditional-variance}
		\Sigma_{j|-j} \defn \Sigma_{j,j} - \mySigma_{j,-j}(\mySigma_{-j,-j})^{-1} \mySigma_{-j,j}.
	\end{equation*} 
	We prove that the false-coverage proportion (FCP) concentrates around $q$, where 
	\begin{equation}
		\FCP := \frac1p \sum_{j = 1}^p \indic{\theta^*_j \not \in \mathsf{CI}^\mathrm{d}_j}
		=
		\frac1p \sum_{j=1}^p \mathbf{1}\{|\widehat\theta_j^{\mathrm{d}} - \theta^*_j| > \Sigma_{j|-j}^{-1/2} \widehat \tau z_{1-\alpha/2} / \sqrt{n} \}.
	\end{equation}
	In other words, confidence intervals based on the debiased Lasso achieve nominal false coverage. 
	Combining this with the fact that 
	$q = \E\Big[\frac1p \sum_{j=1}^p \mathbf{1}\{|\widehat\theta_j^{f,\mathrm{d}} - \theta^*_j| 
	\geq \Sigma_{j|-j}^{-1/2} \tau z_{1-\alpha/2} / \sqrt{n} \}\Big]$,
	we conclude the $\FCP$ in the random-design model concentrates on the expectation of 
	the analogous quantity in the fixed-design model.

	The above result provides an additional interpretation of the fixed point parameter $\tau^{*2}$ as the effective noise-level for the debiased Lasso estimates. 
	Note that in the low-dimensional limit which takes $p$ fixed, $n \rightarrow \infty$, the asymptotic standard error of the OLS estimate for $\theta_j^*$ is given by $\Sigma_{j|-j}^{-1/2} \sigma / \sqrt{n}$. 
	The first fixed-point equation states that we should inflate this standard error by replacing $\sigma^2$ with $\sigma^2 + \| \thetahat - \thetastar \|^2_{\mySigma}$, which concentrates around $\tau^{*2}$. 
	Of course, under a low-dimensional asymptotics, we expect $\| \thetahat - \thetastar \|_{\mySigma}^2 \stackrel{\mathrm{p}}\rightarrow 0$, recovering the low-dimensional theory.

\end{itemize}

Versions of these results and the corresponding interpretations of $\tau^*,\zeta^*$
have appeared elsewhere \cite{bayati2011lasso,bayati2013estimating,donoho2016high,thrampoulidis2018,miolane2018distribution,bellec2021outofsample}. 
The present paper is the first one establishing these results under correlated Gaussian designs
and optimal sample size requirements.

\section{Preliminaries} 
\label{Sec:preliminary}
This section summarizes several important concepts that shall be used throughout the paper
 and discusses the assumptions under which our main results are derived.

\paragraph*{Gaussian width and the Donoho-Tanner phase transition}

The success probability of $\ell_1$-norm based methods
changes abruptly at a critical sampling rate $\delta_{\mathrm{DT}}$ which depends on the sparsity 
of the signal and the geometry of the covariates. 
We will refer to this phenomenon as the Donoho-Tanner phase transition
\cite{donoho2005neighborliness,donoho2009counting}.
Below the transition (roughly speaking, for $n/p < \delta_{\mathrm{DT}}$),
$\ell_1$-penalized methods fail to achieve exact noiseless recovery, stable noisy recovery, 
bounded minimax noisy recovery over sparse balls, and full power for variable selection 
\cite{donohoTanner2009,donohoMalekiMontanari2011,chandrasekaran2010,tropp2015convex,suCandes2017,wang2020}.
Above the transition (for $n/p > \delta_{\mathrm{DT}}$), $\ell_1$-penalized methods are
 able to succeed according to these metrics.

This paper uses Gaussian comparison techniques \cite{chandrasekaran2010,miolane2018distribution},
and our results hold for all sampling rates  $n/p$ exceeding $\delta_{\mathrm{DT}}$,
where $\delta_{\mathrm{DT}}$ is defined below in terms of a certain Gaussian width. 
We anticipate that our definition of this threshold is (for general $\mySigma$)
slightly different from the standard one in the literature.
Importantly, the restricted eigenvalue conditions which are often used to 
derive estimation error bounds  on the Lasso need not occur near the Donoho-Tanner phase transition.
Hence, our results could not be established using those conditions.

Given a vector $\bx\in\{+1,-1,0\}^p$, define the closed convex cone $\cuK(\bx,\mySigma)$
and the homogeneous convex function $F(\,\cdot\,;\bx,\mySigma):\reals^p\to\reals$ as follows:
\begin{align}
\cuK(\bx,\mySigma) &:=\big\{ \bv\in\reals^p:\; F(\bv;\bx,\mySigma)\le 0\big\}\, ,\\
  F(\bv;\bx,\mySigma) &:=\<\bx,\mySigma^{-1/2}\bv\>+\big\|(\mySigma^{-1/2}\bv)_{S^c}\big\|_1\, 
  \qquad 
  \text{for } S \defn \supp(\bx).
\label{EqnConstraintDef}
\end{align}
(The reader should think of
$\bv$ as $\mySigma^{-1/2}(\mytheta - \thetastar)$, where $\mytheta$ is the argument appearing 
in the Lasso optimization.)

Consider $\thetastar\in\reals^p$ with $\bx = \sign(\thetastar)$, 
i.e., $x_j = 1$ for $\theta^*_j > 0$, $x_j = -1$ for $\theta^*_j < 0$, and $x_j = 0$ 
for $\theta^*_j = 0$.
Then $\cuK(\bx,\mySigma)$ is the descent 
cone of the function $\bv\mapsto \|\thetastar+\mySigma^{-1/2}\bv\|_{1}$ at $\bv=\bzero$.
Namely (denoting by ${\rm cl}(A)$ the closure of set $A$)
\begin{align}
\label{eqn:def-cuk}
\cuK(\bx,\mySigma) &:={\rm cl}\Big(\big\{ \bv\in\reals^p:\; 
\exists\eps>0\;\mbox{s.t.}\; \|\thetastar+\eps \Sigma^{-1/2}\bv\|_{1}\le \|\thetastar\|_1\big\}\Big)\,.
\end{align}
The connection between this cone and the Lasso is most easily seen in the case of 
minimum $\ell_1$-norm interpolation (basis pursuit), corresponding to the $\lambda\to 0$ limit of 
the Lasso \eqref{EqnOrgRisk}:
\begin{align}
	\thetahat_{\sBP} 
		&\defn 
		\argmin_{\mytheta \in \reals^{\usedim}}  \Big\{ \|\mytheta\|_1\;\;\;\;\mbox{s.t.}\;\;\;
		\bX\mytheta=\by\Big\}\, .
\end{align}
In the noiseless case $\sigma=0$ (i.e. $\by=\bX\thetastar$), 
$\thetahat_{\sBP}  = \thetastar$ if and only if ${\rm null}(\bG)\cap \cuK(\bx,\mySigma) = \{\bzero\}$
where $\bG=\bX\mySigma^{-1/2}$ is a Gaussian matrix with i.i.d. entries \cite{amelunxen2014living}.
As proven in \cite{amelunxen2014living}, the 
probability of the event ${\rm null}(\bG)\cap \cuK(\bx,\mySigma) = \{\bzero\}$ 
transitions rapidly from $0$ to $1$ when the sampling 
ratio $n/p$ crosses $\cuG_{d}(\bx,\mySigma)^2$. Specifically, \cite[Theorem II]{amelunxen2014living} ensures that 
\begin{align}
	\text{if }\frac{n}{p} \leq \cuG_{d}(\bx,\mySigma)^2 - \Delta, \qquad
	&\mathbb{P}(\thetahat_{\sBP}  = \thetastar) \leq 4 \exp(-p\Delta^2/8);\\
	\text{if }\frac{n-1}{p} \geq \cuG_{d}(\bx,\mySigma)^2 + \Delta, \qquad
	&\mathbb{P}(\thetahat_{\sBP}  = \thetastar) \geq 1 - 4 \exp(-p\Delta^2/8).
\end{align}
Here $\cuG_{d}(\bx,\mySigma)$ is the Gaussian width of $\cuK(\bx,\mySigma)$
defined as follows \cite{geer2000empirical,chandrasekaran2010,tropp2015convex}:
\begin{align}
\label{Eqn:std-GW}
	\cuG_{d}(\bx,\mySigma)
	= 
	\frac1{\sqrt{p}} \E\Big[ 
		\max_{ 
			\substack{	\bv \in \cuK(\bx,\mySigma) \\
						\|\bv\|_2^2 \leq 1
						} 
			} 
			\langle \bv , \bg \rangle 
		\Big]\,.
\end{align}

We next introduce the modified width that is relevant for our results.
Consider the probability space $(\reals^{p},\cB,\gamma_p)$ with $\cB$ being the Borel $\sigma$-algebra 
and $\gamma_p$ the standard Gaussian measure in $p$ dimensions.
We denote by $L^2 := L^2(\reals^p;\reals^p)$ the space of functions $\f:\reals^p\to \reals^p$
that are square integrable in $(\reals^{p},\cB,\gamma_p)$. 
This space is equipped with the scalar product
\begin{align*}
  \<\f_1,\f_2\>_{L^2} = \E [\<\f_1(\bg),\f_2(\bg)\>]  = \int \<\f_1(\bg),\f_2(\bg)\>\, \gamma_p(\de\bg)\, ,
  \end{align*}
The standard notion of Gaussian width defined in
Eq.~\eqref{Eqn:std-GW} can be rewritten as
\begin{align}\label{Eqn:std-GW-0}
	\cuG_{d}(\bx,\mySigma) 
	:= 
	\sup_{ \bv\in L^2}
        \Big\{\frac1{\sqrt{p}} \langle \bv,\bg \rangle_{L^2}:\;\;\; \mprob(\|\bv\|_2\le 1) = 1\, ,\;\;\; \mprob \big( F(\bv;\bx,\mySigma)\le 0\big)=1\Big\}\,,
\end{align}
where $\bg$ denotes the identity function on $L^2$. Let us emphasize that the supremum is 
taken over functions $\bv:\reals^p\to\reals^p$, $\bg\mapsto\bv(\bg)$.

Instead of  \eqref{Eqn:std-GW-0}, we will make use of the following relaxed version of 
Gaussian width:
\begin{align}\label{EqnGaussianWidth}
	\cuG(\bx,\mySigma) 
	:= 
	\sup_{ \bv\in L^2}
        \Big\{\frac1{\sqrt{p}} \langle \bv,\bg \rangle_{L^2}:\;\;\; \|\bv\|_{L^2}\le 1\, ,\;\;\; \E [F(\bv;\bx,\mySigma)] \le 0\Big\}\,,
\end{align}
In words, $\cuG(\bx,\mySigma)$ is the maximal correlation of a random direction with a 
standard Gaussian vector $\bg$ subject to $F(\bw;\bx,\mySigma)$ being non-positive 
\emph{on average}. 

\paragraph*{Properties of the Gaussian width}

In the case  $\mySigma = \Ind_p$, 
 $\cuG(\bx,\Ind_p)$ depends on $\bx$ only through $\varepsilon := \| \bx \|_0 / p$.
Denote
\begin{align}
\label{eqn:omegastar}
	\text{$\omegastar(\varepsilon) := \cuG_{d}(\bx,\Ind_p)$ for any $\bx$ with $\|\bx\|_0/p = \varepsilon$.}\,
\end{align}
Indeed $\omegastar(\varepsilon)$ can be computed explicitly, and is given in parametric form by
\begin{align}
 \label{eqn:def-omegastar}\omegastar(\varepsilon)^2 & = \eps+2(1-\eps)\Phi(-\alpha)\, ,\\
  \text{where $\alpha$ satisfies} \qquad
 \notag \eps & = \frac{2[\varphi(\alpha)-\alpha\Phi(-\alpha)]}{\alpha+2[\varphi(\alpha)-\alpha\Phi(-\alpha)]}\,.
\end{align}
Here $\varphi(x) = e^{-x^2/2}/\sqrt{2\pi}$ is the standard Gaussian density, and $\Phi(x) =\int_{-\infty}^x\varphi(t) \de t$ is the
Gaussian cumulative distribution function.
One can show that $\omegastar(\varepsilon)$ is increasing and continuous in $\varepsilon$,
goes to 1 as $\varepsilon \rightarrow 1$,
and satisfies
\begin{equation}
	\omega^*(\varepsilon)
		=
		(1 + o(\eps))\sqrt{2 \varepsilon \log(1 / \varepsilon)} .
\end{equation} 
Thus, $n/p \geq \cuG(\sign(\thetastar),\Ind_p)^2$ is equivalent to $2(1+o(s/p))s\log(p/s)/n \leq 1$.

% \vspace{2cm}

% \yutingcomment{check back on this.}

For general Gaussian designs $\mySigma$, 
the critical sampling rate depends not only on the sparsity of $\thetastar$ but also on the 
location and sign of its active coordinates. However, the value of $\cuG(\bx,\mySigma)$
changes at most by a factor equal to the condition number of $\mySigma$, as stated in the next lemma.
\begin{lems}
\label{lem:width-under-corr}
	Assume that $\mySigma$ has condition number upper bounded by $\kappacond$.
	Then for any $\bx \in \{-1,0,1\}^p$,
	\begin{equation}
		\label{EqWidthSigmaToIBound}
		\kappacond^{-1/2} \cdot \omegastar( \|\bx\|_0 / p )\le 
		\cuG(\bx,\mySigma) \le \kappacond^{1/2} \cdot \omegastar( \|\bx\|_0 / p )\,.
	\end{equation}
	In particular,
	if $ 2(1 + o(s/p))s\log(p/s) /n \leq \kappacond^{-1}$, then $n/p \geq \cuG(\bx,\mySigma)^2$.
\end{lems}
\noindent We prove Lemma \ref{lem:width-under-corr} in Appendix \ref{sec:GW-prop}.

% Although the statement $2(1+o(s/p))s\log(p/s)/n \leq 1$ is more easily interpretable than the statement $\cuG_{d}(\sign(\thetastar),\Ind_p)^2 \leq \delta$,
% the latter more precisely identifies the sampling rates at which exact and stable recovery becomes possible. 
% Indeed, when $s/p = \Theta(1)$,
% the threshold $\cuG_{d}(\sign(\thetastar),\Ind_p)^2 \leq \delta$ exactly identifies the threshold $\delta_{\mathrm{DT}}$,
% whereas $2(1+o(s/p))s\log(p/s)/n \leq 1$ specifies only that this threshold is $\Theta(1)$.
% Thus, in order to precisely identify the sampling rate at which our results begin to apply,
% our structural assumptions will be given in terms of a notion of Gaussian width, modified to address the current setting. 
% For correlated designs,
% this will imply that a sufficient condition is $2(1 + o(s/p)) s \log(p/s)/n \leq c$ for condition-number-dependent constant $c$.

%The Donoho-Tanner phase transition as located in \cite{chandrasekaran2010,tropp2015convex} 
%occurs at the critical sampling rate $\cuG_{d}(\bx,\mySigma)^2$.
The definitions \eqref{EqnGaussianWidth} and \eqref{Eqn:std-GW-0} immediately 
imply $\cuG_{d}(\bx,\mySigma)\le\cuG(\bx,\mySigma)$.
The next lemma establishes that the two definitions of
Gaussian width differ by a factor that is often negligible.
\begin{proposition}
\label{prop:std-width}
	For $c'$ depending only on $\kappacond$, we have 
	\begin{align}
	\cuG(\bx,\mySigma) - c' \min\Big(\frac{\sqrt{p}}{s} ; \, \sqrt{\frac{s}{p} \log(p/s) }
	\Big)\le 
	\cuG_{d}(\bx,\mySigma) \le \cuG(\bx,\mySigma) \, ,
	\end{align}
	 where $s = \| \bx \|_0$.
\end{proposition}
\noindent 
We prove Proposition \ref{prop:std-width} in Section \ref{SecPfandDetails}.

%It shows that the gap in the critical sampling rate predicted by the functional Gaussian width and standard Gaussian width vanishes when $p \rightarrow \infty$. 
%Indeed, if $s = o(p)$, then the second term in the minimum vanishes. 
For designs with bounded condition number, $\cuG(\bx,\mySigma)^2\asymp 
(s/p)\log(p/s)$, cf. Lemma \ref{lem:width-under-corr}.
Comparing with the lower bound in Proposition \ref{prop:std-width}, we obtain that 
the difference between  $\cuG_{d}(\bx,\mySigma)$  and $\cuG(\bx,\mySigma)$ is negligible
provided $s\gg p^{2/3}/(\log p)^{1/3}$.

For sub-linear sparsity $s=o(p)$, we do not expect the bound of Proposition \ref{prop:std-width} to be tight. 
Because the results in this paper provide non-trivial control of the Lasso and debiased Lasso 
estimates for sampling rates $n/p$ of order 1 (see parameter
 $\Deltamin$ in Assumption \ref{assump:1}(d) below),
we do not pursue a more careful comparison of the standard and functional Gaussian widths 
for sublinear sparsities here. 
Indeed, under sub-linear sparsity, any sampling rate of order 1 is well above the Donoho-Tanner
 phase transition.

\paragraph*{Assumptions}

We are ready to formally state the assumptions which will hold throughout the paper.
The distribution of the random design $\bX$,  response vector $\by$, and Lasso estimate $\thetahat$ is determined
by the tuple $(\thetastar,\mySigma,\sigma,\lambda)$, the number of samples $n$, and the dimensionality $p$.
Our results hold uniformly over choices of $(\thetastar,\mySigma,\sigma,\lambda)$ and sampling rates $n/p$ that satisfy the following conditions: 
\begin{enumerate}[label={\rm(A\arabic*)}]
	\item \label{assump:1} There exist $0 < \lambdamin \leq \lambdamax < \infty$, $0 < \kappamin \leq \kappamax < \infty$, and $0 < \sigmamin \leq \sigmamax < \infty$, $M < \infty$, $\Deltamin \in (0,1)$
	such that
	\begin{enumerate}[label=(\alph*)]
		\item % a
		The Lasso regularization parameter $\lambda$ is bounded $\lambdamin \leq \lambda \leq \lambdamax$.
		\item % b
		The singular values $\kappa_j(\mySigma)$ of the population covariance $\mySigma$ are bounded $\kappamin \leq \kappa_j(\mySigma) \leq \kappamax$ for all $j$.
		We define $\kappacond := \kappamax/\kappamin \geq 1.$
		\item % c
		The noise variance $\sigma^2$ is bounded $\sigmamin^2 \leq \sigma^2 \leq \sigmamax^2$.
		\item % d
		There exists $\barthetastar \in \reals^p$ such that $\| \thetastar - \barthetastar \|_1/p \leq M / \sqrt{n} $ and 
		\begin{align*}
		\frac{n}{p}  \geq \cuG(\sign(\barthetastar),\mySigma)^2 + \Deltamin\, .
		\end{align*}
	\end{enumerate}
\end{enumerate}
We denote the collections of constants appearing in assumptions  \ref{assump:1} by 
\begin{align}
\label{eqn:Pmodel-parameters}
	\cuPmodel 
		\defn 
		(\lambdamin,\lambdamax,\kappamin,\kappamax,\sigmamin,\sigmamax,\Deltamin,M)\,.
\end{align}
The choice of the constants $\cuPmodel$ determines via Assumption \ref{assump:1} the 
space of parameters $(\thetastar,\mySigma,\sigma,\lambda)$ and sampling rates $n/p$
(the uniformity class) 
within which the results stated below apply. With a slight abuse of language, we will 
occasionally use $\cuPmodel$ to refer to the uniformity class as well.

Assumption \ref{assump:1}(d) can be viewed as an approximate sparsity condition: 
$\thetastar$ is approximated in $\ell_1$-norm by a vector $\barthetastar$ whose sparsity places 
it above the Donoho-Tanner phase transition.
As established in the next proposition,
Assumption \ref{assump:1}(d) is implied by existing popular notions of approximate sparsity 
which appear elsewhere in the Lasso literature. 
\begin{proposition}
	\label{ClaimBallToApproxSparse}
	Assumption \ref{assump:1}(d) (with the specified choice of $M$) is implied by any of the following.
	\begin{enumerate}[label=(\alph*)]

		\item % a
		If $\|\thetastar\|_0 \leq s$, 
		then Assumption \ref{assump:1}(d) is satisfied with $M = 0$ if 
		\begin{equation}
		\label{eq:cond-s-bound}
			\kappacond^{1/2}\omega^*(s/p) \leq 1 - \Deltamin.
		\end{equation}
		In particular, it suffices that
		\begin{equation}
		\label{eq:cond-approx-s-bound}
			2 \kappacond (1 + o(s/p))s \log(p/s) / n \leq (1 + \Deltamin)^{-1}.
		\end{equation}

		\item % b
		If $\thetastar \in \Ball_q(\nu)$ for $q,\nu > 0$, 
		then Assumption \ref{assump:1}(d) is satisfied by taking $M = \sqrt{n}\nu(1 - s/p)/p^{1/q}$ for any $s$ satisfying Eq.~\eqref{eq:cond-s-bound} or Eq.~\eqref{eq:cond-approx-s-bound}.

		\item % c
		If $ \sum_{j=1}^p \min(1, \sqrt{n} |\theta^*_j|/a_0) \leq s$ for a certain $a_0$, 
		then Assumption \ref{assump:1}(d) is satisfied with $M = a_0 s / p$ provided 
		Eq.~\eqref{eq:cond-s-bound} or Eq.~\eqref{eq:cond-approx-s-bound} is satisfied.
	\end{enumerate}	
\end{proposition}
\noindent Proposition \ref{ClaimBallToApproxSparse} follows from Lemma \ref{lem:width-under-corr}. 
Its proof is given in Appendix~\ref{Sec:PFclaimBall}.

In words, 
Assumption \ref{assump:1}(d) allows $\thetastar$ to be unbounded on a certain signed support, 
and requires that it be small in $\ell_1$-norm on its remaining coordinates. 
Here ``small'' means $O(1/\sqrt{n})$ per coordinate on average, with leading constant given by $M$.
The location and sign of the coordinates on which $\thetastar$ can be unbounded is determined 
by the Gaussian width $\cuG(\mySigma,\bx)$ of the corresponding vector $\bx$.
Assumption \ref{assump:1}(d) permits that the number of coordinates in which $\thetastar$ is 
unbounded is proportional to $p$, 
but does not allow for arbitrarily large proportionality constant.
For example, as is clear from Proposition \ref{ClaimBallToApproxSparse},
we require at least that $s \leq n$, and in fact will require something stronger than this. 

Proposition \ref{ClaimBallToApproxSparse} uses Lemma \ref{lem:width-under-corr} to
bound $\cuG(\mySigma,\bx)$ with a suitable $\bx=\sign(\barthetastar)$.
Since Lemma \ref{lem:width-under-corr} is loose in general, the sufficient notions of 
approximate sparsity in Proposition \ref{ClaimBallToApproxSparse} are not sharp
and do not identify
the whole domain of validity of our results.
In contrast, Assumption \ref{assump:1}(d) will imply that our results hold down to 
the Donoho-Tanner phase transition for a good $\ell_1$-approximation of $\thetastar$.

%%%%%%%%%%%%%%%%%%%%%%%%%%%%%%%%%%%%%%%%%%%%%%%%%%%%%%%%%%%%%%%%%%%%%%

\section{Main results}
\label{sec:main}

We now turn to the statement of our main results and a discussion of some of their consequences. 
The proof details are deferred to the appendix. 

\subsection{Control of the fixed-point parameters}
\label{sec:solution-property}

Each of our results involves a comparison of the Lasso or debiased Lasso estimators in the random- and fixed-design models.
The comparison will be valid provided we choose $\tau,\zeta$ to be the solution to the fixed-point equations \eqref{EqnEqn1} and \eqref{EqnEqn2}.
This solution we call $\tau^*,\zeta^*$.
The next lemma establishes that the solution is unique, and satisfies uniform
bounds under Assumption \ref{assump:1}.
\begin{lems}
\label{LemTauZetaBounds}
	If $\mySigma$ is invertible and $\sigma^2 > 0$,
	then Eqs.~\eqref{EqnEqn1} and \eqref{EqnEqn2} have a unique solution $\tau^*,\zeta^*$. 
	Under Assumption \ref{assump:1}, 
	there exists $\taumax < \infty$ and $\zetamin > 0$ depending only on 
	$\cuPmodel$ and $\delta$ such that 
	$\sigma^2 \leq \tau^{*2} \leq \taumax^2$ and $\zetamin \leq \zeta^* \leq 1$.
\end{lems}
\noindent We prove Lemma \ref{LemTauZetaBounds} in Appendix \ref{SecPfandDetails}.
An important consequence of Lemma \ref{LemTauZetaBounds} is that, due to the fixed-point equations \eqref{EqnEqn1} and \eqref{EqnEqn2},
the quantity $\risk(\tau^{*2},\zeta^*)$ is bounded above by $\taumax^2 - \sigma^2$ and the quantity $\df(\tau^{*2},\zeta^*)/n$ is bounded away from 1 by $1 - \zetamin$. 
As we will see (and as described in Section \ref{Sec:preliminary}), $\risk(\tau^{*2},\zeta^*)$ and $\df(\tau^{*2},\zeta^*)$ are good approximations of the prediction risk $\| \thetahat - \thetastar \|^2_{\mySigma}$ and the degrees-of-freedom $\| \thetahat \|_0$ of the Lasso estimator in the random-design model \eqref{EqnOrgRisk}.
Thus, Lemma \ref{LemTauZetaBounds}, in addition to being a technical tool which shall be used repeatedly in our proofs,
has substantive consequences on the behavior of the Lasso: 
under an arbitrarily small separation from the Donoho-Tanner phase transition, 
it gives non-trivial upper bounds on the Lasso prediction error and model size.

\begin{remark}
\label{rmk:fixed-pt-bounds}
	The challenge in proving Lemma \ref{LemTauZetaBounds} lies in the fact that 
	$\tau^*,\zeta^*$ are implicitly defined as the solutions to the fixed-point equations
	 \eqref{EqnEqn1} and \eqref{EqnEqn2}. 
While in the case of iid Gaussian designs, one can exploit the explicit analytic formulas for
 $\risk(\tau^2,\zeta)$ and $\df(\tau^2,\zeta)$ as in \cite{miolane2018distribution}, 
	no such formulas are available under correlated designs. 
	Thus, we resort to a novel argument based on viewing Eqs.~\eqref{EqnEqn1} and \eqref{EqnEqn2} 
	as KKT conditions for an infinite-dimesional optimization problem defined in Section
	 \ref{sec:proof-ingredients} (see also Section \ref{SecPfandDetails}).
	The Gaussian width plays a central and natural role in the analysis of this optimization problem.
	Restricted eigenvalues or similar ideas do not yield a tight analysis 
	of this optimization problem.
\end{remark}

\noindent For the remainder of the document, we always assume $\thetahat^f$ and 
$\thetahat^{f,\mathrm{d}}$ are computed with parameters $\tau^*,\zeta^*$.

\subsection{Control of the Lasso estimate}
\label{sec:LassoEstimate}

Our first result states that the random-design Lasso behaves like the fixed-design Lasso from the point of view of Lipschitz test functions.
 The proof of this result is deferred to Section \ref{SecPfThmLassoL2}. 
\begin{theos}
\label{ThmControlLassoEst}
	Assume \ref{assump:1} holds.
	Then there exist constants $C,c,c' > 0$ depending only on $\cuPmodel$ and $\delta$ such that the following holds: if $n \geq \sqrt{2}/\Deltamin$, 
	then for any $1$-Lipschitz function $\phi:\real^p \rightarrow \reals$
	we have for all $\epsilon < c'$
	\begin{align}
		\mprob \left(\exists \lambda \in [\lambdamin, \lambdamax],~
		\Big| 
		\phi\big(\thetahat \big) 
		- \E\Big[ \phi\big(\thetahat^f\big)\Big]\Big|
		 > \epsilon\right)
		\leq 
		\frac{C}{\epsilon^4} e^{-cp\epsilon^4}\,.
	\end{align}
	Here $\thetahat^f$ is the fixed-design Lasso with $\tau^*,\zeta^*$ solving Eqs.~\eqref{EqnEqn1} and \eqref{EqnEqn2}.
\end{theos}
%
% \noindent We provide the proof of Theorem \ref{ThmControlLassoEst} in .
\noindent The proof of this theorem is presented in Section~\ref{SecUniform}.

Theorem \ref{ThmControlLassoEst} has an obvious corollary which we spell out for future reference.
For any 
fixed $\lambda \in [\lambdamin, \lambdamax]$:
	\begin{align}
		\mprob \left(
		\Big| 
		\phi\big(\thetahat \big) 
		- \E\Big[ \phi\big(\thetahat^f\big)\Big]\Big|
		 > \epsilon\right)
		\leq 
		\frac{C}{\epsilon^4} e^{-cp\epsilon^4}\,.\label{eq:FixedLambda}
	\end{align}
Namely, any Lipschitz function of the Lasso estimate concentrates around its expectation in the fixed-design model with high probability --- provided that the sampling rate exceeds the Donoho-Tanner phase transition for a good $\ell_1$ approximation of $\thetastar$ and $p$ is large. 
In particular, 
this concentration holds true even in the case where the sparsity $s$ and dimension $p$ are proportional to $n$, although the proportionality constants cannot be arbitrary. 

We make note that since $\thetastar$ is deterministic, $\phi$ may depend implicitly on $\thetastar$. 
In particular, Theorem \ref{ThmControlLassoEst} applies, for example, to the estimation error and prediction error by taking $\phi(\mytheta) = \| \mytheta - \thetastar \|_2$ and $\phi(\mytheta) = \| \mytheta - \thetastar \|_{\mySigma}$, respectively.
(In the latter case, the constants must be adjusted to account for the fact that $\mytheta \mapsto \| \mytheta - \thetastar \|_{\mySigma}$ does not have Lipschitz constant equal to 1. The adjustment is by at most constant factors because the Lipschitz constant is bounded under  \ref{assump:1}.)
Thus, the $\ell_2$-estimation error and the prediction error concentrate on their expectations in the fixed-design models.
By Eq.~\eqref{EqnEqn1},
this implies that the prediction error $\| \thetahat - \thetastar \|_{\mySigma}^2$ concentrates on $\risk(\tau^*,\zeta^*) = \tau^{*2} - \sigma^2$.

\paragraph*{Comparison with earlier results} 
It is worth comparing this result to the existing fixed-design results for the Lasso 
(e.g.~\cite{bickel2009simultaneous,buhlmann2011statistics,negahban2012,bellec2018}).
To be definite, we consider $\ell_q$-estimation error for $1 \leq q \leq 2$. 
The optimal fixed-design results establish the existence of
constants $c,C>0$ such that 
\begin{align}
\label{eq:ellq-standard}
\lambda \geq c \sqrt{\log(2ep/s)}
\;\;\Rightarrow\;\;
\| \thetahat - \thetastar \|_q \leq C \frac{s^{1/q} \lambda}{\RE^2\sqrt{n}}\, ,
\end{align}
where $\RE$ is an appropriate restricted eigenvalue of $\bX$ (see \cite{bellec2018} 
for precise statements), and $C$ may depend on $q$.

Consider the proportional sparsity regime $s =\Omega(p)$, which is our focus in the present paper.
 We make the following comparisons:

 \textbf{Regularization parameter.} When $s$ is proportional to $p$,  
$c \sqrt{\log(2ep/s)}$ is of order one, 
	so that $\lambda \geq c \sqrt{\log(2ep/s)}$ implies Assumption \ref{assump:1}(d).
	On the other hand, Assumption \ref{assump:1}(d) permits smaller regularization parameters than are 
	permitted by \cite{bellec2018}, since $\lambdamin$ in Assumption \ref{assump:1}(d) can be 
	arbitrarily small (but nonvanishing as $n,p,s\rightarrow \infty$), while  $c$ in 
	Eq.~\eqref{eq:ellq-standard} and  \cite{bellec2018} is a fixed numerical constant
	 bounded away from 0. 
	The case when $\lambda$ is taken to be exactly zero is considered in recent works 
	(see e.g.~\cite{li2021minimum}). 

\textbf{Estimation error.} Because $\mytheta \mapsto 
	\| \mytheta - \thetastar \|_q / p^{1/q - 1/2}$ is $1$-Lipschitz, we can apply
	Theorem \ref{ThmControlLassoEst}. Further using the bound on $\tau^*$ from Lemma \ref{LemTauZetaBounds},
	one can show  that $\E[\| \thetahat^f - \thetastar \|_q] = O(p^{1/q}/n^{1/2})$ under 
	Assumption \ref{assump:1}, where $O$ hides constants depending on $\cuPmodel$
	 (see \eqref{eqn:Pmodel-parameters}). Summarizing, we obtain, with probability at least $1-p^{-A}$
	 for any constant $A$,
	\begin{align}
	\| \thetahat - \thetastar \|_q &= \E[\| \thetahat^f - \thetastar \|_q]+
	 O(p^{1/q-3/4}\log(p))\, ,\label{eq:ellq-concentration}\\
	 \E[\| \thetahat^f - \thetastar \|_q] &= O(p^{1/q}/n^{1/2})\, .
	\end{align}
	In the present setting
	$p^{1/q}/n^{1/2}$ is of the same order as $C s^{1/q}\lambda/(\RE^2\sqrt{n})$, 
	so that the estimate is consistent with the results of \cite{bellec2018}.
	If in addition $n = \tilde o(p^{3/2})$, then the error term in Eq.~\eqref{eq:ellq-concentration}
	is much  smaller than $\E[\| \thetahat^f - \thetastar \|_q]$.
	In other words, we obtain a more precise concentration around a deterministic 
	theoretical prediction, which we characterize.
	
	 \textbf{Restricted eigenvalues and sampling rates.} 
	The previous bullet point describes a scenario in which the restricted eigenvalue 
	$\mathrm{RE}$ is of order 1 (and, in particular, is bounded away from 0).
	In the random-design setting, 
	this implicitly corresponds to an assumption on the number of samples.
	In Section \ref{SecCompare},
	we show that restricted eigenvalues can be 0 for $n/p\ge (1+\eps)\cuG(\mySigma,\bx)$
	with $\eps$ a positive constant.
	Our results provide precise control in an interval of sampling rates that is excluded by
	\cite{bellec2018} and related work \cite{bickel2009simultaneous,buhlmann2011statistics,negahban2012}.
	
	 \textbf{Exact characterization.} 
	By establishing that $\| \thetahat - \thetastar \|_q $ concentrates on 
	$\E[\| \thetahat^f - \thetastar \|_q]$,
	Theorem \ref{ThmControlLassoEst} establishes upper and lower bounds on the risk that
	hold pointwise with respect to $\thetastar$ and match up to negligible errors.
	It is a promising research direction to analyze $\E[\| \thetahat^f - \thetastar \|_q]$
	for	specific correlation structures $\mySigma$ (e.g., 
	block diagonal or low-rank plus identity).
	
	Theorem \ref{ThmControlLassoEst} and the later results in this paper
	can be used to design estimators for $\taustar, \zetastar$, derive
	the distribution of the debiased Lasso, and construct confidence intervals for single coordinates.
	A recent example of this strategy was given in \cite{celentano2021cad} in a different setting.
	These exact concentration results are inaccessible from existing results 
	like those in \cite{bickel2009simultaneous,buhlmann2011statistics,negahban2012,bellec2018} 
	which are loose in their leading constants.

\begin{remark}
\label{rmk:gaussian-error}
	Although we assume that the error $\bz$ in the linear model is Gaussian with independent components,
	this assumption is not necessary,
	and Theorem \ref{ThmControlLassoEst} holds provided that $\| \bz \|_2 / \sqrt{n}$ concentrates on $\sigma$
	(the rate of this concentration may affect the right-hand side of Eq.~\eqref{eq:FixedLambda}). 
	This results from the rotational invariance of the $\ell_2$-norm.
	In settings similar to ours, the extension to non-Gaussian noise is common in the literature 
	(see, for example, \cite{celentano2021}).
	We choose to develop theory with Gaussian noise to simplify the exposition and proofs.
\end{remark}

\begin{remark}
\label{rmk:rate}
	Up to logarithmic factors, Theorem \ref{ThmControlLassoEst} demonstrates a concentration at the rate $p^{-1/4}$.
	Such a rate is typical of results proved using Gordon's comparison inequality, which we use to derive all the results in this paper (see Section \ref{sec:proof-ingredients}). We suspect this rate is an artifact of our proof technique, and the correct rate should be $p^{-1/2}$.
	Recently, \cite{li2022non,li2023approximate} developed a non-asymptotic theory to analyze the approximate message passing algorithm, which offers another possible path to improve upon the current rate. 
	
	At a high level, the source of the rate appearing in Theorem \ref{ThmControlLassoEst} is as follows. 
	Gordon's proof technique allows us to localize $\thetahat$ within a region across which the growth of the objective value exceeds the size of its typical fluctuations.
	The size of the typical fluctuations are $O_{\mathrm{p}}(n^{-1/2})$, 
	and, as a function of distance $r$ from the minimizer, we expect to growth to be $O_{\mathrm{p}}(r^2)$.
	Thus, we get the rate $n^{-1/4}$.
	This rate appears again in Theorem \ref{ThmControlLassoEst} and \ref{ThmLassoResidual}.
	Theorem \ref{SecThmLassoSparsity}, Theorem \ref{ThmDBLasso}, and Corollary \ref{CorDBLassoCI} require approximations which degrade the rate further. We expect that here, too, the rate appearing in the theorem is not optimal.
\end{remark}

\paragraph*{Simultaneous control over $\lambda$} 
So far, we only discussed the consequences of Theorem~\ref{ThmControlLassoEst}
for a fixed value of $\lambda$, namely Eq.~\eqref{eq:FixedLambda}.
However, Theorem~\ref{ThmControlLassoEst} establishes a characterization
which holds simultaneously over all  $\lambda$ in a bounded interval $[\lambda_{\min},\lambda_{\max}]$.
This is particularly useful to analyze adaptive procedures to select $\lambda$.

In particular, it implies that with high probability the minimum estimation error over choices 
of $\lambda\in[\lambda_{\min},\lambda_{\max}]$,
is nearly-achieved at a deterministic value $\lambda_*$. 
Namely, writing $\thetahat_{\lambda}$ and $\thetahat^f_{\lambda}$ for
the Lasso estimator and fixed-design estimator at regularization $\lambda$, we have
\begin{gather}
	\mprob \left(
  \Big| \frac{1}{\sqrt{p}}\|\thetahat_{\lambda_*}-\thetastar\|_2 -\min_{\lambda\in[\lambda_{\min},\lambda_{\max}]}
  \frac{1}{\sqrt{p}}\|\thetahat_{\lambda}-\thetastar\|_2\Big| > \epsilon\right)
		\leq 
  \frac{C}{\epsilon^4} e^{-cp\epsilon^4}\, ,\\
  \text{for }~~ \lambda_* := \argmin_{\lambda\in[\lambda_{\min},\lambda_{\max}]} \frac{1}{\sqrt{p}}\E[\|\thetahat^f_{\lambda}-\thetastar\|_2]\, .
\end{gather}
Recall that it is standard to choose $\lambda$ on the order of $\sqrt{\log(p/s)}$ 
(see, e.g., \cite{bellec2018}). 
As we have already described, 
applying existing fixed-design analysis to the current setting where $s$ is proportional to $p$ 
requires taking $\lambdamin \geq c > 0$ for an explicit constant $c$ that is bounded away from 0.   
As shown in \cite{miolane2018distribution},
choosing $\lambda$ based on such conservative lower bounds can be suboptimal by a large factor.
By allowing $\lambdamin$ to be arbitrarily close to 0,
our results can capture the full range of regularization parameters on which the Lasso behaves well.

\paragraph*{Control of the empirical distribution}

Previous work on iid covariates has mainly focused on establishing the convergence of the
 joint empirical distribution of the coordinates of the Lasso estimator and the true parameter vector: 
 \begin{align}
\hmu_{n,p}:=  \frac{1}{p} \sum_{i=1}^p \delta_{\sqrt{n}\theta_i^*,\sqrt{n}\widehat \theta_i}\, ,
 \end{align}
 to a limiting distribution either weakly or in Wasserstein distance \cite{bayati2011lasso,miolane2018distribution}.
When covariates are iid, the behavior of $\hmu_{n,p}$ captures all non-trivial behavior of the distribution of $\thetahat$:
indeed, the exchangeability of the model implies that conditional on $\hmu_{n,p}$, the distribution of 
$\thetahat$ is uniform over permutations of the coordinates which map each coordinate of $\thetastar$ to a
 coordinate with the same value.
 This is no longer the case for correlated covariates, and
Theorems \ref{ThmControlLassoEst} capture this  this additional structure.

Nevertheless, the empirical distribution $\hmu_{n,p}$ may be of interest, in part because it is easily interpretable.
By applying Theorem \ref{ThmControlLassoEst} to several test functions at once,
we can establish concentration of the empirical distribution simultaneously in $\lambda$.
We use a particular metrization of the weak-topology\footnote{The metric $d_{w^*}$ metrizes weak convergence in the sense that $\mu_i \stackrel{\mathrm{d}}\rightarrow \mu$ if and only if 
$d_{w^*}(\mu_i,\mu) \rightarrow 0$.}
 on the space of probability measures on $\reals^2$, namely 
\begin{equation*}
	d_{w^*}(\mu,\nu) = \sum_{k=1}^\infty 2^{-k} | \E_{\bA \sim \mu}[\phi_k(\bA)] - \E_{\bB\sim \nu}[\phi_k(\bB)] |.
\end{equation*}
Here $\{\phi_k\}$ denotes a countable subset of the $1$-Lipschitz functions $\reals^2 \to \reals$ such that for any compact set $K \subset \reals^2$, $\{\phi_k|_K\}$ is dense with respect to the $\ell_\infty$-norm.
\begin{corollary}\label{CorEmpDist}
	Assume Assumption \ref{assump:1} and additionally that $n/p \leq \Deltamax$.
	There exists $\mu_*$ --- a probability distribution on $\reals^2$ --- and constants $C,C',c > 0$ depending only on $\cuPmodel$ and $\Deltamax$ such that 
	\begin{equation*}
		\mprob\left(\exists \lambda \in [\lambdamin,\lambdamax], 
				\;\;
				d_{w^*}\left(\frac1p \sum_{i=1}^p \delta_{\sqrt{n}\theta_i^*,\sqrt{n}\widehat \theta_i},\mu_*\right) \geq \frac{C'}{\sqrt{p}} + \epsilon
			\right)
			\leq 
			\frac{C}{\epsilon^4}e^{-cn\epsilon^4},
	\end{equation*}
	and 
	\begin{equation*}
		\mprob\left(\exists \lambda \in [\lambdamin,\lambdamax], 
				\;\;
				d_{w^*}\left(\frac1p \sum_{i=1}^p \delta_{\sqrt{n}\theta_i^*,\sqrt{n}\widehat \theta_i^f},\mu_*\right) \geq \frac{C'}{\sqrt{p}} + \epsilon
			\right)
			\leq 
			2e^{-cn\epsilon^2}.
	\end{equation*}
\end{corollary}
\noindent Corollary \ref{CorEmpDist} states that in both the random-design model and the fixed-design model, 
the joint empirical distribution of the estimate and the true parameter concentrates with respect to weak-$*$ distance, and that moreover, they concentrate on the same value. 
Using Theorem \ref{ThmControlLassoEst},
one can also control properties of $\mu_*$ such as its second moments in terms of $\cuPmodel$. We prove Corollary \ref{CorEmpDist} in Appendix \ref{sec:proofOfCorEmpDist}.

\begin{remark}
\label{rmk:miolane-lasso-est}
	The proof of Theorem \ref{ThmControlLassoEst} is similar to the proof of Theorem 3.1 of \cite{miolane2018distribution} 
	in the iid design case. 
	The proof of simultaneous control over $\lambda$ (Theorem \ref{ThmControlLassoEst}) and the control 
	of the Lasso residual (Theorem \ref{ThmLassoResidual}), stated below, are similar to the proofs of 
	analogous results in \cite{miolane2018distribution}.
	We emphasize, however, that these proofs rely heavily on the boundedness and uniqueness of the fixed-point parameters 
	$\tau^*$ and $\zeta^*$ (see Lemma \ref{LemTauZetaBounds}). 
	Regarding the Lasso estimate, establishing these properties of the fixed-design characterization
	is the main technical innovation of the present paper (see Remark \ref{rmk:fixed-pt-bounds}).
	Below we will see that further technical innovations are required for analyzing the Lasso sparsity and the debiased Lasso.
	
 Note that the $\epsilon^4$ appearing in the exponent in Theorem \ref{ThmControlLassoEst} is faster than the rate appearing in Theorem 3.1 of \cite{miolane2018distribution}.
	This is because \cite{miolane2018distribution} provide a good approximation of the empirical distribution of the coordinates of $\thetahat$ in Wasserstein metric, 
	which is more complex object to control than a single Lipschitz function (see \cite[Proposition F.2]{miolane2018distribution}). 
	Corollary \ref{CorEmpDist} controls the empirical distribution of the coordinates of $\thetahat$, but in a metric which is weaker than the Wasserstein metric.
\end{remark}

%%%%%%%%%%%%%%%%%%%%%%%%%%%%%%%%%%%%%%%%%%%%%%%%%%%%%%%%%%%%%%%%%%%%%%

% \input{residual}
\subsection{Control of the Lasso residual}

In this section, we establish control for the residual of the Lasso estimator. 
The behavior of this residual is of interest because it can be used in estimators of important quantities. 
For example, we shall use it to construct an empirical estimate $\widehat \tau$ of $\tau^*$.
Informally, the Lasso residual behaves like a normally distributed random vector with zero mean and covariance $(\taustar\zetastar)^2 \Ind_n$. 
\begin{theos}
\label{ThmLassoResidual}
	Under Assumption \ref{assump:1},
	there exist constants $C,c,c' > 0$ depending only on $\cuPmodel$ and $\delta$ such that for any $1$-Lipschitz function $\phi:\reals^p \rightarrow \reals$, 
	we have for all $\epsilon < c'$
	\begin{equation}
	\label{EqnGradResidual}
		\mprob \left(
				\Big|\phi\Big(\frac{\by - \bX\thetahat}{\sqrt n}\Big) - \E\Big[\phi\Big(\frac{\taustar\zetastar \bh}{\sqrt n}\Big)\Big]\Big| > 
				\epsilon
			\right) 
			\leq 	
			\frac{C}{\epsilon^2} e^{-c\numobs \epsilon^4}\,,
	\end{equation}
	where $\bh \sim \normal(\bzero, \Ind_n)$.
	Consequently,
	\begin{equation}
		\mprob \left(
				\Big|\frac{\|\by - \bX \thetahat\|_2}{\sqrt n} - \taustar\zetastar\Big| > 
				\epsilon
			\right) 
			\leq 	
			\frac{C}{\epsilon^2} e^{-c\numobs \epsilon^4}\,.
	\end{equation}
\end{theos}
\noindent The proof of Theorem \ref{ThmLassoResidual} is provided in Section~\ref{SecPfThmLassoResidual}. 

%%%%%%%%%%%%%%%%%%%%%%%%%%%%%%%%%%%%%%%%%%%%%%%%%%%%%%%%%%%%%%%%%%%%%%

\subsection{Control of the Lasso sparsity}

This section characterizes the sparsity of the Lasso estimator. In particular, we show that the number of selected parameters per observation $\|\thetahat\|_0/n$ concentrates on $\E[\| \thetahat^f \|_0]/n = 1-\zetastar$.
\begin{theos}
\label{ThmLassoSparsity}
	Under Assumption \ref{assump:1},
	there exist constants $C,c,c' > 0$ depending only on $\cuPmodel$ and $\delta$ such that for all $\epsilon < c'$,
	\begin{equation}
		\mprob\left(\Big|\frac{\|\thetahat\|_0}{n} - (1-\zetastar)\Big| > \epsilon\right)
		\leq 
		\frac{C}{\epsilon^3}e^{-cn\epsilon^6}\,.
	\end{equation}
\end{theos}
\noindent The proof of this result is presented in Section~\ref{SecThmLassoSparsity}. 

Note that the $\epsilon^6$ in the exponent in Theorem \ref{ThmLassoSparsity} is worse than the $\epsilon^4$ appearing in the exponent in Theorem \ref{ThmControlLassoEst}, Theorem \ref{ThmControlLassoEst}, Corollary \ref{CorEmpDist}, and Theorem \ref{ThmLassoResidual}. 
This is because the function $\| \thetahat \|_0 / n$ is not a Lipschitz function.
%As in \cite{miolane2018distribution} in the case of iid Gaussian designs,
The proof involves instead analyzing the subgradient of the $\ell_1$ penalty at the Lasso solution
 and applying certain Lipschitz approximations for indicator functions. 
Because the Lipschitz constants diverge as $\epsilon \rightarrow 0$, 
this results in a weaker probability bound (see Section~\ref{SecThmLassoSparsity} for details).
We suspect this rate is not tight, and a dependence of $\epsilon^2$ may be possible, but proving such a tighter dependence may require new tools.
The estimators in the coming sections which involve $\| \thetahat \|_0/n$ will also suffer this degraded rate.

We make a note that recently Bellec and Zhang \cite[Section 3.4]{Bellec2018SecondOS} establish that $\frac{1}{n}\|\thetahat\|_0 \mid \bX$ concentrates around its expectation with deviations of order $O(n^{-1/2})$ using the second-order Stein's formula.
Our result is different and complementary, in that it shows that $\frac{1}{n}\|\thetahat\|_0$ has large-deviation probabilities (w.r.t. randomness of both the noise and the design) which decay exponentially,
and characterizes the value around which it concentrates.
Moreover, our result also implies that under Assumption \ref{assump:1} (and, in particular, above the Donoho-Tanner phase transition), 
the value on which $\frac{1}{n}\|\thetahat\|_0$ concentrates is uniformly bounded away from 1.

\begin{remark}
\label{rmk:miolane-sparsity}
	The proof of Theorem~\ref{ThmLassoSparsity} is fundamentally different from the proof of
	the analogous result for iid designs \cite[Theorem F.1]{miolane2018distribution}. 
	Indeed, the proof of \cite[Theorem F.1]{miolane2018distribution} draws heavily on
	simple expression for the empirical distribution of the coordinates of $\thetahat$ and of the subgraident 
	of the $\ell_1$-norm at the Lasso solution. 
	For general covariances, such simple expressions are unavailable due to the non-exchangeability of the model.
	See Section \ref{SecThmLassoSparsity} for details.
\end{remark}

\paragraph*{Prediction error and hyperparameter tuning}
Using Theorem \ref{ThmLassoResidual} and \ref{ThmLassoSparsity},
we can construct an estimator $\widehat \tau$ of $\tau^*$.
This gives rise to a provably optimal method for parameter tuning and a consistent estimate of the standard error of the debiased Lasso, which can be used to construct confidence intervals.
In particular, Theorem \ref{ThmLassoResidual} shows that, the residuals $\by - \bX \thetahat$ are 
approximately $\normal(0,(\tau^*\zeta^*)^2\Ind_n)$,
and moroever, that $\|\thetastar\|_0/n$ concentrates on $1 - \zeta^*$.
Thus, the parameters $\tau^*$ is consistently estimated by 
 \begin{align}
 \label{eq:tau-hat}
   \widehat{\tau}(\lambda) := \frac{\|\by-\bX\thetahat\|_2}{\sqrt{n}(1-\|\thetahat\|_0/n)}\, .
 \end{align}

 Since $\tau^*$ controls the noise in the fixed design model, its
 estimation is of particular interest.
 Indeed, because $\tau^{*2} = \sigma^2 + \E[\| \thetahat^f - \thetastar\|_{\mySigma}^2]$ and
  $\| \thetahat - \thetastar \|_{\mySigma}^2$ concentrates on $\E[\| \thetahat^f - \thetastar\|_{\mySigma}^2]$,
 $\widehat{\tau}(\lambda)^2$ concentrates, up to an additive constant which does not depend on $\lambda$, on the prediction error. 
 Because of their importance, we collect these facts in the next theorem. 
\begin{theos}
\label{ThmLassoTau}
	Under Assumption \ref{assump:1}, let $\tau^*=\tau^*(\lambda)$ be the unique solution 
	of the system of equations \eqref{EqnEqn1}, \eqref{EqnEqn2}.
	Then there exist constants $C,c,c' > 0$ depending only on $\cuPmodel$ and $\delta$ such that for all $\epsilon < c'$,
	\begin{equation}
		\mprob\left(\exists \lambda \in [\lambdamin, \lambdamax], \big|\widehat{\tau}(\lambda) -\tau^*(\lambda)\big|\ge \epsilon\right)
		\leq 
		\frac{C}{\epsilon^6}e^{-cn\epsilon^6}\,.
	\end{equation}
	Further defining $\hlambda:= \arg\min\big\{\widehat{\tau}(\lambda) :\; \lambda\in [\lambdamin,\lambdamax]\big\}$,
	we have 
	\begin{equation}
		\mprob\left(\| \thetahat_{\hlambda} - \thetastar \|_{\mySigma}^2
		\ge \min_{\lambda\in[\lambdamin,\lambdamax]}\| \thetahat_{\lambda} - \thetastar \|_{\mySigma}^2+ \epsilon\right)
		\leq 
		\frac{C}{\epsilon^6}e^{-cn\epsilon^6}\,.
	\end{equation}
\end{theos}
 Thus, minimizing $\widehat{\tau}(\lambda)^2$ over $\lambda$ gives a provably optimal 
 parameter tuning method. 
 Importantly, $\widehat{\tau}(\lambda)$ does not depend on any unknown model parameters, namely, 
 $\sigma,\mySigma$, or $\thetastar$.
It was already observed in \cite{miolane2018distribution} that minimizing $\widehat{\tau}(\lambda)$ over $\lambda$ provides
 a good selection procedure for the regularization parameter.
 Our results provide theoretical support for this approach under general Gaussian designs.
 After the current paper was posted, similar results were recently obtained for a wide class of losses and penalties in \cite{bellecShen2022}.
 
%%%%%%%%%%%%%%%%%%%%%%%%%%%%%%%%%%%%%%%%%%%%%%%%%%%%%%%%%%%%%%%%%%%%%%

\subsection{Control of the debiased Lasso}
\label{sec:DB-lasso}

Recall that the debiased Lasso with degrees-of-freedom adjustment is defined according to expression~\eqref{EqnDBlasso}
\begin{align}
% \label{EqnDBlasso}
	%
	\bthetahatd \defn \thetahat +  \frac{1}{n - \normzero{\thetahat}}
	\mySigma^{-1} \bX^\top(\by - \bX\thetahat)\,.
\end{align}
The next theorem establishes that the debiased Lasso behaves like the debiased Lasso in the fixed-design model 
 $\thetahat^{f,\de}$ (defined in Eq.~\eqref{EqnDBf}), which follows a Gaussian distribution with mean $\thetastar$ and covariance ${\taustar}^2 \mySigma^{-1}/n$.
The proof of this result is provided in Section~\ref{SecPfControlDBLasso}
\begin{theos}
\label{ThmDBLasso}
	Under Assumption \ref{assump:1}, 
	there exist constants $C,c,c' > 0$ 
	depending only on $\cuPmodel$ and $\delta$ such that for any $1$-Lipschitz $\phi: \real^p \rightarrow \reals$,
	we have for all $\epsilon < c'$
	\begin{align}
	\label{EqnMainDBLasso}
		\mprob \left(
			\big|\phi\big(\bthetahatd\big) -  \E\Big[\phi\big(  \thetahat^{f,\de}\big)\Big]\big| > \epsilon \right)
		\leq 
		\frac{C}{\epsilon^3} e^{-cp\epsilon^6}\,,
	\end{align}
	where $\bg \sim \normal(\bzero,\Ind_p)$.
\end{theos}
%
% \noindent. 
 Note that the rate of convergence obtained here is faster than the one appearing in Theorem 3.3 of \cite{miolane2018distribution} in the case of iid Gaussian designs.
The results, however, are not directly comparable, since \cite[Theorem 3.3]{miolane2018distribution} controls the empirical distribution of the coordinates of $\thetahat$ in Wasserstein distance, whereas we control a single Lipschitz function (see Remark \ref{rmk:rate} for a similar discussion).
Further, our proof techniques differ substantially from that of \cite{miolane2018distribution}.
While their results rely on a gluing argument (see Section F.2 of the Supplementary Material to \cite{miolane2018distribution}), we connect the debiased Lasso to a ``smoothed Lasso'' estimator (see Section \ref{SecPfControlDBLasso}).
In neither this paper nor in \cite{miolane2018distribution} do we expect the rates of concentration to be tight.

\paragraph*{Confidence intervals using the debiased Lasso}
Equipped with Theorem \ref{ThmDBLasso}, one may construct confidence intervals for any individual coordinate of $\thetastar$ with guaranteed coverage-on-average.
Because $\taustar$ is unknown, we use the estimator $\widehat\tau(\lambda)$ given by Eq.~\eqref{eq:tau-hat}.
We refer to the resulting intervals as the \emph{debiased confidence intervals}.
\begin{cors}
\label{CorDBLassoCI}
	Fix $q \in (0,1)$.
	For each $j \in [p]$,
	define the interval
	\begin{align}
	\label{EqnDebiasedCI}
		\mathsf{CI}^\mathrm{d}_j 
		&:= 
           \left[ \thetahatd_j - \Sigma_{j|-j}^{-1/2} \widehat\tau(\lambda)  z_{1-q/2} /\sqrt{n},~~ \thetahatd_j + \Sigma_{j|-j}^{-1/2}\widehat\tau(\lambda)\,
           z_{1-q/2} / \sqrt{n}\right]\,,
	\end{align}
	where $z_{1-q/2}$ is the $(1-q/2)$-quantile of the standard normal distribution, $\widehat\tau(\lambda)$ is given by Eq.~\eqref{eq:tau-hat}, and 
	\begin{equation*}
	% \label{eqn:conditional-variance}
		\Sigma_{j|-j} = \Sigma_{j,j} - \mySigma_{j,-j}(\mySigma_{-j,-j})^{-1} \mySigma_{-j,j}.
	\end{equation*} 
	Define the false-coverage proportion 
	\begin{equation}
		\FCP := \frac1p \sum_{j = 1}^p \indic{\theta^*_j \not \in \mathsf{CI}^\mathrm{d}_j}\,.
	\end{equation}
	Under assumptions  \ref{assump:1} and if $n/p \leq \Deltamax < \infty$, 
	there exist constants $C,c,c' > 0$ depending only on $\cuPmodel$ and $\Deltamax$ such that for all $\epsilon < c'$
	\begin{equation}
		\mprob\left( | \FCP - q | > \epsilon \right) \leq \frac{C}{\epsilon^6} e^{-cn\epsilon^{12}}\,.
	\end{equation}
\end{cors}
\noindent We prove Corollary \ref{CorDBLassoCI} in Section \ref{SecPfControlDBLasso}.
Importantly, we are able to show that the debiased Lasso is successful, at least in the sense of Corollary \ref{CorDBLassoCI}, down to the Donoho-Tanner phase transition and allow $\lambda$ to be arbitrarily close to zero (though not vanishing asymptotically).

As we have already described in Section \ref{Sec:preliminary}, in the low-dimensional limit which takes $p$ fixed, $n \rightarrow \infty$, the asymptotic standard error of the OLS estimate for $\theta_j^*$ is given by $\Sigma_{j|-j}^{-1/2} \sigma / \sqrt{n}$. 
The first fixed-point equation \eqref{EqnEqn1} states that we should inflate this standard error by replacing $\sigma^2$ with $\sigma^2 + \| \thetahat - \thetastar \|_{\mySigma}$. 
By Lemma \ref{LemTauZetaBounds},
we have that $\tau^*$ is $O(1)$.
Thus, Theorem \ref{ThmDBLasso} shows above the Donoho-Tanner phase transition the debiased Lasso achieves the parametric $n^{-1/2}$ rate in most coordinates, with standard error inflated at most by a constant.

It is worth emphasizing that the debiasing construction of Eq.~\eqref{EqnDBlasso} assumes that the population covariance
$\mySigma$ is known. 
In practice, $\mySigma$ often needs to be estimated from data. 
Replacing $\mySigma$ with $\widehat{\mySigma}$ in Eq.~\eqref{EqnDBlasso} introduces an error
$(\mySigma^{-1} - \widehat{\mySigma}{}^{-1})\bX^\top(\by - \bX \thetahat) / (n - \| \thetahat \|_0 )$, 
which we can crudely bound as $O_p(\| \mySigma^{-1} - \widehat{\mySigma}{}^{-1} \|_{\mathrm{op}})$ (because, under Assumption \ref{assump:1}, $\|\bX^\top(\by - \bX \thetahat)\|_2 / (n - \| \thetahat \|_0 ) = O_p(1)$).
Operator norm consistency of $\widehat{\mySigma}$ can be achieved under
two scenarios: $(i)$~When sufficiently strong information is known about the structure of $\mySigma$ (for instance $\mySigma$ or $\mySigma^{-1}$ are band diagonal or very sparse), see, for example, \cite{caiZhang2010,elKaroui2008,bickelLevina2008}; $(ii)$~ When additional `unlabeled' data $(\bx'_i)_{i\ge 1}$ is available.
Alternatively, 
if one is interested in a particular coordinate $j$ of $\bthetahatd$,
one needs only to control the corresponding row of $\widehat{\mySigma}{}^{-1}$,
which can be achieved using, for example, the node-wise Lasso and sufficient sparsity conditions \cite[Section 3.3.2]{javanmard2018debiasing}.
Finally, we remark that the recent paper \cite{celentano2021cad} studies the problem of debiasing in a regime where the inverse covariance matrix $\mySigma^{-1}$ cannot be estimated well, although much about this difficult regime remains open.

\begin{remark}
\label{rmk:dof}
  It is instructive to compare the degrees-of-freedom adjusted debiased Lasso of Eq.~\eqref{EqnDBlasso} with the more standard construction without
adjustment \cite{zhang2014confidence,van2014asymptotically,javanmard2014hypothesis,javanmard2014confidence,javanmard2018debiasing}:
\begin{align}
\label{EqnNaiveDebiasing}
    \bthetahatd_0 = \thetahat + \frac{1}{n}\mySigma^{-1}\bX^\top (\by - \bX \thetahat)\,.
\end{align}
The degrees-of-freedom adjustment adjusts the second term by a factor $1/(1 - \| \thetahat \|_0/n)$.
Intuitively, when the sparsity $s$ is much smaller than $n$, this factor should be close to 1, and the two constructions $\bthetahatd_0, \bthetahatd$ should behave comparably.
The paper \cite{bellec2019biasing} made this precise by showing that the impact of the adjustment on a single coordinate $\widehat \theta_{0j}^{\mathrm{d}}$ is $o_{\mathrm{p}}(n^{-1/2})$ provided $s = o(n^{2/3}/\log(p/s)^{1/3})$.
For larger values of $s$, 
the impact of the adjustment on a single coordinate can be non-negligble on the $n^{-1/2}$ scale, so becomes relevant for inference on a single coordinate (see next section).
In the proportional regime $s = \Theta(n)$, we can have $\| \thetahat^\de - \thetahat^\de_0 \|_2 = \Theta(1)$, 
whence we expect the degrees of freedom adjustment to have a non-negligible impact on all or almost all coordinates simultaneously.
The degrees-of-freedom adjustment in Eq.~\eqref{EqnDBlasso} is crucial for Theorem \ref{ThmDBLasso} and Corollary \ref{CorDBLassoCI}.
\end{remark}

%%%%%%%%%%%%%%%%%%%%%%%%%%%%%%%%%%%%%%%%%%%%%%%%%%%%%%%%%%%%%%%%%%%%%%

%%%%%%%%%%%%%%%%%%%%%%%%%%%%%%%%%%%%%%%%%%%%%%%%%%%%%%%%%%%%%%%%%%%%%%

% \input{CI_body}

\subsection{Inference on a single coordinate}
\label{SecSingleCoordCI}

While Theorem \ref{ThmDBLasso} and Corollary \ref{CorDBLassoCI} establish coverage of the debiased confidence intervals $\mathsf{CI}^\mathrm{d}_j$ on average across coordinates,
they do not guarantee the coverage of $\mathsf{CI}^\mathrm{d}_j$ for a fixed $j$.
To illustrate the problem, recall that
Theorem \ref{ThmDBLasso} implies that for any $1$-Lipschitz $\phi:\reals^p \rightarrow \reals$, we have, with high probability,
 $\phi\big(\bthetahatd\big) - \E\big[\phi\big(\bthetahatd\big)\big]= \tilde O(p^{-1/6})$, 
where $\tilde O$ hides factors which only depend on $\cuPmodel$ and $\delta$ or are poly-logarithmic in $p$.
Applied to $\phi(\bthetahatd) = \thetahatd_j - \theta^*_j$, 
this implies that the difference $\sqrt{n}(\thetahatd_j - \theta^*_j)$ lies with high-probability in an interval of length 
 $\tilde O(\sqrt{n}/p^{1/6})$.
In contrast, Theorem \ref{ThmDBLasso} and Corollary \ref{CorDBLassoCI} suggest that the typical fluctuations of 
$\sqrt{n}(\thetahatd_j - \theta^*_j)$ are of order $O(1)$.
Thus, the control of a single coordinate provided by Theorem \ref{ThmDBLasso} is at a larger scale than the scale 
of its typical fluctuations.

In fact, the na\"ive guess based on Theorem \ref{ThmDBLasso} that $\sqrt{n}\Sigma_{j|-j}^{1/2}(\widehat \theta_j^\de - \theta_j^*) \sim \normal(0,\tau^{*2})$ can be incorrect.
For example, the recent paper \cite{Bellec2019SecondOP} studies the distribution of a single coordinate of the debiased 
Lasso (and other penalized estimators),
and establishes that $\sqrt{n}\Sigma_{j|-j}^{1/2}(\thetahatd_j - \theta^*_j)/\tau^* \stackrel{\mathrm{d}}\rightarrow \normal(0,1)$ for most, but not all, coordinates of the debiased Lasso.
They show that the variance of $\sqrt{n}\Sigma_{j|-j}^{1/2}(\thetahatd_j - \theta_j^*)$ is approximately given by (see Eq.~(3.19) of \cite{Bellec2019SecondOP})
\begin{equation}
\label{eq:corrected-se}
	\E\Big[
		\frac{\| \by - \bX \bthetahat \|_2^2}{n(1-\| \bthetahat \|_0/n)^2} + \frac{(\widehat \theta_j - \theta^*_j)^2}{1 - \| \bthetahat \|_0/n}
	\Big].
\end{equation}
In particular, the standard error estimate $\widehat \tau$ will be too small by a non-negligible amount when
$(\widehat \theta_j - \theta^*_j)^2/(1 - \| \bthetahat \|_0/n)$ does not vanish relative to 
$\htau(\lambda)^2 = \| \by - \bX \bthetahat \|_2^2/n(1-\| \bthetahat \|_0/n)^2$.
Under a proportional asymptotics,
we have shown that both $\|\bthetahat - \thetastar\|_2^2/(1 - \| \bthetahat \|_0/n)$ 
 and 
$\htau(\lambda)^2$ are of order 1,
which implies that for most coordinates, $(\widehat \theta_j - \theta^*_j)^2/(1 - \| \bthetahat \|_0/n)$
vanishes relative to $\htau(\lambda)^2$. 
Nevertheless,
there may exist a sublinear number of coordinates $j$ for which $(\hat \theta_j - \theta_j^*)^2 = \Omega_\mathrm{p}(1)$
 \cite{bellec2019biasing}.
Note that this can occur even above the Donoho-Tanner phase transition or when restricted eigenvalue conditions are satisfied.
For such coordinates, the standard error $\widehat \tau$ will be too small.
The bounds $\max_j \| \mySigma^{-1}\be_j \|_1$ used by \cite{javanmard2018debiasing} prohibit the existence of such coordinates,
but need not hold under the Assumption \ref{assump:1}.

In Fig.~1 of \cite{Bellec2019SecondOP}, 
the authors demonstrate a case in which $\widehat \tau$ systematically underestimates the variance of
 $\widehat \theta_j^\mathrm{d}$.
For convenience, we also include a similar simulation here.
Let $\bv = (0, \mathbf{1}_s, \bzero_{p-s-1})/\sqrt{s}$.
That is, $\bv$ has unit $\ell_2$-norm, sparsity $s$, and is constant on its active set. 
We take $s = 100$, $n = 500$, $p = 1000$, $\rho^2 = 0.75$, $\sigma = 1$, $\lambda = \sqrt{2\sigma \log(p/s)}$, 
$\thetastar = 3\sqrt{s}\lambda \bv$, and $\mySigma = \Ind_p + \rho \be_1 \bv^\top + \rho \bv \be_1^\top$.
One can check that $\mySigma$ is positive definite.
For 5000 replications,
we generate data from the model \eqref{EqnLM},
fit the debiased Lasso estimate $\hat \theta_1^{\mathrm{d}}$,
compute the estimated standard error $\widehat\tau$,
and compute
\begin{equation}
\label{eq:sd-bz}
	\mathrm{sd}^2_{\mathrm{BZ}}
		:= 
		\frac{\| \by - \bX \bthetahat \|_2^2}{n(1-\| \bthetahat \|_0/n)^2} + \frac{(\widehat \theta_j - \theta^*_j)^2}{1 - \| \bthetahat \|_0/n}.
\end{equation}
In Figure \ref{FigFailure},
from left to right,
we plot histograms of $\sqrt{1-\rho^2}\sqrt{n}(\widehat\theta_1^\mathrm{d} - \theta_1^*)/\widehat \tau$, $\sqrt{n}(\widehat \theta_1^\mathrm{d} - \theta_1^*)/\mathrm{sd}_{\mathrm{BZ}}$, and $\hat \theta_1 - \theta_1^*$.
In the first two plots, we superimpose the standard normal density.
In the left plot, we see an overdispersion of $\sqrt{1-\rho^2}\sqrt{n}(\widehat\theta_1^\mathrm{d} - \theta_1^*)/\widehat \tau$ relative to the normal density which is no longer present when the errors are instead normalized by $\mathrm{sd}_{\mathrm{BZ}}$ in the second plot.
This validates that for the first coordinate, $\sqrt{1-\rho^2}\widehat \tau/\sqrt{n}$ underestimates the standard error.
The right-most histogram show that the error $\hat \theta_1 - \theta_1^*$ is of order $1$,
whence the second term in $\mathrm{sd}_{\mathrm{BZ}}$ is non-negligible.
(Precisely, the standard devision in this plot is about $2.2$).
We emphasize that $\mathrm{sd}_{\mathrm{BZ}}$ is not an empirical quantity.
Our purpose is simply to display evidence that the standard error $\Sigma_{1|-1}^{-1/2}\widehat\tau$ is incorrect for the first coordinate.
The paper \cite{bellec2019biasing} also provides an empirical standard error which agrees with $\mathrm{sd}_{\mathrm{BZ}}$ to first order.
\begin{figure}[h!]
\centerline{\includegraphics[width=.33\textwidth]{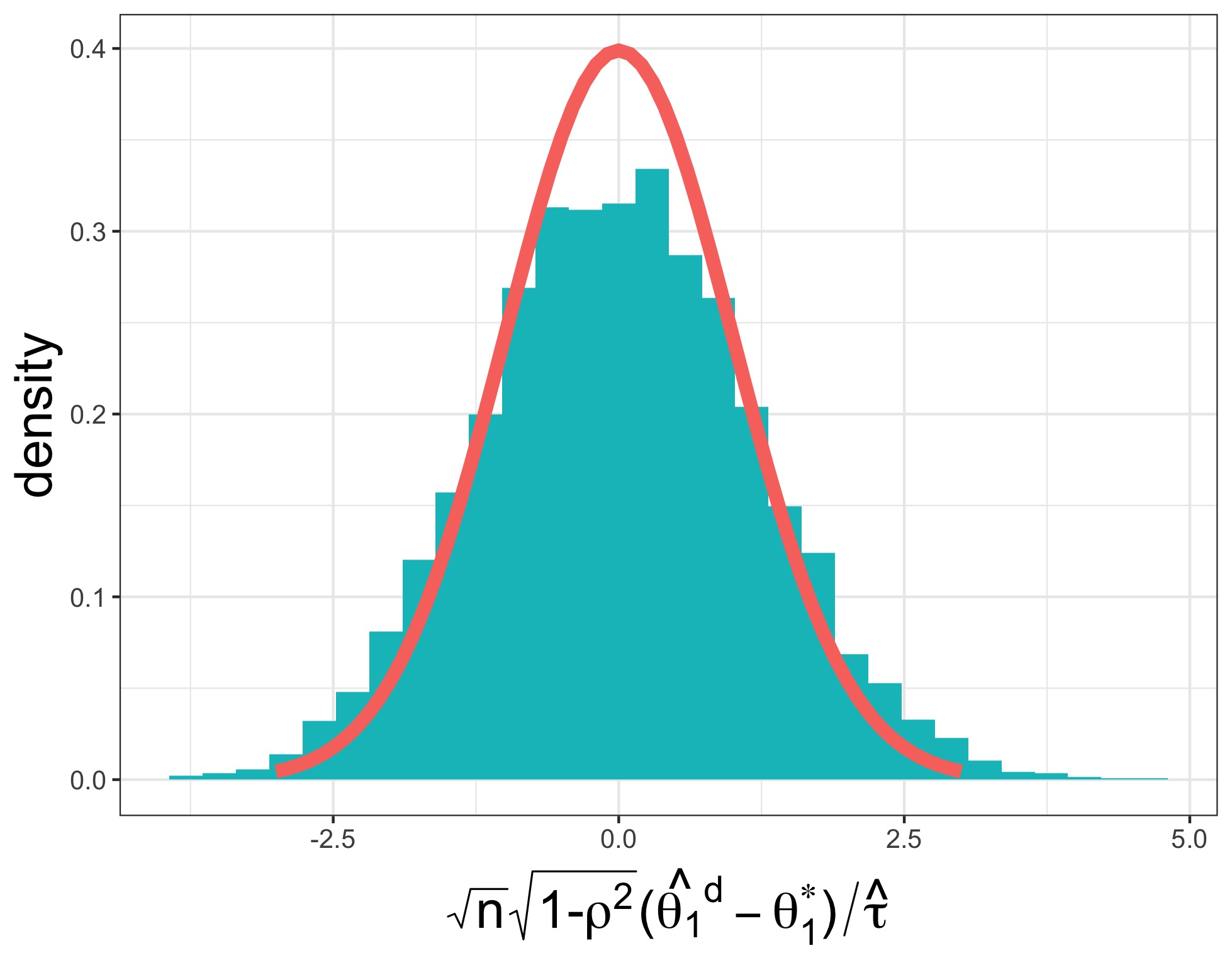}
\includegraphics[width=.33\textwidth]{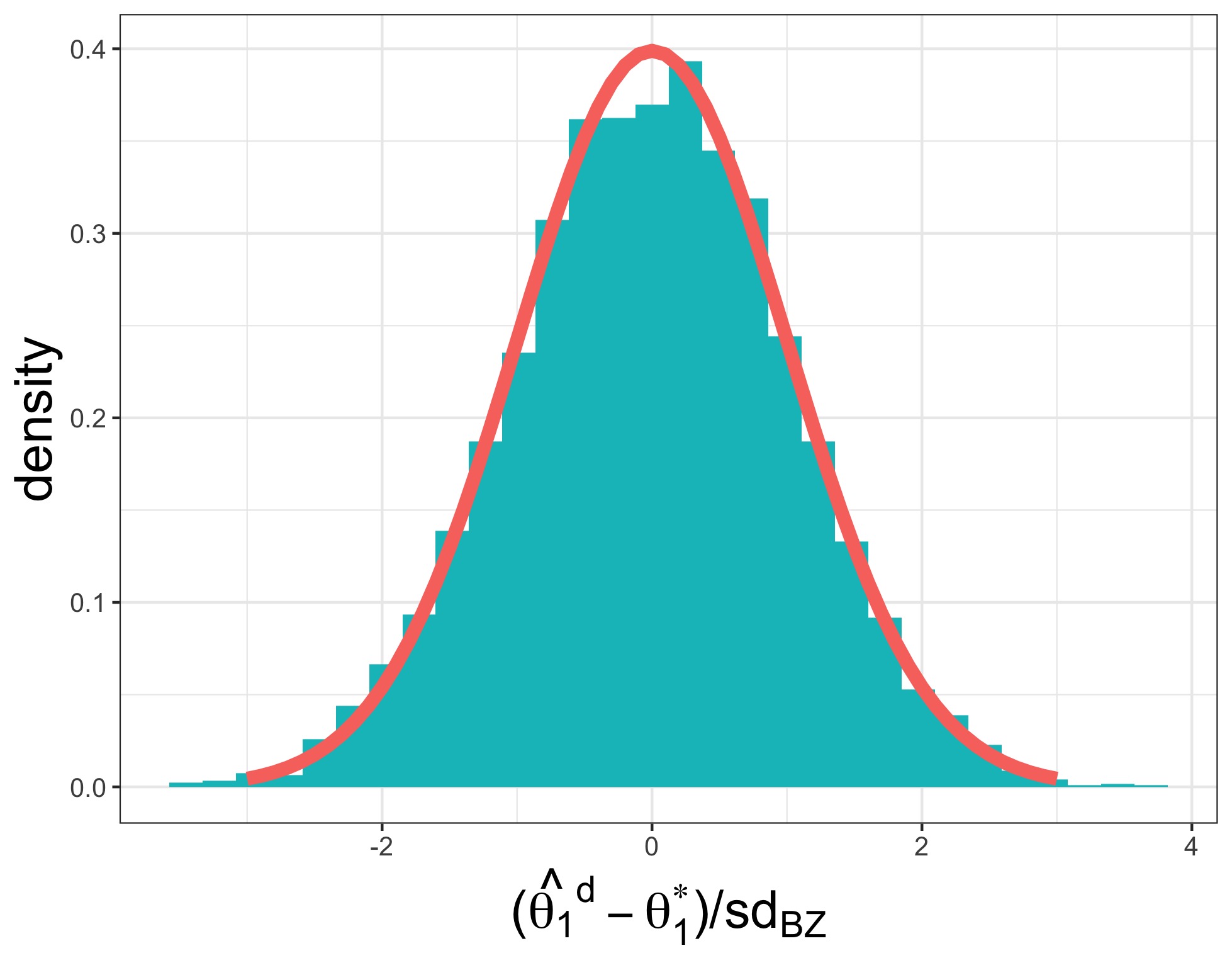}
\includegraphics[width=.33\textwidth]{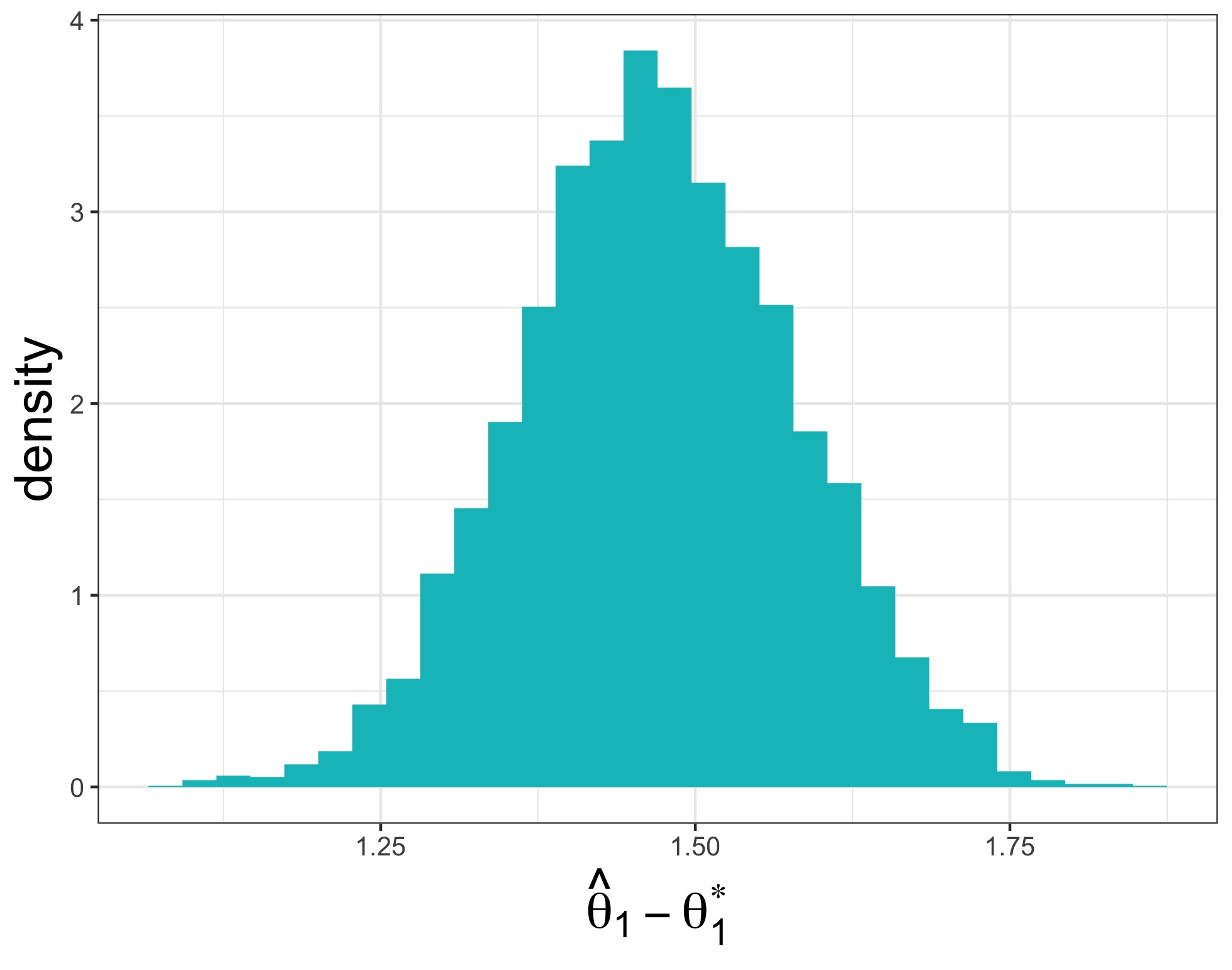}}
\caption{The debiased Lasso test statistic $\widehat \theta_1^{\de}$ for $p = 1000$, $n=500$, $s=100$, 
	$\rho^2 = .5$, $\sigma = 1$, $\lambda = \sqrt{2\sigma \log(p/s)/n} = .096$, 
	$\thetastar = 3\sqrt{s}\lambda \bv$, 
	where $\bv = (0, \mathbf{1}_s, \bzero_{p-s-1})/\sqrt{s}$ and $\mySigma = \Ind_p + \rho \be_1 \bv^\top + \rho \bv \be_1^\top$. 
	On the left, we plot a histogram of the debiased Lasso centered and normalized based on the effective nosie $\widehat \tau$ and the theory in this paper, and we superimpose the standard normal density.
	In the center plot, we normalize instead by the standard deviation derived in \cite{Bellec2019SecondOP} (see Eq.~\eqref{eq:sd-bz}). 
	On the right, we plot a histograms of Lasso error $\widehat \theta_1 - \theta_1^*$ without centering or standarization, demonstrating that the error of the Lasso in this coordinate is $O(1)$.}
\label{FigFailure}
\end{figure}

Figure \ref{FigFailure} suggests the conjecture that while  $\widehat \theta_1^\mathrm{d}$ may have standard error larger that $\tau^*$ in some coordinates, 
it is still approximately normally distributed and unbiased.
We do not establish this fact, and
as far as we know, establishing it remains open.
We expect that completing this theory will require different techniques than those in the current work.

\paragraph*{An alternative approach} 
In the current paper, we instead provide an alternative construction of confidence intervals for a single coordinate using a leave-one-out technique.
We are able to establish the coordinate-wise validity of these confidence intervals even in cases where the Lasso error $\widehat \theta_j - \theta_j^*$ is of order 1. 
We call these confidence intervals, defined below, the \emph{leave-one-out confidence intervals}, denoted by $\mathsf{CI}^{\mathrm{loo}}_j$.
According to simulation,
the leave-one-out confidence intervals often approximately agree with the debiased confidence intervals,
though for some coordinates they may have a larger or smaller width.

To facilitate the construction, let us write the observation vector $\by$ as 
\begin{align}
\label{EqnModelA}
  \by = 
  \begin{pmatrix}
    \cdots & \brevebx_j & \cdots
  \end{pmatrix}
  \begin{pmatrix}
    \vdots \\
    \theta^*_j\\
    \vdots
  \end{pmatrix}
  +
  \sigma \bz
  =
  \theta^*_j \brevebx_j + \bX_{-j} \thetastar_{-j} + \sigma \bz\,,
\end{align}
where $\bX_{-j}\in\real^{\numobs\times (\usedim-1)}$ denotes the original design matrix 
excluding the $j$-th column
and $\brevebx_j$ denotes the $j$-th column.
Define $\xper_j \defn \brevebx_j - \bX_{-j} \mySigma_{-j,-j}^{-1}\mySigma_{-j,j} \in \real^{\numobs}$
so that $\xper_j$ is independent of $\bX_{-j}$ (see Section~\ref{SecPropX}).
Let $\widehat \theta_{j,\mathrm{init}}$ be any deterministic real number that is chosen \emph{a priori}; for instance, $\widehat \theta_{j,\mathrm{init}}$ can be set as $0.$
According to decomposition~\eqref{EqnModelA},
\begin{align}
\label{EqnTransform}
  \by - \xper_j \widehat \theta_{j,\mathrm{init}} = \bX_{-j} \underbrace{(\mytheta^*_{-j}+(\theta^*_j-\widehat \theta_{j,\mathrm{init}}) \mySigma_{-j,-j}^{-1}\mySigma_{-j,j})}_{= : \thetastarloo} + \xper_j (\theta^*_j- \widehat \theta_{j,\mathrm{init}}) + \sigma \bz\,,
\end{align}
and 
\begin{align*}
  \xper_j (\theta^*_j - \widehat \theta_{j,\mathrm{init}}) +\sigma \bz \sim \normal(0, \sigma_{\mathrm{loo}}^2\Ind_{\numobs})\;\; \text{with}\;\;
  \sigma_{\mathrm{loo}}^2 \defn \sigma^2 + \Sigma_{j|-j}(\theta_j^*-\widehat \theta_{j,\mathrm{init}})^2\,.
\end{align*}
Expression~\eqref{EqnTransform} can be viewed as defining a linear-model with $p-1$ covariates, with true parameter $\thetastarloo$, noise variance $\sigma_{\mathrm{loo}}^2$, and outcome $\by - \xper_j \widehat \theta_{j,\mathrm{init}}$.
We call this the \emph{leave-one-out model},
 and call
\begin{equation*}
	\by_{\mathrm{init}} := \by - \xper_j \widehat \theta_{j,\mathrm{init}}
\end{equation*}
the pseudo-outcome.
Let $\tau^*_{\mathrm{loo}}$, $\zeta^*_{\mathrm{loo}}$ be the solution to the fixed point equations \eqref{EqnEqn1} and \eqref{EqnEqn2} in the leave-one-out model,
and $\thetahatloo$ be the Lasso fit on $\by_{\mathrm{init}}$ to $\bX_{-j}$. 

The leave-one-out confidence interval is then constructed based on the variable importance statistic
\begin{align}
\label{EqnXi}
  \xi_j \defn \widehat \theta_{j,\mathrm{init}} + \frac{(\xper_j)^\top (\by_{\mathrm{init}} - \bX_{-j} \thetahatloo) }{\Sigma_{j|-j}(n - \|\thetahatloo\|_0)}\,.
\end{align}
Note the statistic $\xi_j$ is a renormalized empirical correlation between residuals from two regressions: 
the population regression of feature $j$ on the other features (i.e., $\xper_j$), and a sample 
regression of the pseudo-outcome $\by_{\mathrm{init}}$ on the other features (i.e., $\by_{\mathrm{init}} - \bX_{-j} \thetahatloo$).
If $\widehat \theta_{j,\mathrm{init}} = \theta_j^*$,
these residuals will be independent.
Indeed, in this case $\xper_j$ is independent of $(\by_{\mathrm{init}},\bX_{-j})$,
and because $\thetahatloo$ is a function of $(\by_{\mathrm{init}},\bX_{-j})$,
$\xper_j$ is also independent of $\by_{\mathrm{init}} - \bX_{-j} \thetahatloo$.
In this case, the distribution of $\xi_j$ is easy to understand.
We will also quantify the distribution of the variable importance statistic $\xi_j$ when $\widehat \theta_{j,\mathrm{init}}$ is sufficiently close to $\theta^*_j$,
which will allow us to construct tests and confidence intervals.

Similarly to $\widehat\tau(\lambda)$ defined in Eq.~\eqref{eq:tau-hat}, we estimate the effective noise level in the leave-one-out model by
\begin{align*}
  \htauloo^j:=\frac{\|\by_{\mathrm{init}} - \bX_{-j} \thetahatloo \|_2 }{ \sqrt{n} (1 - \|\thetahatloo\|_0/n) }\,.
\end{align*}
The leave-one-out confidence interval is then defined as 
\begin{align}
\label{EqnSingleCoordCI}
    \mathsf{CI}^{\mathrm{loo}}_j :=
    \left[\xi_j - \Sigma_{j|-j}^{-1/2}\htauloo^j\, z_{1-\alpha/2} / \sqrt{n},~~ \xi_j + \Sigma_{j|-j}^{-1/2}\htauloo^j\, z_{1-\alpha/2} / \sqrt{n}\right]\,.
\end{align}
As asserted by the following result, this confidence interval $\mathsf{CI}^{\mathrm{loo}}_j$ achieves approximate coverage for fixed $j$ provided $\widehat \theta_{j,\mathrm{init}} - \theta_j^* = o(1)$.
We prove this result in Section~\ref{sec:thm12}.

\begin{theos}
\label{ThmCILengthAndPower}
  Assume $p\geq 2$ and that the leave-one-out model and Lasso estimators satisfy  \ref{assump:1}.
 Recall $\tau^*_{\mathrm{loo}}$, $\zeta^*_{\mathrm{loo}}$ are the solution to the fixed point equations \eqref{EqnEqn1} and \eqref{EqnEqn2} in the leave-one-out model.
    \begin{enumerate}[label=(\alph*)]
      
      \item % a
      \emph{(Coverage and power of the leave-one-out confidence interval)}
      For any $\gamma > 0$, 
      there exist constants $C,c,c' > 0$ depending only on $\cuPmodel$ and $\gamma$ such that for all $\epsilon < c'$, $|\theta_j^* - \widehat \theta_{j,\mathrm{init}}| < c'$,
      and $\theta \in \reals$, we have 
      \begin{equation}
      \label{EqnCIPower}  
      \begin{aligned}
        &\Big|\mprob\left(
          \theta \not \in \mathsf{CI}^{\mathrm{loo}}_j
          \right) - \mprob\left(|\theta_j^* + \Sigma_{j|-j}^{-1/2}\tau^{*}_{\mathrm{loo}} G / \sqrt{n} - \theta| > \Sigma_{j|-j}^{-1/2}\tau^{*}_{\mathrm{loo}} z_{1-\alpha/2} / \sqrt{n}\right)\Big|
        \\
        &\qquad\qquad\qquad\qquad\qquad\qquad\leq 
          C\left(|\theta_j^* - \widehat \theta_{j,\mathrm{init}}|^{2/3} + n^{2/6 + \gamma}|\theta_j^* - \theta| + \frac1n\right)\,,
      \end{aligned}
      \end{equation}
      where $G \sim \normal(0,1)$. (See discussion following theorem for an interpretation of this bound).

      \item % b
      \emph{(Length of the leave-one-out confidence interval).}
      There exist constants $C,c,c' > 0$ depending only on $\cuPmodel$, $M'$, and $\delta_{\mathrm{loo}}$ such that for all $\epsilon < c'$,
      \begin{equation}  
      \label{EqnCILengthBound}
        \mprob_{\theta^*_j}\left(
          \Bigg| \frac{\htauloo^j}{\tau^*_{\mathrm{loo}}} - 1 \Bigg| > \epsilon
          \right)
          \leq 
          \frac{C}{\epsilon^3}e^{-cn\epsilon^6}\,.
      \end{equation}

    \end{enumerate}
\end{theos}
%
% \noindent 
Note that $\mprob\left(|\theta_j^* + \Sigma_{j|-j}^{-1/2}\tau^*_{\mathrm{loo}} G/\sqrt{n} - \theta| > 
\tau^*_{\mathrm{loo}} z_{1-\alpha/2}\Sigma_{j|-j}^{-1/2} / \sqrt{n}\right)$ is the power of the standard
 two-sided confidence interval under Gaussian observations $\Sigma_{j|-j}^{1/2}\theta_j^* + \tau^*_{\mathrm{loo}} G/\sqrt{n}$ 
 against alternative $\theta$.
This normal approximation holds provided $\theta_j^* - \widehat \theta_{j,\mathrm{init}} = o(1)$ and 
$\theta - \theta_j^* = o(n^{-2/6-\gamma}) $ for some $\gamma > 0$.
In particular, it holds for $\theta - \theta_j^*$ on the $n^{-1/2}$ scale.

It is convenient to consider a few special cases of Theorem \ref{ThmCILengthAndPower}:
\begin{enumerate}
\item $\widehat \theta_{j,\mathrm{init}}=0$ and   $\theta_j^*=0$. In this case, setting $\theta=0$
yields $|\mprob\left(  0 \not \in \mathsf{CI}^{\mathrm{loo}}_j
          \right) -\alpha|\le C/n$. In fact a moment of reflection shows that this bound can be improved to yield
      \begin{equation}  
\mprob\left(  0 \not \in \mathsf{CI}^{\mathrm{loo}}_j
          \right) = \alpha\, .
      \end{equation}  
    That is, we have exact control of type I errors.
\item  $\widehat \theta_{j,\mathrm{init}}=0$ and   $\theta_j^*=o(1)$. Setting again
$\theta=\theta^*_j$, we obtain 
      \begin{equation}  
        |\mprob\left(  \theta^*_j \not \in \mathsf{CI}^{\mathrm{loo}}_j
          \right) -\alpha| = o(1).  
    \end{equation}
    That is, we obtain asymptotic coverage for all non-zero coefficients that are small (note
    that if $\|\thetastar\|_2=O(1)$, this is the case for most non-zero coefficients). 
\item  Generally leave-one-out confidence intervals are successful provided $\widehat \theta_{j,\mathrm{init}}$ 
 is consistent for $\theta_j^*$. 
Note that we assume $\widehat \theta_{j,\mathrm{init}}$ is deterministic, 
which accommodates settings in which it is based on prior knowledge or is an estimate based on an independent data set.
Note that consistency is a rather weak requirement (indeed $\|\thetahat-\thetastar\|_2=O(1)$). We
also point at the next section for a construction of exact confidence intervals that do not require the initialization 
$\widehat \theta_{j,\mathrm{init}}$.
\end{enumerate}

\begin{remark}
\label{rmk:db-fail-loo-succeed}
Even when $\theta_j^*$ is 0, it is possible that $\widehat \theta_j$ as estimated by the Lasso is of order 1;
indeed, 
Figure \ref{FigFailure} presents a simulation of such a scenario.
In this case, the na\"ive standard error for the debiased Lasso is too small,
but our leave-one-out construction with $\widehat \theta_{j,\mathrm{init}} = 0$ achieves coverage. 
Moreover, in Section \ref{SecSimsCI}, we provide simulation evidence that in this scenario, the leave-one-out estimates $\xi_j$ have smaller variance than the debiased estimates $\widehat \theta_j^{\mathrm{\de}}$,
indicating that they permit more precise inference.
Characterizing in which scenarios the leave-one-out intervals are more or less precise than the debaised confidence intervals is a promising avenue for future work.
\end{remark}

In concurrent work, Bellec and Zhang \cite{Bellec2019SecondOP} consider debiasing with a arbitrary convex penalties, 
and establish success of the debiased confidence intervals when (among other assumptions) the initial estimate 
$\widehat \theta_j$ is consistent in coordinate $j$.
Our result is comparable with theirs (for a special choice of the penalty) but has the
advantage of holding down to the Donoho-Tanner phase transition and 
permitting that taking $\lambda$ be arbitrarily close to 0.
We also do not require that $\| \thetahat \|_0 / n \leq 1/2$ with high-probability.

The leave-one-out construction is a renormalized empirical correlation between the residuals of the regression of $\by_{\mathrm{init}}$ on $\bX_{-j}$ and of $\bx_j$ on $\bX_{-j}$.
It is thus similar to a method proposed by \cite{sun2012comment,ren2015asymptotic}, in which the partial correlation between two features in a Gaussian graphical model is estimated by regressing each of these features on the remaining features.
For each regression, \cite{sun2012comment,ren2015asymptotic} use the scaled Lasso and must assume $s = o(\sqrt{n})$ (up to logarithmic terms) to achieve normal inference.
In contrast, we assume that one of the regressions --- that of $\bx_j$ on $\bX_{-j}$ --- is known perfectly, 
whereas the second regression --- that of $\by_{\mathrm{init}}$ on $\bX_{-j}$ --- must be estimated and can have much less structure (possibly linear sparsity). 
For this reason, we require a degrees-of-freedom correction, which is not present in \cite{sun2012comment,ren2015asymptotic}.

\paragraph*{Relation to the conditional randomization test}
It is worth remarking that exact tests and confidence intervals for $\theta^*_j$ may be constructed in our setting.
In fact, when the feature distribution is known, one can perform an exact test of
\begin{equation}
\label{eq:conditional-independence-null}
  \by \indep \brevebx_j \mid \bX_{-j}\,,
\end{equation}
even without Gaussianity or any assumption on the conditional distribution of the outcome $\by$ given the features $\bX$ (see, e.g., \cite{candes2016panning,Katsevich2020ATT,Liu2020FastAP}).
The test which achieves this is called the \emph{conditional randomization test}
and is feasible to use for any arbitrary variable importance statistic $T(\by,\bX)$.
The key observation leading to the construction of the conditional randomization test is that under the null, 
the distribution of $T(\by,\bX) \mid \bX_{-j}$ is equal to the distribution of $T(\by, \bx_j' , \bX_{-j})\mid \bX_{-j}$ 
where $\bx_j'$ is drawn by the statistician from the distribution $\bx_j | \bX_{-j}$ without using $\by$.
Under the null, this distribution can be computed to arbitrary precision by Monte Carlo sampling.
We refer the reader to \cite{candes2016panning,Katsevich2020ATT,Liu2020FastAP} for more details about how these observations lead to the construction of an exact test.

When the linear model is well-specified, the null \eqref{eq:conditional-independence-null} corresponds to $\theta^*_j = 0$,
and our leave-one-out procedure with $\widehat \theta_{j,\mathrm{init}} = 0$ implements the conditional randomization test under this null, as we now explain.
The statistic $\xi_j$, defined in Eq.~\eqref{EqnXi} and used in the construction of the leave-one-out interval,
can also be used as the variable importance statistic in the conditional randomization test.
The Gaussian design assumption and the choice of statistic $\xi_j$ permit an explicit description of the null conditional distribution $\xi_j | \by,\bX_{-j}$. 
Indeed, because $\xper_j$ is independent of $(\by,\bX_{-j},\thetahatloo)$ under the null $\theta^*_j = 0$, one has 
\begin{equation*}
  \sqrt{n}\,\xi_j | \by,\bX_{-j} \sim \normal\Big(0,\Sigma_{j|-j}^{-1} (\htauloo^j)^2 \Big) \,.
\end{equation*}
In our setting, we can access the null conditional distribution through its analytic form rather than through Monte Carlo sampling.
The test which rejects when $0 \not\in \mathsf{CI}^{\mathrm{loo}}_j$ is exactly the conditional randomization test for the null \eqref{eq:conditional-independence-null} based on the variable importance statistic $|\xi_j|$.\footnote{This holds provided that the statistician computes $\xi_j \mid \by,\bX_{-j}$ exactly by taking an arbitrarily large Monte Carlo sample.}
As a consequence, the leave-one-out confidence intervals have exact finite sample coverage under the null $\theta_j^* = 0$ when $\widehat \theta_{j,\mathrm{init}} = 0$.
Moreover, Theorem \ref{ThmCILengthAndPower} provides more than what existing theory on the conditional randomization test can provide:
it gives confidence intervals which are valid under proportional asymptotics and a power analysis for the corresponding tests.

The linearity assumption in our setting allows us to push this rationale further.
When $\theta^*_j = \widehat \theta_{j,\mathrm{init}}$,
the $j^\text{th}$ residualized covariate $\xper_j$ is independent of the pseudo-outcome $\by_{\mathrm{init}}$ and $\bX_{-j}$. 
Thus, by the same logic as above,
the leave-one-out confidence interval achieve exact coverage when $\widehat \theta_{j,\mathrm{init}} = \theta_j^*$. 
In particular, we have an exactly valid test of $\theta_j^* = \widehat \theta_{j,\mathrm{init}}$ for all values of $\widehat \theta_{j,\mathrm{init}}$.
The inversion of this collection of tests, indexed by $\widehat \theta_{j,\mathrm{init}}$, produces an exact confidence interval. 
Details of this construction are provided in Appendix \ref{SecPfSingleCoordCI}.

We prefer the approximate interval $\mathsf{CI}_j^{\mathrm{loo}}$ to the exact interval outlined in the preceding paragraph for computational reasons. 
The construction of these exact confidence intervals requires recomputing the leave-one-out Lasso estimate using pseudo-outcome $\by - \xper_j \widehat \theta_{j,\mathrm{init}}$ for each value of $\widehat \theta_{j,\mathrm{init}}$.
In contrast, the leave-one-out confidence interval we provide requires only computing a single leave-one-out Lasso estimate.
It achieves only approximate coverage,
but our simulations in Section \ref{SecSimsCI} show that coverage is good already for $n,p,s$ on the order of 10s or 100s.
An additional benefit of Theorem \ref{ThmCILengthAndPower} is its quantification of the length of the leave-one-out confidence intervals and the power of the corresponding tests, 
which are not in general accessible for the conditional randomization test or confidence intervals based on it.
In fact, because the test $0 \not \in \mathsf{CI}^{\mathrm{loo}}_j$ is exactly the conditional randomization test,
Theorem \ref{ThmCILengthAndPower}(a) applied under $\theta^*_j$ provides an estimate of the power of the conditional randomization test under alternative $\theta^*_j = \omega$.

% \mc{Should we update what we say about this? Currently it addresses Item 6 from Reviewer 3, but I think it feels a bit out of place, as I don't think that paper really does a leave-one-out construction as far as I can tell.} \yutingcomment{agreed}
% Statistics $\xi_{j}$ constructed here (cf.~\eqref{EqnXi}) shares the same structure as the decorrelated score statistics constructed in \cite{ning2017general}, except a minor difference in fitting the nuisance parameter. 
% The latter establishes the asymptotic equivalence between the leave-one-out confidence interval and debiased lasso confidence interval in the sparse regime.

%%%%%%%%%%%%%%%%%%%%%%%%%%%%%%%%%%%%%%%%%%%%%%%%%%%%%%%%%%%%%%%%%%%%%%

\subsection{Restricted eigenvalues and the Donoho-Tanner phase transition}
\label{SecCompare}

An important feature of our results is that they hold down to the Donoho-Tanner phase transition, which can be weaker than the requirement based on restricted eigenvalue conditions.

Specifically, the standard restricted eigenvalue on support $S \subset [p]$ of a matrix $\bX \in \reals^{n \times p}$ is defined as (see, for example, \cite{bickel2009simultaneous,bellec2018}) 
\begin{equation}
	\RE(S,c)
	:=
	\RE(S,c; \bX )
	:=
	\min_{\mytheta \in \cuC_{\sRE}(s,c)}
		\frac1{\sqrt{n}}\| \bX \mytheta \|_2 
		>0,
\end{equation}
where $\cuC_{\sRE}(S,c) := \{ \mytheta \in \reals^p : \| \mytheta_{S^c} \|_1 \leq c \| \mytheta_S \|_1,\, \| \mytheta\|_2 = 1\}$.
In order for bounds based on restricted eigenvalues to yield the correct estimation error rate, 
one typically needs $\RE(S,c)$ to be bounded away from zero for some $c$ strictly larger than $1.$

In the random design setting of the present paper, we illustrate by the following example that, $\RE(S,c) = 0$ with high-probability for some non-vanishing interval of sampling rates above the Donoho-Tanner phase transition.

\begin{proposition}
\label{prop:RE-and-DT}
	Consider a block diagonal matrix $\mySigma \in \reals^{p \times p}$ 
	whose first $s/2$ diagonal blocks are $\bK = \begin{pmatrix} 1 & \rho \\ \rho & 1 \end{pmatrix}$ for some constant $\rho>0$, 
	and whose lower right $(p-s)\times(p-s)$ diagonal block is $\Ind_{p-s}$.
	Let $S = \{1,2,\dots,s\}$ and $\bx^* = \mathbf{1}_S\in\reals^p$ be the indicator vector on $S$. 

	Consider the limit $s,p,n\rightarrow \infty$ with $s/p = \eps$ and $n/p = \delta$ fixed.
	In this setting, the Gaussian width $\cuG(\bx^*,\mySigma)= \overline{\cuG}(\eps,\delta,\rho)
	\in(0,\infty)$ only depends on $n$, $p$, $s$ through the ratios $\eps$, $\delta$.
	Further, there exists $\Delta(\varepsilon,\delta,\rho) > 0$ such that if 
	$\cuG(\bx^*,\mySigma)^2 <\delta < \cuG(\bx^*,\mySigma)^2+ \Delta(\varepsilon,\delta,\rho),$ 
	then with probability going to 1 as $p\to\infty$, $\RE(S,c;\bX) = 0$ for all $c \geq 1$.
\end{proposition}
\noindent We prove Proposition \ref{prop:RE-and-DT} in Appendix \ref{sec:RE-to-DT}.
% Hence, the results of the present paper apply to regimes in which bounds based on restricted eigenvalues would fail. To circumvent this bottleneck, it is possible to modify the definition of the restricted eigenvalue by considering 
% cone $\cuK(\bx,\mySigma)$ (cf.~\eqref{eqn:def-cuk}) rather than $\cuC_{\sRE}(S,c)$. However, generalizing our results under restricted eigenvalue conditions is highly nontrivial.}
% \red{[Michael's proposal:]
	We remark that the set $\cuC_{\sRE}(S,1)$ is closely related to the cone $\cuK(\bx,\mySigma)$ used in defining the Gaussian width $\cuG(\bx,\mySigma)$:
	the former is based on the cone constraint $ \| \mytheta_{S^c} \|_1 \leq \| \mytheta_S \|_1$,
	whereas the latter is based on the cone constraint $\| \mytheta_{S^c} \|_1 \leq \< \sign(\bx) , \mytheta \>$, where $S = \supp(\bx)$.
	The right-hand side $\| \mytheta_S \|_1$ is the supremum of $\< \sign(\bx) , \mytheta \>$ over all sign vectors $\bx$ with support $S$.
	Existing proofs based on the restricted eigenvalue condition 
	\cite{bickel2009simultaneous,bellec2018} go through if $\| \mytheta_S \|_1$
	 were replaced by $\< \sign(\bx) , \mytheta \>$ in the definition of the restricted 
	 eigenvalue condition (indeed, in these proofs, this quantity serves only as a bound
	  on $\| \mytheta_S^* \|_1 - \| \mytheta_S \|_1$).
	Thus, Proposition \ref{prop:RE-and-DT} as demonstrates the importance of using 
	$\< \sign(\bx) , \mytheta \>$ instead of $\| \mytheta_S \|_1$ in definitions 
	of the Gaussian width or restricted eigenvalue rather than demonstrating a
	 fundamental limitation of prior analyses.
	A fundamental improvement of our analysis relative to prior analyses is that we can take 
	$c = 1$ rather than $c > 1$. For fixed $c>1$, even a modified restricted 
	 eigenvalue condition using   $\< \sign(\bx) , \mytheta \>$  results in a gap with respect to our condition 
	 $\cuG(\bx^*,\mySigma)^2 <\delta$.
% }

A natural question is whether our results hold for sampling rates below the 
Donoho-Tanner phase transition. 
The following proposition gives a partial answer, in the negative direction.
\begin{proposition}
\label{prop:DT-necessary}
	Consider $\bx \in \{-1,0,1\}^p$ with $\| \bx \|_0 \geq 1$ and $\epsilon > 0$.
	If
	\begin{equation}
		\cuG(\bx,\mySigma) \geq \sqrt{\frac{n}{p}} + \epsilon,
	\end{equation}
	then, for any $r > 0$,
	there exists $\thetastar$ (depending on $r,\lambda,\sigma,\kappamin,\kappamax,n,p$, and 
	$\| \bx \|_0$) with $\sign(\thetastar) = \bx$ 
	such that
	if the data is generated according to \eqref{EqnLM}, 
	then
	\begin{equation}
		\mprob(\| \thetahat - \thetastar \|_2 \geq r) \geq 1 - Ce^{-cp\epsilon^2},
	\end{equation}
	where $C,c > 0$ depend only on $\kappamax$.
\end{proposition}
\noindent In particular, the Lasso has unbounded risk on sparse balls below the Donoho-Tanner phase transition, whence Theorem \ref{ThmControlLassoEst} cannot hold with bounded fixed-point parameters.
We prove Proposition \ref{prop:DT-necessary} in Section \ref{sec:DT-necessary}.
 	
%%%%%%%%%%%%%%%%%%%%%%%%%%%%%%%%%%%%%%%%%%%%%%%%%%%%%%%%%%%%%%%%%%%%%%

\section{Numerical simulations}

This section contains numerical experiments which \emph{(i)} illustrate the success of the degrees-of-adjustment for 
$n,p,s$ in the 100s to 1000s,
\emph{(ii)}
compare the leave-one-out and debiased confidence intervals,
and \emph{(iii)} support the expectation that our results may hold for a more general class of
 feature distributions than Gaussian. 
We present here some representative simulations and refer to Appendix~\ref{SecAdditionalSims}
for further results.

%%%%%%%%%%%%%%%%%%%%%%%%%%%%%%%%%%%%%%%%%%%%%%%%%%%%%%%%%%%%%%%%%%%%%%

\subsection{Debiasing with degrees-of-freedom adjustment}
\label{sec:DebiasNum}

We compare the degrees-of-freedom adjusted debiased Lasso of Eq.~\eqref{EqnDBlasso} with the
 unadjusted estimator of Eq.~\eqref{EqnNaiveDebiasing}.

Figure \ref{FigQQplots_and_histograms} reports results on the distribution of the two estimators. 
We set $p = 1000$, $n = 500$, and $s = 200$, and fix $ \thetastar \in \reals^p$ with $s/2$ coordinates equal to $10/\sqrt{s}$ and the rest equal to $-10/\sqrt{s}$ chosen uniformly at random. 
We repeat the following steps $N_\mathrm{sim} = 500$ times.
First, we generate data from the linear model \eqref{EqnLM} where $\bx_i \sim \normal(\bzero,\mySigma)$,
$\sigma = 1$ and $\mySigma$ comes from the
autoregressive model $\mathsf{AR}(0.5)$:
\begin{align}
\label{Eqnballon}
    \Sigma_{ij} = 0.5^{|i-j|}\,.
\end{align}
In each simulation, we keep the same $\thetastar$ vector but independent draws of $\bX,\bz$.
We compute for each $j\leq p$ the standardized values $\frac{\sqrt{n}\Sigma_{j|-j}^{1/2}(1-\|\bthetahat\|_0/n)(\thetahatd_j - \theta^*_j)}{\|\by - \bX \bthetahat\|_2/\sqrt{n}}$ and $\frac{\sqrt{n}\Sigma_{j|-j}^{1/2}(\thetahatd_{0j} - \theta^*_j)}{\|\by - \bX \bthetahat\|_2/\sqrt{n}}$ corresponding to the debiased Lasso with and without degrees-of-freedom adjustment respectively.
Aggregating over coordinates and simulations (giving $p \cdot N_{\mathrm{sim}} = 500,000$ observations of single coordinates),
we plot histograms and quantile plots for all coordinates corresponding to $\theta_j = -10/\sqrt{s}$, $0$, $10/\sqrt{s}$ separately. 
In the quantile plots, the empirical quantiles are compared with the theoretical quantiles of the standard normal distribution $\normal(0,1)$.

Without the degrees-of-freedom adjustment, visible deviations from normality occur.
For active coordinates, we observe bias and skew;
for inactive coordinates, we observe tails which are too fat.
The fattening of the tails occurs around and beyond the quantiles corresponding to two-sided confidence intervals constructed at the 0.05 level.
Thus, failure to implement degrees-of-freedom adjustments will lead to under-coverage in standard statistical practice even prior to corrections for multiple testing.
In contrast, with degrees-of-freedom adjustment, no obvious deviations from normality occur for either the inactive or active coordinates.
Normality is retained well into the normal tail.
Our simulations are well into the proportional regime.
In agreement with \cite{javanmard2014hypothesis,bellec2019biasing,Bellec2019SecondOP,miolane2018distribution}, 
these simulations confirm that the degrees-of-freedom adjustment suffices to recover normality.

The simulations presented Figure \ref{FigQQplots_and_histograms} are representative of simulations conducted at differing sample sizes, sparsity levels, and correlation parameters $\rho$. 
We present these simulations in Appendix~\ref{SecAdditionalSims}.
We also present simulations for equicorrelated designs.
%%%%%%%%%%%%%%%%%%%%%%%%%%%%%%%%%%%%%%%%%%%%%%%%%%%%%%%%%%%%%%%%%%%%%%

\subsection{Confidence interval for a single coordinate}
\label{SecSimsCI} 

In this section, we consider the behavior of the debiased confidence interval $\mathsf{CI}^{\mathrm{d}}_j$ (defined in Eq.~\eqref{EqnDebiasedCI})
and leave-one-out confidence interval $\mathsf{CI}^{\mathrm{loo}}_j$ (defined in Eq.~\eqref{EqnSingleCoordCI}).

In Figure \ref{FigSingleCoord}, we examine the coverage of the confidence interval for both an active coordinate and an inactive coordinate.
We consider $p = 100$, $n = 25$, and $s = 20$, and fix $ \thetastar \in \reals^p$ with $s/2$ coordinates equal to $10/\sqrt{s}$ and the rest equal to $-10/\sqrt{s}$.
The locations of the active coordinates are chosen uniformly at random. 
We set the coordinate of interest to be $\theta_{50}$.
For each model specification, we perform the following $N_\mathrm{sim} = 1000$ times.
First, we generate data from the linear model \eqref{EqnLM} with $\sigma = 1$ and $\mySigma$ the $\mathsf{AR}(0.5)$ covariance
$\Sigma_{ij} = 0.5^{|i-j|}$.
We construct for $j = 50$ the $(1-\alpha)$-confidence intervals $\mathsf{CI}^{\mathrm{d}}_j$ and $\mathsf{CI}^{\mathrm{loo}}_j$ at level $\alpha = 0.05$.
We also construct the following interval based on the debiased Lasso without degrees-of-freedom adjustment given by Eq.~\eqref{EqnNaiveDebiasing}:
\begin{equation*}
    \mathsf{CI}_j^{\mathrm{d,noDOF}} := \left[\widehat{\theta}^{\mathrm{d}}_{0j} - \frac{\Sigma_{j|-j}^{-1/2} \|\by - \bX \thetahat\|_2}{n}z_{1-\alpha/2},~~\widehat{\theta}^{\mathrm{d}}_{0j} + \frac{\Sigma_{j|-j}^{-1/2} \|\by - \bX \thetahat\|_2}{n}z_{1-\alpha/2}\right]\,.
\end{equation*}

The confidence intervals from the first 40 of the 1000 simulations are plotted in Figure \ref{FigSingleCoord} for the cases 
$\theta^*_{50} = 0$ and $\theta^*_{50} = 10/\sqrt{s} \approx 2.24$. 
Both the debiased Lasso and the leave-one-out confidence intervals achieve coverage.
Although in some simulations  there appears to be a small difference between the intervals computed by the two methods, in most cases the confidence intervals almost exactly agree.
 In contrast, when $\theta^*_{50} = 10/\sqrt{s}$, the confidence interval without degrees-of-freedom adjustment is uncentered
 and too narrow, leading to large  under-coverage. 
 When $\theta_{50}^* = 0$, the empirical coverage (for 1000 simulations) is $95.5\%$ for the debiased Lasso with degrees-of-freedom
 adjustment, $95.2\%$ for leave-one-out confidence interval, and $90.05\%$ for the debiased Lasso without degrees-of-freedom adjustment.
When $\theta_{50}^* = 10/\sqrt{s}$, these coverages are $94.3\%$, $93.9\%$, and $36.8\%$, respectively.
Note that confidence interval with degrees of freedom adjustment is undefined when $\| \thetahat \|_0 = n$.
 Because we take $s/n$ very large, this occurs in some of our simulations.
 When this occurs, we count this as an non-coverage event, and omit to draw the confidence interval in our plots.

\begin{figure}[h!]
\centerline{\includegraphics[width=.8\textwidth]{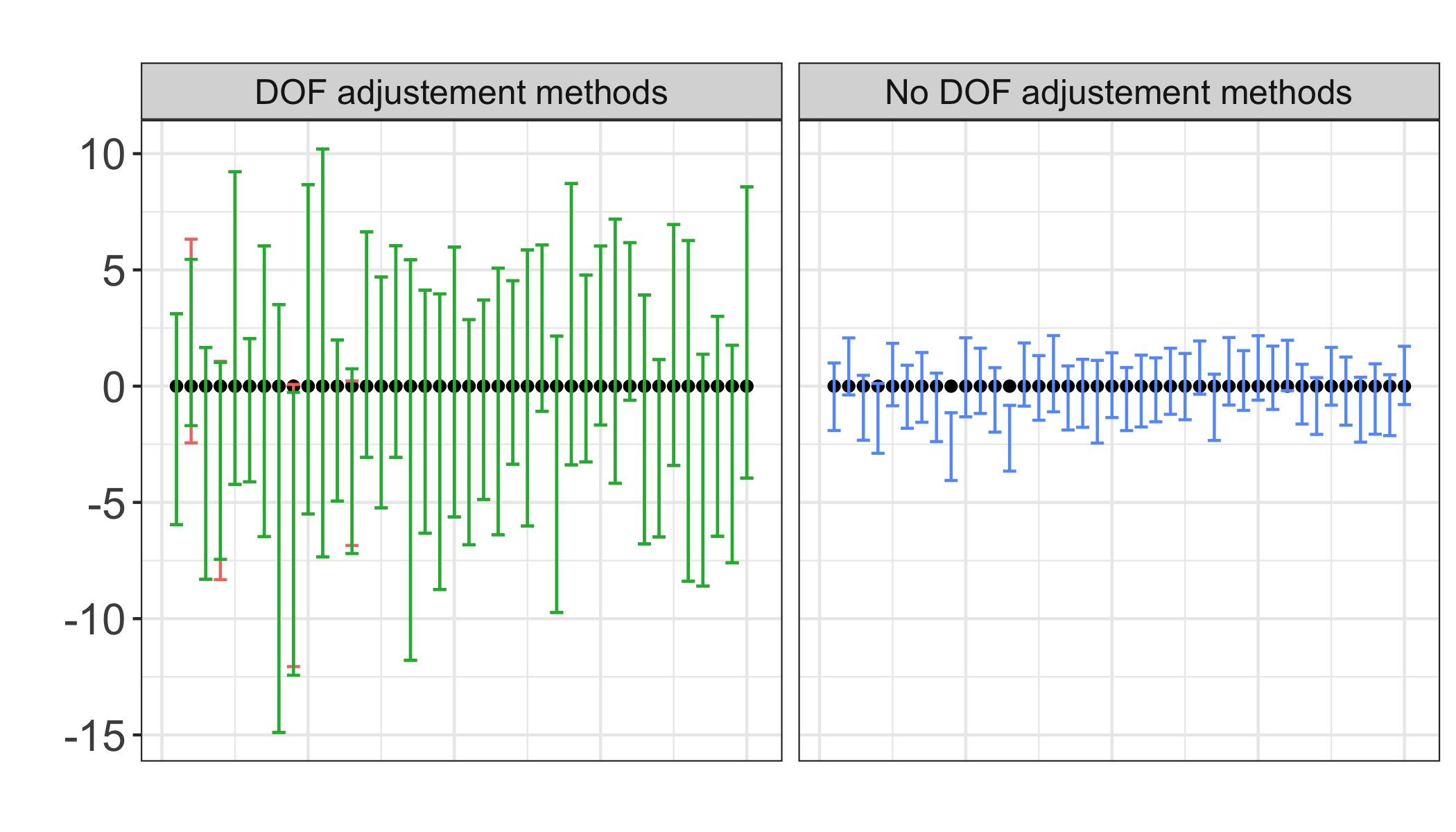}}
\centerline{\includegraphics[width=.8\textwidth]{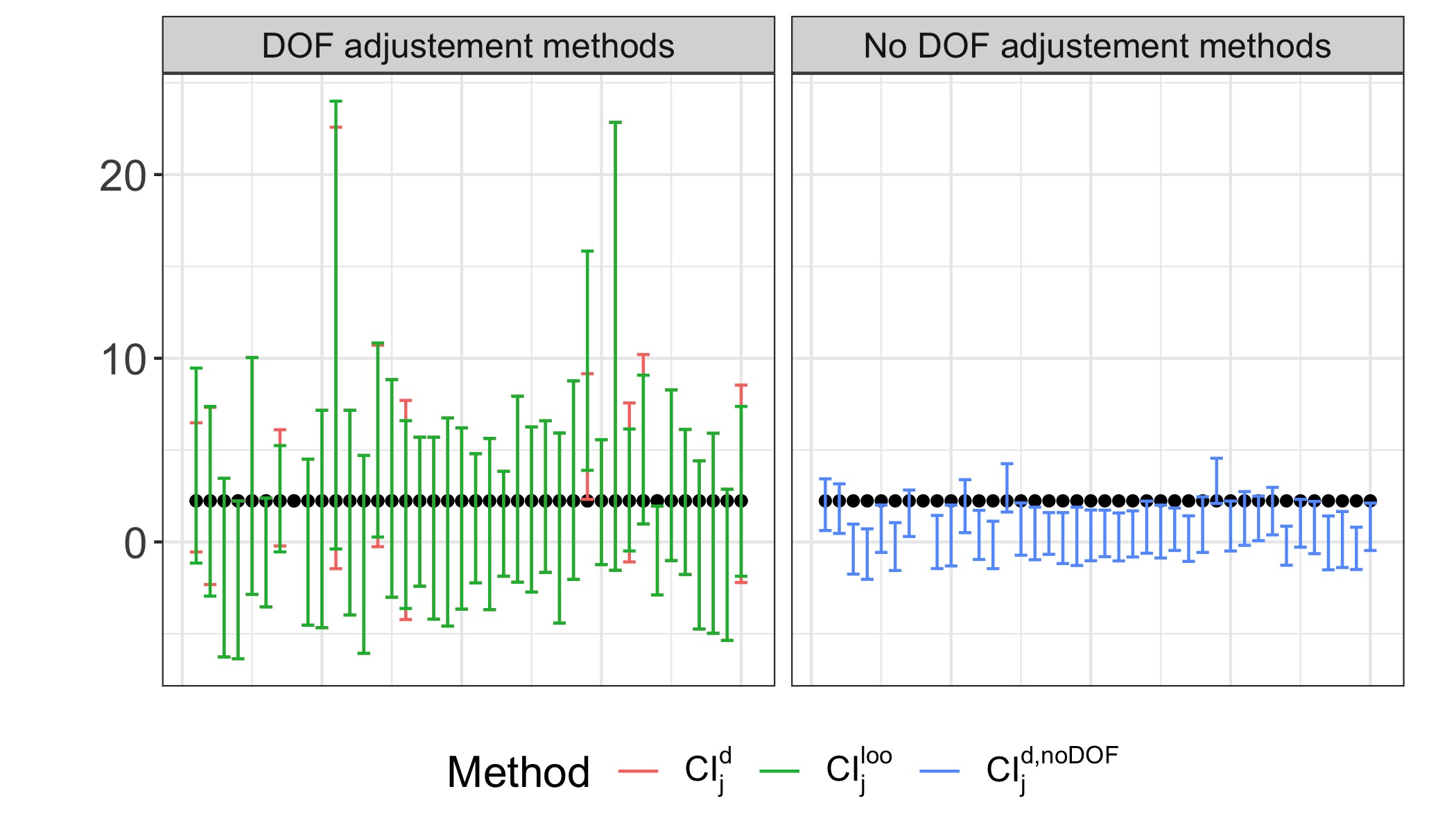}}
\caption{Confidence interval for a single coordinate $\theta^*_{50}$. Here $p = 100$, $n = 25$, $s = 20$, $\Sigma_{ij} = 0.5^{|i-j|}$, $\lambda = 4$, $\sigma = 1$. In the top plots, the truth is $\theta_{50}^* = 0$, and in the bottom plots the truth is $\theta_{50}^* = 10/\sqrt{s} \approx 2.24$.}
\label{FigSingleCoord}
\end{figure}

These simulations provide evidence that the leave-one-out confidence intervals $\mathsf{CI}^{\mathrm{loo}}_j$ are valid for fixed coordinate $j$, already
for moderate values of $n,p$.
In this case, the debiased confidence intervals  $\mathsf{CI}^{\mathrm{d}}_j$ appear to achieve
coverage per-coordinate and not only on average across coordiantes.
Moreover, in this case, the confidence intervals and $\mathsf{CI}^{\mathrm{d}}_j$ and $\mathsf{CI}^{\mathrm{loo}}_j$ appear to be nearly equivalent.

We also consider the simulation set-up in Figure \ref{FigFailure},
in which the effective noise $\widehat \tau$ for the debiased confidence intervals is too small. 
In the same simulations used to generate Figure \ref{FigFailure},
we display in the left plot of Figure \ref{FigCompare} 40 realizations of the debiased confidence interval with degrees of freedom adjustment and the leave-one-out confidence intervals. 
The empirical coverage across these 5000 replications was 89.78\% for the debiased confidence intervals and 94.78\% for the leave-one-out confidence intervals. 

As expected, the debiased confidence intervals with width computed based on $\widehat \tau$ are too narrow and undercover,
whereas the leave-one-out confidence intervals are correctly calibrated and achieve coverage.
Perhaps surprisingly, this occurs even though the debiased confidence intervals are wider than the leave-one-out confidence intervals.
Indeed,
across 5000 replications,
the average value of $\widehat \tau / \sqrt{1-\rho^2}$ was $2.97$ $(1.6\mathrm{e}{-3})$ and $\widehat \tau_{\mathrm{loo}}^1$ 
was $2.88$ $(1.5\mathrm{e}{-3})$, which gives a $p$-value for a non-zero difference in means of $2.2\mathrm{e}{-16}$ and a $95\%$ confidence interval for the difference in means of $[.080,.088]$.
This discrepancy is also visually apparent in the right plot of Figure \ref{FigCompare},
in which the debiased confidence intervals tend to be wider. 
If the correct standard error of \cite{Bellec2019SecondOP} were used for the debiased Lasso, so that the intervals would achieve coverage, these intervals would be wider still.
On the right plot of Figure \ref{FigCompare}, we display histograms of the test statistic $\widehat \theta_1^{\de}$ and $\xi_j$ across 5000 replications.
We see that the leave-one-out confidence interval's test statistic has smaller variance than the debaised Lasso test statistic, indicative of the fact that in this case we may achieve more precise inference with the leave-one-out construction than the debiased construction.

\begin{figure}[h!]
	\centerline{
	\includegraphics[width=.5\textwidth]{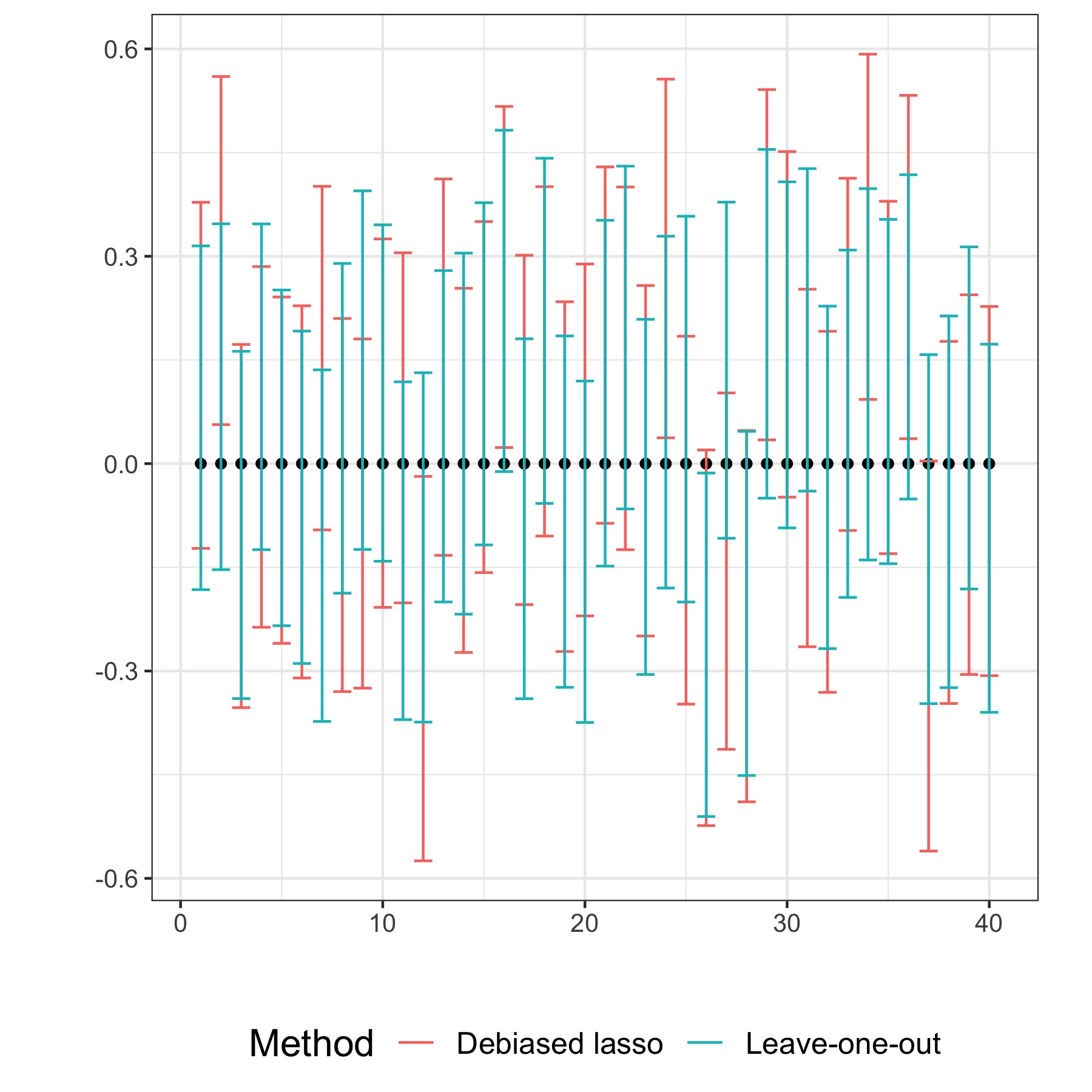}
	\includegraphics[width=.5\textwidth]{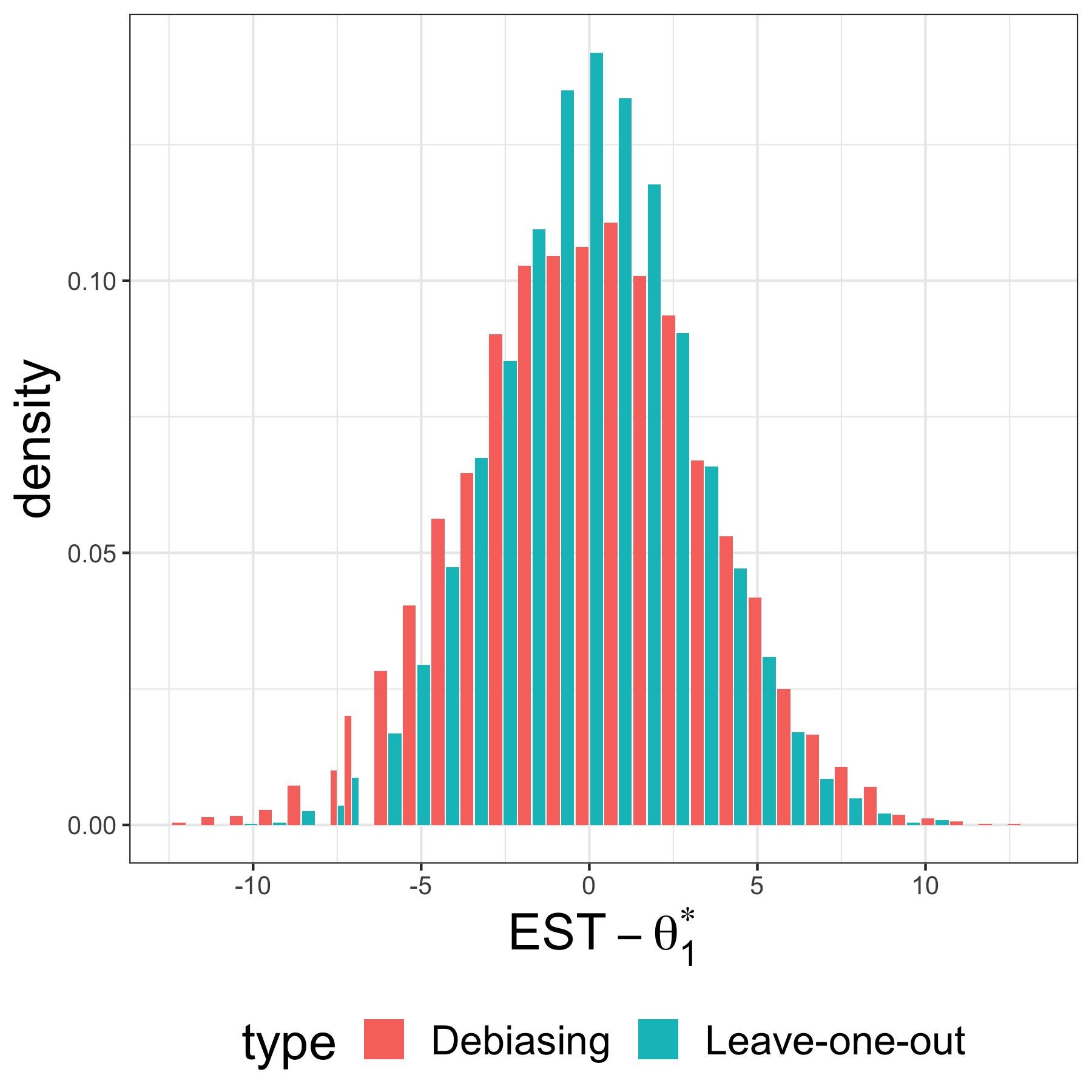}}
	\caption{The debiased Lasso and leave-one-out confidence intervals for $p = 1000$, $n=500$, $s=100$, 
	$\rho^2 = .5$, $\sigma = 1$, $\lambda = \sqrt{2\sigma \log(p/s)/n} = .096$, 
	$\thetastar = 3\sqrt{s}\lambda \bv$, 
	where $\bv = (0, \mathbf{1}_s, \bzero_{p-s-1})/\sqrt{s}$ and $\mySigma = \Ind_p + \rho \be_1 \bv^\top + \rho \bv \be_1^\top$. 
	On the left, we plot the debiased confidence interval $\mathsf{CI}_1^\de$ and leave-one-out confidence intervals $\mathsf{CI}_1^{\mathsf{loo}}$ for $\theta_1^*=0$. On the right, we plot histograms of $\widehat \theta_1^\de$ and $\xi_1$ across 5000 replications.}
	\label{FigCompare}
\end{figure}

\subsection{Non-Gaussian designs}

The results described in this work are proven under correlated Gaussian designs. 
When covariates are independent, numerical simulations and universality arguments in previous work suggest exact asymptotic characterizations
still hold for independent but possibly non-Gaussian covariates (see e.g.~\cite{bayati2015universality,oymak2018universality,montanari2017universality} for rigorous universality results).
Moreover, such universality phenomena are also expected to hold beyond the linear models: 
for instance, \cite{sur2019modern} (in Figure 9) present simulations for logistic regression with independent but non-Gaussian covariates whose behavior agrees with the corresponding  asymptotic predictions for independent Gaussian covariates. 
%Nevertheless, these predictions are incorrect when covariates are correlated.
%This suggests that the most severe limitation of the existing exact asymptotic theory  is not the Gaussianity
% assumption but rather the independence assumption. 
%It is this assumption that the current paper weakens.

Here we provide some numerical evidence which suggests that our theory describes the behavior of the Lasso under
 some realistic data generating distributions (when the Gaussianity assumption breaks). 
%Theoretical investigations of universality is left for future work. 
We consider the design matrix with covariates generated according to a hidden Markov model.
Hidden Markov models are frequently used for modeling the covariates in genetics applications (see, e.g.~\cite{sesia2018}).
The specification of the hidden Markov model used in our simulation is described in details in Appendix~\ref{SecAdditionalSims}.
The model is such that covariates with indices differing by approximately 10 or less have non-negligible correlation.
The response is generated according to model~\eqref{EqnLM}, with $n = 1280,~p=2000,~s = 0.128 p,$ and $\sigma = 0.2$, and
 all active coordinates of $\mytheta^*$ are set to $1$. 
We run our debiasing procedure with degrees of freedom adjustment for $N_\mathrm{sim} = 10$ 
independent realizations of the data, with the knowledge of the underlying covariance matrix for the covariates.
We then aggregate the standardized and centered debiased Lasso estimates across coordinates and across simulations, 
separately for the inactive and active coordinates, and provide a qq-plot for each; the results
are presented in Figure \ref{FigQQplots_debiasing_hmm}.
It is worth noting that from the simulations, one can see the success of the debiasing procedure with degrees of 
freedom adjustment carries even into the tails of the distribution.  
This phenomenon cannot be justified using prior theory based on independent Gaussian covariates.
\begin{figure}[h!]
  \centerline{\includegraphics[width=.8\textwidth]{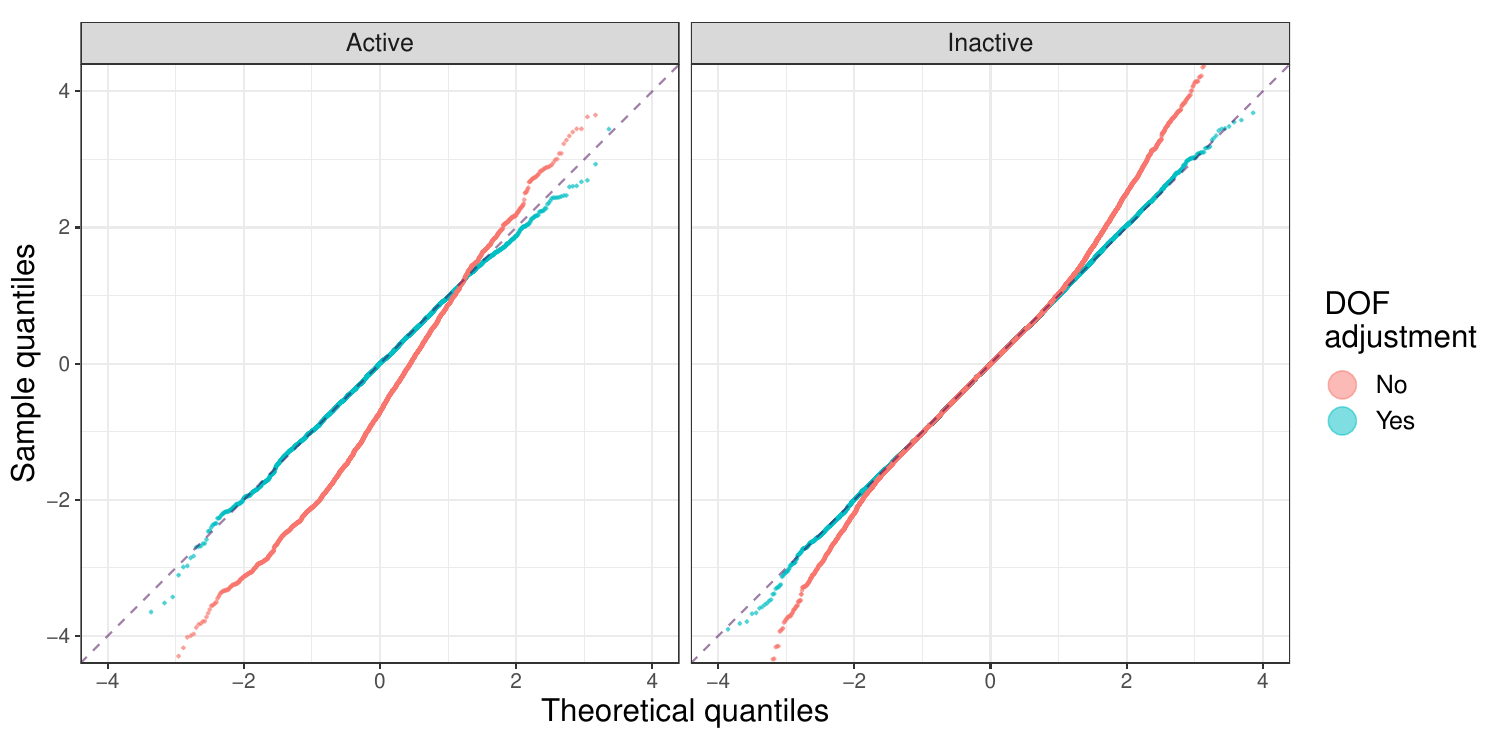}}
\caption{The debiased Lasso with and without degrees-of-freedom (DOF) adjustment for hidden Markov model features.
  Here $n = 1280,p=2000,s = .128 \cdot p,$ and $\sigma = .2$, and all active coordinates of $\mytheta^*$ equal to 1. Quantiles
  and densities are compared with the ones of the standard normal distribution. }
\label{FigQQplots_debiasing_hmm}
\end{figure}

%%%%%%%%%%%%%%%%%%%%%%%%%%%%%%%%%%%%%%%%%%%%%%%%%%%%%%%%%%%%%%%%%%%%%%
\section{Main proof ingredients}
\label{sec:proof-ingredients}

Our proofs are built upon a tight version of Gordon's min-max theorem for convex functions.
Gordon's original theorem \cite{gordon1985some,Gordon1988} is a Gaussian comparison inequality for the minimization-maximization of two related Gaussian processes,
and has several applications in random matrix theory and convex optimization \cite{rudelson2010,raskutti10a}.
In a line of work initiated by \cite{stojnic2013framework} and formalized by \cite{thrampoulidis2015regularized}, the comparison inequality was shown to be tight when the underlying Gaussian process is convex-concave.
This observation has led to several works establishing exact asymptotics for high-dimensional convex procedures, 
including general penalized M-estimators in linear regression \cite{thrampoulidis2015regularized,thrampoulidis2018} and binary classification \cite{Deng2019AMO,Montanari2019,Liang2020APH}.
(We also refer to \cite{bayati2011lasso,amelunxen2014living,donoho2016high,el2018impact,reeves2016replica, barbier2019optimal} for alternative proof techniques to obtain
sharp results in high-dimensional regression 
models, in the proportional asymptotics.)

Earlier work has so far focused on the case of independent features 
or correlated features with unpenalized or ridge-penalized procedures.
Analyzing the Lasso estimator under general Gaussian designs, however,  requires overcoming several technical challenges, as the $\ell_{1}$-penalty breaks the isometry underlying the procedure. 
In this section, we summarize our proof strategy, 
emphasizing the technical innovations that are required in the context of general correlated designs. 
Our work builds on the approach of \cite{miolane2018distribution}, which  studied the Lasso and debiased Lasso estimators in the case $\mySigma = \Ind_p$.

\paragraph*{Control of the Lasso estimate} 
We find it useful to first rewrite the Lasso optimization objective as
\begin{equation}
\label{eq:cuC-def}
	\cuC(\bv):= \frac1{2n} \|\sigma \bz - \bX\mySigma^{-1/2}\bv\|_2^2 + \frac{\lambda}{\sqrt{n}}\Big( \| \thetastar + \mySigma^{-1/2} \bv \|_1 - \| \thetastar \|_1 \Big)\,.
\end{equation}
Here we introduce the \emph{prediction error} vector $\bv := \mySigma^{1/2} (\mytheta - \thetastar)$.
The variable $\bv$ is used to whiten the design matrix and isolate the dependence of the objective on it.
Indeed, $\bX \mySigma^{-1/2}$ has entries distributed i.i.d. from $\normal(0,1)$, and we have expanded $\by$ to reveal its dependence on $\bX$.
We denote by $\bvhat$ the minimizer of $\cuC(\bv)$, i.e., $\bvhat := \mySigma^{1/2}(\thetahat - \thetastar)$.
By a standard argument, Gordon's min-max theorem implies that the Lasso optimization behaves, in a certain sense, like the optimization of the simpler objective
\begin{equation}
\label{eq:def-cuL-main}
	\cuL(\bv):= \frac12 \left(\sqrt{\sigma^2 + \|\bv\|_2^2} - \frac{\langle \bg , \bv \rangle}{\sqrt{n}}\right)_+^2 + \frac\lambda{\sqrt{n}} \left(\|\thetastar + \mySigma^{-1/2}\bv \|_1 - \|\thetastar\|_1\right)\,,
\end{equation}
which we call \emph{Gordon's objective}.
The precise statement is as follows.
\begin{lemma}[Gordon's lemma]
\label{LemGordonMain}
	The following statements hold true. 

	\begin{enumerate}[label=(\alph*)]
		
		\item % a
		Let $D \subset \reals^p$ be a closed set.
		For all $t \in \reals$,
		\begin{equation}
			\mprob\left(  \min_{\bv \in D} \cuC(\bv) \leq t \right)
			\leq 
			2\mprob\left( \min_{\bv \in D} \cuL(\bv) \leq t \right)\,.
		\end{equation}

		\item % b
		Let $D \subset \reals^p$ be a closed, convex set.
		For all $t \in \reals$,
		\begin{equation}
			\mprob\left(  \min_{\bv \in D} \cuC(\bv) \geq t \right)
			\leq 
			2\mprob\left( \min_{\bv \in D} \cuL(\bv) \geq t \right)\,.
		\end{equation}

	\end{enumerate}
\end{lemma}
By studying  Gordon's objective, and comparing the value of $ \min_{\bv \in D} \cuL(\bv)$ for suitable choices of the set $D$, we can extract properties of $\bvhat$ and hence $\thetahat$.
In particular, in Theorem \ref{ThmControlLassoEst}, we compare the value taken for $D=\reals^p$ and 
	\begin{equation}
	\label{eq:suboptimality-set}
		D 
		= 
		\left\{\mytheta \in \reals^p \Bigm| \Big| 
		\phi\big(\thetastar + \mySigma^{-1/2}\bv\big) 
		- \E\Big[ \phi\big(\thetahat^f\big)\Big]\Big|
		 > \epsilon \right\}\,,
	\end{equation}
	where $\thetahat^f$ is defined by Eq.~\eqref{EqnSTH} with $\taustar,\zetastar$ the unique solution to Eqs.~\eqref{EqnEqn1} and \eqref{EqnEqn2}.       
The argument is carried out in detail in Appendix \ref{SecPfThmLassoL2}.

This discussion clarifies that we can control the behavior of the Lasso objective only insofar as we can control the behavior of Gordon's objective. 
The major technical challenge to apply this approach to general correlated designs
is in relating the minimizer of Gordon's objective to the fixed design estimator $\thetahat^f$.
In particular, this requires showing that the solution $(\taustar, \zetastar)$ of
Eqs.~\eqref{EqnEqn1} and \eqref{EqnEqn2} is unique and bounded in terms of simple model parameters (see Lemma \ref{LemTauZetaBounds}).

Although several parts of our argument are similar to the arguments of \cite{miolane2018distribution},
establishing existence, uniqueness, and boundedness of $\tau^*,\zeta^*$ requires entirely new techniques. 
Generalizing an idea introduced in \cite{Montanari2019}, 
we control the solutions Eqs.~\eqref{EqnEqn1} and \eqref{EqnEqn2} by 
showing that these equations are the  KKT conditions for a certain convex optimization problem on the infinite dimensional Hilbert space $L^2(\reals^p;\reals^p)$. 
To be more specific, the optimization problem is
\begin{align}\label{eq:Hilbert}
	\min_{\bv \in L^2} \cuE(\bv) := \min_{\bv \in L^2} \left\{\frac12 \Big(\sqrt{\|\bv\|_{L^2}^2 + \sigma^2} - \frac{\langle \bg , \bv \rangle_{L^2}}{\sqrt{n}}\Big)_+^2 + \frac\lambda{\sqrt{n}} \E\left[\|\thetastar + \mySigma^{-1/2}\bv(\bg)\|_1 - \|\thetastar\|_1\right]\right\}\,.
\end{align}
The objectives $\cuL$ and $\cuE$ are closely related, but their arguments belong to different spaces. 
The objective $\cuL$ takes vectorial arguments $\bv \in \reals^p$;
the objective $\cuE$ takes functional arguments $\bv: \reals^p \rightarrow \reals^p$.
Both objectives are convex.
In Appendix \ref{SecFixedPtSoln},
we show that $\bv \in L^2$ is a minimizer of $\cuE$ if and only if $\bv(\bg) = \eta(\mySigma^{1/2}\thetastar + \taustar \bg ; \zetastar)$ for $\taustar$, $\zetastar$ a solution to the fixed point equations. 
This follows from showing that Eqs.~\eqref{EqnEqn1} and \eqref{EqnEqn2} correspond to KKT conditions for the minimization of $\cuE$.
Further, we show that $\cuE$ diverges to infinity as $\|\bv\|_{L^2}\rightarrow \infty$ and is strictly convex in a neighborhood of any minimizer, whence a minimizer exists, and it is unique.
We are then able to conclude that the fixed point equations also have a unique solution. 
We defer the details of this argument to Appendix \ref{SecFixedPtSoln}.

Controlling the size of the fixed point parameters relies on bounding the norm of the minimizer of $\cuE$.
Again, our approach is geometric: rather than analyzing the fixed point equations directly, we study the growth of the objective $\cuE$ as $\|\bv\|_{L^2}$ diverges.
The functional Gaussian width \eqref{EqnGaussianWidth} controls this growth.
This explains the centrality of the Gaussian width $\cuG(\bx,\mySigma)$ in our analysis.
In fact, under only a sparsity constraint on $\thetastar$, we can control the growth $\cuE$ in $\|\bv\|_{L^2}$ in an $n$-independent way only when $\cuG(\bx,\mySigma) < \sqrt{\delta}$ where $\bx \in \partial \|\thetastar\|_1$ (see, also, the proof of Proposition \ref{prop:DT-necessary}).
The detailed argument bounding the fixed point parameters is in Appendix \ref{SecFixedPtSoln}.

The present approach is significantly more general both than the one of \cite{miolane2018distribution}, which studies the Lasso for $\mySigma = \Ind_p$, and of \cite{Montanari2019} which studies binary
classification under a ridge-type regularization.
When $\mySigma = \Ind_p$, the Lasso estimator in the fixed-design model is separable, and  Eqs.~\eqref{EqnEqn1} and \eqref{EqnEqn2} simplify because
\begin{equation}
\begin{gathered}
	\risk(\tau^2,\zeta) = \E_{\Theta,G}[ (\eta_{\mathrm{soft}}(\Theta^* + \tau G,\lambda/\zeta) - \Theta^*)^2]\,,\\
	\frac{\df(\tau^2,\zeta)}{p} = \mprob(\eta_{\mathrm{soft}}(\Theta^* + \tau G,\lambda/\zeta) \neq 0)\,,
\end{gathered}
\end{equation}
where $\Theta^*\sim \frac1p \sum_{j=1}^p \delta_{\sqrt{n}\theta^*_j}$ independent of $G \sim \normal(0,1)$,
and $\eta_{\mathrm{soft}}(y;\zeta) :=(|y|-\zeta)_+\sign(y)$.
Hence --- in that case --- existence and uniqueness of the solution of Eqs.~\eqref{EqnEqn1} and \eqref{EqnEqn2} can be proved by analyzing the explicit form of these equations.

Also, our approach is more general than the one of \cite{Montanari2019}, which constructs a
Hilbert-space optimization problem by taking the $n,p\to\infty$ limit of the Gordon's
problem. In the present case, since we intend to establish a non-asymptotic control, for finite $n,p$ there is no natural sequence of covariances in which to embed $\mySigma$.

\paragraph*{Control of the Lasso sparsity}
It is not feasible to directly control quantity $\|\thetahat\|_0/n$ using Theorem \ref{ThmControlLassoEst} with $\phi(\mytheta) = \|\mytheta\|_0/n$ because this function 
is not Lipschitz or even continuous.
Instead, we establish lower and upper bounds on the sparsity separately.
%We establish the lower bound by analyzing the Lasso estimate $\thetahat$.
%We establish the upper bound by analyzing the subgradient $\bthat$.

To explain the argument, 
define
\begin{equation}
	\widehat \bt = \frac1{\lambda \sqrt{n}} \bX^\top (\by - \bX \thetahat),
\end{equation}
and observe that by the KKT conditions for Eq.~\eqref{EqnOrgRisk},
$\widehat \bt \in \partial \| \thetahat \|_1$.
Define for any $\mytheta\in\reals^p$ the \emph{$\epsilon$-strongly active} coordinates of $\mytheta$ to be $\{j \in [p] \mid |\theta_j| > \epsilon / \sqrt{n}\}$.
Likewise, for any $\bt\in\reals^p$ define the \emph{$\epsilon$-strongly inactive} coordinates of $\bt$ to be $\{j \in [p] \mid |t_j| < 1 - \epsilon\}$ (this definition is motivated by the fact that if $\bt$ is the sub-gradient of the Lasso, if $|t_j| < 1-\epsilon$ then $\theta_j = 0$ and $t_j$ would have to change by at least $\epsilon$ for $\theta_j$ to become active).
Our argument relies on the following two facts (here $\thetahat$ is, as always, the Lasso estimate,
and $\bthat$ is the subgradient of Eq.~\eqref{EqnDefSG}): 
\begin{equation}
\label{eq:strongly-active-inf}
	\text{if}\;\; \frac{\|\thetahat\|_0}{n} \leq 1 - \zetastar - \epsilon, ~~
	\text{then}\; \inf_{\mytheta}\left\{ \|\thetahat - \mytheta\|_2 \Bigm| \frac{|\{j\mid |\theta_j| > \epsilon / \sqrt{n}\}|}{n} > 1 - \zetastar - \frac\epsilon2\right\} \geq \sqrt{\frac{\epsilon^3}{2}}\,,
\end{equation}
and
\begin{equation}
\label{eq:strongly-inactive-inf}
	\text{if}\;\; \frac{\|\thetahat\|_0}{n} \geq 1 - \zetastar + \epsilon\,,~~
	\text{then}\; \inf_{\bt} \left\{ \|\bthat - \bt\|_2 \Bigm| \frac{|\{j\mid |t_j| < 1 - \epsilon\}|}{n} > 1 - \zetastar - \frac\epsilon2\right\} \geq \sqrt{\frac{\epsilon^3}{2}}\,.
\end{equation}

The first implication holds because the vectors $\mytheta$ and $\thetahat$ differ by at least $\epsilon/\sqrt{n}$ in $n\epsilon/2$ coordinates; namely, in those coordinates in which $\mytheta$ is $\epsilon$-strongly active and $\thetahat$ is inactive.
The second implication holds similarly.
In words, vectors which are very sparse are separated in Euclidean distance from vectors with many $\epsilon$-active coordinates; similarly, subgradients with many active coordinates are separated in Euclidean distance from vectors with many $\epsilon$-inactive coordinates.

To proceed, we leverage the following fact: for any set $D \subset \reals^p$ which contains the fixed-design Lasso estimate $\thetahat^f$ with high-probability, the random design Lasso estimate $\thetahat$ is close to $D$ with high-probability.
Similarly, for any set $D \subset \reals^p$ which contains the fixed-design subgradient $\bthat^f$ with high-probability, 
the random-design subgradient $\bthat$ is close to $D$ with high-probability.
We provide control of the subgradient which is analogous to the control we provide of the Lasso estimate in Lemma \ref{LemSubgradient} of the appendices.
A similar statement holds for the Lasso estimate, and developed in the proof of Theorem \ref{ThmControlLassoEst}.
Taking $D$ to be the set over which the infimum in Eq.~\eqref{eq:strongly-active-inf} (resp.\ Eq.~\eqref{eq:strongly-inactive-inf}) is taken, 
we can conclude $\|\thetahat\|_0/n > 1 - \zetastar - \epsilon$ (resp.\ $\|\thetahat\|_0/n < 1 - \zetastar + \epsilon$) with high-probability
as soon as we can show $\thetahat^f \in D$ (resp.\ $\bthat^f \in D$) with high-probability.
The details of this argument are carried out in Appendix \ref{SecThmLassoSparsity}.

\paragraph*{Control of the debiased Lasso}
We may write the debiased Lasso as a function of the Lasso estimate $\thetahat$, 
the subgradient $\bthat$,
and the Lasso sparsity $\|\thetahat\|_0/n$:
\begin{align}
	\bthetahatd = \thetahat + \frac{\lambda \mySigma^{-1} \bthat / \sqrt{n} }{1 - \|\thetahat\|_0/n}\,.
\end{align}
Because $1-\|\thetahat\|_0/n$ concentrates on $\zetastar$ by Theorem \ref{ThmLassoSparsity},
the debiased Lasso is with high-probability close to 
\begin{align}
\label{eq:debiased-lasso-estimate-sg-form}
	\thetahat + \frac{\lambda\mySigma^{-1} \bthat/\sqrt{n}}{\zetastar} \,.
\end{align}
Our goal is to show that $\thetahat + \frac{\lambda\mySigma^{-1} \bthat/\sqrt{n}}{\zetastar} - \thetastar$ is approximately Gaussian noise with zero mean and covariance ${\taustar}^2\mySigma^{-1}$.
Heuristically, if we replace the Lasso estimate and subgradient by their fixed-design counterparts,
we get 
\begin{align}
	\thetahat^f + \frac{\lambda\mySigma^{-1} \bthat^f/\sqrt{n}}{\zetastar} - \thetastar
	= \thetahat^f - \thetastar + \mySigma^{-1} \mySigma^{1/2} (\by^f - \mySigma^{1/2} \thetahat^f) = \taustar \mySigma^{-1/2} \bg / \sqrt{n}\,,
\end{align}
where in the first inequality we have used that $\frac{\lambda}{\sqrt{n}\zetastar} \bthat^f = \mySigma^{1/2}(\by^f-\mySigma^{1/2}\thetahat)$ by the KKT conditions for the optimization \eqref{EqnSTH}.
Thus, we would like to justify the heuristic replacement of the random design quantities with their fixed-design counterparts.

It turns out that it is not straightforward to justify this heuristic,
and here again we require an entirely new arugment compared to that which appears in \cite{miolane2018distribution}.
The challenge is as follows: Theorem \ref{ThmControlLassoEst} and Lemma \ref{LemSubgradient} compare the distributions of $\thetahat$ and $\bthat$ to their fixed design counterparts individually, but does not say anything about their joint distribution.
The paper \cite{miolane2018distribution} addresses this challenge by providing a simple characterization of $\hat \mu = \frac1p \sum_{j=1}^p \delta_{\sqrt{n} \widehat \theta_j}$ 
and $\hat \mu' := \frac1p\sum_{j=1}^p \delta_{\widehat t_j}$,
and arguing that their is only one joint distribution $\frac1p \sum_{j=1}^p \delta_{\sqrt{n} \widehat \theta_j,\widehat t_j}$ which is consistent with the Lasso KKT conditions and is consistent with the marginal distributions.
Because we are unable to arrive at a simple characterization of the empirical distributions $\hat \mu = \frac1p \sum_{j=1}^p \delta_{\widehat \theta_j}$ 
and $\hat \mu' := \frac1p\sum_{j=1}^p \delta_{\widehat t_j}$ in the correlated design case, 
we were unable to follow a strategy similar to \cite{miolane2018distribution}.

Instead, we resort to a smoothing argument.
For penalized regression estimators with differentiable penalties, 
the subgradient $\bthat$ is a function of the estimate $\thetahat^f$.
Thus, for smooth procedures, the expression corresponding to Eq.~\eqref{eq:debiased-lasso-estimate-sg-form} is a deterministic\footnote{In particular, it has no dependence on the design $\bX$ except through $\thetahat$.} function only of the estimate.
Thus, the replacement of the quantities in \eqref{eq:debiased-lasso-estimate-sg-form} by their fixed-design counterparts can be justified via analysis of the distribution of the estimate $\thetahat$ individually.
The Lasso penalty $\| \mytheta\|_1$ is not smooth, so that $\widehat \bt$ is not a deterministc function of $\thetahat$.
To handle this, we introduce the $\alpha$-smoothed Lasso, in which we replace the $\ell_1$-penalty by a smooth approximation in the original Lasso objective \eqref{EqnOrgRisk}.
We prove a characterization of the $\alpha$-smoothed Lasso analogous to Theorem \ref{ThmControlLassoEst}, 
and use this to establish the success of the debiasing procedure corresponding to the smoothed estimator.
Finally, we argue that the debiased Lasso estimate is well-approximated by the debiased $\alpha$-smoothed Lasso estimate for small enough smoothing parameter, and show that Theorem \ref{ThmDBLasso} follows.
The details of this argument are provided in Appendix~\ref{SecPfControlDBLasso}.

\begin{supplement}
\sname{Supplement A}
itle{Supplement to `The Lasso with general Gaussian designs with applications to hypothesis testing.'}
\slink[doi]{COMPLETED BY THE TYPESETTER}
\sdescription{The supplement contains proofs and technical details that were omitted from the main text.}
\end{supplement}

\bibliographystyle{abbrv}
\bibliography{lasso}

\newpage

\setcounter{page}{1}
\setcounter{section}{0}
\renewcommand\thesection{\Alph{section}}

\begin{center}
\large \textbf{ Supplement to `The Lasso with general Gaussian designs with applications to hypothesis testing.'}
\end{center}

\section{Preliminaries}
\label{SecPfandDetails}

%%%%%%%%%%%%%%%%%%%%%%%%%%%%%%%%%%%%%%%%%%%%%%%%%%%%%%%%%%%%%%%%%%%%%%

\subsection{A Gaussian width tradeoff}

Define the descent cone
\begin{equation}
\label{EqnDescentCone}
	\cuD(\bx,\mySigma) := \left\{ \bv \in L^2 \Bigm| \E\big[F(\bv;\bx,\mySigma) \big] \leq 0\right\}\,.
\end{equation}
The Gaussian width in Eq.~\eqref{EqnGaussianWidth} can be seen as the maximal value of the correlation $\< \bv , \bg \>_{L^2} / \| \bv \|_{L^2} \| \bg \|_{L^2}$ subject to the constraint $\bv \in \cuD(\bx,\mySigma)$. 
In this section, 
we quantify the sensitivity of this maximal correlation to small relaxations of this constraint.
This is a central tool in establishing bounds on the solutions to Eqs.~\eqref{EqnEqn1} and \eqref{EqnEqn2} (see the proof of Theorem \ref{ThmControlLassoEst}(a) in Section \ref{SecFixedPtSoln}).

\begin{lem}
\label{LemGaussWidthTradeoff}
Fix $\mySigma \in \S_{\ge 0}^p$ with singular values bounded $0 < \kappamin \leq \kappa_j(\mySigma) \leq \kappamax < \infty$ for all
$j$. Define $\kappacond:=\kappamax/\kappamin$.
	Let $\bx \in \{-1,0,1\}^p$ with $\|\bx\|_0/p \geq \numin > 0$.
	Let $S = \supp(\bx)$.
	Then, for any $\bv \in L^2$ and any $\epsilon > 0$, we have either
	\begin{equation}\label{EqnWidthTradeoff1}
		 \frac1{\sqrt{p}} \langle \bv , \bg \rangle_{L^2} \leq \big(\cuG(\bx,\mySigma) + \epsilon\big) \|\bv\|_{L^2}\,,
	\end{equation}
	or 
	\begin{equation}\label{EqnWidthTradeoff2}
	    \frac1{\sqrt{p}} \E\Big[F(\bv;\bx,\mySigma)\Big] \geq \frac{\numin^{1/2}\epsilon}{ \kappamax^{1/2}(2+\kappacond)} \|\bv\|_{L^2}\,.
	\end{equation}
\end{lem}

\begin{proof}[Proof of Lemma \ref{LemGaussWidthTradeoff}]
	Under the change-of-variables $\bw = \sqrt{p}\,  \mySigma^{-1/2} \bv$,\footnote{Note that, for the purposes of this proof, we have normalized by $\sqrt{p}$, which is distinct from the usage of $\bw$ elsewhere.} 
	we may alternatively write \eqref{EqnGaussianWidth} as
	\begin{equation}\label{EqnGaussianWidthW}
		\cuG(\bx,\mySigma) 
		= 
		\sup_{ \substack{\bw \in \cuD(\bx,\Ind_p) \\ \|\mySigma^{1/2}\bw\|_{L^2}^2 \leq p} } \frac1p \E\Big[ \bw^\top \tilde \bg \Big]\,,
	\end{equation}
	where $\tilde \bg = \mySigma^{1/2}\bg$ and $\bg$ is interpreted as the identity function in $L^2$.
	The Lagrangian for this problem reads:
	\begin{equation}
		\cuL(\bw; \kappa,\xi ) := \frac1p\E\left[\bw^\top \tilde \bg\right] + \frac\kappa2 \left(1 -  \frac1p \E\left[\|\mySigma^{1/2}\bw\|_2^2\right] \right) - \frac{\xi}{p} \E \Big[ \sum_{j \in S} x_jw_j + \|\bw_{S^c}\|_1\Big]\,,
	\end{equation}
	where the Langrange multipliers $\kappa,\xi$ are restricted to be non-negative.
	First, we bound the dual optimal Lagrange multipliers.
	We bound
	\begin{align}
		\frac\kappa2 + \frac1p&\E\left[\bw^\top \tilde \bg - \frac{\kappa\kappamin}{2} \|\bw\|_2^2 - \xi \Big(\sum_{j \in S} x_j w_j + \|\bw_{S^c}\|_1\Big)\right]
			\geq \cuL(\bw;\kappa,\xi) 
			\nonumber\\ 
			&\;\;\;\;\;\;\;\;\;\;\;\;\geq 
			\frac\kappa2 + \frac1p\E\left[\bw^\top \tilde \bg - \frac{\kappa \kappamax}{2} \|\bw\|_2^2 - \xi \Big( \sum_{j \in S} x_j w_j + \|\bw_{S^c}\|_1 \Big)\right]\,.\label{EqnLagrBounds}
	\end{align}
	The expected value appearing in the upper bound is maximized by maximizing the integrand for each value of $\tilde \bg$.
	Because the integrand is separable across coordinates, we may do this explicitly.
	The maximal value of the integrand at fixed $\tilde \bg$ is 
	\begin{align}
		\frac\kappa2 + \frac1{2p\kappa\kappamin}\sum_{j \in S}(\tilde g_j - \xi x_j)^2 + \frac{1}{p\kappa \kappamin }\sum_{j \in S^c} \left( \frac{\tilde g_j^2}{2} - \xi \sfM_\xi(\tilde g_j) \right)\,,
	\end{align}
	where $\sfM_\xi(\tilde g_j)$ is the Moreau envelope of the $\ell_1$-norm
	\begin{align}
		\sfM_\xi(y)
		:=
		\inf_{x \in \reals} \left\{\frac1{2\xi}(y-x)^2 + |x|\right\}\,.
	\end{align}
	Because $\xi \sfM_\xi(\tilde g_j) \geq 0$, 
	we have $\E[\tilde g_j^2/2 - \xi\sfM_\xi(\tilde g_j)] \leq \E[\tilde g_j^2/2] \leq \E[(\tilde g_j - \xi x)^2/2] \leq (\kappamax + \xi^2)/2$ whenever $x =\pm1$.
	Thus,
	\begin{equation}
		\frac\kappa2 + \frac1{\kappa} \frac{\kappacond + \xi^2/\kappamin}{2} \geq \sup_{\bw \in L^2} \cuL(\bw;\kappa,\xi).
	\end{equation}
	This further implies that 
	\begin{equation}
	\label{eq:cuL-ub}
		\inf_{\kappa,\xi \geq 0} \sup_{\bw \in L^2} \cuL(\bw;\kappa,\xi) \leq \sup_{\bw \in L^2} \cuL(\bw;1,\kappamin^{1/2}) = 1 + \frac{\kappacond}2\,.
	\end{equation}
	Similarly, maximizing the right-hand side of Eq.~\eqref{EqnLagrBounds} explicitly and using $\tilde g_j^2/2 - \xi\sfM_\xi(\tilde g_j) \geq 0$,
	\begin{equation}
	\label{eq:cuL-lb}
		\sup_{\bw \in L^2} \cuL(\bw;\kappa,\xi) 
		\geq 
		\frac\kappa2 + \frac1{\kappa} \frac{|S|(1/\kappacond + \xi^2/\kappamax)}{2p}\,.
	\end{equation}
	If either $\kappa/2 > 1 + \kappacond/2$ or $\xi^2/\kappamax > 4(1 + \kappacond/2)^2/(|S|/p)$, then $ \sup_{\bw \in L^2} \cuL(\bw;\kappa,\xi) > 1 + \kappacond/2$.
	Combining the previous two displays,
	we conclude that $\inf_{\kappa,\xi \geq 0} \sup_{\bw \in L^2} \cuL(\bw;\kappa,\xi)$ is achieved at some 
	\begin{equation}
	\label{EqnLagrMultUB}
		\kappastar \leq 2 + \kappacond \;\;\;\;\;\; \text{and} \;\;\;\;\;\; \xistar \leq \frac{\kappamax^{1/2}(2 + \kappacond)}{(|S|/p)^{1/2}}\,.
	\end{equation}
	Since the constraints on $\bw$ are strictly feasible, strong duality holds:
	\begin{equation}
		\sup_{\bw \in L^2} \cuL(\bw;\kappastar,\xistar) = \cuG(\bx,\mySigma)\,.
	\end{equation}

	The dual optimal variable $\xi^*$ quantifies the tradeoff we seek to control, as we now show.
	For any function $\bw:\reals^p\rightarrow \reals^p$,
	let $\bar \bw : \reals^p \rightarrow \reals^p$ be defined by $\bar \bw(\bg) = \sqrt p \bw(\bg) / \E[\|\bw(\bg)\|_{\mySigma}^2]^{1/2}$,
	where $\|\bw\|_{\mySigma}^2 = \bw^\top \mySigma \bw$.
	Then
	\begin{equation}
	\begin{aligned}
		&\frac1p \langle \bar \bw , \tilde \bg \rangle_{L^2} 
			- \frac{\xistar}p \E\left[ \sum_{j \in S} x_j \bar w_j(\bg) 
				+ \|\bar \bw_{S^c}(\bg)\|_1\right] \\
		&\quad\quad\quad\quad =
		\frac1p \E[\bar \bw(\bg)^\top \tilde \bg  ] 
			+ \frac{\kappastar}2 \left(1 - \frac1p \E[\|\bar\bw\|_{\mySigma}^2] \right)
			- \frac{\xistar}p \E\left[ \sum_{j \in S} x_j \bar w_j(\bg) 
				+ \|\bar\bw_{S^c}(\bg)\|_1\right] \\
		&\quad\quad\quad\quad \leq 
		\sup_{\bw \in L^2} \cuL(\bw;\kappastar,\xistar) 
			= \cuG(\bx,\mySigma)\,,
	\end{aligned}
	\end{equation}
	where in the first equality we used that $\E[\|\bar \bw(\bg)\|_{\mySigma}^2]/p = 1$.
	We conclude that for any $\epsilon > 0$,
	\begin{equation}
		\text{either} 
		\;\;\;\;\;\; 
			\frac1p \langle \bar \bw , \tilde \bg \rangle_{L^2} 
			\leq 
			\cuG(\bx,\mySigma) + \epsilon 
		\;\;\;\;\;\; 
		\text{or} 
		\;\;\;\;\;\; 
			\frac1p \E\left[\sum_{j \in S} x_j \bar w_j(\tilde\bg) + \|\bar \bw(\bg)_{S^c}\|_1\right] 
			\geq 
			\frac{\epsilon}{\xistar}\,.
	\end{equation}
	Plugging in $\bw = \E[\|\bw(\bg)\|_{\mySigma}^2]^{1/2}{\bar\bw}/\sqrt p$ 
    and the upper bound on $\xistar$ in \eqref{EqnLagrMultUB},
	the lemma follows.
\end{proof}

\subsection{Some properties of the Gaussian width}
\label{sec:GW-prop}

\begin{proof}[Proof of Lemma \ref{lem:width-under-corr}]
	Notice that 
	\begin{align}
	\label{EqSigmaToIComparison}
		\sup_{\substack{
						\bv \in \cuD(\bx,\mySigma)\\ 
						\|\bv\|_{L^2}^2\leq 1}
						} \frac1{\sqrt{p}} \langle \bv , \bg \rangle_{L^2}
			= 
			\sup_{\substack{
						\bw \in \cuD(\bx,\Ind_p)\\ 
						\|\bw\|_{L^2}^2\leq 1}
						} \frac{\|\bw\|_{L^2}\langle \mySigma \bw , \bg \rangle_{L^2}}{\sqrt{p}\|\mySigma \bw\|_{L^2}} 
			\leq 
			\frac{\kappamax^{1/2}}{\kappamin^{1/2}}\sup_{\substack{
						\bw \in \cuD(\bx,\Ind_p)\\ 
						\|\bw\|_{L^2}^2\leq 1}
						} \frac{\langle \mySigma \bw , \bg \rangle_{L^2}}{\sqrt{p}\kappamax^{1/2}}\,,
	\end{align}
	where in the equality we have used that $\bw \leftrightarrow \|\bw\|_{L^2}\mySigma \bw /\|\mySigma \bw\|_{L^2}$ is a bijection between the sets over which the suprema are taken,
	and in the inequality we have used that the supremum is positive (because $\bw = \bzero$ is feasible)  and $\|\bw\|_{L^2}/\|\mySigma \bw\|_{L^2} \leq \kappamin^{-1/2}$.
	The Lagrangian for the maximization on the right-hand side is 
	\begin{equation}
	\begin{aligned}
		\cuL_{\mySigma}(\bw; \kappa,\xi ) 
			&:= 
			\frac1{\sqrt{p}\kappamax^{1/2}}\E\big[\bw^\top \mySigma^{1/2} \bg\big] 
			+ 
			\frac\kappa2 \left(1 - \E\left[\|\bw\|_2^2\right] \right) 
			- 
			\frac{\xi}{\sqrt{p}} \E \Big[ \sum_{j \in S} x_jw_j + \|\bw_{S^c}\|_1\Big]\\
			&=
			\frac\kappa2 + \E\left[\frac{\bw^\top \mySigma^{1/2} \bg}{\sqrt{p}\,\kappamax^{1/2}} - \frac\kappa2\|\bw\|_2^2 - \frac{\xi}{\sqrt{p}}\left(\sum_{j \in S} x_jw_j + \|\bw_{S^c}\|_1\right) \right].
	\end{aligned}
	\end{equation}
	The optimal $\bw \in L^2$ maximizes the integrand for almost every $\bg$,
	whence
	\begin{equation}
	\begin{aligned}
		\sup_{\bw \in L^2}\cuL_{\mySigma}(\bw; \kappa,\xi ) 
			=
			\frac\kappa2 + \E\left[\sup_{\bw \in \reals^p} \left\{\frac{\bw^\top \mySigma^{1/2} \bg}{\sqrt{p}\,\kappamax^{1/2}} - \frac\kappa2\|\bw\|_2^2 - \frac{\xi}{\sqrt{p}}\left(\sum_{j \in S} x_jw_j + \|\bw_{S^c}\|_1\right) \right\} \right].
	\end{aligned}
	\end{equation}
	We emphasize that the dummy variable $\bw$ is in $L^2$ on the left-hand side and $\reals^p$ on the right-hand side.
	We apply the Sudakov-Fernique inequality \cite[Theorem 2.2.3]{adler2009random} to upper bound the expectation in the preceding display.
	Indeed, for $\bw,\bw'\in\reals^p$, 
	we have $\E[\bw^\top \mySigma^{1/2}\bg/\kappamax^{1/2}] = \E[\bw^\top \bg] = 0$ and $\Var((\bw-\bw')^\top \mySigma^{1/2} \bg / \kappamax^{1/2}) \leq \Var((\bw - \bw')^\top \bg)$ because $\|\mySigma^{1/2}/\kappamax^{1/2}\|_{\mathrm{op}} \leq 1$.
	Thus, the Sudakov-Fernique inequality implies
	\begin{equation*}
	\begin{aligned}
		\sup_{\bw \in L^2}\cuL_{\mySigma}(\bw; \kappa,\xi ) 
			&\leq 
			\frac\kappa2 + \E\left[\sup_{\bw \in \reals^p} \left\{ \frac{\bw^\top  \bg}{\sqrt{p}} - \frac\kappa2\|\bw\|_2^2 - \frac{\xi}{\sqrt{p}}\left(\sum_{j \in S} x_jw_j + \|\bw_{S^c}\|_1\right) \right\} \right]
		\\
			&=
			\sup_{\bw \in L^2}\cuL_{\Ind_p}(\bw; \kappa,\xi ).
	\end{aligned}
	\end{equation*}
	For any $\kappa,\xi \geq 0$, $\sup_{\bw \in L^2}\cuL_{\mySigma}(\bw; \kappa,\xi ) \geq \sup_{\substack{
						\bw \in \cuD(\bx,\Ind_p)\\ 
						\|\bw\|_{L^2}^2\leq 1}
						} \frac{\langle \mySigma \bw , \bg \rangle_{L^2}}{\sqrt{p}\kappamax^{1/2}}$,
	whence, by Eq.~\eqref{EqSigmaToIComparison},
	\begin{equation}
		\cuG(\bx,\mySigma)
		\leq 
		\kappacond^{1/2} \sup_{\bw \in L^2} \cuL_{\Ind_p}(\bw;\kappa,\xi).
	\end{equation}
	Note that $\cuL_{\Ind_p}(\bw;\kappa,\xi)$ is the Lagrangian for the optimization Eq.~\eqref{EqnGaussianWidth} defining $\cuG(\bx,\Ind_p)$.
	Because the constraints on $\bw$ in this optimization are strictly feasible, strong duality holds.
	Thus, Eq.~\eqref{EqWidthSigmaToIBound} follows by taking the infimum over $\kappa,\xi \geq 0$ in the preceding display.
\end{proof}

\begin{proof}[Proof of Proposition \ref{prop:std-width}]
	By the Lagrangian calculations in the proof of Lemma \ref{LemGaussWidthTradeoff},
	\begin{equation}
		\cuG(\bx,\mySigma)
			:=
			\frac1p \< \bv^* , \bg \>_{L^2},
	\end{equation}
	where 
	\begin{equation}
		\bv^*(\bg) 
			=
			\argmin_{\bv \in \reals^p} 
			\Big\{
				\frac{\kappa^*}2 \| (\kappa^*)^{-1}\bg - \bv \|_2^2
				+ \xi^* F(\bv;\bx,\mySigma)
			\Big\},
	\end{equation}
	and $\bg$ is the identity function of $L_2$.
	We have
	% \begin{equation}
	% 	\kappa^*
	% 		\geq 
	% 		\frac{(|S|/p)/(2\kappacond)}{1 + \kappacond/2}
	% 		\geq 
	% 		\frac{(\omega^*)^{-1}\big(\cuG(\bx,\mySigma)/\kappacond^{1/2}\big)/(2\kappacond)}{1 + \kappacond/2
	% 		}
	% 		\geq c',
	% \end{equation}
	\begin{equation}
		\kappa^*
			\geq 
			\frac{(|S|/p)/(2\kappacond)}{1 + \kappacond/2}
			\geq 
			\frac{c'|S|}{p},
	\end{equation}
	where we have used \eqref{eq:cuL-ub} and \eqref{eq:cuL-lb} in the first inequality. 
	% and Lemma \ref{lem:width-under-corr} and the monotonicity of $\omega^*$ in the second inequality.
	The constant $c' > 0$ depends only on $\kappacond$, and may change values at each appearance.
	Because proximal operators are $1$-Lipschitz \cite{Parikh2013ProximalAlgorithms},
	we see that $\bv^*$ is $c' p /|S|$-Lipschitz in $\bg$.

	Let $\bw^* := \mySigma^{-1/2} \bv^*$, 
	and define $\bv^{(1)}:= \mySigma^{1/2} \bw^{(1)}$,
	where
	\begin{equation}
		\bw_S^{(1)}
			:=
			\bw_S^*
				- 
				\frac{F(\bv^*;\bx,\mySigma)_+}{|S|} \bx_S ,
		\qquad 
		\bw_{S^c}^{(1)}
			:=
			\bw_{S^c}^*,
	\end{equation}
	and $F(\bv^*;\bx,\mySigma)_+ = \max\{0,F(\bv^*;\bx,\mySigma)\}$ denotes the positive part of $F(\bv^*;\bx,\mySigma)_+$.
	Note that 
	\begin{equation}
	\begin{gathered}
		F(\bv^{(1)};\bx,\mySigma)
			= 
			F(\bv^*;\bx,\mySigma) - F(\bv^*;\bx,\mySigma)_+
			=
			F(\bv^*;\bx,\mySigma)_- \leq 0,
	\\
		\frac{\| \bv^{(1)} - \bv^* \|_2}{\sqrt{p}}
			\leq 
			\frac{\kappamax^{1/2}\| \bw^{(1)} - \bw^* \|_2}{\sqrt{p}}
			= 
			\frac{\kappamax^{1/2}F(\bv^*;\bx,\mySigma)_+}{\sqrt{p}|S|^{1/2}}.
	\end{gathered}
	\end{equation}
	Define $\bv^{(2)} = \sqrt{p} \bv^{(1)} / \| \bv^{(1)} \|_2$.
	Note 
	\begin{equation}
	\begin{gathered}
		\frac{\|\bv^{(2)}\|_2}{\sqrt{p}} = 1,
		\qquad 
		F(\bv^{(2)};\bx,\mySigma)
			= \frac{\sqrt{p}}{\|\bv^{(1)}\|_2} F(\bv^{(1)};\bx,\mySigma) \leq 0,
	\\
		\frac{\|\bv^{(2)} - \bv^{(1)}\|_2}{\sqrt{p}}
			= 
			\Big|
				1 - \frac{\| \bv^{(1)}\|_2}{\sqrt{p}}
			\Big|
			\leq 
			\Big|
				1 - \frac{\| \bv^*\|_2}{\sqrt{p}}
			\Big|
			+ 
			\frac{\kappamax^{1/2}F(\bv^*;\bx,\mySigma)_+}{\sqrt{p}|S|^{1/2}}.
	\end{gathered}
	\end{equation}
	The first line of the previous display implies that $\bv^{(2)}$ satisfies the constraints in Eq.~\eqref{Eqn:std-GW-0}.
	Moreover, using Cauchy-Schwartz
	\begin{equation}
	\begin{aligned}
		\frac1p \big| \< \bv^{(1)} , \bg \>_{L^2} - \< \bv^* , \bg \>_{L^2} \big|
			&\leq 
			\frac{\E[\| \bv^{(1)} - \bv^* \|_2^2]^{1/2} }{\sqrt{p}}
		\\
			&\leq 
			\E\Big[
				\Big(
					\Big|
						1 - \frac{\| \bv^*\|_2}{\sqrt{p}}
					\Big|
					+ 
					\frac{2\kappamax^{1/2}F(\bv^*;\bx,\mySigma)_+}{\sqrt{p}|S|^{1/2}}
				\Big)^2
			\Big]^{1/2}
	\end{aligned}
	\end{equation}
	Because $\E[\| \bv^*\|_2^2]/p = 1$, $\E[F(\bv^*;\bx,\mySigma)] \leq 0$, 
	the quantity $\| \bv^* \|_2 / \sqrt{p}$ is $1/\sqrt{p}$-Lipschitz in $\bv^*$, the quantity $F(\bv^*;\bx,\mySigma)_+/(\sqrt{p}|S|^{1/2})$ is $\kappamin^{-1/2}/\sqrt{p}$-Lipschitz in $\bv^*$, and $\bv^*$ is $c'p/|S|$-Lipschitz in $\bg$,
	we have by Guassian concentration of Lipschitz functions that the right-hand side is bound above by $c' \sqrt{p}\, / |S|$ for some $c'$ depending on $\kappacond$.
	Because $\< \bv^*,\bg\>_{L^2} = \cuG(\bx,\mySigma)$, we conclude that $\cuG_{d}(\bx,\mySigma) \geq \cuG(\bx,\mySigma) - c' \sqrt{p}/s$.
	The second bound comes from the fact that $\cuG(\bx,\mySigma) \leq c' \sqrt{s \log(p/s)/p}$ by Lemma \ref{lem:width-under-corr} and the fact that $\omega^*(\eps) \sim (1 + o(\eps))\sqrt{2\eps \log(1/\eps)}$, and the $\cuG_{d}(\bx,\mySigma) \geq 0$.
\end{proof}

\subsection{The $\alpha$-smoothed Lasso}
\label{SecSmoothedLasso}

Controlling the debiased Lasso (Theorem \ref{ThmDBLasso}) will require a smoothing argument in which we replace the $\ell_1$-penalty by a differentiable approximation.
We call the estimator using the differentiable approximation the $\alpha$-\emph{smoothed Lasso}.
Throughout the appendix, all theory is developed for non-smoothed and smoothed Lasso in a unified way.
Results about the Lasso estimate and residual will be instances of these general results.

For $\alpha > 0$, 
define the Moreau envelope of the $\ell_1$-norm
\begin{equation}
\label{EqnMoreau}
	\sfM_\alpha(\mytheta) 
	:= 
	\inf_{\bb \in \reals^p} \left\{\frac1{2\alpha}\|\mytheta - \bb\|_2^2 + \|\bb\|_1\right\}\,,
\end{equation}
and define $\sfM_0(\mytheta) = \|\mytheta\|_1$. Notice that this coincides with the H\"uber loss.
In particular, for all $\mytheta \in \reals^p$, 
\begin{equation}
\label{EqnMoreauRelation}
	\|\mytheta\|_1 - \frac{p\alpha}2\le \sfM_\alpha(\mytheta) \le \|\mytheta\|_1\,.
\end{equation}
For all $\alpha \geq 0$, 
define the \emph{$\alpha$-smoothed Lasso} in the random-design model
\begin{align}
\label{EqnAlphaSmoothRandomD}
	\thetahat_\alpha
		&:=
		\arg\min_{\mytheta\in\reals^p} \left\{
			\frac1{2n} \|\by - \bX\mytheta \|_2^2 + \frac\lambda{\sqrt{n}} (\sfM_\alpha(\mytheta) - \|\thetastar\|_1) 
		\right\}\\
		&=:
		\arg\min_{\mytheta\in\reals^p} \Risk_\alpha(\mytheta)\nonumber
	\,,
\end{align}
where the term $- \|\thetastar\|_1$ is added to the definition of $\Risk_\alpha(\mytheta)$ for future convenience.
Define the $\alpha$-smoothed Lasso in the fixed-design model 
\begin{gather}
\label{EqnAlphaSmoothLassoFixed}
	\thetahat^f_{\alpha} := \eta_\alpha(\by^f,\zeta) := \arg\min_{\mytheta \in \reals^p} \left\{\frac\zeta2 \| \by^f - \mySigma^{1/2} \mytheta\|_2^2 + \frac{\lambda}{\sqrt{n}} \sfM_\alpha(\mytheta)\right\}\,,\\
	\yhat_\alpha(\by^f,\zeta) := \mySigma^{1/2}\eta_\alpha(\by^f,\zeta)\,.\label{EqnAlphaSmoothedLassoPredFixed}	
\end{gather}
Denote the in-sample prediction risk and degrees-of-freedom of the $\alpha$-smoothed Lasso in the fixed-design model by
\begin{align}
	\label{EqnRiskAndDFAlpha1 }
		\risk_\alpha(\tau^2,\zeta)& \defn 
		\E\left[\ltwo{\yhat_\alpha(\mySigma^{1/2} \thetastar + \tau \bg / \sqrt{n}, \zeta) - \mySigma^{1/2}\thetastar}^2\right]\,,\\
		\label{EqnRiskAndDFAlpha2}	
		\df_\alpha(\tau^2,\zeta) &\defn 
		\frac{\sqrt{n}}{\tau} \E\left[\inprod{\yhat_\alpha(\mySigma^{1/2}\thetastar + \tau \bg / \sqrt{n},\zeta)}{ \bg} \right]\\
		&= \E[\div \yhat_\alpha(\mySigma^{1/2}\thetastar + \tau \bg / \sqrt{n} ) ]\,,\nonumber
	\end{align}
where the expectation is over $\bg \sim \normal(\bzero_p,\Ind_p)$.
Let $\taustar_\alpha, \zetastar_\alpha$ be solutions
to the system of equations
\begin{subequations}
	\begin{align}
	\label{EqnAlphaFixedPt1}
		\tau_\alpha^2 &= \sigma^2 + \risk_\alpha(\tau_\alpha^2, \zeta_\alpha)\,,\\
	\label{EqnAlphaFixedPt2}
		\zeta_\alpha &= 1 - \frac{\df_\alpha(\tau_\alpha^2, \zeta_\alpha)}{n}\,.
	\end{align}	
\end{subequations}
We refer to these equations as the {\it $\alpha$-smoothed fixed point equations}.
For $\alpha = 0$, 
these definitions agree with the corresponding definitions for the Lasso.
The solutions $\taustar_\alpha$, $\zetastar_\alpha$ are well-defined,
and can be bounded in terms of model parameters. 
\begin{lem}
\label{LemAlphaFixedPtSoln}
	The fixed-point equations \eqref{EqnAlphaFixedPt1} and \eqref{EqnAlphaFixedPt2} satisfy the following properties.

	\begin{enumerate}[label=(\alph*)]

		\item % a 
		\emph{(Existence and uniqueness of solutions)}
		For all $\alpha \geq 0$,
		if $\mySigma$ is invertible and $\sigma^2 > 0$, 
		then Eqs.~\eqref{EqnAlphaFixedPt1} and \eqref{EqnAlphaFixedPt2} have a unique solution.

		\item % b
		\emph{(Boundedness of solutions)}
		Consider $\alphamax \geq 0$ and the $\alpha$-smoothed Lasso with $\alpha \leq \alphamax / \sqrt{n}$.
		Assume  \ref{assump:1}.
		For $\cuPmodel$-dependent constant $C$, 
		if $p > C$ then
		then there exist $0 < \taumax < \infty$
		and 
		$0 < \zetamin < 1$
		depending only on $\cuPmodel$ and $\alphamax$
		such that the unique solution 
		$\taustar_\alpha$, $\zetastar_\alpha$ to 
		Eqs.~\eqref{EqnAlphaFixedPt1} and \eqref{EqnAlphaFixedPt2} satisfies 
		$\sigma \leq \taustar_\alpha \leq \taumax$ and 
		$\zetamin \leq \zetastar_\alpha \leq 1$.

	\end{enumerate}

\end{lem}
\noindent 
Note that Lemma \ref{LemTauZetaBounds} is the $\alphamax = 0$ instance of Lemma \ref{LemAlphaFixedPtSoln}.
We prove Lemma \ref{LemAlphaFixedPtSoln} in the next section.

%%%%%%%%%%%%%%%%%%%%%%%%%%%%%%%%%%%%%%%%%%%%%%%%%%%%%%%%%%%%%%%%%%%%%%

\subsection{Control on fixed-point parameters: proofs of Lemma \ref{LemTauZetaBounds} and Lemma \ref{LemAlphaFixedPtSoln}}
\label{SecFixedPtSoln}
	
Lemma \ref{LemTauZetaBounds} is the $\alphamax = 0$ instance of Lemma \ref{LemAlphaFixedPtSoln}.
Thus, we only prove Lemma \ref{LemAlphaFixedPtSoln}.
First, we prove Lemma \ref{LemAlphaFixedPtSoln}(a).

\begin{proof}[Proof of Lemma \ref{LemAlphaFixedPtSoln}(a)]
	Define functions $\cT,\cuZ:L^2(\reals^p;\reals^p) \rightarrow \reals$ by
	\begin{gather}
		\label{Eqntaustar}
	    \mathcal{T}(\bv)^2 := \sigma^2+\|\bv\|^2_{L^2}\, ,\\
	    \label{Eqnzetastar}
	    \cuZ(\bv) := \left(1-\frac{\<\bg,\bv\>_{L^2}}{\sqrt{n}\mathcal{T}(\bv)}\right)_+\,,
	\end{gather}
	where $\bg$ is interpreted as the identity function in $L^2$.
	Define $\cuE_\alpha:L^2(\reals^p;\reals^p)\to\reals$ by
	\begin{equation}
	\begin{aligned}
	\label{EqFunctionalObj}
		\cuE_\alpha(\bv) 
			&:= 
			\frac12\Big(
				\sqrt{\sigma^2 + \|\bv\|_{L^2}^2 }\,-\frac{\langle \bg , \bv \rangle_{L^2}}{\sqrt{n}}
			\Big)^2_+
			+
			\frac{\lambda}{\sqrt{n}}
			\E\Big\{
				\sfM_\alpha(\thetastar+\mySigma^{-1/2}\bv(\bg))-\|\thetastar\|_1
			\Big\}
		\\
		&=: 
			\cuF(\bv) 
			+ 
			\frac{\lambda}{\sqrt{n}}
			\E\Big\{
				\sfM_\alpha(\thetastar+\mySigma^{-1/2}\bv(\bg))-\|\thetastar\|_1
			\Big\}\,.
	\end{aligned}
	\end{equation}
	Let us emphasize the argument of $\cuE_\alpha$  is not a vector but a function $\bv:\reals^p\to\reals^p$.

	Each of the two terms in the definition of $\cuE_\alpha$ are convex and continuous. 
	Moreover, for all $\bg$ we have, by Eq.~\eqref{EqnMoreauRelation},
	\begin{align}
		\sfM_\alpha(\thetastar + \mySigma^{-1/2} \bv(\bg)) 
			&\geq 
			\|\thetastar + \mySigma^{-1/2}\bv(\bg)\|_1 - \frac{p\alpha}2 
		\\
		&\geq
			\|\mySigma^{-1/2} \bv(\bg)\|_1 - \|\thetastar\|_1  - \frac{p\alpha}2
			\geq 
			\kappamax^{-1/2} \|\bv(\bg)\|_2 - \|\thetastar\|_1  - \frac{p\alpha}2
		\,.
	\end{align}
	For any $M > 0$,
	\begin{align}
		|\langle \bv , \bg \rangle_{L^2}| &= |\E[\langle \bv(\bg) , \bg \indic{\|\bg\|_2 > M} \rangle ]
			+ \E[\langle \bv(\bg) , \bg \indic{\|\bg\|_2 \leq M }\rangle]| \nonumber \\
		&\le \|\bv\|_{L^2} \E[\|\bg\|_2^2\indic{\|\bg\|_2 > M}]^{1/2} + M \E[\|\bv(\bg)\|_2].
	\end{align}
	Take $M$ large enough that $\E[\|\bg\|^2 \indic{\|\bg\| > M}]^{1/2} < \sqrt{n}/2$.
	Then
	\begin{equation}
	\begin{aligned}
		\cuE_\alpha(\bv) 
		&\ge 
		\frac12\left(\frac{\|\bv\|_{L^2}}{2} - \frac{M}{\sqrt{n}} \E[\|\bv(\bg)\|_2] \right)_+^2 
			+ \frac{\lambda \kappamax^{-1/2}}{\sqrt{n}} \E[\|\bv(\bg)\|_2] 
			- \frac{2\lambda}{\sqrt{n}} \|\thetastar\|_1
			- \frac{p\lambda\alpha}{2\sqrt{n}}
		\\
		&\ge 
			\min\left\{
				\frac{\|\bv\|_{L^2}^2}{32}, 
				\frac{\lambda \kappamax^{-1/2}}{4M} \|\bv\|_{L^2}
			\right\} 
			- \frac{2\lambda}{\sqrt{n}} \|\thetastar\|_1 
			- \frac{p\lambda \alpha}{2\sqrt{n}}\,,
	\end{aligned}
	\end{equation}
	where the second inequality holds by considering the cases that $\|\bv\|_{L^2}/4$ is no smaller and no larger than $M\E[\|\bv(\bg)\|_2]/\sqrt{n}$, respectively.
	We see that $\cuE_\alpha(\bv) \rightarrow \infty$ as $\|\bv\|_{L^2} \rightarrow \infty$, whence by \cite[Theorem 11.9]{bauschke2011convex}, $\cuE_\alpha$ has a minimizer. 
	Let $\bv^*_\alpha$ be one such minimizer.

	Consider the following convex function in $L^2$ parameterized by $\tau,\zeta \ge 0$:
	\begin{align}
		\tcuE_\alpha(\bv;\zeta,\tau) 
			&:= 
			\frac{\zeta}{2}\big\|\bv-\frac{\tau}{\sqrt{n}}\bg\big\|_{L^2}^2
			+ \frac{\lambda}{n}\E\Big\{\sfM_\alpha(\thetastar+\mySigma^{-1/2}\bv)
			- \|\thetastar\|_1\Big\}
		\nonumber\\
		&= 
			\E\left\{\frac{\zeta}{2} \Big\|\bv - \frac{\tau}{\sqrt{n}} \bg\Big\|_2^2 
			+ \frac\lambda {\sqrt{n}} (\sfM_\alpha(\thetastar + \mySigma^{-1/2} \bv) 
			- \|\thetastar\|_1)\right\}
		\,.\label{Eqnvhat:eq}
	\end{align}
	For fixed $\zeta,\tau\ge 0$, the function $\bv^*_\alpha$ minimizes $\tcuE_\alpha$ if and only if $\bv^*_\alpha(\bg)$ minimizes the objective inside the expectation for almost every $\bg$.
	That is, if and only if 
	\begin{equation}\label{EqnTcuEtoProx}
		\bv^*_\alpha = \mySigma^{1/2}(\eta_\alpha(\mySigma^{1/2}\thetastar + \tau \bg / \sqrt{n} ; \zeta ) - \thetastar) \;\;\text{almost surely.}
	\end{equation}
	For any $\bv_0,\bv_1\in L^2$ fixed, 
	we have by differentiation of $\cuF$ with respect to $\eps \in
        \reals$ that
	\begin{equation}
	\begin{aligned}
		\tcuE_\alpha(\bv_0 + \eps \bv_1;\cuZ(\bv_0),\mathcal{T}(\bv_0)) - \cuE_\alpha(\bv_0 + \eps \bv_1) 
			&= 
			\frac{\cuZ(\bv_0)}{2}\Big\|\bv_0 + \eps \bv_1 - \frac{\mathcal{T}(\bv_0)}{\sqrt{n}} \bg\Big\|_{L^2}^2 - \cuF(\bv_0 + \eps \bv_1) \\
		&= \tcuE_\alpha(\bv_0 ; \cuZ(\bv_0);\cuT(\bv_0))-\cuE_\alpha(\bv_0 )+O(\epsilon^2)\,.
	\end{aligned}
	\end{equation}
	Thus, $\bv_1$ is a descent direction of $\bv \mapsto \tcuE_\alpha(\bv;\cuZ(\bv_0),\cT(\bv_0))$ at $\bv_0$ if and only if it is also a descent direction of $\bv \mapsto \cuE_\alpha(\bv)$ at $\bv_0$.
	In particular, $\bv_0$ minimizes $\cuE_\alpha$ if and only if it minimizes $\tcuE(\bv;\zeta,\tau)$ for $\zeta = \cuZ(\bv_0)$ and $\tau = \cuT(\bv_0)$.
	By \eqref{EqnTcuEtoProx}, we conclude that $\bv^*_\alpha$ is a minimizer of $\cuE_\alpha$ if and only if
	\begin{equation}\label{EqnGordonToProx}
		\bv^*_\alpha(\bg) = \mySigma^{1/2} ( \eta_\alpha(\mySigma^{1/2}\thetastar + \mathcal{T}(\bv^*_\alpha) \bg / \sqrt{n} ; \cuZ(\bv^*_\alpha) ) - \thetastar ) \;\; \text{almost surely.}
	\end{equation}
	That is, if and only if $\taustar_\alpha = \mathcal{T}(\bv^*_\alpha)$, $\zetastar_\alpha = \cuZ(\bv^*_\alpha)$ 
	is a solution to equations \eqref{EqnAlphaFixedPt1} and \eqref{EqnAlphaFixedPt2}.
	Because $\cuE_\alpha$ has minimizers, solutions to equations \eqref{EqnAlphaFixedPt1} and \eqref{EqnAlphaFixedPt2} exist.
      
	To complete the proof, we only need to show that the minimizer $\bv^*_\alpha$ of $\cuE_\alpha$ is unique.
	First, we claim $\cuZ(\bv^*_\alpha) > 0$ for all minimizers $\bv^*_\alpha$.
	Assume otherwise that $\cuZ(\bv^*_\alpha) = 0$ for some minimizer $\bv^*_\alpha$.
	Then, by property \eqref{EqnGordonToProx}, 
	\begin{align*}
		\bv^*_\alpha = \mySigma^{1/2}(\eta_\alpha(\mySigma^{1/2}\thetastar + \mathcal{T}(\bv^*_\alpha)\bg/\sqrt{n};0) - \thetastar) = - \mySigma^{1/2} \thetastar.
	\end{align*}
	Thus, we have $\cuZ(\bv^*_\alpha) = \left(1 - \frac1{\sqrt{n}\mathcal{T}(\bv^*_\alpha)}\langle \bg , -\mySigma^{1/2} \thetastar \rangle_{L^2}\right)_+ = 1$, a contradiction.
	We conclude $\cuZ(\bv^*_\alpha) > 0$ for all minimizers $\bv^*_\alpha$ of $\cuE_\alpha$.

	The function $\cuE_\alpha$ is strictly convex on $\cuZ(\bv) > 0$.	
	Indeed, for any $\bv \neq \bv'$,
	the function 
	\begin{align*}
	   t \mapsto \sqrt{\|(1-t)\bv + t\bv'\|_{L^2}^2 + \sigma^2} = \sqrt{\|\bv\|_{L^2}^2 - 2t \langle \bv , \bv - \bv' \rangle_{L^2} + t^2 \|\bv'\|_{L^2}^2 + \sigma^2} 
	\end{align*} is strictly convex by univariate calculus.
	Because $x \mapsto x_+^2$ is convex and strictly increasing on $x > 0$, strict convexity of $\cuE_\alpha$ on $\cuZ(\bv) > 0$ follows.
	Because all minimizers $\bv^*_\alpha$ satisfy $\cuZ(\bv^*_\alpha) > 0$, 
	strict convexity on $\bv^*_\alpha > 0$ implies the minimizer is unique.
\end{proof}

To prove Lemma \ref{LemAlphaFixedPtSoln}(b),
we require the next lemma, which states that the degrees-of-freedom of the estimator cannot be large without also the risk or noise-variance also being large.
\begin{lem}
\label{LemDFControl}
	For any $\tau,\zeta,\delta > 0$ and $\alphamax / \sqrt{n} \geq \alpha \geq 0$ and if the eigenvalues of $\mySigma$ are bounded as $0 < \kappamin \leq \kappa_j(\mySigma) \leq \kappamax < \infty$,
	then 
	\begin{equation}
		\kappamax
		\left(\tau \kappacond^{1/2} + \sqrt{\tau^2 \kappacond + \frac{n}{p} \risk_\alpha(\tau^2,\zeta)}\right)^2
		\geq
			\frac{n}{p}
			\frac{\lambda^2}{\zeta^2} \frac{\df_\alpha(\tau^2,\zeta)}{n}
			- 
			\frac{\alphamax \lambda \kappamax}{\zeta} \, .
	\end{equation}
\end{lem}
\noindent We prove Lemma \ref{LemDFControl} in at the end of this section.

\begin{proof}[Proof of Lemma \ref{LemAlphaFixedPtSoln}(b)]
	By general properties of proximal operators \cite{bauschke2011convex}, the Jacobian matrix $\nabla \yhat_\alpha(\by^f,\zeta)$ of $\yhat_{\alpha}(\,\cdot\,,\zeta)$ is positive-semidefinite.
	Therefore, $\df_\alpha({\taustar_\alpha}^2,\zetastar_\alpha) \ge 0$. 
	Also, $\zetastar_\alpha \leq 1$ is immediate from
	Eq.~\eqref{EqnAlphaFixedPt2}.
	Further,  $\taustar_\alpha \geq \sigmamin$ is immediate from Eq.~\eqref{EqnAlphaFixedPt1}.

    \paragraph*{Establishing the bound $\taumax^2$}
    By Assumption \ref{assump:1}(d),
    there exists $\barthetastar\in\reals^p$ such that 
    $n/p \geq \cuG(\bx,\mySigma)^2 + \Deltamin$ for $\bx = \sign(\barthetastar)$ and $\|\barthetastar - \thetastar\|_1/p \leq M / \sqrt{n}$.
    Let $\bx'$ be such that $x_j' = x_j$ for $x_j \neq 0$ and 
    \begin{equation}
    	|\supp(\bx')| = \big\lfloor p (\omega^*)^{-1}\big(\sqrt{(n/p - \Deltamin/2)/\kappacond}\big) \big\rfloor \vee \| \barthetastar \|_0
    \end{equation} 
    and assume $p (\omega^*)^{-1}\big(\sqrt{\Deltamin/(2\kappacond)}\big) \geq 1 $.\footnote{This is where we require $p > C$}
    In particular, $\bx' \in \partial \| \barthetastar \|_1$, so that for any $\bw \in \reals^p$ we have $\| \barthetastar + \bw \|_1 - \| \barthetastar \|_1 \geq \< \bx' , \bw \>$.
    Denote $S' = \supp(\bx') \subset[p]$.
    Because $p (\omega^*)^{-1}\big(\sqrt{(n/p - \Deltamin/2)/\kappacond}\big) \geq p (\omega^*)^{-1}\big(\sqrt{\Deltamin/(2\kappacond)}\big) \geq 1$,
    we conclude that $|S'|/p \geq (\omega^*)^{-1}\big(\sqrt{\Deltamin/(2\kappacond)}\big)/2 =: \numin$.
    If $\big\lfloor p (\omega^*)^{-1}\big(\sqrt{(n/p - \Deltamin/2)/\kappacond}\big) \big\rfloor < \| \barthetastar \|_0$,
    then by assumption $n/p \geq \cuG(\bx',\mySigma)^2 + \Deltamin$. 
    Otherwise,
    by Lemma \ref{lem:width-under-corr},
    $\cuG(\bx',\mySigma) \leq \sqrt{\delta}/2$.
    Otherwise, by assumption we have $n/p \geq \cuG(\bx',\mySigma)^2 + \Deltamin / 2$.
    In summary, 
    we have
    \begin{equation}
    	\frac{n}{p} \geq \cuG(\bx',\mySigma)^2 + \frac{\Deltamin}{2}, 
    	\quad
    	\frac{\| \bx' \|_0}p \geq \frac{(\omega^*)^{-1}\big(\sqrt{\Deltamin/(2\kappacond)}\big)}{2},
    	\quad 
    	\text{and}
    	\quad 
    	\bx' \in \partial \| \barthetastar \|_1.
    \end{equation}

	We may equivalently write the objective in \eqref{EqFunctionalObj} as a function of $\bw := \mySigma^{-1/2}\bv$.
	Note  that
	\begin{equation}
	\begin{aligned}
	\label{EqnL1NormLB}
		\sfM_\alpha(\thetastar + \bw) - \|\thetastar\|_1
			&\geq 
			\|\barthetastar + \bw\|_1 - \|\barthetastar\|_1
			- 2\|\barthetastar - \thetastar\|_1 
			- \frac{p\alphamax}{2 \sqrt{n}}
		\\
		&\geq 
			\sum_{j \in S'} \Big(x_j' w_j + \|\bw_{S'^c}\|_1\Big) 
			- \frac{2Mp}{\sqrt{n}}
			- \frac{p\alphamax}{2 \sqrt{n}}\,,
	\end{aligned}
	\end{equation}
	where the first inequality uses the relation~\eqref{EqnMoreauRelation}.
	Let $\tilde \bg := \mySigma^{1/2} \bg$.
	If $n/p \leq 2$, let $\epsilon = \sqrt{n/p}\, \Deltamin / 16$, so that $\cuG(\bx',\mySigma) + \epsilon \leq \sqrt{n/p - \Deltamin/2} + \sqrt{n/p}\,\Deltamin/16 \leq \sqrt{n/p}(1-\Deltamin/16)$.
	Otherwise, let $\epsilon = (\sqrt{2}-1)/2$, so that $\cuG(\bx',\mySigma) + \epsilon \leq (\sqrt{2}+1)/2 \leq (3/4) \sqrt{n/p}$,
	where we have used that $\cuG(\bx',\mySigma) \leq 1$ in the first inequality and $n/p \geq 2$ in the second inequality.
	Using Lemma~\ref{LemGaussWidthTradeoff},  Eqs.~\eqref{EqnWidthTradeoff1} and \eqref{EqnWidthTradeoff2},
	we have that either 
	\begin{equation}
		 \frac1{\sqrt{n}} \langle \tilde \bg  , \bw \rangle_{L^2} \leq \Big(1 - \frac{\Deltamin}{16} \wedge \frac14 \Big)\;\;  \E[\|\bw\|_{\mySigma}^2]^{1/2}\,,
	\end{equation}
	or
	\begin{equation}
	    \frac1{\sqrt{n}} \E\left[\sum_{j \in S} x_j' w_j + \|\bw_{S^c}\|_1\right] \geq c \E[\|\bw\|_{\mySigma}^2]^{1/2}\,,
	\end{equation}
	where $c = \numin^{1/2} ( \Deltamin/16 \wedge (\sqrt{2}-1)/2 ) / ( \kappamax^{1/2}(2+\kappacond))$.
	Then, Eq.~\eqref{EqnL1NormLB} gives
	\begin{equation}
	\begin{aligned}
		&\cuE_\alpha(\bv) 
		= 
		\frac12 \left(\sqrt{\E[\|\bw\|_{\mySigma}^2] + \sigma^2} - \frac{\langle \tilde \bg , \bw \rangle_{L^2}}{\sqrt{n}}\right)_+^2 
			+ \frac\lambda{\sqrt{n}} \E\left\{\sfM_\alpha(\thetastar + \bw) - \|\thetastar\|_1\right\} \\
		&\geq 
		\min\left\{\frac12\Big(\frac{\Deltamin}{16} \wedge \frac14 \Big)^2 \E[\|\bw\|_{\mySigma}^2] ,\;\;  \lambda c \| \E[\| \bx \|_{\mySigma}^2]^{1/2}\right\}- \frac{2\lambda M p}{n} - \frac{p\alphamax\lambda}{2n}.
	\end{aligned}
	\end{equation}
	As in the proof of Lemma \ref{LemAlphaFixedPtSoln} (see Eq.~\eqref{EqnTcuEtoProx}), let $\bv_\alpha^*$ be the minimizer of $\cuE_\alpha$. 
	Because $\sfM_\alpha(\thetastar) \leq \|\thetastar\|_1$,
	we bound $\sigma^2/2 = \cuE_0(\bzero) \geq \cuE_\alpha(\bzero) \geq \cuE_\alpha(\bv^*_\alpha)$.
	Combining this bound with the previous display applied at $\bw = \mySigma^{-1/2}\bv_\alpha^*$, 
	some algebra yields
	\begin{align}
		&\frac{\|\bv^*_\alpha\|_{L^2}^2}{n} = \frac{\E[\|\bw\|_{\mySigma}^2]}{n} 
		\leq  
		\max\left\{\frac{(\frac{\sigma^2}2 + \frac{2\lambdamax M}{\Deltamin} + \frac{\alphamax\lambdamax}{2\Deltamin})}{\frac12(\frac{\Deltamin}{16} \wedge \frac14)^2}, \frac{(\frac{\sigma^2}{2\lambdamin} + \frac{2M}{\Deltamin} + \frac{\alphamax}{2\Deltamin} )^2}{c^2} \right\}.
	\end{align}
	Recalling the fixed point equation~\eqref{EqnAlphaFixedPt1}, we may set $\taumax^2$ to be the sum of $\sigma^2$ and the right-hand side above. 

	\paragraph*{Establishing the bound $\zetamin$}
	If $\df_\alpha({\taustar_\alpha}^2,\zetastar_\alpha)/n \leq 1/2$, 
	then by Eq.~\eqref{EqnRiskAndDFAlpha2}, $\zetastar_\alpha \geq 1 - \df_\alpha(\tau_\alpha^{*2},\zeta_\alpha^*)/n \geq 1/2$.
	Alternatively,
	if $\df_\alpha(\tau_\alpha^{*2},\zetastar_\alpha)/n \geq 1/2$, 
	then by Lemma \ref{LemDFControl}, it is guaranteed that 
	\begin{equation}
	\begin{aligned}
		\kappamax\taumax^2 
		\left(\kappacond^{1/2} + \sqrt{ \kappacond + n/p }\right)^2
			&\geq
			\frac{\lambdamin^2\Deltamin}{2\zeta_\alpha^{*2}}
			- 
			\frac{\alphamax \lambdamax \kappamax}{\zeta_\alpha^*} .
	\end{aligned}
	\end{equation}
	Note that because $\df_{\alpha}(\tau_\alpha^{*2},\zeta_\alpha^*) \leq p$,
	in this case we have $n/p \leq 2$,
	so that we may replace $n/p$ with $2$ on the left-hand side.
	Because the right-hand side of the preceding equation diverges to $+\infty$ and $\zeta_\alpha^* \downarrow 0$,
	the inequality implies a lower bound $\zetamin$ on $\zeta_\alpha^*$ which depends only on $\cuPmodel$.
	The proof is complete.
\end{proof}

\begin{proof}[Proof of Lemma \ref{LemDFControl}]
	The KKT conditions for the $\alpha$-smoothed Lasso in the fixed-design model \eqref{EqnAlphaSmoothLassoFixed} are
	\begin{equation}
	\label{EqnKKTfix}
		\yhat_\alpha(\by^f,\zeta) - \mySigma^{1/2}\thetastar = \frac{\tau}{\sqrt{n}} \bg - \frac\lambda{\sqrt{n}\,\zeta} \mySigma^{-1/2} \bdelta\,,
	\end{equation}
	where $\by^f = \mySigma^{1/2}\thetastar + \tau \bg / \sqrt{n}$ and $\bdelta \in \partial \sfM_\alpha(\eta_\alpha(\by^f,\zeta))$.
	Therefore,
	\begin{equation}
		\|\yhat_\alpha(\by^f,\zeta) - \mySigma^{1/2}\thetastar\|_2^2 
			\geq 
			\frac{\lambda^2}{\zeta^2\kappamax} \frac{\|\bdelta\|_2^2}{n} - \frac{2\lambda \tau}{\zeta \kappamin^{1/2}} \frac{\|\bg\|_2 \|\bdelta\|_2}{n}\,.
	\end{equation}
	Taking expectations and applying Cauchy-Schwartz yields
	\begin{equation}
		\risk_\alpha(\tau^2,\zeta) 
		\geq 
		\frac{\lambda^2}{\zeta^2\kappamax} \frac{\E[\|\bdelta\|_2^2]}{n} 
		- 
		\frac{2\lambda \tau}{\zeta \sqrt{n/p}\, \kappamin^{1/2}} \frac{\E[\|\bdelta\|_2^2]^{1/2}}{\sqrt{n}}\,,
	\end{equation}
	Solving the resulting quadratic equation for $\frac{\lambda\E[\|\bdelta\|_2^2]^{1/2}}{\zeta\sqrt{n}}$, 
	we conclude
	\begin{equation}
	\label{EqnGradToRiskBound}
		\frac{\E[\|\nabla \sfM_\alpha(\eta_\alpha(\by^f,\zeta)\|_2^2]^{1/2}}{\sqrt{n}} 
		\leq 
		\frac{\zeta \kappamax^{1/2}}{\lambda \sqrt{n/p}} 
		\left(\tau \kappacond^{1/2} + \sqrt{\tau^2 \kappacond + (n/p) \risk_\alpha(\tau^2,\zeta)}\right)\,.
	\end{equation}

	Now we divide the analysis into two cases. 
	First, consider the case $\alpha>0$.
	Then $\sfM_\alpha$ is differentiable and $\bdelta = \nabla\sfM_\alpha(\eta_\alpha(\by^f,\zeta))$. We compute
	\begin{gather}
		\nabla \sfM_\alpha(\mytheta) = (\mytheta - \softthreshold(\mytheta,\alpha))/\alpha\,,
		\;\;\;\;\;\;
		\nabla^2 \sfM_\alpha(\mytheta) = \diag( (\indic{|\theta_j| \leq \alpha})_j )/\alpha \,.
	\end{gather}
	Because $|\theta - \softthreshold(\theta,\alpha)|/\alpha = 1$ for $|\theta| \geq \alpha$,
	we bound
	\begin{equation}
		\|\nabla \sfM_\alpha(\mytheta)\|_2^2 \geq |\{j \in [p] \mid |\theta_j| \geq \alpha\}|\,.
	\end{equation}
	The KKT condition~\eqref{EqnKKTfix} are alternatively written
	\begin{equation}
		\zeta \mySigma^{1/2}(\by^f - \yhat_\alpha(\by^f,\zeta)) = \frac{\lambda}{\sqrt{n}} \nabla \sfM_\alpha( \eta_\alpha(\by^f,\zeta) )\,.
	\end{equation}
	%f
	Differentiating with respect to $\by^f$,
	\begin{equation}
		\zeta \mySigma^{1/2} - \zeta \mySigma^{1/2} \nabla \yhat_\alpha(\by^f;\zeta) = \frac{\lambda}{\sqrt{n}} \nabla^2 \sfM_\alpha( \eta_\alpha(\by^f,\zeta) ) \mySigma^{-1/2} \nabla \yhat_\alpha(\by^f;\zeta)\,.
	\end{equation}
        (More precisely, $\yhat_\alpha(\by^f,\zeta)$ and $\eta_\alpha(\by^f,\zeta)$ are continuous and piecewise linear in
        $\by^f$, and the above identity holds in the interior of each linear region.)
	We therefore get
	\begin{align}
		\nabla \yhat_\alpha(\by^f,\zeta)
		&=
			\left(\Ind_p + \frac{\lambda}{\sqrt{n}\,\zeta} \mySigma^{-1/2} \nabla^2 \sfM_\alpha( \eta_\alpha(\by^f,\zeta) ) \mySigma^{-1/2} \right)^{-1} 
		\\
		&= 
			\left(\Ind_p + \frac{\lambda}{\sqrt{n}\,\alpha\zeta} (\mySigma^{-1/2})_{\cdot,S^c} (\mySigma^{-1/2})_{S^c,\cdot} \right)^{-1}
		\\
		&= \Ind_p 
			- \frac{\lambda}{\sqrt{n}\, \alpha\zeta} 
			(\mySigma^{-1/2})_{\cdot,S^c} 
			\Big( \Ind_{|S^c|} + \frac{\lambda}{\sqrt{n}\,\alpha\zeta}(\mySigma^{-1/2})_{S^c,\cdot} (\mySigma^{-1/2})_{\cdot,S^c}\Big)^{-1} 
			(\mySigma^{-1/2})_{S^c,\cdot}\,,
	\end{align}
	where $S = \{j \in [p] \mid |\eta_\alpha(\by^f,\zeta)| \geq \alpha\}$.
	Thus,
	\begin{equation}
		\div \yhat_\alpha(\by^f,\zeta)
		=
		\trace( \nabla \yhat_\alpha(\by^f,\zeta) )
		\leq
		p - \frac{|S^c|}{1 + \sqrt{n}\,\alpha \zeta\kappamax/\lambda}
		\leq p - \frac{ p - \|\nabla \sfM_\alpha(\eta_\alpha(\by^f,\zeta))\|_2^2 }{1 + \alphamax\zeta\kappamax/\lambda}\,.
	\end{equation}
	Rearranging and taking expectations,
	\begin{equation}
	\label{eq:grad-g-df}
	\begin{aligned}
		\frac{\E[\|\bdelta\|_2^2]}{n} 
		&\geq 
		\frac{\df_\alpha(\tau^2,\zeta)}{n}
		-
		\frac{\alphamax \zeta \kappamax}{\lambda}
		\Big(
			\frac{p}{n}
			-
			\frac{\df_\alpha(\tau^2,\zeta)}n
		\Big)
	\\
		&\geq 
		\frac{\df_\alpha(\tau^2,\zeta)}{n}
		-
		\frac{\alphamax \zeta \kappamax}{\lambda}
		\frac{p}{n}.
	\end{aligned}
	\end{equation}
	Squaring Eq.~\eqref{EqnGradToRiskBound}, chaining it with the previous display, and multiplying by $\lambda^2(n/p)/\zeta^2$ 
	gives the result of the lemma in the case $\alpha > 0$.

	Now consider the case $\alpha = 0$.
	In this case, $\sfM_0(\mytheta) = \| \mytheta \|_1$, 
	Note that $\bdelta \in \partial \| \thetahat^f \|_1$ then implies that $\| \bdelta \|^2 \geq \| \thetahat^f \|_0^2$.
	Moreover,
	$\nabla \yhat_\alpha(\by^f,\zeta)$ exists almost surely with respect to $\bg$ by \cite[Proposition 5.3]{Bellec2019SecondOP},
	and is equal to (see, for example, \cite[Table 1]{Bellec2019SecondOP}),
	\begin{equation}
		\nabla \yhat_\alpha(\by^f,\zeta)
			= 
			(\mySigma^{1/2})_{\cdot,S}(\mySigma_{S,S})^{-1}(\mySigma^{1/2})_{\cdot,S},
	\end{equation}
	where $S = \{ j \in [p] \mid \widehat{\theta}_j^f \neq 0\}$.
	Thus, $ \div \yhat_\alpha(\by^f,\zeta) = \| \thetahat^f \|_0$ almost surely.
	Taking expectations,
	we get Eq.~\eqref{eq:grad-g-df},
	and the result follows as before.
\end{proof}

As a consequence, one arrives at the following result. 
\begin{corollary}
  Under Assumption \ref{assump:1} and if $\alpha \leq \alphamax / \sqrt{n}$, then  $\df_\alpha({\taustar_\alpha}^2,\zetastar_\alpha)/n$ is uniformly bounded away from one. Namely
  \begin{align}
    \frac{\df_\alpha({\taustar_\alpha}^2,\zetastar_\alpha)}n = 1 - \zetastar \leq 1 - \zetamin\, .
  \end{align}
  with $\zetamin$ depending uniquely on $\cuPmodel$ and $\alphamax$.
\end{corollary}

%%%%%%%%%%%%%%%%%%%%%%%%%%%%%%%%%%%%%%%%%%%%%%%%%%%%%%%%%%%%%%%%%%%%%%

\subsection{Continuity of fixed point solutions in smoothing parameter}

\begin{lem}
\label{LemFixedPtContinuity}
	If Assumption \ref{assump:1} holds,
	then there exist constants $\alphamax$, $L_\tau$, and $L_\zeta$ depending only on $\cuPmodel$ such that 
	for $\alpha \leq \alphamax/\sqrt{n}$,
	\begin{equation}
		|\taustar_0 - \taustar_\alpha| \leq L_\tau \sqrt{\sqrt{n}\alpha}\,,
		\;\;\;\;\;\;\;\;
		|\zetastar_0 - \zetastar_\alpha| \leq L_\zeta \sqrt{\sqrt{n}\alpha}\,.
	\end{equation}
\end{lem}
\begin{proof}[Proof of Lemma \ref{LemFixedPtContinuity}]
	The function 
	\begin{equation}
		f:L^2 \rightarrow \reals\,, 
		\;\;\;\;
		\bv \mapsto \sqrt{\|\bv\|_{L^2}^2 + \sigma^2} - \frac{\langle \bg, \bv\rangle_{L^2}}{\sqrt{n}}\,,
	\end{equation}
	is $(1 + \sqrt{p/n})\leq (1 + \Deltamin^{-1/2})$-Lipschitz.
	Evaluated at the minimizer $\bv^*_0$ of $\cuE_0$ defined in \eqref{EqFunctionalObj},
	$f(\bv^*_0)$ is equal to $\taustar_0\zetastar_0 \geq \sigmamin\zetamin$ by the proof of Lemma \ref{LemAlphaFixedPtSoln} in Section \ref{SecFixedPtSoln}.
	Thus, for $\|\bv - \bv^*_0\|_{L^2} \leq \sigmamin\zetamin/(2(1 + \Deltamin^{-1/2}))$, it is guaranteed that 
	\begin{equation}
		f(\bv)
		\geq 
		\frac{\sigmamin\zetamin}{2}\,.	
		\,
	\end{equation}
	Let $r:= \min\left\{\frac{\sigmamin\zetamin}{2(1 + \Deltamin^{-1/2})},\frac{\sigmamin}2\right\}$.
	By differentiation along affine paths, 
	the function 
	\begin{equation}
		\frac12f(\bv)_+^2
			\;\;\; 
			\text{is} 
			\;\;\;
			\frac{\sigmamin^2\inf_{\bv \in B}f(\bv)_+}{(\sup_{\bv \in B}\|\bv\|_{L^2}^2 + \sigmamin^2)^{3/2}} 
			\;\;\;
			\text{strongly convex on $\bv \in B$ for any $B \subset L^2$}.
	\end{equation}
	Thus, 
	$\cuE_0$ is $a:= \frac{\sigma^2 \sigmamin\zetamin/2}{(R^2 + \sigma^2)^{3/2}}$-strongly convex on 
	$\|\bv - \bv^*_0\|_{L^2} \leq r$, 
	where $R = \taumax + r$.

	By Eq.~\eqref{EqnMoreauRelation}, for any $\bv \in L^2$ and $\alpha \geq 0$,
	$\E[\sfM_0(\thetastar + \mySigma^{-1/2}\bv)] 
	\geq \E[\sfM_\alpha(\thetastar + \mySigma^{-1/2}\bv)] 
	\geq \E[\sfM_0(\thetastar + \mySigma^{-1/2}\bv)] - p\alpha/2$.
	Thus, $\cuE_\alpha(\bv_0^*) \leq \cuE_0(\bv_0^*)$ 
	and for $\|\bv - \bv_0^*\|_{L^2} \leq r$,
	$\cuE_\alpha(\bv) \geq \cuE_0(\bv) - p\lambda\alpha/(2\sqrt{n}) \geq \cuE_0(\bv_0^*) + a\|\bv - \bv_0^*\|_{L^2}^2/2 - p\lambda\alpha/(2\sqrt{n})$.
	Thus,
	if $\sqrt{\frac{p \lambdamax \alpha}{a\sqrt{n}}} \leq r$,
	we have
	$\|\bv^*_\alpha - \bv^*_0\|_{L^2} \leq \sqrt{\frac{p \lambdamax \alpha}{a\sqrt{n}}}$.
	Since, by the proof of Lemma \ref{LemAlphaFixedPtSoln},
        $\taustar_\alpha = \sqrt{\sigma^2 + \|\bv_\alpha^*\|_{L^2}^2}$ and $\zetastar_\alpha = (1 - \langle \bg, \bv_\alpha^*\rangle_{L^2}/\sqrt{n})$,
	we conclude
	\begin{equation}
		|\taustar_0 - \taustar_\alpha|
			\leq
			\sqrt{\frac{p\lambdamax\alpha}{a\sqrt{n}}} \,,
			\;\;\;\;
			|\zetastar_0 - \zetastar_\alpha|
			\leq
			\sqrt{\frac{p}{n}} \sqrt{\frac{p\lambdamax\alpha}{a\sqrt{n}}} 
			\;\;\;\;
			\text{for}
			\;\;\;\;
			\alpha \leq \frac{r^2a\sqrt{n}}{p\lambdamax}\,.
	\end{equation}
	Using that $n/p \geq \Deltamin$,
	we may set $\alphamax = r^2a\Deltamin/\lambdamax$, $L_\tau = \sqrt{\lambdamax/(a\Deltamin)}$, and $L_\zeta = \sqrt{\lambdamax/(a\Deltamin^2)}$.
\end{proof}

%%%%%%%%%%%%%%%%%%%%%%%%%%%%%%%%%%%%%%%%%%%%%%%%%%%%%%%%%%%%%%%%%%%%%%

\subsection{The fixed point solutions as a saddle point}

A crucial role in our analysis is played by the max-min problem
\begin{gather}
\label{EqnMinimaxPsi}
	\max_{\beta>0} \min_{\tau\geq \sigma} \psi_\alpha(\tau, \beta)\,, \\
\nonumber	\psi_\alpha(\tau, \beta) \defn - \frac{1}{2}\beta^2 - \frac{p-n}{2n}\tau\beta + 	\frac{\sigma^2\beta}{2\tau} + 
	\E\min_{\mytheta \in \real^{\usedim}} \left\{\frac{\beta}{2\tau}  
	\Big\|\mytheta - \thetastar - \frac{\tau}{\sqrt{n}}\mySigma^{-1/2}\bm{g}\Big\|_{\mySigma}^2 + 
	\frac{\lambda}{\sqrt{n}} (\sfM_\alpha(\mytheta) - \lone{\thetastar}) \right\}\,,
\end{gather}
where the expectation is taken over $\bg \sim \NORMAL(0,\Ind_\usedim)$.
We establish that Eqs.~\eqref{EqnAlphaFixedPt1} and \eqref{EqnAlphaFixedPt2} are first-order conditions for the solution to this max-min problem,
and in the non-smoothed ($\alpha = 0$) case, 
Eqs.~\eqref{EqnEqn1} and \eqref{EqnEqn2} are first-order conditions for the solution to this max-min problem.

\begin{lem}
\label{LemMinmax}
	Let $\taustar_\alpha$, $\zetastar_\alpha$ be the unique solution to 
	Eqs.~\eqref{EqnAlphaFixedPt1} and \eqref{EqnAlphaFixedPt2}, 
	and let $\betastar_\alpha = \taustar_\alpha\zetastar_\alpha$.
	Then $(\taustar_\alpha, \betastar_\alpha)$ is a saddle point for the  max-min value in Eq.~\eqref{EqnMinimaxPsi}.
        Namely, for all $\beta>0$, $\tau\ge\sigma$,
        \begin{gather}
          \psi_\alpha(\taustar_\alpha,\beta) \leq \psi_\alpha(\taustar_\alpha,\betastar_\alpha) \leq \psi_\alpha(\tau,\betastar_\alpha)\,,\label{eq:SaddlePoint}\\
           \psi_\alpha(\taustar_\alpha,\betastar_\alpha)=	\max_{\beta>0} \min_{\tau\geq \sigma} \psi_\alpha(\tau, \beta) =\min_{\tau\geq \sigma} \max_{\beta>0} \psi_\alpha(\tau, \beta)\, .\label{eq:Exchange}
          \end{gather}
\end{lem}

\begin{proof}[Proof of Lemma \ref{LemMinmax}]
	Let us define function 
	\begin{equation}
		\Xi_\alpha(\tau,\beta) 
			:=
			- \frac{1}{2}\beta^2 - \frac{p-n}{2n}\tau\beta + 	\frac{\sigma^2\beta}{2\tau} + 
			\min_{\mytheta \in \real^{\usedim}} \left\{
				\frac{\beta}{2\tau}  
				\Big\|\mytheta - \thetastar - \frac{\tau}{\sqrt{n}}\mySigma^{-1/2}\bm{g}\Big\|_{\mySigma}^2 + 
				\frac{\lambda}{\sqrt{n}} (\sfM_\alpha(\mytheta) - \lone{\thetastar})
			\right\} \, ,
	\end{equation}
        so that $\psi_\alpha(\tau, \beta)  = \E_{\bm{g}}\Xi_\alpha(\tau,\beta)$.
	It is easily seen that $\Xi_\alpha$ is convex-concave in $(\tau,\beta)$ for $\tau,\beta > 0$ because prior to the minimization over $\mytheta$ it is jointly convex in $(\tau,\mytheta)$ and concave in $\beta$.
	By the envelope theorem \cite[Theorem 1]{milgrom2002},
	\begin{gather}
        \label{EqnXiwrtBeta}
		\frac{\partial \Xi_\alpha}{\partial \beta} 
			= 
			-\beta
			- \frac{p-n}{2n}\tau
			+ \frac{\sigma^2}{2\tau} 
			+ \frac1{2\tau} 
			\Big\| \eta_\alpha ( \mySigma^{1/2}\thetastar + \tau \bg / \sqrt{n} , \beta/\tau ) - \thetastar - \frac{\tau}{\sqrt{n}} \mySigma^{-1/2}\bg \Big\|_{\mySigma}^2 \,,
		\\
		\label{EqnXiwrtTau}
		\frac{\partial \Xi_\alpha}{\partial \tau} 
			= 
			-\frac{p-n}{2n}\beta
			- \frac{\sigma^2\beta}{2\tau^2} 
          - \frac{\beta}{2\tau^2} 
			\| \eta_\alpha ( \mySigma^{1/2}\thetastar + \tau \bg/\sqrt{n} , \beta/\tau ) - \thetastar\|_{\mySigma}^2 
			+ \frac{\beta\|\bg\|_2^2}{2n}\,.
	\end{gather}
	Taking expectations with respect to $\bg$, exchanging expectations and derivatives by dominated convergence,
        and expanding the square in the first line, 
	we conclude
	\begin{align}
		\frac{\partial\psi_{\alpha}(\tau,\beta)}{\partial \beta} 
			&= 
			-\beta + \frac{\tau}2 
			+ \frac{\sigma^2}{2\tau} 
			+ \frac1{2\tau}\risk_\alpha(\tau^2,\beta/\tau) 
			- \tau\, \frac{\df_\alpha(\tau^2,\beta/\tau)}{n}
		\nonumber\\
		&= 
			\tau\left(
				-\frac{\beta}{\tau} 
				+ 1 - \frac{\df_\alpha(\tau^2,\beta/\tau)}{n}
			\right) 
			+ \frac1{2\tau}\left(
				-\tau^2 + \sigma^2 
				+ \risk_\alpha(\tau^2,\beta/\tau)
			\right)\,,
		\\
		\frac{\partial \psi_\alpha(\tau,\beta)}{\partial \tau} 
			&= 
			\frac{\beta}2 
			- \frac{\sigma^2 \beta}{2\tau^2} 
			- \frac{\beta}{2\tau^2} 
			\risk_\alpha(\tau^2,\beta/\tau) 
			= 
			\frac{\beta}{2\tau^2}
			(\tau^2 - \sigma^2 
			- \risk_\alpha(\tau^2,\beta/\tau))
		\,.
	\end{align}
	Thus, if $(\tau^*_\alpha,\zeta^*_\alpha) = (\taustar_\alpha,\betastar_\alpha/\taustar_\alpha)$ solves Eqs.~\eqref{EqnAlphaFixedPt1} and \eqref{EqnAlphaFixedPt2},
	the derivatives in the preceding display are 0. 
	Because $\psi_\alpha(\tau,\beta)$ is convex-concave in $(\tau,\beta)$, we conclude that, for any $\tau,\beta > 0$,
        Eq.~\eqref{eq:SaddlePoint} holds.
	Thus, $(\taustar_\alpha,\betastar_\alpha)$ is a saddle-point of $\psi_\alpha$ (see, e.g., \cite[pg.\ 380]{rockafellar1970convex}).
	By \cite[Lemma 36.2]{rockafellar1970convex}, the max-min value of \eqref{EqnMinimaxPsi} is achieved at $(\taustar_\alpha,\betastar_\alpha)$, and the maximization and minimization may be exchanged as in Eq.~\eqref{eq:Exchange}.
\end{proof}

%%%%%%%%%%%%%%%%%%%%%%%%%%%%%%%%%%%%%%%%%%%%%%%%%%%%%%%%%%%%%%%%%%%%%%

\section{Proofs of main results}

\subsection{Control of $\alpha$-smoothed Lasso estimate and proof of Theorem~\ref{ThmControlLassoEst} for a fixed $\lambda$}
\label{SecPfThmLassoL2}

The following theorem controls the behavior of the $\alpha$-smoothed lasso.
\begin{theo}
\label{ThmControlAlphaSmoothEstimate}
	If Assumption \ref{assump:1} holds and $\alpha \leq \alphamax / \sqrt{n}$,
	then there exist constants $C,c,c',\gamma > 0$ depending only on $\cuPmodel$ and $\alphamax$
	such that for any $1$-Lipschitz function $\phi:\real^p \rightarrow \reals$,
	we have for all $\epsilon < c'$
	\begin{align}
		\mprob \left(
			\exists \mytheta \in \reals^{\usedim},\,
			\Big|\phi\big(\mytheta\big) 
			- 
			\E\Big[\phi\big(\thetahat_\alpha^f\big)\Big]\Big| > 
			\epsilon
			\text{ and } 
			\Risk_\alpha(\mytheta) \leq \min_{\mytheta\in \reals^\usedim} \Risk_\alpha(
			\mytheta) + \gamma \epsilon^2
			\right) 
			\leq
			\frac{C}{\epsilon^2} e^{-c\numobs \epsilon^4}.
	\end{align}
\end{theo}
\noindent Theorem \ref{ThmControlLassoEst} for a fixed $\lambda$ is an immediate corollary of Theorem \ref{ThmControlAlphaSmoothEstimate} (apply Theorem \ref{ThmControlLassoEst} with $\alpha = 0$).
Uniformity over $\lambda$ is achieved in the next section (Section \ref{SecUniform}), completing the proof of Theorem \ref{ThmControlLassoEst}.

Define the error vectors of the $\alpha$-smoothed Lasso in the random-design model,
\begin{equation}
\label{EqnLassoErrVec}
	\bwhat_\alpha \defn \thetahat_\alpha - \thetastar\,, \;\;\;\;\;\; \bvhat_\alpha \defn \mySigma^{1/2}(\thetahat_\alpha - \thetastar)\,,
\end{equation}
where $\thetahat_\alpha$ is defined by \eqref{EqnAlphaSmoothRandomD}.
The error vector $\bvhat_\alpha$ is the minimizer of the reparameterized objective
\begin{align}
	\cuC_\alpha
		&(\bv) 
		:= 
		\frac1{2n} \|\bX \mySigma^{-1/2} \bv - \sigma \bz\|_2^2 + 
		\frac{\lambda}{\sqrt{n}}(\sfM_\alpha(\thetastar + \mySigma^{-1/2}\bv) - \|\thetastar\|_1)
	\nonumber\\
	&=
		\max_{\bu \in \reals^n} \left\{ \frac{1}{n} \bu^\top( \bX \mySigma^{-1/2} \bv - \sigma \bz)  - \frac{1}{2n}\|\bu\|_2^2 + \frac{\lambda}{\sqrt{n}}(\sfM_\alpha(\thetastar+\mySigma^{-1/2}\bv) - \|\thetastar\|_1) \right\}
		=:
		\max_{\bu \in \reals^n} C_\alpha(\bv,\bu)\,. \label{EnqCalpha}
\end{align}
We also define the error vector of the $\alpha$-smoothed Lasso in the fixed-design model
\begin{equation}
\label{EqnVf}
	\hbv_\alpha^f 
	:= 
	\mySigma^{1/2}(\eta_\alpha(\mySigma^{1/2}\thetastar + \taustar_\alpha \bg,\zetastar_\alpha) - \thetastar)\,.
\end{equation}
We control the behavior of $\alpha$-smoothed Lasso error $\hbv_\alpha$ in the random-design model using Gordon's minimax theorem \cite{thrampoulidis2015regularized,miolane2018distribution}.
Define Gordon's objective by
\begin{equation}
\begin{aligned}
	\cuL_\alpha(\bv)
		&:= 
		\frac12 \left(\sqrt{\|\bv\|_2^2 + \sigma^2} \frac{\|\bh\|_2}{\sqrt n} - \frac{\bg^\top\bv}{\sqrt{n}} \right)_+^2 + \frac\lambda{\sqrt{n}}( \sfM_\alpha(\thetastar + \mySigma^{-1/2}\bv) - \|\thetastar\|_1)\\
	&=
		\max_{\bu\in\reals^n} 
			\left\{
				-\frac{\|\bu\|_2 \bg^\top \bv}{n} + \sqrt{\|\bv\|_2^2 + \sigma^2}\, \frac{\bh^\top \bu}{n} - \frac{\|\bu\|_2^2}{2n} 
				+ \frac\lambda{\sqrt{n}} (\sfM_\alpha(\thetastar + \mySigma^{-1/2}\bv) - \|\thetastar\|_1)
			\right\}\\
	&=: \max_{\bu \in \reals^n} L_\alpha(\bv,\bu)\,,\label{EqnGLoss}
\end{aligned}
\end{equation}
where $\bg \sim \normal(\bzero_p,\Ind_p)$ and $\bh \sim \normal(\bzero_n,\Ind_n)$ are 
all independent.  
Gordon's lemma compares the (possibly constrained) minimization of $\cuC_\alpha(\bv)$ 
with the corresponding minimization of $\cuL_\alpha(\bv)$.
\begin{lem}[Gordon's lemma]
\label{LemGordon}
	The following hold.
	\begin{enumerate}[label=(\alph*)]
		
		\item % a
		Let $D \subset \reals^p$ be a closed set.
		For all $t \in \reals$,
		\begin{equation}
			\mprob\left(\min_{\bv \in D} \cuC_\alpha(\bv) \leq t\right)
			\leq 
			2\mprob\left( \min_{\bv \in D} \cuL_\alpha(\bv) \leq t \right)\,.
		\end{equation}

		\item % b
		Let $D \subset \reals^p$ be a closed, convex set.
		For all $t \in \reals$,
		\begin{equation}
			\mprob\left(\min_{\bv \in D} \cuC_\alpha(\bv) \geq t\right)
			\leq 
			2\mprob\left( \min_{\bv \in D} \cuL_\alpha(\bv) \geq t \right)\,.
		\end{equation}

	\end{enumerate}
\end{lem}
\noindent We prove Lemma \ref{LemGordon} later in this section.
\begin{proof}[Proof of Theorem \ref{ThmControlAlphaSmoothEstimate}]
	For any set $D$, define $D_\epsilon := \{\bx \in \reals^p \mid \inf_{\bx' \in D}\|\bx-\bx'\|_2 \geq \epsilon\}$.
	Denote $\Lstar_\alpha \defn \psi_\alpha(\taustar_\alpha, \betastar_\alpha)$ where $\taustar_\alpha,\betastar_\alpha$ are as in Lemma \ref{LemMinmax}.
	To control $\hbv_\alpha$ using Gordon's lemma,
	we show
	that with high probability
	the minimal value of $\cuL_\alpha$ is close to $\Lstar_\alpha$, and that if $D$ contains $\hbv_\alpha^f$ with high probability, the objective $\cuL_\alpha$ is uniformly sub-optimal on $D_\epsilon$ with high probability.
	We need the following lemma.
	\begin{lem}\label{LemGordonProbBounds}
		There exist constants $C,c,c',\gamma > 0$, 
		depending only on $\cuPmodel$ and $\alphamax$,
		such that for $\epsilon \in (0,c')$, 
		we have
		\begin{equation}
		\label{EqnGordonObjControl}
			\min_{\bv \in \mathsf{B}_2^c(\hbv_\alpha^f;\epsilon/2)} \cuL_\alpha(\bv) > \Lstar_\alpha + 2\gamma\epsilon^2\,,
			\;\;\;\;\;\;\;\;\;\;\;\;
			|\min_{\bv \in \reals^p} \cuL_\alpha (\bv) - \Lstar_\alpha| \leq \gamma\epsilon^2 \,,
		\end{equation}
		with probability at least $1 - \frac{C}{\epsilon^2}\exp(-cn\epsilon^4)$.
	\end{lem}
	\noindent We prove Lemma \ref{LemGordonProbBounds} at the end of this section.

	With $C,c,c',\gamma > 0$ as in Lemma \ref{LemGordonProbBounds},
	we have for $\epsilon < c'$
	\begin{align}
		&\mprob\left(
				\min_{\bv\in D_{\epsilon/2}} \cuC_\alpha(\bv) \leq \min_{\bv \in \reals^p} \cuC_\alpha(\bv) + \gamma\epsilon^2
			\right)
		\nonumber\\
		&\quad\leq 
			\mprob \left(
				\min_{\bv\in D_{\epsilon/2}} \cuC_\alpha(\bv) \leq \min_{\bv \in \reals^p} \cuC_\alpha(\bv) + \gamma\epsilon^2 
				~\text{ and }~ 
				\min_{\bv \in \reals^p} \cuC_\alpha(\bv) \leq \Lstar_\alpha + \gamma\epsilon^2
			\right)
			+
			\mprob\left( 
				\min_{\bv \in \reals^p} \cuC_\alpha(\bv) > \Lstar_\alpha + \gamma\epsilon^2 
			\right)
		\nonumber\\
		&\quad\leq\mprob\left( 
				\min_{\bv\in D_{\epsilon/2}} \cuC_\alpha(\bv) \leq \Lstar_\alpha + 2\gamma\epsilon^2
			\right)
			+
			\mprob\left(
				\min_{\bv \in \reals^p} \cuC_\alpha(\bv) > \Lstar_\alpha + \gamma\epsilon^2 
			\right)
		\nonumber\\
		&\quad\leq 2\mprob\left( 
				\min_{\bv\in D_{\epsilon/2}} \Loss_{\alpha}(\bv) \leq \Lstar_\alpha + 2\gamma\epsilon^2
			\right)
			+
			2\mprob\left( 
				\min_{\bv \in \reals^p} \Loss_{\alpha}(\bv) > \Lstar_\alpha + \gamma\epsilon^2 
			\right)
		\nonumber\\
		&\quad 
			\leq 2 \mprob \left(
				\hbv_\alpha^f \not \in D
			\right)
			+
			2\mprob\left( 
				\min_{\bv\in \mathsf{B}_2^c(\hbv_\alpha^f;\epsilon/2)} \Loss_{\alpha}(\bv) \leq \Lstar_\alpha + 2\gamma\epsilon^2
			\right)
			+
			2\mprob\left( 
				\min_{\bv \in \reals^p} \Loss_{\alpha}(\bv) > \Lstar_\alpha + \gamma\epsilon^2 
			\right)
		\nonumber\\
		&\quad \leq 2 \mprob \left(
				\hbv_\alpha^f \not \in D
			\right)
			+
			\frac{4C}{\epsilon^2} e^{-cn\epsilon^4}\,,
		\label{EqnMinimizationBound}
	\end{align}
	where the third-to-last inequality holds by Gordon's Lemma (Lemma~\ref{LemGordon});
	the second to last inequality holds because either $\hbv_\alpha^f \not \in D$ or $D_{\epsilon/2} \subset \mathsf{B}_2^c(\hbv_\alpha^f;\epsilon/2)$;
	and the last inequality holds by Lemma \ref{LemGordonProbBounds}.

	Define $\tilde\phi\left(\bv\right) := \kappamin^{1/2}\phi\left(\thetastar + \mySigma^{-1/2}\bv\right)$ (recall that $\thetastar$ is deterministic), with $\phi$ as in the statement of Theorem \ref{ThmControlAlphaSmoothEstimate}.
	Define the set
	\begin{align}
	\label{EqnDefnD}
		D \defn \left\{
			\bv \in \reals^p
			\Bigm| 
			\Big|
				\tilde\phi\big(\bv\big)
				-
				\E\Big[\tilde\phi\big(\hbv_\alpha^f\big)\Big]
			\Big| \leq \frac{\epsilon}{2}
		  \right\}\,.
	\end{align}
	By Eq.~\eqref{EqnVf} and recalling that $\betastar_{\alpha} = \zetastar_{\alpha}\taustar_{\alpha}$, we have
	\begin{equation}\label{EqnVLamProxForm}
		\hbv_\alpha^f = \arg\min_{\bv \in \reals^p} \left\{\frac{\betastar_\alpha}{2\taustar_\alpha}\Big\|\bv - \frac{\taustar_\alpha}{\sqrt{n}} \bg\Big\|_2^2 + \frac{\lambda}{\sqrt{n}} \|\thetastar + \mySigma^{-1/2} \bv\|_1\right\}.
	\end{equation}
	Thus, $\hbv_\alpha^f$ as a function of $\taustar_\alpha\bg/\sqrt{n}$ is a proximal operator, 
	whence $\hbv_\alpha^f$ is a $\tau_\alpha^*/\sqrt{n}$-Lipschitz function of $\bg$ \cite[pg.~131]{Parikh2013ProximalAlgorithms}.
	Gaussian concentration of Lipschitz functions \cite[Theorem 5.6]{boucheron2013concentration} guarantees that
	\begin{equation}
	\label{EqnVLamConcentration}
		\mprob\left(
				\hbv_\alpha^f \not \in D
			\right) 
			\leq 
			2\exp\left(
				-\frac{n\epsilon^2}{8{\taustar_\alpha}^2}
			\right) 
			\leq 
			2\exp\left(
				-\frac{n\epsilon^2}{8\taumax^2\delta}
			\right)
		\,.
	\end{equation}
	Combined with Eq.~\eqref{EqnMinimizationBound} and appropriately adjusting constants, 
	for $\epsilon < c'$
	\begin{equation}
		\mprob\left(
				\min_{\bv\in D_{\epsilon/2}} \cuC_\alpha(\bv) \leq \min_{\bv \in \reals^p} \cuC_\alpha(\bv) + \gamma\epsilon^2
			\right)
		\leq \frac{C}{\epsilon^2}e^{-cn\epsilon^4}\,.
	\end{equation}
	Because $\cuC_\alpha$ is a reparameterization of the $\alpha$-smoothed Lasso objective, 
	the preceding display is equivalent to
	\begin{equation}
		\mprob \left(
			\exists \mytheta \in \reals^{\usedim}, 
			\Big|\phi\big(\mytheta\big) 
			- \E\Big[\phi\big(\thetahat_\alpha^f\big)\Big]\Big| > 
			\kappamin^{-1/2}\epsilon
			\text{ and } 
			\Risk_\alpha(\mytheta) 
			\leq 
			\min_{\mytheta\in \reals^\usedim} \Risk_\alpha(
			\mytheta) + \gamma \epsilon^2
			\right) 
			\leq
			\frac{4C}{\epsilon^2} e^{-c\numobs \epsilon^4}.
	\end{equation}
	Theorem \ref{ThmControlAlphaSmoothEstimate} follows by a change of variables.	
\end{proof}

\begin{proof}[Proof of Lemma \ref{LemGordon}]
	Because $\sfM_\alpha(\thetastar + \mySigma^{-1/2}\bv) \rightarrow \infty$ as $\|\bv\|_2 \rightarrow \infty$,
	\begin{equation}
		\min_{\bv \in D} \cuC_\alpha(\bv) 
		= 
		\lim_{R\rightarrow \infty} \min_{ 
										\substack{ 	\bv \in D \\ 
													\|\bv\|_2 \leq R }
										} \cuC_\alpha(\bv)\,.
	\end{equation}
	Note that $\arg\max_{\bu \in \reals^n} C_\alpha(\bv,\bu) = \bX\mySigma^{-1/2}\bv - \sigma \bz$ has $\ell_2$-norm no larger than $\|\bX\mySigma^{-1/2}\|_{\mathrm{op}}\|\bv\|_2 + \sigma \|\bz\|_2$.
	In particular, for any realization of $\bX,\bz$,
	we have for $R$ sufficiently large that $\|\bv\|_2 \leq R$ implies $\|\arg\max_{\bu \in \reals^n} C_\alpha(\bv,\bu)\|_2 \leq R^2$.
	In particular, for any realization of $\bX,\bz$
	\begin{equation}
		\min_{ 
			\substack{ 	\bv \in D \\ 
						\|\bv\|_2 \leq R }
			} \cuC_\alpha(\bv)
		=
		\min_{ 
			\substack{ 	\bv \in D \\ 
						\|\bv\|_2 \leq R }
			} \max_{\|\bu\|_2\leq R^2} C_\alpha(\bv,\bu)\;\;\;\;
		\text{for $R$ sufficiently large,}
	\end{equation}
	where ``sufficiently large'' can depend on $\bX,\bz$.
	Thus, almost surely
	\begin{equation}
		\min_{ \bv \in D } \cuC_\alpha(\bv)
		=
		\lim_{R \rightarrow \infty} 
		\min_{ 
			\substack{ 	\bv \in D \\ 
						\|\bv\|_2 \leq R }
			} \max_{\|\bu\|_2\leq R^2} C_\alpha(\bv,\bu)\,.
	\end{equation}
	An equivalent argument shows that almost surely
	\begin{equation}
		\min_{ \bv \in D } \cuL_\alpha(\bv)
		=
		\lim_{R \rightarrow \infty} 
		\min_{ 
			\substack{ 	\bv \in D \\ 
						\|\bv\|_2 \leq R }
			} \max_{\|\bu\|_2\leq R^2} L_\alpha(\bv,\bu)\,.
	\end{equation}
	Because $\sqrt{n}\bX\mySigma^{-1/2}$ has iid standard Gaussian entries, 
	by Gordon's min-max lemma (see, e.g., \cite[Corollary G.1]{miolane2018distribution}),
	for any finite $R$ and closed $D$
	\begin{equation}
		\mprob\left(
			\min_{ 
			\substack{ 	\bv \in D \\ 
						\|\bv\|_2 \leq R }
			} \max_{\|\bu\|_2\leq R^2} C_\alpha(\bv,\bu)
			< t
		\right)
		\leq 
		2\mprob\left(
				\min_{ 
				\substack{ 	\bv \in D \\ 
							\|\bv\|_2 \leq R }
				} \max_{\|\bu\|_2\leq R^2} L_\alpha(\bv,\bu)
				< t
			\right)\,,
	\end{equation}
	and if $D$ is also convex
	\begin{equation}
		\mprob\left(
			\min_{ 
			\substack{ 	\bv \in D \\ 
						\|\bv\|_2 \leq R }
			} \max_{\|\bu\|_2\leq R^2} C_\alpha(\bv,\bu)
			> t
		\right)
		\leq 
		2\mprob\left(
				\min_{ 
				\substack{ 	\bv \in D \\ 
							\|\bv\|_2 \leq R }
				} \max_{\|\bu\|_2\leq R^2} L_\alpha(\bv,\bu)
				> t
			\right)\,.
	\end{equation}
	Although \cite[Corollary G.1]{miolane2018distribution}) states Gordon's lemma with weak inequalities inside the probabilities, 
	strict inequalities follow by applying \cite[Corollary G.1]{miolane2018distribution}) with $t' \uparrow t$and $t'\downarrow t$ in the previous two displays respectively.
	Taking $R \rightarrow \infty$, we conclude that the previous two displays hold without norm bounds on for $R$ sufficiently large $\bv$ and $\bu$.
	The strict inequalities can be made weak by applying the result with $t' \downarrow t$ and $t' \uparrow t$ respectively.
\end{proof}

\begin{proof}[Proof of Lemma \ref{LemGordonProbBounds}]
	The proof follows almost exactly the proof of Theorem B.1 in the supplementary material of \cite{miolane2018distribution}.
	For convenience and completeness, we walk the reader through the main steps, 
	but omit some steps which can be taken verbatim from \cite{miolane2018distribution}.

	Recall by Lemma \ref{LemMinmax} that the max-min value of \eqref{EqnMinimaxPsi}
	is achieved at $\taustar_\alpha$, $\betastar_\alpha$.
	We have $\betamin \leq \betastar_\alpha \leq \betamax$, 
	where $\betamin := \sigmamin\zetamin$ and $\betamax := \taumax$.
	Let $t = \min(\betamin/16,\sigmamin)$. 
	Define events
	\begin{equation}
	\begin{gathered}
          \Aevent_1
			\defn \left\{
				~\|\bg\|_2 \leq 2\sqrt{\usedim} + 2 \sqrt{n},
				~\left(1 - \frac{\betamin}{8\taumax}\right) \leq \frac{\|\bh\|_2}{\sqrt n} \leq 2
			\right\}\,,
		\\
		\Aevent_2 
			\defn \left\{
				\left|\ltwo{\hbv_\alpha^f}^2 - \Exs\left[\ltwo{\hbv_\alpha^f}^2\right]\right|\leq t^2,
				~\frac{\bg^\top \hbv_\alpha^f}{\sqrt{n}} \leq   \Exs\left[\frac{\bg^\top \hbv_\alpha^f}{\sqrt{n}}\right] + t\right\}
		\,.
	\end{gathered}
	\end{equation}
	There exist $r,a > 0$, 
	depending only on $\betamin,\betamax,\sigmamin,\taumax$ 
	such that on the event $\Aevent_1 \cap \Aevent_2$
	the objective $\cuL_\alpha$ is $a$-strongly convex on 
	$\mathsf{B}_2(\hbv_\alpha^f;r)$.
	This follows verbatim from the proof of Theorem B.1 in the supplementary material of \cite{miolane2018distribution} up to the fifth display on pg.~20,
	except for one small change:
	because we bound $\| \bg \|_2$ by $2\sqrt{p} + 2 \sqrt{n}$ rather than $2 \sqrt{p}$ as they do,
	the function $\bv \mapsto \sqrt{\| \bv \|_2^2 + \sigma^2} \| \bh \|_2 / \sqrt{n} - \bg^\top \bv / \sqrt{n}$ is $2 \sqrt{p/n} + 4$-Lipschitz rather than $2 \sqrt{p/n} + 2$-Lipschitz.
	This impacts the proof only by requiring adjusted constants.
	(We make this change to achieve a probability that decays exponentially in $n$ when $n/p \rightarrow \infty$. Note that \cite{miolane2018distribution} use a different but completely equivalent normalization to ours).

	Let $R = \taumax + r$, $\gamma = a/(96\Deltamin)$, $c' = \sqrt{ar^2/(24\gamma)}$, and $\epsilon \in (0, c')$.
	Define events 
	\begin{equation}
	\begin{gathered}
		\Aevent_3 \defn \left\{ \min_{\|\bv\|_2 \leq R} \cuL_\alpha(\bv) \geq \Lstar_\alpha - \gamma\epsilon^2 \right\}\,,\\
		\Aevent_4 \defn \left\{  \cuL_\alpha(\hbv_\alpha^f ) \leq \Lstar_\alpha + \gamma\epsilon^2 \right\}\,.
	\end{gathered}
	\end{equation}
	On event $\Aevent_3 \cap \Aevent_4$, 
	\begin{equation}
		\cuL_\alpha(\hbv_\alpha^f) \leq \min_{\|\bv\|_2 \leq R} \cuL_\alpha(\bv) + 2\gamma\epsilon^2 
			< 
			\min_{\|\bv\|_2 \leq R} \cuL_\alpha(\bv) + 3\gamma\epsilon^2\,.
	\end{equation}
	Because $3\gamma\epsilon^2 < a r^2 / 8$, 
	the previous display corresponds to (B.7) of the supplementary material of \cite{miolane2018distribution}.
	Thus, by Lemma B.1 of the supplementary material of \cite{miolane2018distribution},
	we have that on $\bigcap_{i=1}^4 \Aevent_i$,
	\begin{equation}
		\min_{\bv \in \mathsf{B}_2^c(\hbv_\alpha^f;\epsilon/2)} \cuL_\alpha(\bv) 
		= 
		\min_{\bv \in \mathsf{B}_2^c(\hbv_\alpha^f;\sqrt{24\gamma \epsilon^2/a})} \cuL_\alpha(\bv) \geq \min_{\bv \in \reals^p} \cuL_\alpha(\bv) + 3\gamma\epsilon^2\,.
	\end{equation}
	and
	\begin{equation}
		\min_{\bv \in \reals^p} \cuL_\alpha(\bv) = \min_{\|\bv\|_2 \leq R} \cuL_\alpha(\bv)\,.
	\end{equation}
	We conclude that on $\bigcap_{i=1}^4 \Aevent_i$,
	\begin{equation}
		\min_{\bv \in \mathsf{B}_2^c(\hbv_\alpha^f; \epsilon/2)} \cuL_\alpha(\bv) 
		\geq 
		\min_{\bv \in \reals^p} \cuL_\alpha(\bv) + 3\gamma\epsilon^2 
		= 
		\min_{\|\bv\|_2 \leq R} \cuL_\alpha(\bv) + 3\gamma\epsilon^2 
		\geq 
		\Lstar_\alpha + 2\gamma\epsilon^2\,,
	\end{equation}
	and
	\begin{equation}
		\Lstar_\alpha + \gamma\epsilon^2 
		\geq 
		\cuL_\alpha(\hbv_\alpha^f) 
		\geq 
		\min_{\bv \in \reals^p} \cuL_\alpha (\bv) 
		= 
		\min_{\|\bv\|_2 \leq \sqrt{n} R} \cuL_\alpha (\bv) 
		\geq 
		\Lstar_\alpha - \gamma\epsilon^2\,.
	\end{equation}
	Lemma \ref{LemGordonProbBounds} follows as soon as we show there exists $C,c,c' > 0$ depending only on $\cuPmodel$ and $\alphamax$ such that for $\epsilon < c'$ we have $\mprob(\cap_{i=1}^4 \Aevent_i) \geq 1 - \frac{C}{\epsilon^2}\exp(-cn\epsilon^4)$.

	Now to complete the proof of Lemma \ref{LemGordonProbBounds}, it is only left for us to control the probability of each $\Aevent_i$ respectively. 
	
	\subsubsection*{Event $\cuA_1$ occurs with high probability depending on $\betamin,\taumax,\delta$}
	Because $\bg \mapsto \|\bg\|_2$
	and $\bh \mapsto \|\bh\|_2$ are Lipschitz functions of standard Gaussian random vectors,
	there exist $C,c$ depending only on $\betamin,\taumax,\delta$ such that
	\begin{equation}\label{EqnA1Bound}
		\mprob(\Aevent_1) \geq 1 - C\exp(-cn)\,.
	\end{equation}

	\subsubsection*{Event $\cuA_2$ occurs with high probability depending on $\sigmamin,\betamin,\taumax$.}
	The function $\bg \mapsto \|\hbv_\alpha^f\|_2$ is $n^{-1/2}\taumax$-Lipschitz because $\hbv_\alpha^f$ is a proximal operator applied to $\taustar_\alpha \bg / \sqrt{n}$ by Eq.~\eqref{EqnVLamProxForm}
	\cite[pg.~131]{Parikh2013ProximalAlgorithms}.
	By Gaussian concentration of Lipschitz functions, $n^{-1/2} \|\hbv_\alpha^f\|_2$ is $\taumax^2/n$-sub-Gaussian.
	By the fixed point equations \eqref{EqnAlphaFixedPt1}, 
	we bound its expectation $\E[\|\hbv_\alpha^f\|_2] \leq \E[\|\hbv_\alpha^f\|^2]^{1/2} \leq \taumax $.
	Combining its sub-Gaussianity and bounded expectation, we conclude by Proposition G.5 of \cite{miolane2018distribution} that 
	\begin{equation}
		\text{$\|\hbv_\alpha^f\|_2^2$ is $(C/n,C/n)$-sub-Gamma for some $C$ depending only on $\taumax$.}
	\end{equation}

	Write 
	\begin{align*}
	    \taumax\bg^\top \hbv_\alpha^f/\sqrt{n} = (\|\hbv_\alpha^f - \taumax\bg / \sqrt{n}\|_2^2 - \|\hbv_\alpha^f\|_2^2 - \taumax^2\|\bg\|_2^2/n)/2\,. 
	\end{align*}
	Because $\bg \mapsto \hbv_\alpha^f - \taumax\bg/\sqrt{n}$ is $2\taumax/\sqrt{n}$-Lipschitz, 
	the first term is $(C/n,C/n)$-sub-Gamma for some $C$ depending only on $\taumax$.
	We conclude\footnote{We remark that the argument establishing that $\|\hbv_\alpha^f\|_2^2$ is sub-Gamma is exactly as it occurs in the proof of Lemma F.1 of the supplementary material of \cite{miolane2018distribution}.
	The argument establishing $\|\hbv_\alpha^f\|_2^2$ requires a slightly modified argument to that appearing in the proof of Lemma F.1 in \cite{miolane2018distribution} due to the presense of the matrix $\mySigma$.}
	\begin{equation}
		\text{$\taumax\bg^\top \hbv_\alpha^f/\sqrt{n}$ is $(C/n,C/n)$-sub-Gamma for some $C$ depending only on $\taumax$.}
	\end{equation}
	By standard bounds on the tails of sub-Gamma random variables,
	we deduce that there exist $C,c > 0$ depending only on $\taumax$, such that 
	\begin{equation}\label{EqnNormAndWidthBound}
		\begin{gathered}
                  \mprob\left(\Big|\|\hbv_\alpha^f\|_2^2 - \E\left[\|\hbv_\alpha^f\|_2^2\right] \Big|> \epsilon \right) \leq C
                  \exp\big(-cn(\epsilon^2\vee\epsilon)\big)\,,\\
			\mprob\left(\Big|\frac{\bg^\top\hbv_\alpha^f}{\sqrt{n}} - \E\left[\frac{\bg^\top\hbv_\alpha^f}{\sqrt{n}}\right]\Big| > \epsilon \right) \leq C\exp\big(-cn(\epsilon^2\vee\epsilon)\big)\,.
		\end{gathered}
	\end{equation}
	Because $t$ depends only on $\sigmamin,\betamin$,
	there exists $C,c > 0$ depending only on $\sigmamin,\betamin,\taumax$ such that
	\begin{equation}\label{EqnA2Bound}
		\mprob(\Aevent_2) \geq 1 - C\exp(-cn)\,.
	\end{equation}

	\subsubsection*{Event $\cuA_3$ occurs with high probability depending on $\cuPmodel$}
	Our control on the probability of $\cuA_3$ closely follows the proof of Proposition B.2 in the supplementary material of \cite{miolane2018distribution}.
	Consider for any $\epsilon > 0$ the event
	\begin{equation}
	\label{EqnGordonNoiseConc}
		\cuA_3^{(1)}
		:= 
		\left\{\left|\frac{\|\bh\|_2}{\sqrt n} - 1\right| \leq \epsilon\right\} \,.
	\end{equation}
	By Gaussian concentration of Lipschitz functions, 
	$\mprob(\cuA_3^{(1)}) \geq Ce^{-cn\epsilon^2}$ for all $\epsilon \geq 0$.
	
	By maximizing over over $\bu$ for which $\|\bu\|/\sqrt{n} = \beta$ in Eq.~\eqref{EqnGLoss}, 
	we compute
	\begin{align*}
		\Loss_\alpha(\bv) 
		&= 
			\max_{\beta \geq 0} 
			\left(\sqrt{\|\bv\|_2^2 + \sigma^2} \frac{\|\bh\|_2}{\sqrt n} - \frac{\bg^\top \bv}{\sqrt{n}} \right) \beta 
			- \frac12 \beta^2 
			+ \frac{\lambda}{\sqrt{n}}(\sfM_\alpha(\thetastar + \mySigma^{-1/2}\bv) - \lone{\thetastar})
		\\
		&=: 
			\max_{\beta \ge 0} \ell_\alpha(\bv,\beta)\,.
	\end{align*}
	Consider the slightly modified objective
	\begin{align*}
		\ell_\alpha^0(\bv,\beta) 
			:= 
			\left(\sqrt{\ltwo{\bv}^2 + \sigma^2} 
			- \frac{\bg^\top \bv}{\sqrt{n}} \right)\beta 
			- \frac{1}{2} \beta^2 
			+ \frac{\lambda}{\sqrt{n}}(\sfM_\alpha(\thetastar + \mySigma^{-1/2}\bv) - \lone{\thetastar})\,.
	\end{align*}
	On the event \eqref{EqnGordonNoiseConc}, for every $\|\bv\|_2 \leq  R$ and $\beta\in[0,\betamax]$,
	\begin{align*}
		|\ell_\alpha(\bv,\beta) - \ell_\alpha^0(\bv,\beta)| 
		\leq 
		\betamax (R^2 + \sigma^2)^{1/2} \epsilon\,.
	\end{align*}
	Thus, on the event \eqref{EqnGordonNoiseConc},
	\begin{align}
		\min_{\|\bv\|_2 \leq R} \Loss_\alpha(\bv) 
			&= 
			\min_{\|\bv\|_2 \leq R} \max_{\beta\geq 0} \ell_\alpha(\bv,\beta) 
			\geq 
			\min_{\|\bv\|_2 \leq R}\ell_\alpha(\bv,\betastar_\alpha) 
		\\
		&\geq 
			\min_{\|\bv\|_2 \leq R} \ell_\alpha^0(\bv,\betastar_\alpha) - \betamax(R^2 + \sigma^2)^{1/2}\epsilon\,.
		\label{EqEllToEll0}
	\end{align}
	For $\|\bv\|_2 \leq R$, 
	\begin{align*}
		\sqrt{\ltwo{\bv}^2 + \sigma^2} 
		= 
		\min_{\tau \in [\sigma,\sqrt{\sigma^2 + R^2}] }
		\left\{\frac{\|\bv\|_2^2 + \sigma^2}{2\tau} + \frac{\tau}{2}\right\}.
	\end{align*}
	Thus, we obtain that 
	\begin{align*}
		\ell_\alpha^0(\bv,\betastar_\alpha) 
			= \min_{\tau \in [\sigma,\sqrt{\sigma^2 + R^2}]} 
			\left\{ 
				\left(
					\frac{\ltwo{\bv}^2 + \sigma^2}{2\tau} + \frac{\tau}{2}
				\right)\betastar_\alpha
		 		- \frac{\bg^\top \bv}{\sqrt{n}} \betastar_\alpha 
		 		- \frac{1}{2} {\betastar_\alpha}^2 
		 		+ \frac\lambda{\sqrt{n}}(\sfM_\alpha(\thetastar + \mySigma^{-1/2}\bm{\bv}) - \lone{\thetastar})
		 	\right\}\,,
	\end{align*}
	which further implies that 
	\begin{align*}
		&\min_{\|\bv\|_2 \leq R} \ell_\alpha^0(\bv,\betastar_\alpha)
		\\
		&= 
		\min_{\tau \in [\sigma,\sqrt{\sigma^2 + R^2}]} 
			\left\{
				\frac{\betastar_\alpha}{2}
				\left(\frac{\sigma^2}{\tau} + \tau \right)
				- \frac12 {\betastar_\alpha}^2 
				+ \min_{\|\bv\|_2 \leq R}
				\left\{ 
					\frac{\betastar_\alpha}{2\tau}\ltwo{\bv}^2 
					- \betastar_\alpha \frac{\bg^\top \bv}{\sqrt{n}}
					+ \frac{\lambda}{\sqrt{n}}(\sfM_\alpha(\thetastar + \mySigma^{-1/2}\bm{\bv}) - \lone{\thetastar}) 
				\right\} 
			\right\}
		\\
	    &=:
	    	\min_{\tau \in [\sigma,\sqrt{\sigma^2 + R^2}]} F(\tau,\bg)\,.
	\end{align*}
	We claim that $F(\tau,\bg)$ concentrates around its expectation.
	In order to see this, first note that for every $\tau \in [\sigma,\sqrt{\sigma^2 + R^2}]$, the function 
	\begin{align*}
		\bg 
			\mapsto 
			\min_{\|\bv\|_2 \leq R}
			\left\{
				\frac{\betastar_\alpha}{2\tau}\ltwo{\bv}^2 
				- \betastar_\alpha \frac{\bg^\top \bv}{\sqrt{n}}
				+ \frac{\lambda}{\sqrt{n}}(\sfM_\alpha(\thetastar + \mySigma^{-1/2}\bv) - \lone{\thetastar}) 
			\right\}
	\end{align*}
	is $\betamax R/\sqrt{\numobs}$-Lipschitz,
	whence $\bg \mapsto F(\tau, \bg)$ is as well.
	By Gaussian concentration of Lipschitz functions \cite[Theorem 5.6]{boucheron2013concentration},
	\begin{align*}
		\mprob(|F(\tau,\bg) -  \E[F(\tau,\bg)]| > \epsilon) \leq 2e^{-c\numobs\epsilon^2},
	\end{align*}
	for $c = 1/(2{\betamax}^2R^2)$.
	Because $\tau \geq \sigmamin > 0$, 
	for all $\bg$ the function $\tau \mapsto F(\tau,\bg)$ is 
	$(\betamax + \betamax R^2/(2\sigmamin^2) )$-Lipschitz on $[\sigma,\sqrt{\sigma^2 + R^2}]$,
	so that by an $\epsilon$-net argument, 
	we conclude that for $C,c$ depending only on $R,\betamax,\sigmamin$ that
	\begin{align*}
		\mprob\left(\cuA_3^{(2)}\right) := \mprob\left(\sup_{\tau \in [\sigma,\sqrt{\sigma^2 + R^2}]} |F(\tau,\bg) -  \E[F(\tau,\bg)]| \leq \epsilon\right) 
			\geq 
			1 - \frac{C}{\epsilon}e^{-c\numobs\epsilon^2}\,.
	\end{align*}
	On $\cuA_3^{(2)}$, 
	\begin{align}
		\min_{\|\bv\|_2 \leq R}
			\ell_\alpha^0(\bv,\betamax)
			= 
			\min_{\tau \in [\sigma,\sqrt{\sigma^2 + R^2}]} F(\tau,\bg) 
			\geq 
			\min_{\tau \in [\sigma,\sqrt{\sigma^2 + R^2}]}
		 	\E[F(\tau,\bg)] - \epsilon\,.
		\label{EqEll0ToF}
	\end{align}

	We compute 
	\begin{align} 
		&F(\tau,\bg) 
			= 
			\frac{\betastar_\alpha}{2}\left( \frac{\sigma^2}{\tau} + \tau \right) 
			- \frac12 {\betastar_\alpha}^2 
			+ 
	    	\min_{\|\bv\|_2 \leq R}\left\{ 
	    		\frac{\betastar_\alpha}{2\tau}\ltwo{\bv}^2 
	    		- \betastar_\alpha \frac{\bg^\top \bv}{\sqrt{n}}
	    		+ \frac{\lambda}{\sqrt{n}}(\sfM_\alpha(\thetastar + \mySigma^{-1/2}\bv) - \lone{\thetastar}) 
	    	\right\} 
	    \nonumber
	    \\ 
		&=  
			\frac{\betastar_\alpha}{2}\left(\frac{\sigma^2}{\tau} + \tau \right) 
			- \frac12 {\betastar_\alpha}^2 
			- \frac{\betastar_\alpha\tau\|\bg\|_2^2}{2n} 
			+ 
			\min_{\|\bv\|_2 \leq R}\left\{ 
				\frac{\betastar_\alpha}{2\tau}\Big\|\bv - \frac{\tau}{\sqrt{n}} \bg\Big\|_2^2 
				+ \frac{\lambda}{\sqrt{n}}(\sfM_\alpha(\thetastar + \mySigma^{-1/2}\bv) - \lone{\thetastar}) 
			\right\} 
		\nonumber
		\\ 
		&\geq 
			\frac{\betastar_\alpha}{2}\left(\frac{\sigma^2}{\tau} + \tau \right) 
			- \frac12 {\betastar_\alpha}^2 
			- \frac{\betastar_\alpha\tau\|\bg\|_2^2}{2n} 
			+ 
			\min_{\bv \in \reals^p}\left\{ 
				\frac{\betastar_\alpha}{2\tau}\Big\|\bv - \frac{\tau}{\sqrt{n}} \bg\Big\|_2^2 
				+ \frac{\lambda}{\sqrt{n}}(\sfM_\alpha(\thetastar + \mySigma^{-1/2}\bv) - \lone{\thetastar}) 
			\right\}\,. 
	\end{align}
	Taking expectations, we have for any $\tau \geq 0$
	\begin{align} 
		\E[F(\tau, \bg)] 
			=
			\psi_\alpha(\tau, \betastar_\alpha)\,,
		\label{EqFToPsi}
	\end{align}
	where $\psi_\alpha$ is defined in \eqref{EqnMinimaxPsi}.

	Combining 
	Eqs.~\eqref{EqEllToEll0}, \eqref{EqEll0ToF}, 
	and \eqref{EqFToPsi}, 
	we conclude that on $\cuA_3^{(1)} \cap \cuA_3^{(2)}$
	\begin{equation}
		\min_{\|\bv\|_2 \leq R} \Loss_\alpha(\bv) 
			\geq 
			\psi_\alpha(\taustar_\alpha,\betastar_\alpha) - K \epsilon\,,
	\end{equation}
	with $K =\betamax \sqrt{R^2 + \sigma^2}+1$.
	By a change of variables 
	and applying the probability bounds on $\cuA_3^{(1)}$ and $\cuA_3^{(2)}$ establishes
	\begin{equation}
		\mprob(\cuA_3) 
			\geq 
			\mprob(\cuA_3^{(1)} \cap \cuA_3^{(2)}) 
			\geq
			1 - \frac{C}{\epsilon^2}\exp\left(-cn\epsilon^4\right)\,,
	\end{equation}
	for some $C,c$ depending only on $R$, $\sigmamax$, $\betamax$, $\sigmamin$, and $\gamma$,
	and hence only on $\cuPmodel$.
	
	\subsubsection*{Event $\cuA_4$ occurs with high probability depending on $\cuPmodel$}
	There exist $C,c>0$ depending only on $\kappamin,\Deltamin,\taumax$ such that for $\epsilon > 0$,
	\begin{equation}\label{EqnPenaltyAndGaussBounds}
		\begin{gathered}
			\mprob\left(
					\frac{\sfM_\alpha(\thetastar + \mySigma^{-1/2}\hbv_\alpha^f)}{\sqrt{n}} - \E\left[\frac{\sfM_\alpha(\thetastar + \mySigma^{-1/2}\hbv_\alpha^f)}{\sqrt{n}}\right] > \epsilon 
				\right) 
				\leq 
				C\exp(-cn\epsilon^2)\,,
			\\
		\end{gathered}
	\end{equation}
	because 
	$\bg \mapsto \sfM_\alpha(\thetastar + \mySigma^{-1/2}\hbv_\alpha^f)/\sqrt{n}$ 
	is $\kappa_{\mathrm{min}}^{-1/2}\sqrt{p/n}\taumax/\sqrt{n}$-Lipschitz. 
	For any $x_0 \geq 0$, note that $x \mapsto x_+^2$ is locally Lipschitz in any ball around $x_0$ 
	with Lipschitz constant and ball radius depending only on an upper bound on $|x_0|$.
	Thus, considering $x_0 = \betastar_\alpha$, there exists $L,c' > 0$ depending only on $\betamax $ such that for $\epsilon < c'$,
	if $\cuA_2$ and $\cuA_3^{(1)}$ occur, 
	then 
	\begin{equation}
		\left|\left(\sqrt{\|\bv\|_2^2 + \sigma^2} \frac{\|\bh\|_2}{\sqrt n} 
		- \frac{\bg^\top \bv}{\sqrt{n}} \right)_+^2 - {\betastar_\alpha}^2\right| \leq L\epsilon\,.
	\end{equation}
	Using the probability bounds on $\cuA_2$ and $\cuA_3^{(1)}$ and absorbing $L$ into constants,
	we may find $C,c,c' > 0$ depending only on $\gamma,\sigmamin,\kappamin,\Deltamin,\betamin,\taumax$
	such that for $\epsilon < c'$,
	\begin{equation}
		\mprob(\Aevent_4) \geq 1 - C\exp(-cn\epsilon^4)\,.
	\end{equation}

	Lemma \ref{LemGordonProbBounds} is established now follows by combining the probability bounds on $\cuA_i$ for $1\leq i\leq 4$.
\end{proof}

%%%%%%%%%%%%%%%%%%%%%%%%%%%%%%%%%%%%%%%%%%%%%%%%%%%%%%%%%%%%%%%%%%%%%%

% \input{uniform}
\subsection{Uniform control over $\lambda$: proof of Theorem~\ref{ThmControlLassoEst}}
\label{SecUniform}

In this section, 
we complete the proof of Theorem \ref{ThmControlLassoEst} by showing that Theorem \ref{ThmControlAlphaSmoothEstimate} holds uniformly over $\lambda$, at a very small cost in the rate of concentration.

To make the dependence of the Lasso objective on $\lambda$ explicit, we write $\Risk^\lambda(\mytheta)$ for Eq.~\eqref{EqnOrgRisk}.
As before, $\CLoss^\lambda(\bv)$ is a re-parametrization of $\Risk^\lambda(\mytheta)$, namely  $\CLoss^\lambda(\bv)\defn \Risk^\lambda(\thetastar+
\mySigma^{-1/2}\bv) - \lambda \| \thetastar \|_1 / \sqrt{n}$. 
We also write $\bthetahat^\lambda$ for the minimizer of $\Risk^\lambda(\mytheta)$ and $\hbv^\lambda$ for the minimizer of $\CLoss^\lambda(\bv)$ (in particular   $\bthetahat^\lambda=\thetastar+\mySigma^{-1/2}\hbv^\lambda$).  
Finally, in order to expose the full dependency 
of $\eta$ on the regularization parameter $\lambda$, we redefine
\begin{gather}
\label{EqnSTH-BIS}
	\eta(\by^f,\zeta/\lambda) \defn \argmin_{\mytheta \in \reals^{\usedim}}
	\left\{\frac{\zeta/\lambda}{2}\ltwo{\by^f - \mySigma^{1/2}\mytheta}^2 + \frac{1}{\sqrt{n}}\lone{\mytheta} \right\}\,.
\end{gather}
Throughout this section we will use this definition instead of Eq.~\eqref{EqnSTH}. 
We denote the Lasso error vector in the fixed-design model at regularization $\lambda$ by 
\begin{equation}
	\hbv^{f,\lambda} 
	:=
	\mySigma^{1/2}(\eta(\mySigma^{1/2}\thetastar + \taustar \bg / \sqrt{n},\zetastar/\lambda) - \thetastar)\,,
\end{equation}
where implicitly $\taustar,\zetastar$ depend on $\lambda$ via the fixed-point Eqs.~\eqref{EqnEqn1} and \eqref{EqnEqn2}.
For simplicity, we write $\tilde\phi\big(\bv\big)=\phi\big(\thetastar+\mySigma^{-1/2 }\bv\big)$.
For $\lambda\in [\lambda_{\min},\lambda_{\max}]$, let
\begin{align}
		D_\epsilon^\lambda \defn \left\{
			\bv \in \reals^p
			\Bigm| 
			\Big|
				\tilde\phi\big(\bv\big)
				-
				\E\big[\tilde\phi\big(\hbv^{f,\lambda}\big)\big]
			\Big| > \epsilon
		  \right\}\,.
\end{align}
Define $\cuE^\lambda: L^2(\reals^p;\reals^p) \rightarrow \reals$ as $\cuE(\bv) = \cuE_0(\bv)$, where $\cuE_0$ is as in the proof of Lemma \ref{SecFixedPtSoln}, and we make dependence on $\lambda$ explicit in the notation.
In particular,
\begin{equation}
	\cuE^\lambda(\bv) 
			:= 
			\frac12\Big(
				\sqrt{\|\bv\|_{L^2}^2+\sigma^2 }\,-\frac{\langle \bg , \bv \rangle_{L^2}}{\sqrt{n}}
			\Big)^2_+
			+
			\frac{\lambda}{\sqrt{n}}
			\E\Big\{
				\|\thetastar+\mySigma^{-1/2}\bv(\bg)\|_1-\|\thetastar\|_1
			\Big\}.
\end{equation}
We emphasize that the argument $\bv$ is not a vector but a function $\bv: \reals^p \rightarrow \reals^p$.
Recall, by the proof of Lemma \ref{LemAlphaFixedPtSoln}, that $\hbv^{f,\lambda}$, viewed as a function of $\bg$ and thus a member of $L^2(\reals^p;\reals^p)$, is the unique minimizer of $\cuE^\lambda$.

The proof of Theorem \ref{ThmControlLassoEst} relies on two lemmas.
The first quantifies the sensitivity of the Lasso problem \eqref{EqnOrgRisk} to the regularization parameter $\lambda$. 
The second quantifies the continuity of the minimizer of the objective function $\cuE^\lambda$ in the regularization parameter $\lambda$.
\begin{lem}
\label{LemCostComparisonLambda}
 Under Assumption \ref{assump:1}, there exist constants $K,C_0,c_0 > 0$ depending only on $\cuPmodel$ such that  
\begin{align*}
	\mprob\left( \forall \lambda, \lambda' \in [\lambdamin, \lambdamax],
	~\CLoss^{\lambda'} (\hbv^\lambda) \leq \min_{\bv \in \real^p}\CLoss^{\lambda'} (\bv)
	+ K  |\lambda - \lambda'|\right)
	\geq 
	1 - C_0e^{-c_0n}.
\end{align*}
\end{lem}
\begin{lem}
\label{LemPopCostComparisonLambda}
	Under Assumption \ref{assump:1}, there exists constants $K,c' > 0$ depending only on $\cuPmodel$ such that for all $\lambda,\lambda' \in [\lambdamin,\lambdamax]$ with $|\lambda-\lambda'| < c'$ we have
	\begin{equation}
		\text{for all $\lambda,\lambda' \in [\lambdamin,\lambdamax]$,}
		\;\;\;\;\;\;
		\|\hbv^{f,\lambda'} - \hbv^{f,\lambda}\|_{L^2}  
		\leq 
		K|\lambda'-\lambda|^{1/2}\,,
	\end{equation}
	where in the previous display we view $\hbv^{f,\lambda}$, $\hbv^{f,\lambda'}$ as functions of the same random vector $\bg$ and thus as members of $L^2(\reals^p;\reals^p)$.
\end{lem}
\noindent The characterization of the Lasso solution involves only the distribution of $\hbv^{f,\lambda}$.
The preceding lemma implicitly constructs a coupling between these distributions defined for different values of $\lambda$ by using the same source of randomness $\bg$ in defining $\hbv^{f,\lambda}$ and $\hbv^{f,\lambda'}$.
We prove Lemma \ref{LemCostComparisonLambda} and \ref{LemPopCostComparisonLambda} in Sections~\ref{sec:pfLemmaCostComparisonLambda} and \ref{sec:pfPopCostComparisonLambda} respectively. 

To achieve a uniform control over $\lambda \in [\lambdamin, \lambdamax]$, 
we invoke an $\epsilon$-net argument.
Consider $\epsilon < c'$, where $c'$ is as in Theorem \ref{ThmControlAlphaSmoothEstimate}.
Let $C_0,c_0$ be as in Lemma \ref{LemCostComparisonLambda} and Lemma 
and let $K_1,K_2$ be the $K$'s which appear in Lemma \ref{LemCostComparisonLambda} and Lemma \ref{LemPopCostComparisonLambda}, respectively.
Set $\epsilon' = \min\left\{\gamma\epsilon^2/K_1,\epsilon\kappamin^{1/2}/K_2\right\}$.
Define  $\lambda_{i} = \lambdamin + i \epsilon'$ for $i = 1,\ldots, k$, $k\defn \lfloor \frac{\lambdamin - \lambdamax}{\epsilon'} \rfloor$ and $\lambda_{k+1} = \lambdamax$. 

By a union bound over $\lambda_{i}$,
Theorem~\ref{ThmControlAlphaSmoothEstimate} implies that, 
for $C,c,c',\gamma > 0$ depending only on $\cuPmodel$,
with probability at least $1 - \frac{C(k+1)}{\epsilon^2}\exp(-cn\epsilon^4)$,
\begin{align}
\label{EqnChopin}
	\forall \bv \in \real^\usedim,~\forall \lambda_i, \quad
	\CLoss_{\lambda_i}(\bv) \leq \min_{\bv \in \real^p} \CLoss_{\lambda_i}(\bv) + \gamma \epsilon^2
	~\Rightarrow~ \bv \in (D_\epsilon^{\lambda_i})^c\,.
\end{align}
Further, Lemma \ref{LemCostComparisonLambda} implies that
with probability at least $1 - C_0 e^{-c_0n}$,
the following occurs:
for all $\lambda \in [\lambdamin,\lambdamax]$
\begin{align*}
	\CLoss_{\lambda_i} (\hbv^\lambda) \leq \min_{\bv \in \real^p}\CLoss_{\lambda_i} (\bv)
	+ K |\lambda - \lambda_i| 
	\leq 
	\min_{\bv \in \real^p}\CLoss_{\lambda_i} (\bv) + \gamma \epsilon^2,
	% + K p^{t} |\lambda - \lambda'|  \leq 
\end{align*}
where $i = i(\lambda)$ is chosen such that $\lambda \in [\lambda_i, \lambda_{i+1}]$
and the inequality holds by the choice of $\epsilon'$.
Combining with inequality~\eqref{EqnChopin}, we conclude that 
\begin{equation}
	\text{for all $\lambda$,}\; \hbv^\lambda \in (D_{\epsilon}^{\lambda_i})^c \;\;\; \text{where $i=i(\lambda)$ is such that } \lambda \in [\lambda_i,\lambda_{i+1}]\,.
\end{equation}
Because $\phi$ is $1$-Lipschitz,
\begin{equation}
\begin{aligned}
	\left|\Exs \tilde\phi\big(\hbv^{f,\lambda}\big)
	-  \Exs \tilde\phi\big(\hbv^{f,\lambda_i}\big)\right| 
		&\leq 
		\Exs \left[\ltwo{\mySigma^{-1/2} (\hbv^{f,\lambda} - \hbv^f(\lambda_i)})\right]
		\leq  
		\kappamin^{-1/2} \| \hbv^{f,\lambda} - \hbv^{f,\lambda_i}\|_{L^2}
		\\
	&\leq K\kappamin^{-1/2}|\lambda_i - \lambda|^{1/2}
		\leq 
		\epsilon\,,
\end{aligned}
\end{equation}
where the third-to-last inequality holds by Jensen's inequality,
and the second-to-last inequality holds by Lemma \ref{LemPopCostComparisonLambda},
and the last inequality holds by the choice of $\epsilon'$.
Note we have compared the two expectations on the left-hand side by constructing a coupling between the distribution of $\hbv^{f,\lambda}$ defined for different values of $\lambda$; see comment following Lemma \ref{LemPopCostComparisonLambda}. 
By the triangle inequality, if $\hbv^\lambda \in (D_{\epsilon}^{\lambda_i})^c$,
then $\hbv^\lambda \in (D_{2 \epsilon}^\lambda)^c$.
Thus,
we conclude that with $C,c,c' > 0$ depending only on $\cuPmodel$
\begin{equation}
	\mprob\left(
		\exists \lambda \in [\lambdamin,\lambdamax],
		\;\;
		\hbv^\lambda \in D_{2\epsilon}^\lambda
	\right)
	\geq 1 - \frac{C(k+1)}{\epsilon^2} e^{-cn\epsilon^4}.
\end{equation}
For $\epsilon < c'$, 
we have $(k+1) \leq C/\epsilon^2$ for some $C$ depending only on $\cuPmodel$.
Absorbing constants appropriately,
the proof of Theorem~\ref{ThmControlLassoEst} is complete.

\subsection{Control of Lasso residual: proof of Theorem \ref{ThmLassoResidual}}
\label{SecPfThmLassoResidual}

Like the proof of Theorem~\ref{ThmControlLassoEst}, the proof of Theorem \ref{ThmLassoResidual} uses Gordon's lemma. 
Specifically, denote
\begin{align}
	\uhat \defn \bX \widehat{\bw} - \sigma \bz = \bX\thetahat - \by,
\end{align}
where $ \bwhat := \bwhat_0 = \thetahat - \thetastar$
 as defined in Eq.~\eqref{EqnLassoErrVec}.
Then $\uhat$ is the unique maximizer of 
\begin{align}
\bm{u} \mapsto	\min_{\bw \in \reals^{p}} 
	\left\{\frac{1}{\numobs} \inprod{\bX \bw - \sigma \bz}{\bm{u}} - \frac{1}{2\numobs}\ltwo{\bm{u}}^2 + \frac{\lambda}{\sqrt{n}}(\lone{\bw +\thetastar} - \lone{\thetastar}) \right\},
\end{align}
where the function on the right hand side (before minimizing over $\bw$) is defined as $C_{0}(\bv,\bu) =:C(\bv,\bu)$ in expression~\eqref{EnqCalpha} with re-parametrization $\bv \defn \mySigma^{1/2} \bw.$ Compared with the analysis in Theorem~\ref{ThmControlLassoEst} which focuses on the behavior of $\widehat{\bv}$, the focus of this section is the behavior of $\uhat$.
% Denote $\bv \defn \mySigma^{1/2} \bw$ and $\bm{G} = n^{1/2}\bX \mySigma^{-1/2}$. 

\paragraph*{Study of the corresponding Gordon's problem} 
Recall Gordon's optimization problem defined in expression \eqref{EqnGLoss} with $\alpha = 0$ and $\sfM_\alpha(\mytheta) = \lone{\mytheta}$. 
For every $(\bv,\bu)$, we have (defining $L(\bv,\bu) = L_0(\bv,\bu)$, cf. Eq~\eqref{EqnGLoss}):
\begin{align*}
L(\bv,\bu) \defn - \frac{\ltwo{\bm{u}} \bg^\top \bv}{n} + 
	\sqrt{\ltwo{\bv}^2 + \sigma^2} \frac{\bm{h}^\top \bm{u}}{n}
	- \frac{\ltwo{\bm{u}}^2}{2n}
	+ \frac{\lambda}{\sqrt{n}}(\lone{\mySigma^{-1/2}\bv +\thetastar} - \lone{\thetastar}).
\end{align*}
Denote $U(\bm{u}) = \min_{\bm{v}\in \reals^p} L(\bv,\bm{u})$
and $\widetilde{U}(\bm{u}) = L(\hbv^f,\bm{u})$ where $\hbv^f$ is defined in expression~\eqref{EqnVf} with $\alpha = 0$, namely 
\begin{align*}
	\hbv^f \defn \mySigma^{1/2} \left[\eta\left(\thetastar + \frac{\taustar}{\sqrt{n}}\mySigma^{-1/2}\bg, ~\frac{\betastar}{\taustar}\right) - \thetastar\right].
\end{align*}
By definition, $U(\bm{u}) \leq \widetilde{U}(\bm{u}).$ 
From direct calculations, the maximizer of $\widetilde{U}(\bm{u})$ is 
\[
	\frac{\bm{u}}{\sqrt{n}} = \left(\sqrt{\ltwo{\hbv^f}^2 + \sigma^2} \frac{\|\bh\|_2}{\sqrt{n}} - \frac{\bg^\top \hbv^f}{\sqrt{n}} \right)_+ \frac{\bh}{\|\bh\|_2}.
\]
Let us define quantity $\util \defn \taustar \zetastar \bh$.
By the concentration of $\hbv^f$ (given by inequality  \eqref{EqnNormAndWidthBound}) and the definition of the $(\taustar, \zetastar)$ in~\eqref{EqnEqn1} and \eqref{EqnEqn2},  $\util$ is $\epsilon$-close to the maximizer of $\widetilde{U}(\bm{u})$
(in the sense that $\|\util-\bu^*\|_2/\sqrt{n}\le \epsilon$).
In particular, 
Lemma D.1~\cite{miolane2018distribution} holds verbatim here.

Define the set
\begin{align}
	D_\epsilon \defn \left\{ \bu \in \real^p ~\mid~
	\Big|\phi\Big(\frac{\bu}{\sqrt n}\Big) - \E\Big[\phi\Big(\frac{\taustar\zetastar \bh}{\sqrt n}\Big)\Big]\Big| > \epsilon \right\}\,.
\end{align}
The probability $\mprob(\uhat \in D_\epsilon)$ can be controlled as 
\begin{align*}
    \mprob(\uhat \in D_\epsilon)
    &= \mprob(\max_{\bu \in D_\epsilon} \min_{\bw} C(\bv,\bu) \geq 
    \max_{\bu} \min_{\bw} C(\bv,\bu)) \\
    &\leq  \mprob(\max_{\bu \in D_\epsilon} \min_{\bv} C(\bv,\bu) \geq \Lstar - \epsilon^2
    ) + \mprob(\max_{\bu} \min_{\bv}C(\bv,\bu) \leq \Lstar - \epsilon) \\
   &\leq  2\mprob(\max_{\bu \in D_\epsilon} \min_{\bv} L(\bv,\bu) \geq \Lstar - \epsilon^2
    ) + 2\mprob(\max_{\bu} \min_{\bv}L(\bv,\bu) \leq \Lstar - \epsilon^2),
\end{align*}
where the last inequality follows by Gordon's lemma (Lemma~\ref{LemGordon}). 
The second term in the last expression is upper bounded $\frac{C}{\epsilon^2} e^{-cn\epsilon^4}$ using same argument as in Theorem~\ref{ThmControlLassoEst}, more concretely, in Lemma~\ref{LemGordonProbBounds}.
The first term is upper bounded as 
\begin{align*}
     2 \mprob(\max_{\bu \in D_\epsilon} \min_{\bv} L(\bv,\bu) \geq \Lstar - \epsilon^2) 
     = 2 \mprob(\max_{\bu \in D_\epsilon} U(\bu) \geq \Lstar - \epsilon^2) \le 2 \mprob(\max_{\bu \in D_\epsilon} \widetilde{U}(\bu) \geq \Lstar - \epsilon^2).
\end{align*}
We control the right-hand side following verbatim from the proof of Theorem D.1 and Lemma D.1 of \cite{miolane2018distribution}.
Putting the details above together yields 
$\mprob(\uhat \in D_\epsilon) \leq \frac{C}{\epsilon^2} \exp(-cn\epsilon^4).$

%%%%%%%%%%%%%%%%%%%%%%%%%%%%%%%%%%%%%%%%%%%%%%%%%%%%%%%%%%%%%%%%%%%%%%

\subsection{Control of the subgradient}
\label{SecPfLemSubgrad}
The proof of Theorem \ref{ThmLassoSparsity} is based on controlling the vector
\begin{align}
\label{EqnDefSG}
	\subg = \frac{1}{\sqrt{n} \lambda}\bmx^\top(\by - \bmx \thetahat)\,, 
\end{align}
which is a subgradient of the $\ell_{1}$-norm at $\thetahat$. 
Since controlling this subgradient may be of independent interest, 
we state our result formally below.
Similarly, we prove that $\subg$ behaves approximately like the corresponding subgradient in the fixed-design model
\begin{align}
\label{EqnDefSGFixed}
	\bthat^f := \frac{\sqrt{n} \zetastar}{\lambda} \
	\mySigma^{1/2} (\by^f - \mySigma^{1/2} \thetahat^f )\,,
\end{align}
where $\by^f = \mySigma^{1/2}\thetastar + \frac{\taustar}{\sqrt{n}} \bg$, 
$\thetahat^f = \eta(\by^f,\zetastar)$, and $\bg \sim \normal(\bzero,\Ind_p)$.
The quality of the approximation is controlled uniformly over models and estimators satisfying Assumption \ref{assump:1}.

For any measurable set $D \subset \reals^p$, define its $\epsilon$-enlargement $D_\epsilon := \{\bx \in \reals^p \mid \inf_{\bx' \in D} \|\bx - \bx'\|_2 \geq \epsilon\}$. The following result makes the connection between $\subg$ and $\bthat^f$ precise. 

\begin{lems}
\label{LemSubgradient}
	Under Assumption \ref{assump:1}, 
	there exist constants $C,c,c' > 0$ depending only on $\cuPmodel$ such that for any measurable set $D \subset \reals^p$ and for all $\epsilon < c'$ 
	\begin{equation}
		\mprob\left(\bthat \in D_{\sqrt{n}\epsilon}\right) 
			\leq 
			2\mprob\left(\bthat^f \not \in D\right)
			+
			\frac{C}{\epsilon^2} e^{-cn\epsilon^4}\,.\label{eq:SubGradient1}
	\end{equation}
	Consequently, there exist (possibly new) constants $C,c,c' > 0$ depending only on $\cuPmodel$ such that for any $1$-Lipschitz function $\phi:\reals^p\rightarrow \reals$ and for $\epsilon < c'$
	\begin{align}
		\label{EqnSubgradient}
		\mprob \left(
			\Big|\phi\Big(\frac{\subg}{\sqrt{n}}\Big) - \E\Big[\phi\Big(\frac{\bthat^f}{\sqrt{n}}\Big) \Big]  \Big| \geq \epsilon
			\right) 
			\leq 
			\frac{C}{\epsilon^2} e^{-c\numobs \epsilon^4}.
	\end{align}
\end{lems}

The proof of Lemma \ref{LemSubgradient} relies on concentration results established in Lemma \ref{LemGordonProbBounds}.
To begin with, define for $\|\bt\|_\infty \leq 1$
\begin{equation}
	\cuV(\bt) 
	:=
	\min_{\bw \in \reals^p} \left\{\frac1{2n} \|\bX \bw - \sigma \bz \|_2^2 + \frac\lambda{\sqrt{n}} \bt^\top (\thetastar + \bw) - \frac\lambda{\sqrt{n}} \|\thetastar\|_1\right\} =: \min_{\bw \in \reals^p} V(\bw,\bt)\,.
\end{equation}
Define, for $\bg\sim\normal(0,\Ind_p)$,  $\bh\sim\normal(0,\Ind_n)$, 
\begin{align}
	\cuT(\bt)
	&:=
	\min_{\bv \in \reals^p} \left\{\frac12 \left(\sqrt{\|\bv\|_2^2 + \sigma^2 } \frac{\|\bh\|_2}{\sqrt n} - \frac{\bg^\top \bv}{\sqrt{n}} \right)_+^2 + \frac\lambda{\sqrt{n}} \bt^\top(\thetastar + \mySigma^{-1/2}\bv) - \frac\lambda{\sqrt{n}} \|\thetastar\|_1\right\}\,\\
	&=: \min_{\bv\in \reals^p} T(\bv,\bt)\,.
\end{align}
We may compare the maximization of $\cuV(\bt)$ with the maximization of $\cuT(\bt)$ using Gordon's lemma.
\begin{lem}
\label{LemGordonSG}
	Let $D \subset \{\bt \in \reals^p \mid \|\bt\|_\infty \leq 1\}$ be a closed set.
	\begin{enumerate}[label=(\alph*)]

		\item % a
		For all $t \in \reals$,
		\begin{equation}\label{EqnGordonSGUB}
			\mprob\left(\max_{\bt \in D} \cuV(\bt) \geq t\right) \leq 2 \mprob\left(\max_{\bt \in D} \cuT(\bt) \geq t\right)\,.
		\end{equation}

		\item % b
		If $D$ is also convex, 
		then for any $t \in \reals$,
		\begin{equation}\label{EqnGordonSGLB}
			\mprob\left(\max_{\bt \in D} \cuV(\bt) \leq t\right) \leq 2 \mprob\left(\max_{\bt \in D} \cuT(\bt) \leq t\right)\,.
		\end{equation}

	\end{enumerate}
\end{lem}
\noindent We prove Lemma \ref{LemGordonSG} at the end of this section.
The maximization of $\cuT(\bt)$ can be controlled because $\cuT(\bt)$ is strongly-concave with high probability.
We first establish this strong-concavity.

\begin{lem}
\label{LemGordonSBObjStronglyConcave}
	Under Assumption \ref{assump:1},
	the objective $\cuT(\bt)$ is $c_0/n$-strongly concave
	on the event
	\begin{equation}
		\left\{\frac{\|\bh\|_2^2}{n} \leq 2,\;
                  \frac{\|\bg\|_2^2}{p} \leq 2 \right\}\,,
	\end{equation}
	where $c_0 > 0$ is a constant depending only on $\cuPmodel$.
\end{lem}
\noindent We prove Lemma \ref{LemGordonSBObjStronglyConcave} at the end of this section.
We are ready to prove Lemma \ref{LemSubgradient}.

	Consider $\alpha=0$, and let $\bv^* \in \reals^p$ be a minimizer of $\cuL(\bv) :=\cuL_{\alpha=0}(\bv)$ defined in \eqref{EqnGLoss}.
	Let
	\begin{equation}
	\begin{aligned}
		\bt^* 
			&:= 
			- \frac{\sqrt{n}}{\lambda} \mySigma^{1/2} \nabla \left(\bv \mapsto \frac12 \left( \sqrt{\ltwo{\bv}^2 + \sigma^2} \frac{\|\bh\|_2}{\sqrt n} 
	    	- \frac{\bg^\top \bv}{\sqrt{n}} \right)_{+}^2\right)\Bigg|_{\bv = \bv^*}\,\\
	    &= 
	    	-\frac{\sqrt{n}}\lambda  \mySigma^{1/2}\left( \sqrt{\ltwo{\bv^*}^2 + \sigma^2} \frac{\|\bh\|_2}{\sqrt n} 
	    	- \frac{\bg^\top \bv^*}{\sqrt{n}} \right)_+ 
	    	\left(\frac{\|\bh\|_2/\sqrt n}{ \sqrt{\|\bv^*\|_2^2 + \sigma^2}} \bv^* - \frac{\bg}{\sqrt{n}} \right)\,.\label{EqnTstarDef}
	\end{aligned}
	\end{equation}
	By the KKT conditions,
	$\frac{\lambda}{\sqrt{n}} \mySigma^{-1/2}\bt^* \in \frac{\lambda}{\sqrt{n}}\partial( \bv\mapsto \|\thetastar + \mySigma^{-1/2}\bv\|_1)$ at $\bv = \bv^*$.
	With this definition, $\bzero_p$ is in the subdifferential with respect to $\bv$ of $T(\bv,\bt)$ at $(\bv^*,\bt^*)$.
	Moreover, $t^*_j = 1$ whenever $(\thetastar + \mySigma^{-1/2}\bv^*)_j > 0$ and $t^*_j = -1$ whenever $(\thetastar + \mySigma^{-1/2}\bv^*)_j < 0$, 
	whence $\bt^* \in \argmax_{\|\bt\|_\infty \leq 1} T(\bv^*,\bt)$.
	Because $T$ is convex-concave, we have
	\begin{equation}
		T(\bv^*,\bt) \leq T(\bv^*,\bt^*) \leq T(\bv,\bt^*)\,,
	\end{equation}
	for all $\bv \in \reals^p$, $\|\bt\|_\infty \leq 1$.
	Thus, $(\bv^*,\bt^*)$ is a saddle-point and by \cite[pg.~380]{rockafellar1970convex}
	\begin{equation}
		\max_{\|\bt\|_\infty \leq 1}\min_{\bv \in \reals^p} T(\bv,\bt) = \min_{\bv \in \reals^p}\max_{\|\bt\|_\infty \leq 1} T(\bv,\bt)\,,
	\end{equation}
	and
	\begin{equation}
		\bt^* \in \arg\max_{\|\bt\|_\infty\leq 1} \cuT(\bt)\,.
	\end{equation}

	Fix $\epsilon > 0$.
	Define the events 
	\begin{gather}
		\cuA_1 
			:= 
			\left\{ \bthat^f \in D \right\}\,,
			\;\;\;\;\;\;\;\;
			\cuA_2 
			:= 
			\left\{ \frac{\|\bt^* - \bthat^f\|_2}{\sqrt n} \leq \frac{\epsilon}{2}\right\}\,,\\
		\cuA_3 
			:= 
			\left\{\frac{\|\bh\|_2^2}{n} \leq 2,\; \frac{\|\bg\|_2^2}{p} \leq 2\right\}\,,
			\;\;\;\;\;\;\;\;
			\cuA_4 
			:= 
			\left\{|\cuT(\bt^*) - \Lstar| \leq \frac{c_0}{16} \epsilon^2 \right\}\,.
	\end{gather}
	We claim that on the event $\bigcap_{a=1}^4 \cuA_a$, 
	\begin{equation}
	\label{cuT-small-away-from-maximizer}
		\max_{\bt \in D_\epsilon}\cuT(\bt) \leq \Lstar - \frac{c_0}{16} \epsilon^2\,.
	\end{equation}
	Indeed, because $\cuA_1$ occurs,
	$\bt \in D_{\sqrt{n}\epsilon}$ implies $\frac{\|\bt - \bthat^f\|_2}{\sqrt n} \geq \epsilon$.
	Because $\cuA_2$ occurs, 
	also $\frac{\|\bt - \bt^*\|_2}{\sqrt n} \geq \frac{\epsilon}{2}$.
	Because $\cuA_3$ occurs, $\cuT(\bt)$ is $\frac{c_0 }{n}$-strongly concave by Lemma \ref{LemGordonSBObjStronglyConcave}, 
	whence because $\bt^*$ maximizes $\cuT$
	\begin{equation}
		\cuT(\bt) \leq \cuT(\bt^*) - \frac12 \frac{c_0 \epsilon^2}{4} \,.
	\end{equation}
	Because $\cuA_4$ occurs, 
	we conclude Eq.~\eqref{cuT-small-away-from-maximizer}.

	By Gordon's lemma for the subgradient (Lemma \ref{LemGordonSG}) and because $D_{\sqrt{p}\epsilon}$ is closed,
	\begin{equation}
		\mprob\left(\max_{\bt \in D_{\sqrt{p}\epsilon}} \cuV(\bt) \geq \Lstar -  \frac{c_0}{16} \epsilon^2 \right) 
		\leq 
		2 \left(1 - \mprob\left(\bigcap_{a = 1}^4 \cuA_a\right)\right) \leq 
		2 \sum_{a=1}^4 \mprob(\cuA_a^c)\,.\label{eq:UnionBoundV}
	\end{equation}
	We control the probabilities in the sum one at a time.

	\subsubsection*{Event $\cuA_2$ occurs with high probability depending on $\cuPmodel$}

	By Lemma \ref{LemGordonProbBounds},
	there exists $C,c,c' > 0$ depending only on $\cuPmodel$ such that for $\epsilon \in (0,c')$ we have 
	\begin{equation} 
		\mprob\left(\|\bv^* - \hbv^f\|_2 > \frac\epsilon2\right) 
		\leq 
		\frac{C}{\epsilon^2}e^{-cn\epsilon^4}.
	\end{equation}
	Indeed, the event in the preceding display occurs when the two conditions in Eq.~\eqref{EqnGordonObjControl} are met.
	Also, $\|\hbv^f\|_2^2+\sigma^2$, $\bg^\top \hbv^f/\sqrt{n}$, and $\|\bh\|_2/\sqrt{n}$ concentrate on 
	${\taustar}^2$, $\taustar(1-\zetastar)$, and $1$ at sub-Gamma or sub-Gaussian rates depending only on $\taumax$ 
	(see, e.g., Eq.~\eqref{EqnNormAndWidthBound} in the proof of Lemma \ref{LemGordonProbBounds}). 
	Combining this with the previous display and updating constants appropriately, 
	we conclude there exists $C,c,c' > 0$ depending only on $\cuPmodel$ such that for $\epsilon \in (0,c')$ we have 
	\begin{equation}
		\mprob\left(
			\frac1{\sqrt{n}} \left\|\bt^* 
			- 
			\frac1\lambda \mySigma^{1/2} \left(
				\taustar - \taustar(1-\zetastar)
			\right)
			\left(\frac{\hbv^f}{\taustar} - \bg\right)\right\|_2
			>
			\frac{\epsilon}2
		\right)
		\leq 
		\frac{C}{\epsilon^2}e^{-cn\epsilon^4}.
	\end{equation}
	By the definition of $\bthat^f$ (Eq.~\eqref{EqnDefSGFixed}) and of $\hbv$ (Eq.~\eqref{EqnLassoErrVec}),
	the preceding display is equivalent to
	\begin{equation}
		\mprob\left(\cuA_2^c\right) \leq \frac{C}{\epsilon^2}e^{-cn\epsilon^4}\,.
	\end{equation}

	\subsubsection*{Event $\cuA_3$ occurs with high probability depending on $\delta$}
	By Gaussian concentration of Lipschitz functions, $\mprob(\cuA_3) \leq Ce^{-cn}$ for some $C,c$ depending only on $\delta$.
	
	\subsubsection*{Event $\cuA_4$ occurs with high probability depending on $\delta$}
	Observe that $\cuT(\bt^*) = \cuL(\bv^*)$.
	Then, by Lemma \ref{LemGordonProbBounds} there exist constants $C,c,c' > 0$, depending only on $\cuPmodel$ such that for $\epsilon \in (0,c')$, 
	\begin{equation}
	\label{eq:T-global-min-concentrates}
		\mprob(\cuA_4^c) = \mprob\left(\left|\cuT(\bt^*) - \Lstar\right| > \epsilon \right) = \mprob\left(\left|\max_{\|\bt\|_\infty \leq 1} \cuT(\bt) - \Lstar \right| > \epsilon \right) \leq \frac{C}{\epsilon^2}e^{-cn\epsilon^4}\,.
	\end{equation}

	Combining the established probability bounds on $\cuA_i$, $i=2,3,4$, Eq.~\eqref{eq:UnionBoundV} implies that for all $\epsilon < c'$,
	\begin{equation}
		\mprob\left(\max_{\bt \in D_{\sqrt{n}\epsilon}} \cuV(\bt) \geq \Lstar - \frac32\gamma\epsilon\right) 
			\leq 
			2\mprob\left(\bthat^f \not \in D\right) + \frac{C}{\epsilon^2}e^{-cn\epsilon^4} 
			\;\; 
			\text{and} 
			\;\; 
			\mprob\left(\max_{\|\bt\|_\infty\leq 1} \cuV(\bt) < \Lstar - \gamma\epsilon\right) \leq \frac{C}{\epsilon^2}e^{-cn\epsilon^4}\,,
	\end{equation}
	where the second probability bound holds by Eq.~\eqref{eq:T-global-min-concentrates}.
	Thus, $\mprob\left(\bthat \in D_{\sqrt{n}\epsilon}\right) \leq 2\mprob\left(\bthat^f \not \in D\right) + \frac{C}{\epsilon^2} e^{-cn\epsilon^4}$.
	Using the definition of $D_{\sqrt{n}\epsilon}$ and a change of variables (which absorbs certain constants into $c$),
	we conclude that Eq.~\eqref{eq:SubGradient1} holds.

	To complete the proof of Lemma \ref{LemSubgradient}, we prove Eq.~\eqref{EqnSubgradient}.
	Define
	\begin{equation}
		D = \left\{\bt \in \reals^p \Bigm| \Big| \phi\Big(\frac{\bt}{\sqrt n}\Big) 
		- 
		\E\Big[\phi\Big(\frac{\bthat^f}{\sqrt n}\Big)\Big] \Big| \leq \epsilon\right\}\,.
	\end{equation}
	By Eq.~\eqref{EqnDefSGFixed},
	$\bthat^f$ is $\frac{\taumax\zetamax\kappamax^{1/2}}{\lambdamin}$-Lipschitz in $\bg$, 
	whence
	\begin{equation}
		\mprob\left(\bthat^f \not \in D\right) \leq 2\exp\left(-\frac{3\gamma n}{c_0 \taumax \zetamax } \epsilon^2\right) \leq \frac{C}{\epsilon^2}e^{-cn\epsilon^4}\,,
	\end{equation}
	where the last inequality holds for $\epsilon < c'$ with $C,c,c'>0$ depending only on $\cuPmodel$.
	Eq.~\eqref{EqnSubgradient} is then a special case of Eq.~\eqref{eq:SubGradient1}.
	The proof of Lemma \ref{LemSubgradient} is complete. \hfill$\square$

\begin{proof}[Proof of Lemma \ref{LemGordonSG}]
	Fix $R > 0$.
	The function $\bt \mapsto \min_{\|\bw\|_2 \leq R} V(\bw,\bt)$ is concave and continuous and is defined on a compact set $D$.
	Moreover, $\min_{\|\bw\|_2 \leq R} V(\bw,\bt)$ is non-increasing in $R$.
	Because the maximum of a non-increasing limit of continuous functions defined on a compact set is equal to the limit of the maxima of these functions,
	\begin{equation}
		\max_{\bt \in D} \cuV(\bt) = \max_{\bt \in D} \lim_{R \rightarrow \infty} \min_{\|\bw\|_2 \leq R} V(\bw,\bt) = \lim_{R \rightarrow \infty} \max_{\bt \in D}  \min_{\|\bw\|_2 \leq R} V(\bw,\bt)\,.
	\end{equation}
	We may write 
	\begin{equation}
		V(\bw,\bt) 
		:=
		\max_{\|\bu\|_2 \leq R' } \breve{V}(\bw,\bt,\bu)\,,
	\end{equation}
	for any $R' > \|\bX\|_{\mathrm{op}} \|\bw\|_2 + \sigma \|\bz\|_2$,
	where 
	\begin{equation}
		\breve{V}(\bw,\bt,\bu) 
		=
		\frac1n \bu^\top (\bX \bw - \sigma \bz) - \frac1{2n} \|\bu\|_2^2 + \frac\lambda{\sqrt{n}} \bt^\top (\thetastar + \bw) - \frac\lambda{\sqrt{n}} \|\thetastar\|_1\,.
	\end{equation}	
	Because almost surely $R^2 > \|\bX\|_{\mathrm{op}} R + \sigma \|\bz\|_2$ for sufficiently large $R$,
	we conclude
	\begin{align}
		\max_{\bt \in D} \cuV(\bt) 
		&= 
		\lim_{R \rightarrow \infty} \max_{\bt \in D}  \min_{\|\bw\|_2 \leq R} \; \max_{\|\bu\|_2 \leq R^2 } \breve{V}(\bw,\bt,\bu) \\
		&= 
		\lim_{R \rightarrow \infty} \max_{\bt \in D} \max_{\|\bu\|_2 \leq R^2} \; \min_{\|\bw\|_2 \leq R}  \breve{V}(\bw,\bt,\bu)\,,\label{cuV-R-lim}
	\end{align}
	almost surely, 
	where we may exchange minimization and maximization because they are taken over compact sets and $\breve{V}$ is convex-concave and continuous.

	Similarly,
	\begin{equation}
		\max_{\bt \in D} \cuT(\bt) 
		= 
		\lim_{R \rightarrow \infty} \max_{\bt \in D} \min_{\|\bv\|_2 \leq R} T(\bv,\bt)\,.
	\end{equation}
	We may write
	\begin{equation}
		T(\bv,\bt) = \max_{\|\bu\|_2 \leq R'} \breve{T}(\bv,\bt,\bu)\,,
	\end{equation}
	for any $R' > \sqrt{n} \left(\sqrt{\|\bv\|_2^2 + \sigma^2} \frac{\|\bh\|_2}{\sqrt n} + \frac{\|\bg\|_2 \|\bv\|_2}{\sqrt{n}} \right)$, where
	\begin{equation}
		\breve{T}(\bv,\bt,\bu) 
		:=
		-  \frac{\ltwo{\bm{u}}}{\sqrt{n}} \frac{\bg^\top \bv}{\sqrt{n}} + 
			\sqrt{\ltwo{\bv}^2 + \sigma^2} \cdot \frac{\bm{h}^\top \bm{u}}{n}
			- \frac{\ltwo{\bm{u}}^2}{2\numobs} 
			+ \frac{\lambda}{\sqrt{\numobs}}\bt^\top(\thetastar + \mySigma^{-1/2}\bv) - \frac\lambda{\sqrt{n}}\lone{\thetastar}\,.
	\end{equation}
	Because almost surely $R^2 >  \sqrt{n} \left(\sqrt{\frac{R^2} + \sigma^2} \frac{\|\bh\|_2}{\sqrt n} + \frac{\|\bg\| R}{\sqrt{n}} \right)$ for sufficiently large $R$,
	we conclude
	\begin{align}
		\max_{\bt \in D} \cuT(\bt) 
			&= 
			\lim_{R \rightarrow \infty} \max_{\bt \in D} \min_{\|\bv\|_2 \leq R} \max_{\|\bu\|_2 \leq R^2} \breve{T}(\bv,\bt,\bu)\\
		&= 
			\lim_{R \rightarrow \infty} \max_{\bt \in D} \max_{\|\bu\|_2 \leq R^2} \min_{\|\bv\|_2 \leq R} \breve{T}(\bv,\bt,\bu)\,,\label{cuT-R-lim}
	\end{align}
	where the second equality holds by the following argument.\footnote{Note that $\breve{T}$ is not convex-concave in $(\bv,\bt,\bu)$, so that the exchange of the minimization and maximization requires a different justification to that in Eq.~\eqref{cuV-R-lim}.}
	For fixed $\bt,\bu$, 
	the function $\breve{T}(\bv,\bt,\bu)$ depends on $\bv$ only through $\bg^\top\bv$, $\bt^\top \mySigma^{-1/2} \bv$, and $\|\bv\|_2$.
	Moreover, $\breve{T}(\bv,\bt,\bu)$ is convex in the triple $(\bg^\top\bv,\bt^\top \mySigma^{-1/2} \bv,\|\bv\|_2)$ and $\{(\bg^\top\bv,\bt^\top \mySigma^{-1/2} \bv,\|\bv\|_2) \mid \|\bv\|_2 \leq R\}$ is a compact, convex set.
	Similarly, for fixed $\bt,\bv$,
	the function $\breve{T}(\bv,\bt,\bu)$ depends on $\bu$ only through $\bh^\top\bu$, $\|\bu\|_2$.
	Moreover, $\breve{T}(\bv,\bt,\bu)$ is convex in the pair $(\bh^\top\bu,\|\bu\|_2)$ and $\{(\bh^\top\bu,\|\bu\|_2) \mid \|\bu\|_2 \leq R^2\}$ is a compact, convex set.
	Thus, 
	the exchange of minimization and maximization in the preceding display is justified.

	By Gordon's Lemma (see \cite[Theorem 3]{thrampoulidis2015regularized}),
	for any finite $R > 0$ and any $t \in \reals$
	\begin{align}
		\mprob\left(\max_{\bt \in D} \max_{\|\bu\|_2 \leq R^2} \; \min_{\|\bw\|_2 \leq R}  \breve{V}(\bw,\bt,\bu) > t\right) \leq 2 \mprob\left(\max_{\bt \in D} \max_{\|\bu\|_2 \leq R^2} \; \min_{\|\bw\|_2 \leq R}  \breve{T}(\bw,\bt,\bu) > t\right)\,.
	\end{align}	
	Taking $R \rightarrow \infty$ and using Eqs.~\eqref{cuV-R-lim} and \eqref{cuT-R-lim},
	we conclude
	\begin{align}
		\mprob\left(\max_{\bt \in D} \cuV(\bt) > t\right) \leq 2 \mprob\left(\max_{\bt \in D}\cuT(\bt) > t\right)\,.
	\end{align}	
	The strict inequalities become weak by considering $t' > t$ in place of $t$ and taking $t' \rightarrow t$.
	We conclude Eq.~\eqref{EqnGordonSGUB}.
	Eq.~\eqref{EqnGordonSGLB} follows similarly.
\end{proof}

\begin{proof}[Proof of Lemma \ref{LemGordonSBObjStronglyConcave}]
	Define 
	\begin{equation}
		f(\bv) := \sqrt{\|\bv\|_2^2 + \sigma^2 } \frac{\|\bh\|_2}{\sqrt n} - \frac{\bg^\top \bv}{\sqrt{n}} \,.
	\end{equation}
	The gradient and Hessian of $f(\bv)$ are
	\begin{gather}
		\nabla f(\bv) = \left(\|\bv\|_2^2 + \sigma^2\right)^{-1/2} \frac{\|\bh\|_2}{\sqrt n}\bv - \frac{\bg}{\sqrt{n}}\,,\\
      \nabla^2 f(\bv) = \left(\|\bv\|_2^2 + \sigma^2\right)^{-1/2}\left(\Ind_p - \left(\|\bv\|_2^2 + \sigma^2\right)^{-1} \bv\bv^\top\right) \frac{\|\bh\|_2}{\sqrt{n}} 
          \preceq \left( \|\bv\|_2^2 + \sigma^2\right)^{-1/2}\frac{\|\bh\|_2}{\sqrt{n}}\Ind_p\,.
	\end{gather}
	We bound 
	\begin{gather}
		\|\nabla f(\bv)\|_2^2 \leq \frac{2\|\bh\|_2^2}{n} + \frac{2\|\bg\|_2^2}{n}\,,\\
		|f(\bv)|\, \|\nabla^2 f(\bv) \|_{\mathrm{op}}
		\leq 
		\frac{\|\bh\|_2}{\sqrt{n}}\left(\frac{\|\bh\|_2}{\sqrt n} + \frac{\|\bg\|_2 \|\bv\|_2}{\sqrt{n}}   \left(\|\bv\|_2^2 + \sigma^2\right)^{-1/2}\right) \,\\
		\leq \frac{\|\bh\|_2^2+\|\bh\|_2\|\bg\|_2}{n}\,.
	\end{gather}
	The Hessian of $\frac12 (f(\bx))_+^2$ is $[\nabla f(\bx) \nabla f(\bx)^\top + f(\bx) \nabla^2 f(\bx)]\,\indic{f(\bx)\ge0}$,
	whence on the event appearing in the statement of the lemma,
	\begin{equation}
      \|\nabla^2\cdot(f(\bv)_+^2/2) \|_{\mathrm{op}} 
      \leq \frac{2\|\bh\|_2^2}{n} + \frac{2\|\bg\|_2^2}{n} + \frac{\|\bh\|_2^2+\|\bh\|_2\|\bg\|_2}{n} 
      \leq \frac{1}{c_0}\,,
	\end{equation}
	where $c_0 = (4 + 4/\Deltamin + 2(1 + \Deltamin^{-1/2}))^{-1} $.
	That is, $\bv \mapsto \frac12 f(\bv)_+^2$ is $1/c_0$-strongly smooth.
	Note that 
	\begin{equation}
		\cuT(\bt) = - \tilde f^*(-\bt) - \frac{\lambda}{\sqrt{n}} \|\thetastar\|_1\,,
	\end{equation}
	where $\tilde f(\tilde \bv) := f( \mySigma^{1/2}(\sqrt{n}\tilde \bv/\lambda - \thetastar ) )^2/2$, 
	and $\tilde f^*$ is the Fenchel-Legendre conjugate of $\tilde f$.
	Because $f(\bv)^2_+/2$ is $1/c_0$-strongly smooth, $\tilde f$ is $\frac{n\kappamax}{c_0\lambdamin^2}$-strongly smooth.
	By the duality of strong smoothness and strong convexity \cite[Theorem 6]{Kakade2009OnTD}, we conclude that $\cuT(\bv)$ is $\frac{c_0\lambdamin^2}{n\kappamax}$-strongly concave.
\end{proof}

%%%%%%%%%%%%%%%%%%%%%%%%%%%%%%%%%%%%%%%%%%%%%%%%%%%%%%%%%%%%%%%%%%%%%%

\subsection{Control of the Lasso sparsity: proof of Theorem \ref{ThmLassoSparsity}}
\label{SecThmLassoSparsity}

For notational convenience, let us write 
\begin{equation}
	 \bar\mySigma \defn \frac{1}{\kappamin}\mySigma\,,
	\;\;\;\;\;\;\;\;
	\bar\tau^* \defn \sqrt{\frac{1}{\kappamin}} \taustar\,,
	\;\;\;\;\;\;\;\;
	\bar\lambda \defn \frac{1}{\kappamin} \lambda\,,
\end{equation}
so that by Eqs.~\eqref{EqnSTH}
\begin{align}
	\thetahat^f 
	&= 
	\arg\min_{\mytheta \in \reals^p}\left\{\frac{\zetastar}{2} \Big\|\frac{\taustar}{\sqrt{n}} \bg + \mySigma^{1/2}(\thetastar - \mytheta)\Big\|_2^2 + \frac{\lambda}{\sqrt{n}} \|\mytheta\|_1\right\} \\
	&=
	\arg\min_{\mytheta \in \reals^p}\left\{\frac{\zetastar}{2} \Big\|\frac{\bar\tau^*}{\sqrt{n}} \bg + \bar\mySigma^{1/2}(\thetastar - \mytheta)\Big\|_2^2 + \frac{\bar\lambda}{\sqrt{n}} \|\mytheta\|_1\right\}\,. \label{eq:ThetaF}
\end{align}
The KKT conditions of this optimization problem are
\begin{equation}
	\bar\mySigma^{1/2}
		\left( \frac{\bar\tau^*}{\sqrt{n}} \bg + \bar\mySigma^{1/2}(\thetastar - \thetahat^f) \right) 
		\in 
		\frac{\bar\lambda}{\sqrt{n}\zetastar} \partial \|\thetahat^f\|_1\,,
\end{equation}
whence
\begin{equation}
\label{EqnThetaHatFBreveY}
	\thetahat^f 
		=
		\softthreshold\left(
			\thetahat^f + \bar\mySigma^{1/2}
			\left( \frac{\bar\tau^*}{\sqrt{n}} \bg + \bar\mySigma^{1/2}(\thetastar - \thetahat^f) \right)
			;
			\frac{\bar\lambda}{\sqrt{n}\zetastar}
		\right)
		=:
		\softthreshold\left(
			\breveby^f
			;
			\frac{\bar\lambda}{\sqrt{n}\zetastar}
		\right)\,,
\end{equation}
and by Eq.~\eqref{EqnDefSGFixed}
\begin{equation}
\label{EqnSGInBreveY}
	\bthat^f 
		=
		\frac{\sqrt{n}\zetastar}{\bar\lambda}
		\left(
			\breveby^f
			-
			\softthreshold\left(
				\breveby^f
				;
				\frac{\bar\lambda}{\sqrt{n}\zetastar}
			\right)
		\right)\,,
\end{equation}
where $\softthreshold(\cdot,\alpha)$ applies $x \mapsto \sign(x)(|x|-\alpha)_+$ coordinates-wise.
This representation is useful because the marginals of $\breveby^f$
have bounded density, which will allow us to control the expected number of coordinates of $\bthat^f$ which are close to 1.
\begin{lem}[Anti-concentration of $\breveby^f$]
\label{LemMargDensBound}
	For each $j$, 
	the coordinate $\brevey^f_j$ has marginal density with respect to Lebesgue measure bounded above by $\frac{\sqrt{n}\kappamin^{1/2}\kappacond}{\sqrt{2\pi}\sigmamin}$.
\end{lem}

\begin{proof}[Proof of Lemma \ref{LemMargDensBound}]
	We compute
	\begin{equation}
	\label{eq:def-breveby}
		\breveby^f
		=
		\thetastar + \bar\mySigma^{1/2}\left(\frac{\bar\tau^*}{\sqrt{n}}\bg + (\Ind_p - \bar\mySigma^{-1})\bar\mySigma^{1/2}(\thetastar - \thetahat^f)\right)
		=:
		\thetastar + \bar\mySigma^{1/2}\left( \frac{\bar\tau^*}{\sqrt{n}} \bg + f(\bar\tau^*\bg/\sqrt{n}) \right)\,.
	\end{equation}
	By definition, all eigenvalues of $\bar\mySigma$ are bounded below and above by $1$ and $\kappacond$, respectively,
	so that all eigenvalues of $(\Ind_p - \bar\mySigma^{-1})$ are between $0$ and $1 - \kappacond^{-1}$.
	Because $\bar\mySigma^{1/2}\thetahat^f$ is $1$-Lipschitz in $\bar\tau^*\bg / \sqrt{n}$ (by Eq.~\eqref{eq:ThetaF}, using \cite[pg.~131]{Parikh2013ProximalAlgorithms}),
	the function $f$ is $(1-\kappacond^{-1})$-Lipschitz.

	Let $\bar\bsigma_i$ be the $i^\text{th}$ row of $\bar\mySigma^{1/2}$.
	Let $\mathsf{P}_i^\perp$ be the projection operator onto the orthogonal complement of the span of $\bar\bsigma_i$.
	Then 
	\begin{equation}
		\brevey_i^f 
			= 
			\theta^*_i 
			+ 
			\bar\tau^* \bar\bsigma_i^\top \bg / \sqrt{n}  
			+
			\bar\bsigma_i^\top f\Big( \bar\tau^*  (\bar\bsigma_i^\top \bg / \sqrt{n}) \bar\bsigma_i / \|\bar\bsigma_i\|_2^2 + \bar\tau^* \mathsf{P}_i^\perp \bg/\sqrt{n} \Big)\,.
	\end{equation}
        Consider the function
	\begin{equation}
          h(x) := \bar\tau^* x
			+
			\bar\bsigma_i^\top f\Big( \bar\tau^*  x \bar\bsigma_i / \|\bar\bsigma_i\|_2^2 + \bar\tau^* \mathsf{P}_i^\perp \bg /\sqrt{n} \Big) \,.
         	\end{equation}
	Since $f$ is $(1-\kappacond^{-1})$-Lipschitz,
	for any $x_1<x_2$, $x_1,x_2 \in \reals$, we have
	\begin{equation}
          h(x_2)-h(x_1)\ge  \bar\tau^*\kappacond^{-1}(x_2-x_1)\, .\label{eq:LowerBoundDerivative}
	\end{equation}
	Because $\bar\bsigma_i^\top\bg / \sqrt{n} \sim \normal(0,\|\bar\bsigma_i\|_2^2/n)$, its density is upper bounded by
        $\sqrt{n/(2\pi\|\bar\bsigma_i\|_2^2)}$. 
    Further, it
        is independent of $\mathsf{P}_i^\perp\bg$.
	Thus, the lower bound  \eqref{eq:LowerBoundDerivative} implies that $\brevey_i^f$ has density  $q(y)$ upper bounded by
        \begin{equation}
          \sup_y q(y)\le  \frac{\sqrt{n}}{\sqrt{2\pi}\|\bar\bsigma_i\|_2}\cdot \frac{1}{\inf_y h'(y)}\leq \frac{\sqrt{n}\kappamin^{1/2}\kappacond}{\sqrt{2\pi}\sigmamin}\, ,
        \end{equation}
	where we have used that $\|\bar\bsigma_i\|_2$ is no smaller than the minimal singular value of $\bar \mySigma^{1/2}$ which is no smaller than 1 by construction,
	and that $\bar \tau^* = \tau^*/\kappamin^{1/2} \geq \sigmamin/\kappamin^{1/2}$.
\end{proof}

We are now ready to complete the proof of Theorem \ref{ThmLassoSparsity}.
We prove high-probability upper and lower bounds on the sparsity separately.
The arguments are almost identical, but establishing the upper bound involves analyzing the subgradient $\bthat$ and establishing the lower bound involves analyzing $\thetahat$.

\medskip
\noindent \textbf{Upper bound on sparsity via the subgradient:}
The lasso sparsity is upper bounded in terms of the lasso subgradient:
\begin{equation}
\label{eq:sparsity-to-sg}
	\frac{\|\thetahat\|_0}n
	\leq 
	\frac{|\{ j \in [p] : |\widehat t_j| = 1\}|}{n}.
\end{equation}
We prove a high-probability upper bound on the right-hand side.
Given any $\Delta \leq 1$,  
define $T(\breveby,\Delta) := \{ j \in [p] \mid |\brevey_j| \geq \bar\lambda(1 + \Delta)/(\sqrt{n}\zetastar)\}$.
We will control quantity $T(\breveby^f,-\Delta)$ for $\Delta \leq 1$.
Consider the function 
\begin{align}
	\phi^{\mathsf{ub}}(\breveby,\Delta) := \frac1n\sum_{j=1}^p \phi^{\mathsf{ub}}_1(\brevey_j,\Delta) 
	\;\;\;\text{where}\;\;\; 
	\phi^{\mathsf{ub}}_1(\brevey,\Delta) := \min(1,\sqrt{n}\zetastar|\brevey|/(\bar\lambda\Delta) - 1/\Delta + 2)_+ \,.
\end{align}
The function $\phi^{\mathsf{ub}}_1$ regarding the first argument equals to 0 on $[-\bar \lambda(1-2\Delta)/(\sqrt{n}\zetastar),\bar \lambda(1-2\Delta)/(\sqrt{n}\zetastar)]$, 
1 on $[-\bar \lambda(1-\Delta)/(\sqrt{n}\zetastar),\bar \lambda(1-\Delta)/(\sqrt{n}\zetastar)]^c$, 
and linearly interpolates between the function values on these sets everywhere else.
Unlike $\breveby\mapsto T(\breveby,\Delta)$,
the function $\phi^{\mathsf{ub}}(\breveby,\Delta)$ 
is $\frac{\sqrt{p}\zetastar}{\bar\lambda\Delta\sqrt{n}} \leq \frac{\zetastar}{\bar\lambda\Delta\sqrt{\Deltamin}} $-Lipschitz in $\breveby$.
For all $\breveby$, by definition we have 
\begin{equation}
	\frac{|T(\breveby,-\Delta)|}{n} \leq \phi^{\mathsf{ub}}(\breveby,\Delta)\,.
\end{equation}
(The preceding display justifies the superscript $\mathsf{ub}$, which stands for ``upper bound''.)
Moreover, by Eq.~\eqref{EqnThetaHatFBreveY}
\begin{equation}
	\phi^{\mathsf{ub}}(\breveby^f,\Delta) \leq \frac{\|\thetahat^f\|_0}{n} + \frac{|\{j \in [p] \mid 1- \sqrt{n}\zetastar|\brevey^f_j|/\bar\lambda \in [0,2\Delta]\}|}{n}\,,
\end{equation}
whence
\begin{equation}
	\E[\phi^{\mathsf{ub}}(\breveby^f,\Delta)] 
	\leq 
	1-\zetastar + \frac{4\bar\lambda\Delta}{(n/p)\zetastar} \frac{\kappamin^{1/2}\kappacond}{\sqrt{2\pi}\sigmamin}
	\leq 
	1-\zetastar + \frac{4\lambdamax\kappacond\Delta}{\Deltamin \sigmamin\zetamin\sqrt{2\pi\kappamin}}\,,
\end{equation}
where we have applied Lemma \ref{LemMargDensBound}.
By the definition of $\breveby^f$ in Eq.~\eqref{eq:def-breveby} and recalling that $\mySigma^{1/2}\thetahat^f$ is $\tau^*/\sqrt{n}$-Lipschitz in $\bg$,
we have that
$\bg \mapsto \breveby^f$ is $\kappacond^{1/2}\bar\tau^* / \sqrt{n} + \kappacond^{1/2} \tau^* / (\sqrt{n}\kappamin^{1/2}) = 2 \kappacond^{1/2}\taumax/(\sqrt{n}\kappamin^{1/2})$-Lipschitz in $\bg$.
By Gaussian concentration of Lipschitz functions,
\begin{align}
	\mprob&\left(
			\frac{|T(\breveby^f,-\Delta)|}{n} 
			\geq 
			1-\zetastar + \frac{4\lambdamax\kappacond\Delta}{\Deltamin\sigmamin\zetamin\sqrt{2\pi\kappamin}}\, + \epsilon 
		\right)
		\leq
		\mprob\Big(
			\phi^{\mathsf{ub}}(\breveby^f,\Delta)
			\geq 
			\E[\phi^{\mathsf{ub}}(\breveby^f,\Delta)] + \epsilon 
		\Big)
	\nonumber\\
	&\qquad\qquad\qquad\qquad\qquad\leq 
		\exp\left(-\frac{n\Deltamin\bar\lambda^2}{2{\zetastar}^2}\Delta^2 \cdot \frac{\kappamin}{4\kappacond\taumax^2}\epsilon^2\right) \nonumber\\
	&\qquad\qquad\qquad\qquad\qquad\leq 
		\exp\left(-\frac{n\Deltamin\lambdamin^2}{8\kappamin\zetamax^2\kappacond\taumax^2}\Delta^2\epsilon^2\right)\,.
\end{align}
Plugging in $\epsilon=\Delta$ and absorbing constants appropriately,
we conclude there exists $c_1,c_2> 0$ depending only on $\cuPmodel$ such that for $\Delta \geq 0$
\begin{equation}
	\mprob\left(
			\frac{|T(\breveby^f,-\Delta)|}{n} 
			\geq 
			1 - \zetastar + c_1\Delta 
		\right) \leq \exp\left(-c_2n\Delta^4\right)\,.
\end{equation}
By Eq.~\eqref{EqnSGInBreveY}, if $\frac{|T(\breveby^f,-\Delta)|}{n} < 1 - \zetastar + c_1\Delta$, 
then for all $\bt \in \reals^p$ with $|\{j \in [p] \mid |t_j| \geq 1 \}|/n \geq 1 - \zetastar + 2c_1\Delta$, 
\begin{equation}
	\frac{\|\bthat^f - \bt\|_2^2}{n} \geq c_1\Delta^3\,, 
\end{equation}
because there are at least $c_1\Delta n$ coordinates where $\bthat^f$ and $\bt$ differ by at least $\Delta$.
Absorbing constants and taking $D = \{ \bt \in \reals^p \mid |\{ j \in [p] \mid 1 - |t_j| \leq \Delta \}|/n \leq 1-\zeta^*+c_1\Delta \}$ in Lemma \ref{LemSubgradient},
there exists $C,c,c' > 0$ depending only on $\cuPmodel$ and $\delta$ such that for $\Delta < c'$
\begin{equation}
	\mprob\left(
			\frac{|\{j \in [p] \mid |\widehat t_j| \geq 1\}|}{n} \geq 1-\zetastar + 2\Delta 
		\right) 
		\leq 
		2\exp\Big(-cn\Delta^4\Big) 
		+ 
		\frac{C}{\Delta^3}\exp\left(-cn\Delta^6\right)\,.
\end{equation}
We may absorb the first term into the second at the cost of changing the constants $C,c,c'$ because the bound applies only to $\Delta < c'$.
By Eq.~\eqref{eq:sparsity-to-sg},
$\mprob(\|\thetahat\|_0/n > 1 - \zetastar + \Delta) \leq \frac{C}{\Delta^3}e^{-cn\Delta^6}$.
A high probability upper bound on the sparsity of the lasso solution has been established.

\medskip
\noindent \textbf{Lower bound on sparsity via the lasso estimate:}
Define 
\begin{align}
	\phi^{\mathsf{lb}}(\breveby,\Delta) := \frac1n\sum_{j=1}^p \phi^{\mathsf{lb}}_1(\brevey_j,\Delta) 
	\;\;\;\text{where}\;\;\; 
	\phi^{\mathsf{lb}}_1(\brevey,\Delta) := \min(1,\sqrt{n}\zetastar|\brevey|/(\bar\lambda\Delta) - 1/\Delta - 1)_+ \, .
\end{align}
The function $\phi^{\mathsf{lb}}_1$ is 0 on $[-\bar \lambda(1+\Delta)/\zetastar,\bar \lambda(1+\Delta)/\zetastar]$, 
1 on 
$[-\bar \lambda(1+2\Delta)/(\sqrt{n}\zetastar),\bar \lambda(1+2\Delta)/(\sqrt{n}\zetastar)]^c$, 
and linearly interpolates between the function values on these sets everywhere else.
The function $\phi^{\mathsf{lb}}$ is a $\frac{\sqrt{p}\zetastar}{\bar\lambda\Delta\sqrt{n}} \leq \frac{\zetastar}{\bar\lambda\Delta\sqrt{\Deltamin}}$-Lipschitz lower bound for $|T(\breveby,\Delta)|/n$:
\begin{equation}
	\frac{|T(\breveby,\Delta)|}{n} \geq \phi^{\mathsf{lb}}(\breveby,\Delta)\,.
\end{equation}
Moreover, by Eq.~\eqref{EqnThetaHatFBreveY}
\begin{equation}
	\phi^{\mathsf{lb}}(\breveby^f,\Delta) \geq \frac{\|\thetahat^f\|_0}{n} - \frac{|\{j \in [p] \mid \sqrt{n} \zetastar|\brevey^f_j|/\bar\lambda - 1 \in [0,2\Delta]\}|}{n}\,,
\end{equation}
whence
\begin{equation}
	\E[\phi^{\mathsf{lb}}(\breveby^f,\Delta)] 
	\geq 
	1-\zetastar - \frac{4\bar\lambda\Delta}{(n/p)\zetastar} \frac{\kappamin^{1/2}\kappacond}{\sqrt{2\pi}\sigmamin}
	\leq 
	1-\zetastar - \frac{4\lambdamax\kappacond\Delta}{\Deltamin\sigmamin\zetamin\sqrt{2\pi\kappamin}}\,,
\end{equation}
where we have applied Lemma \ref{LemMargDensBound}.
Following the same argument used to establish the upper bound,
we conclude there exists $c_1,c_2> 0$ depending only on $\cuPmodel$ such that for $\Delta \geq 0$
\begin{equation}
\label{eq:Gordon-large-sg-bound}
	\mprob\left(
			\frac{|T(\breveby^f,\Delta)|}{n} 
			\leq 
			1 - \zetastar - c_1\Delta 
		\right) \leq \exp\left(-c_2n\Delta^4\right)\,.
\end{equation}
By Eq.~\eqref{EqnThetaHatFBreveY}, 
if $\frac{|T(\breveby^f,\Delta)|}{n} > 1 - \zetastar - c_1\Delta$, 
then $\frac{|\{j \in [p] \mid |\widehat \theta^f_j| \geq \bar \lambda \Delta / (\sqrt{n}\zeta^*)\}|}{n} > 1 - \zetastar - c_1\Delta$.
Then for all $\mytheta \in \reals^p$ with $\|\mytheta\|_0/n \leq 1 - \zetastar - 2c_1\Delta$, 
\begin{equation}
	\|\thetahat^f - \mytheta\|_2^2 \geq c_1\Delta^3\,, 
\end{equation}
because there are at least $c_1\Delta n$ coordinates where $\thetahat^f$ and $\mytheta$ differ by at least $\bar \lambda \Delta / (\sqrt{n}\zeta^*)$.
In particular,
taking $D := \{ \mytheta \in \reals^p \mid \frac{|\{j \in [p] \mid |\widehat \theta^f_j| \geq \bar \lambda \Delta / (\sqrt{n}\zeta^*)\}|}{n} > 1 - \zetastar - c_1\Delta\}$
and $D_\epsilon := \{\bx \in \reals^p \mid \inf_{\bx'\in D}\|\bx-\bx'\|_2 \geq \epsilon\}$,
we have that $\{\mytheta \in \reals^p \mid \|\mytheta\|_0/n \leq 1 - \zetastar - 2c_1\Delta\} \subset D_{\epsilon/2}$ for $\epsilon/2 = \sqrt{c_1\delta\Delta^3}$.
Equation~\eqref{eq:Gordon-large-sg-bound} says $\mprob(\widehat \mytheta^f \not \in D) \leq e^{-c_2n\Delta^4}$.
Thus, 
by the proof of Theorem \ref{ThmControlAlphaSmoothEstimate} in Appendix \ref{SecPfThmLassoL2} ---in particular, Eq.~\eqref{EqnMinimizationBound}---
we conclude
there exists $C,c,c' > 0$ depending only on $\cuPmodel$ such that for $\Delta < c'$
\begin{equation}
	\mprob\left(
			\frac{\|\thetahat\|_0}{n} \leq 1-\zetastar - \Delta 
		\right) 
		\leq 
		\exp\Big(-cn\Delta^4\Big) 
		+ 
		\frac{C}{\Delta^3}\exp\left(-cn\Delta^6\right)\,.
\end{equation}
We may absorb the first term into the second at the cost of changing the constants $C,c,c'$ because the bound applies only to $\Delta < c'$.
(In applying Eq.~\eqref{EqnMinimizationBound}, recall $\hbv^f = \mySigma^{1/2}(\thetahat^f-\thetastar)$, with the definition of $D$ modified according to this change of variables).
A high probability lower bound on the sparsity of the lasso solution has been established.

Theorem \ref{ThmLassoSparsity} follows by putting together the upper and lower bounds.

\subsection{Proof of Theorem \ref{ThmLassoTau}: estimate of effective noise level}

This result is an immediate consequence of the uniform and simultaneous and uniform concentration of $\widehat{\tau}(\lambda)$ and of $\| \thetahat - \thetastar \|_{\mySigma}$.
The simultaneous and uniform concentration of $\| \thetahat - \thetastar \|_{\mySigma}$ was established in Theorem~\ref{ThmControlLassoEst},
with a stronger rate than that which appears in Theorem \ref{ThmLassoTau}.
The uniform and simultaneous concentration of $\widehat{\tau}(\lambda)$ 
follows from the the uniform and simultaneous concentration of the Lasso residual and Lasso sparsity.
To establish these,
we use the same argument as used in Theorem~\ref{ThmControlLassoEst} to establish simultaneous concentration:
we use the point-wise concentration of the residual and sparsity,
as established in Theorems \ref{ThmLassoResidual} and \ref{ThmLassoSparsity},
as well as the fact that these quantities are Lipschitz in $\lambda$,
as we now explain.

Regarding the residual $\| \by - \bX \thetahat \|_2/\sqrt{n}$,
the required Lipschitz property is established using the same property for $\thetahat$, which is established in proving Theorem  combined with a high-probability operator norm bound 
$\mprob(\| \bX \|_2 / \sqrt{n} \geq C) \leq e^{-cp}$ for appropriately chosen and 
$\cuPmodel$-dependent $C,c > 0$.

Consider now the Lasso sparsity.
In the proof of Theorem \ref{SecThmLassoSparsity},
we showed that with the probability given in Theorem \ref{ThmLassoSparsity},
the subgradient $\bthat$ was (Euclidean) distance at least $\sqrt{n} \epsilon^{3/2} $ from the subgradient of any Lasso solution with sparsity $\| \thetahat \|_0/n > 1- \zeta^* + c'\epsilon$.
Moreover, because $\bthat = \frac{1}{\sqrt{n}\lambda}\bX^\top(\by - \bX \thetahat)$ and $\bthat^f = (\sqrt{n}\zeta^*/\lambda)\mySigma^{1/2}(\by^f - \mySigma^{1/2} \thetahat^f)$, 
we have that $\bthat / \sqrt{n}$ and $\bthat^f / \sqrt{n}$ satisfies the same continuity properties in $\lambda$ which are established for $\thetahat$ and $\thetahat^f$ in the proof of Theorem \ref{ThmControlLassoEst} (see, in particular, Lemmas \ref{LemCostComparisonLambda} and \ref{LemPopCostComparisonLambda} and the following paragraphs).
Because we must control the location of the subgradient to within a region of radius $\epsilon^{3/2}$, 
to get the upper bound on the Lasso sparsity uniformly and simultaneously over $[\lambdamin,\lambdamax]$, 
we can consider a mesh of size $C/(\epsilon^{3/2})^2$, 
which incurs an extra factor for $1/\epsilon^3$ in front of the probability bound in Theorem \ref{ThmLassoSparsity}.
An equivalent argument applies to the lower bound on the Lasso sparsity: 
in the proof of Theorem \ref{SecThmLassoSparsity},
we showed that with the probability given in the theorem,
the Lasso estimate $\thetahat$ was (Euclidean) distance at least $ \epsilon^{3/2} $ from any vector with sparsity $\| \thetahat \|_0/n < 1- \zeta^* - c'\epsilon$.
Thus, 
to get a lower bound on the Lasso sparsity uniformly and simultaneously over $[\lambdamin,\lambdamax]$, 
we incur the extra factor of $1/\epsilon^3$ in front of the probability bound in Theorem \ref{ThmLassoSparsity}.

Combining these results gives Theorem \ref{ThmLassoTau}.

%%%%%%%%%%%%%%%%%%%%%%%%%%%%%%%%%%%%%%%%%%%%%%%%%%%%%%%%%%%%%%%%%%%%%%

\subsection{Control of the debiased Lasso: proofs of Theorem \ref{ThmDBLasso} and Corollary \ref{CorDBLassoCI}}
\label{SecPfControlDBLasso}

We control the debiased Lasso by approximating it with the debiased $\alpha$-smoothed Lasso, which turns out to be easier to study due to the Lipschitz differentiability of the $\sfM_\alpha$ (cf.~\eqref{EqnMoreau}).
Define the debiased $\alpha$-smoothed Lasso
\begin{equation}
	\thetahat_\alpha^{\mathrm{d}} := \thetahat_\alpha + \frac{\mySigma^{-1}\bX^\top(\by - \bX\thetahat_\alpha)}{n\zetastar_\alpha}\,.
\end{equation}
This definition is analogous to \eqref{EqnDBlasso} except that $1 - \|\thetahat\|_0/n$ is replaced by the constant $\zetastar_\alpha$.
It is not feasible to calculate $\zetastar_\alpha$ exactly without knowing $\thetastar$, whence $\thetahat_\alpha^{\mathrm{d}}$ cannot be computed either. 
Rather, $\thetahat_\alpha^{\mathrm{d}}$ is a theoretical tool.
Define 
\begin{equation}
	\thetahat_\alpha^{f,\de}
		:= 
		\thetahat_\alpha^f + \mySigma^{-1/2}(\by^f - \mySigma^{1/2} \thetahat_\alpha^f)
		=
		\thetastar + \frac{\tau_\alpha^*}{\sqrt{n}} \mySigma^{-1/2}\bg
\end{equation}

To establish Theorem \ref{ThmDBLasso},
we first characterize the behavior of $\thetahat_\alpha^{\mathrm{d}}$ and second show that $\thetahat^{\mathrm{d}}$ is close to $\thetahat_\alpha^{\mathrm{d}}$ with high-probability.
The next lemma characterizes $\thetahat_\alpha^{\mathrm{d}}$.
\begin{lem}[Characterization of the debiased $\alpha$-smoothed Lasso]
  \label{LemDBSmoothedLasso}
  Let $\alpha>0$.
	Under assumptions  \ref{assump:1} and $\textsf{A2}_\alpha$, 
	there exist constants $C,c,c' > 0$ depending only on $\cuPmodel$ such that for any $1$-Lipschitz $\phi: \real^p\rightarrow \reals$,
	we have for all $\epsilon < c'$
	\begin{align}
		\mprob \left(
			\Big|
				\phi\Big(\bthetahatd_\alpha\Big) 
				-
				\E\Big[\phi\Big(\thetahat_\alpha^{f,\de}\Big)\Big]
			\Big| 
			> 
			\Big(1+\frac{\lambdamax}{\kappamin^{1/2}\zetamin\sqrt{n}\alpha}\Big)
			\epsilon
		\right)
		\leq 	
		\frac{C}{\epsilon^2} e^{-cn\epsilon^4}\,.
	\end{align}
\end{lem}
\noindent We leave the proof of Lemma \ref{LemDBSmoothedLasso} at the end of this section.

The following lemma will allow us to show that $\bthetahatd$ and $\bthetahatd_\alpha$ are close with high probability.
\begin{lem}[Closeness of the Lasso and $\alpha$-smoothed Lasso]
\label{lem:lasso-near-alpha-smoothed-lasso}
	There exists $C_1,C,c,\alphamax > 0$ depending only on $\cuPmodel$ such that
	\begin{equation}
		\mprob\left(\|\thetahat_\alpha - \thetahat\|_2^2 \leq C_1\sqrt{n}\alpha,~ \text{ for all } \alpha \leq \alphamax / \sqrt{n} \right) \geq 1 - Ce^{-cn}\,.
	\end{equation}
\end{lem}
\noindent We prove Lemma \ref{lem:lasso-near-alpha-smoothed-lasso} at the end of this section. 
Equipped with these two lemma, we are now ready to prove Theorem \ref{ThmDBLasso}.

\subsubsection{Proof of Theorem \ref{ThmDBLasso}: characterization of the debiased Lasso}
	For any $\alpha > 0$, direct calculations give (setting $\zeta^*=\zeta^*_0$)
	\begin{equation}
	\begin{aligned}
	\label{eq:debiased-close-to-alpha-debiased}
		\|\thetahat^{\mathrm{d}} - \thetahat_\alpha^{\mathrm{d}}\|_2
			&\leq 
			\Big\|
				\mySigma^{-1}\bX^\top(\by - \bX \thetahat)\Big(
					\frac{1}{n - \|\thetahat\|_0} - \frac{1}{n\zetastar_\alpha}
				\Big)
			\Big\|_2
			+
			\left\|
				(\Ind_p - \mySigma^{-1}\bX^\top\bX/(n\zetastar_\alpha))(\thetahat - \thetahat_\alpha)
			\right\|_2\,
		\\
		&\leq 
		\kappamin^{-1/2} \frac{\|\mySigma^{-1/2}\bX^\top\|_{\mathrm{op}}}{\sqrt{n}} \frac{\|\by - \bX \thetahat\|_2}{\sqrt n}
			\left(
				\left|\frac{1}{1 - \|\thetahat\|_0/n} - \frac1{\zetastar}\right|
				+
				\left|\frac1{\zetastar} - \frac1{\zetastar_\alpha}\right|
			\right)
		\\
		&\qquad+
			\left(1 + \frac{\kappamin^{-1/2}(\|\mySigma^{-1/2}\bX^\top\|_{\mathrm{op}}/\sqrt{n})(\|\bX\mySigma^{-1/2}\|_{\mathrm{op}}/\sqrt{n})\|\mySigma^{1/2}\|_{\mathrm{op}}}{\zetastar_\alpha}\right)\|\thetahat - \thetahat_\alpha\|_2\, \\
		&=: T_1 + T_2.
	\end{aligned}
	\end{equation}
	\paragraph*{Bounding $T_1$} We start with bounding the $T_{1}$ term in Eq.~\eqref{eq:debiased-close-to-alpha-debiased}. 
	By \cite[Corollary 5.35]{vershynin2010} and Theorem \ref{ThmLassoResidual}, 
	there exist $C_1,C,c > 0$ depending only on $\cuPmodel$ such that 
	with probability at least $1 - Ce^{-cn}$
	\begin{align}
	\label{eq:thetad-close-thetadalpha-1}
		\kappamin^{1/2} \frac{\|\mySigma^{-1/2}\bX^\top\|_{\mathrm{op}}}{\sqrt{n}} \frac{\|\by - \bX \thetahat\|_2}{\sqrt n} \leq C_1\,.
	\end{align}
	Let $L_\tau$ and $L_\zeta$ be as in Lemma \ref{LemFixedPtContinuity},
	and let $\alphamax$ be the minimum of the corresponding quantities in Lemma \ref{LemFixedPtContinuity} and Lemma \ref{lem:lasso-near-alpha-smoothed-lasso}.
	Let $\alphamax' = \min\{\alphamax,\sigmamin^2/(4L_\tau^2),\zetamin^2/(4L_\zeta^2)\}$.
	By Lemma \ref{LemFixedPtContinuity},
	for all $\alpha < \alphamax'$, one has 
	\begin{gather}
		\sigmamin/2 \leq \taustar_\alpha \leq \taumax + \sigmamin/2\,,
		\;\;\;\;\;\;\;\;\;\;\;\;
		\zetamin/2 \leq \zetastar_\alpha \leq 1\,.
	\end{gather}
	For $\alpha \leq \alphamax'$, by Lemma \ref{LemFixedPtContinuity} 
	\begin{align}
	\label{eq:thetad-close-thetadalpha-2}
		|1/\zetastar - 1/\zetastar_\alpha| \leq (4/\zetamin^2)L_\zeta \sqrt{\sqrt{n}\alpha}.
	\end{align}
	By Theorem \ref{ThmLassoSparsity},
	there exists $C,c,c' > 0$ depending only on $\cuPmodel$ such that for $\alpha < c'/\sqrt{n}$, 
	with probability $1-\frac{C}{(\sqrt{n}\alpha)^{3/2}} e^{-cn(\sqrt{n}\alpha)^3}$ 
	\begin{equation}
	\label{eq:thetad-close-thetadalpha-3}
		|1/(1 - \|\thetahat\|_0/n) - 1/\zetastar| \leq (4/\zetamin^2) \sqrt{\sqrt{n}\alpha}.
	\end{equation}
	Combining the Eqs.~\eqref{eq:thetad-close-thetadalpha-1}, \eqref{eq:thetad-close-thetadalpha-2}, and \eqref{eq:thetad-close-thetadalpha-3},
	we conclude there exists $C_1,C,c,c' > 0$ depending only on $\cuPmodel$ such that for $\alpha < c'$, 
	with probability $1-\frac{C}{(\sqrt{n}\alpha)^{3/2}} e^{-cn(\sqrt{n}\alpha)^3}$ 
	the first term on the right-hand side of Eq.~\eqref{eq:debiased-close-to-alpha-debiased} is bounded by $C_1 \sqrt{\sqrt{n}\alpha}$.

	\paragraph*{Bounding $T_2$} We now bound the $T_2$ term of Eq.~\eqref{eq:debiased-close-to-alpha-debiased}.
	Because $\zetastar_\alpha \geq \zetamin/2$, 
	by \cite[Corollary 5.35]{vershynin2010} there exist $C_2,C,c > 0$ depending only on $\cuPmodel$ such that
	with probability at least $1 - Ce^{-cn}$
	\begin{equation}
	\label{eq:thetad-close-thetadalpha-4}
		\left(1 + \frac{\kappamin^{-1/2}\|\mySigma^{-1/2}\bX^\top\|_{\mathrm{op}}\|\bX\mySigma^{-1/2}\|_{\mathrm{op}}\|\mySigma^{1/2}\|_{\mathrm{op}}}{\zetastar_\alpha}\right) \leq C_2\,.
	\end{equation}
	Combining this bound with Lemma \ref{lem:lasso-near-alpha-smoothed-lasso}, absorbing parameters into constants, and absorbing smaller terms into larger ones,
	we conclude there exists $C_1,C,c > 0$ depending only on $\cuPmodel$ such that for $\alpha < \alphamax'$, 
	with probability $1-Ce^{-cn}$ the second term on the right-hand side of Eq.~\eqref{eq:debiased-close-to-alpha-debiased} is bounded by $C_1\sqrt{\sqrt{n}\alpha}$.

	Combining the high-probability upper bounds on the terms on the right-hand side of Eq.~\eqref{eq:debiased-close-to-alpha-debiased},
	we conclude there exists $C_1,C,c,\alphamax > 0$ depending only on $\cuPmodel$ such that for $\alpha < \alphamax/\sqrt{n}$,
	\begin{equation}
	\label{eq:thetad-close-thetadalpha-5}
		\mprob\left(
				|\phi(\thetahat^\mathrm{d}) - \phi(\thetahat_\alpha^\mathrm{d})| > C_1\sqrt{\sqrt{n}\alpha}
			\right)
			\leq
			\mprob\left(
				\|\thetahat^\mathrm{d} - \thetahat_\alpha^\mathrm{d}\|_2 > C_1\sqrt{\sqrt{n}\alpha}
			\right)
			\leq 
			\frac{C}{(\sqrt{n}\alpha)^{3/2}}e^{-cn(\sqrt{n}\alpha)^3}\,.
	\end{equation}
	Further, for $\alpha < \alphamax/\sqrt{n}$, by Lemma \ref{LemFixedPtContinuity},
	\begin{equation}
	\label{eq:thetad-close-thetadalpha-6}
		\Big|
		\E\Big[\phi\Big( \thetastar + \frac{\taustar}{\sqrt{n}}\mySigma^{-1/2} \bg\Big)\Big]
		-
		\E\Big[\phi\Big( \thetastar + \frac{\taustar_\alpha}{\sqrt{n}}\mySigma^{-1/2} \bg\Big)\Big]
				\Big| \leq C_1\sqrt{\sqrt{n}\alpha}\,.
	\end{equation}
	Taking $\epsilon = (\sqrt{n}\alpha)^3$ in Lemma \ref{LemDBSmoothedLasso} (and assuming $(\sqrt{n}\alpha)^3 < c'$ for $c'$ in that lemma),
	\begin{align}
	\label{eq:thetad-close-thetadalpha-7}
		\mprob \left(
			\Big|
				\phi\Big(\bthetahatd_\alpha\Big) 
				-
				\E\Big[\phi\Big( \thetastar + \frac{\taustar_\alpha}{\sqrt{n}}\mySigma^{-1/2} \bg\Big)\Big]
			\Big| 
			> 
			C_1 (\sqrt{n}\alpha)^2
		\right)
		\leq 	
		\frac{C}{(\sqrt{n}\alpha)^6} e^{-cn(\sqrt{n}\alpha)^{12}}\,.
	\end{align}
	Combining Eqs.~\eqref{eq:thetad-close-thetadalpha-5}, \eqref{eq:thetad-close-thetadalpha-6}, and \eqref{eq:thetad-close-thetadalpha-7}
	and appropriately adjusting constants,
	we conclude there exists $C,C',c,c' > 0$ depending only on $\cuPmodel$ such 
	that for $\epsilon < c'$
	\begin{align}
		\mprob \left(
			\Big|
				\phi\Big(\bthetahatd\Big) 
				-
				\E\Big[\phi\Big( \thetastar + \frac{\taustar}{\sqrt{n}}\mySigma^{-1/2} \bg\Big)\Big]
			\Big| 
			> 
			C_1 \epsilon
		\right)
		\leq 	
		\frac{C}{\epsilon^3} e^{-cn\epsilon^6}\,.
	\end{align}
	We complete the proof of Theorem \ref{ThmDBLasso}.

%%%%%%%%%%%%%%%%%%%%%%%%%%%%%%%%%%%%%%%%%%%%%%%%%%%%%%%%%%%%%%%%%%%%%%

\subsubsection{Proof of Lemma \ref{LemDBSmoothedLasso}: characterization of the debiased $\alpha$-smoothed Lasso}

	By the KKT conditions for the optimization defining the $\alpha$-smoothed Lasso (cf.~\eqref{EqnAlphaSmoothRandomD}), 
	$\thetahat_\alpha^\mathrm{d} = \thetahat_\alpha + \frac{\lambda\mySigma^{-1}\nabla \sfM_\alpha(\thetahat_\alpha)}{\sqrt{n}\zetastar_\alpha}$.
	Since $\mytheta \mapsto \nabla \sfM_\alpha(\mytheta)$ is $1/\alpha$-Lipschitz, 
	$\thetahat_\alpha^\mathrm{d}$ is a 
    $\left(1+\frac{\lambdamax}{\kappamin\zetamin\sqrt{n}\,\alpha}\right)$-Lipschitz function of $\thetahat_\alpha$.
	Moreover, 
	by the KKT conditions for the optimization defining the $\alpha$-smoothed Lasso in the fixed design model (Eq.~\eqref{EqnAlphaSmoothLassoFixed}),
	\begin{equation} 
		\thetahat_\alpha^f + \frac{\lambda\mySigma^{-1}\nabla \sfM_\alpha(\thetahat_\alpha^f)}{\sqrt{n}\zetastar_\alpha} = \thetahat_\alpha^f + \mySigma^{-1} \mySigma^{1/2}\Big(\frac{\taustar_\alpha}{\sqrt{n}} \bg - \mySigma^{1/2}(\thetahat_\alpha^f - \thetastar)\Big) = \thetahat_\alpha^{f,\de}\,.
	\end{equation}
	Because $\Risk_\alpha(\thetahat_\alpha) \leq \min_{\mytheta \in \reals^p} \Risk_\alpha(\mytheta) + \gamma\epsilon^2$ for any $\gamma,\epsilon > 0$,
	Theorem \ref{ThmControlAlphaSmoothEstimate} and the previous display implies the result.

%%%%%%%%%%%%%%%%%%%%%%%%%%%%%%%%%%%%%%%%%%%%%%%%%%%%%%%%%%%%%%%%%%%%%%

\subsubsection{Proof of Lemma \ref{lem:lasso-near-alpha-smoothed-lasso}: closeness of the Lasso and the $\alpha$-smoothed Lasso}

	The proof of Lemma \ref{lem:lasso-near-alpha-smoothed-lasso} relies on showing that with high-probability, the Lasso objective is strongly convex locally around its minimizer. 
	We then show that
	because the value of the $\alpha$-smoothed Lasso objective is close to that of the Lasso objective pointwise,
	the minimizers of the two objectives must also be close.
	\begin{lem}[Local strong convexity of Lasso objective]
	\label{LemLassoLocalStrongConvexity}
		Assume $n\zetastar/8 \geq 1$.
		Then there exists $C,c,c',c_1>0$ depending only on $\cuPmodel$ such that with probability at least $1-Ce^{-cn}$ the following occurs:
		for all $\mytheta \in \reals^p$ with $\|\mytheta - \thetahat\|_2 \leq c'$,
		\begin{equation}
			\Risk(\mytheta) - \Risk(\thetahat)
			\geq 
			c_1\|\mytheta - \thetahat\|_2^2\,.
		\end{equation}
	\end{lem}
	\noindent We prove Lemma \ref{LemLassoLocalStrongConvexity} below.
	By Eq.~\eqref{EqnMoreauRelation}, 
	$\Risk(\mytheta) \geq \Risk_\alpha(\mytheta) \geq \Risk(\mytheta) - \frac{\lambda p\alpha}{2\sqrt{n}}$ for all $\mytheta \in \reals^p$.
	On the event of Lemma \ref{LemLassoLocalStrongConvexity}, for $\|\mytheta - \thetahat\|_2 \leq c'$
	\begin{equation}
		\Risk_\alpha(\mytheta) 
			\geq 
			\Risk(\mytheta) - \frac{\lambda p \alpha}{2\sqrt{n}}
			\geq 
			\Risk(\thetahat) + c_1 \|\mytheta - \thetahat \|_2^2 - \frac{\lambda p\alpha}{2\sqrt{n}}
			\geq
			\Risk_\alpha(\thetahat) + \frac{c_1}{p} \|\mytheta - \thetahat \|_2^2 - \frac{\lambda p\alpha}{2\sqrt{n}}\,.
	\end{equation}
	Since $\thetahat_\alpha$ minimizes $\Risk_\alpha(\mytheta)$, we conclude that for $\alpha \leq \alphamax / \sqrt{n} $
	\begin{equation}
		\|\thetahat_\alpha - \thetahat\|_2
			\leq 
			\sqrt{\lambdamax \sqrt{n} \alpha / (2c_1\Deltamin)},
	\end{equation}
	provided that also $\alpha$ is small enough for the right-hand side to be less than $c'$.
	The proof of Lemma \ref{lem:lasso-near-alpha-smoothed-lasso} is complete.

%%%%%%%%%%%%%%%%%%%%%%%%%%%%%%%%%%%%%%%%%%%%%%%%%%%%%%%%%%%%%%%%%%%%%%

\subsubsection{Proof of Lemma \ref{LemLassoLocalStrongConvexity}: local strong convexity of the Lasso objective}

We make the observation that with high probability, 
	the Lasso subgradient $\bthat = \frac1{\sqrt{n}\lambda}\bX^\top(\by - \bX \thetahat)$ (cf.~\eqref{EqnDefSG}), 
	cannot have too many coordinates with magnitude close to 1, even off of the Lasso support.
The next lemma makes this precise.
	\begin{lemma}
	\label{lem:large-SG-bound}
		There exists $C,c,\Delta > 0$ depending only on $\cuPmodel$ such that 
		\begin{equation}
		\label{eq:large-SG-bound}
			\mprob \left(
				\frac{|\{j \in [p] \mid |\widehat t_j| \geq 1 -\Delta/2\}|}{n} \geq 1 - \zetastar/2
			\right) 
			\leq 
			Ce^{-cn}\,.
		\end{equation}
	\end{lemma}

	\begin{proof}[Proof of Lemma \ref{lem:large-SG-bound}]
		The proof is as for Theorem \ref{ThmLassoSparsity} with the following minor changes.
		We apply Eq.~\eqref{eq:Gordon-large-sg-bound} with $\Delta = \zetastar / (4c_1)$ with $c_1$ as in in that equation.
		By Eq.~\eqref{EqnSGInBreveY} and with this choice of $\Delta$,
		if $\frac{|T(\breveby^f,\Delta)|}{n} < 1 - 3\zetastar/4$, then for all $\bt \in \reals^p$ with $|\{j \in [p] \mid |t_j| \geq 1 - \Delta/2\}|/n \geq 1 - \zetastar/2$, 
		\begin{equation}
			\frac{\|\widehat \bt^f - \bt\|_2^2}{n} \geq \frac{\Delta^2\zetastar}{16} = \frac{ {\zetastar}^4 }{256 c_1^2}\,,
		\end{equation}
		because there are at least $\zetastar n / 4$ coordinates where $\widehat \bt^f$ and $\bt$ differ by at least $\Delta/2$.
		Absorbing constants and taking 
		$D = \{\bt \in \reals^p \mid |\{ j \in [p] \mid 1 - |t_j| \leq \Delta\}|/n \leq 1 - 3\zetastar/4\}$ in 
		Lemma \ref{LemSubgradient}, there exists $C,c > 0$ depending only on $\cuPmodel$ such that Eq.~\eqref{eq:large-SG-bound} holds.
	\end{proof}

	We are now ready to prove Lemma \ref{LemLassoLocalStrongConvexity}.
	Define the minimum singular value of $\bX$ over a set $S \subset [p]$ by
	\begin{equation}
		\kappa_-(\bX,S) = \inf\left\{\|\bX\bw\|_2/\sqrt{n} \mid \supp(\bw) \subset S,\, \|\bw\|_2 = 1\right\}\,,
	\end{equation}
	and the $s$ sparse singular value by 
	\begin{equation}
		\kappa_-(\bX,s) = \min_{|S|\leq s} \kappa_-(\bX,S)\,.
	\end{equation}
	Consider the event 
	\begin{equation}
		\cuA
		:= 
		\Big\{
			\kappa_-\big(\bX,n(1-\zetastar/4)\big) \geq \kappamin'
		\Big\}
		\cap
		\Big\{
			\frac{\|\bX\|_{\mathrm{op}}}{\sqrt{n}} \leq C
		\Big\}
		\cap
		\Big\{
			\frac{|\{j \in [p] \mid |\widehat t_j| \geq 1 -\Delta/2\}|}{n} \leq 1 - \zetastar/2
		\Big\}\,.
	\end{equation}
	(Note that we need not assume that $(1 - \zetastar/2)n \leq p$ or $(1-\zetastar/4)n \leq p$ for these definitions or events to make sense).
	
	We aim to show that there exist $\kappamin',\Delta,C,c>0$ depending only on $\cuPmodel$ such that
	\begin{equation}
	\label{eq:high-prob-sparse-eigenvalue}
		\mprob(\cuA) \geq 1-Ce^{-cn}\,.
	\end{equation}
	The second event in the definition of $\cuA$ is controlled by \cite[Corollary 5.35]{vershynin2010} and the third event by Lemma \ref{lem:large-SG-bound}.
	Now it is sufficient to consider the first event in the definition of $\cuA$.
	
	\paragraph*{Case $n/p > 1/(1-\zetastar/8)$} Because $1/(1-\zetastar/8) \geq 1/(1-\zetamin/8) > 1$, we have $\mprob(\kappa_-\big(\bX,n(1-\zetastar/4)\big) \geq \kappamin')\geq1-Ce^{-cn}$ because $\kappa_-(\bX,n(1-\zetastar/4)) \geq \kappa_-(\bX,p)$ is the minimum singular value of $\bX/\sqrt{n}$, whence we invoke \cite[Corollary 5.35]{vershynin2010}.

	\paragraph*{Case $n/p \leq 1/(1-\zetastar/8)$}
	Let $k = \lfloor n(1-\zetastar/4) \rfloor$ and note that $k < p$ because $n/ \leq p/(1-\zetastar/8)$.
	Because $\kappa_-(\bX,S') \geq \kappa_-(\bX,S)$ when $S' \subset S$,
	we have that $\kappa_-(\bX,n(1-\zetastar/4)) = \min_{|S| = k} \kappa_-(\bx,S)$.
	By a union bound, for any $t > 0$
	\begin{equation}
	\label{eqn:union}
		\mprob\left(\kappa_-(\bX,n(1-\zetastar/4)) \leq t\right) \leq \sum_{|S| = k} \mprob( \kappa_-(\bX_S) \leq t)\,.
	\end{equation}
	The matrix $\bX_S = \widetilde \bX_S \mySigma_{S,S}^{1/2}$ where $\widetilde \bX_S$ has entries distributed i.i.d. $\normal(0,1)$.
	Thus, one has 
	\begin{equation*}
	 \kappa_-(\bX_S) \geq \kappa_-(\widetilde \bX_S) \kappa_-(\mySigma_{S,S}^{1/2}) \geq \kappa_-(\widetilde \bX_S) \kappamin^{1/2}.
	 \end{equation*}
	Invoking the fact that $\widetilde \bX_S$ has the same distribution for all $|S|=k$, expression~\eqref{eqn:union} implies
	\begin{equation}
		\mprob\left(\kappa_-(\bX,n(1-\zetastar/4)) \leq t\right)
		\leq 
		\binom{p}{k} \mprob(\kappa_-(\widetilde \bX_S) \leq t / \kappamin^{1/2})\,,
	\end{equation}
	where the $S$ appearing on the right-hand side can be any $S$ with cardinality $k$.
	By Lemma 2.9 of \cite{blanchard2010}, 
	\begin{equation}
	\begin{aligned}
		\mprob(\kappa_-(\widetilde \bX_S) \leq t / \kappamin^{1/2})
		\leq 
		C(n,t/\kappamin^{1/2}) \exp\left(n\psi(k/n,t/\kappamin^{1/2})\right)\,,
	\end{aligned}
	\end{equation}
	where $C(a,b)$ is a universal polynomial in $a,b$ and $\psi(a,b) := \frac12[(1-a)\log b + 1 - a + a \log a - b]$.
	(Lemma 2.9 of \cite{blanchard2010} states a bound on the density of $\kappa_-(\widetilde \bX_S)$, but a deviation bound incurs only a factor $t/\kappamin^{1/2}$ which we may absorb into the polynomial term). 
	Note also that $\binom{p}{k} \leq C'(p) \exp(nH(k/p)/(n/p))$, where $C'$ is a universal polynomial.
	We conclude that
	\begin{equation}
	\begin{aligned}
		\mprob\left(\kappa_-(\bX,n(1-\zetastar/4)) \leq t\right)
		\leq 
		C(n,p,t/\kappamin^{1/2}) \exp\left(n(H(k/p)/(n/p) + \psi(k/n,t/\kappamin^{1/2}))\right)\,.
	\end{aligned}
	\end{equation}
	Note that $\psi(a,b) \leq \frac{\zetastar}{8}\log b$ for all $a = k/n \leq 1 - \zetastar/4$ and $b \in (0,1)$.
	Thus, there exists $c,\kappamin' > 0$, depending only on $\Deltamin
	,\kappamin,\zetamin$, such that $H(k/p)/(n/p) + \psi(k/n,
	\kappamin'/\kappamin^{1/2}) < -2c$.
	Because $C(n,p,\kappamin'/\kappamin^{1/2})e^{-cn}$ is upper bounded by a constant $C$ depending only on $\Deltamin,\kappamin,c$,
	we conclude there exists $C,c,\kappamin' > 0$ depending only on $\Deltamin,\kappamin^{1/2},\zetamin$ such that
	\begin{equation}
		\mprob\left(\kappa_-(\bX,n(1-\zetastar/4)) \leq \kappamin'\right) 
		\leq 
		Ce^{-cn}\,.
	\end{equation}
	This conclude the proof of the high-probability bound Eq.~\eqref{eq:high-prob-sparse-eigenvalue}.

	The remainder of the argument takes place on the high-probability event $\cuA$.
	Consider any $\mytheta \in \reals^p$.
	Let $\{j \in [p] \mid |\widehat t_j| \geq 1 -\Delta/2\}$.
	We first construct $S_+ \supset S(\Delta/2)$ such that
	\begin{equation}
	\label{eq:support-decomposition}
		\text{\emph{(i)} } |S_+| \leq n(1-\zetastar/4)
		\;\;\;\;
		\text{and}
		\;\;\;\;
		\text{\emph{(ii)} } 
		\|\mytheta_{S_+^c}\|_2 \leq \frac{2\sqrt{2}}{\sqrt{n \zetastar}} \| \mytheta_{S(\Delta/2)^c}\|_1\,,
	\end{equation}
	where we adopt the convention that $\|\mytheta_\emptyset\|_1 = \|\mytheta_\emptyset\|_2 = 0$.
	We establish this by considering two cases.

	\paragraph*{Case 1: $p \leq n(1-\zetastar / 4)$}

	In this case, let $S_+ = [p]$.
	Then Eq.~\eqref{eq:support-decomposition} holds trivially.

	\paragraph*{Case 2: $p > n(1-\zetastar/4)$}

	Let $S_1,\ldots,S_k$ be a partition of $[p] \setminus S(\Delta/2)$ satisfying the following properties: first, $|S_i| \geq n \zetastar / 8$ for $i = 1,\ldots,k-1$; second, $|S_1| \geq \cdots \geq |S_k|$; third, $|S(\Delta/2) \cup S_1| \leq n(1-\zetastar/4)$; and fourth, $|\theta_j| \geq |\theta_{j'}|$ if $j \in S_i$ and $j' \in S_{i'}$ for $i \leq i'$.
	This is possible because $|S(\Delta/2)| \leq n(1-\zetastar/2)$ and, because $n\zetastar/8 \geq 1$, there exists an integer between $n(1-\zetastar/4)$ an $n(1 - \zetastar/8)$.
	In this case, let $S_+ = S(\Delta/2) \cup S_1$.
	Condition \emph{(i)} holds by construction.
	To verify condition \emph{(ii)}, 
	observe
	\begin{equation}
	\begin{aligned}	
		\|\mytheta_{S_+^c}\|_2^2
			= 
			\sum_{i=2}^k \|\mytheta_{S_i}\|_2^2 
			\leq 
			\sum_{i=2}^k |S_i| \left(\frac{\|\mytheta_{S_{i-1}}\|_1}{|S_{i-1}|}\right)^2
			&\leq 
			\frac1{\min_{i=1,\ldots,k-1}\{|S_{i-1}|\}} \sum_{i=1}^{k-1} \|\mytheta_{S_i}\|_1^2
			\\
		&\leq 
			\frac{8}{n\zetastar  } \|\mytheta_{S(\Delta/2)^c}\|_1^2\,,
	\end{aligned}
	\end{equation}
	where the first inequality holds because $|\theta_j| \leq \|\mytheta_{S_{i-1}}\|_1/|S_{i-1}|$ for $j \in S_i$,
	the second inequality holds because $|S_i| \leq |S_{i-1}|$,
	and the third inequality holds because $|S_i| \geq n\zetastar/8$ for $i \leq k-1$.
	Thus, Eq.~\eqref{eq:support-decomposition} holds in this case as well.

	We lower bound the growth of the Lasso objective by
	\begin{equation}
	\begin{aligned}
		\Risk(\mytheta) - \Risk(\thetahat)
			&=
			\frac1{2n}\|\bX(\mytheta - \thetahat)\|_2^2 + \frac1n\langle \bX^\top (\by - \bX \thetahat) , \thetahat - \mytheta \rangle + \frac\lambda{\sqrt{n}} \left(\|\mytheta\|_1 - \|\thetahat\|_1\right)
			\\
		& = 
			\frac1{2n}\|\bX(\mytheta - \thetahat)\|_2^2 + \frac\lambda{\sqrt{n}}\left(\langle \bthat, \thetahat - \mytheta \rangle + \|\mytheta\|_1 - \|\thetahat\|_1\right)\,.		
	\end{aligned}
	\end{equation}
	We first make the observation that 
	\begin{equation}
		\langle \bthat, \thetahat - \mytheta \rangle + \|\mytheta\|_1 - \|\thetahat\|_1
		\geq
		\frac{\Delta}2 \| \mytheta_{S(\Delta/2)^c} \|_1\,.
	\end{equation}
	Because $\bthat \in \partial \|\thetahat\|_1$ and $|t_j| \leq 1-\Delta/2$ on $S(\Delta/2)^c$ so that $t_j(\widehat \theta_j - \theta_j) + |\theta_j| - |\widehat \theta_j|\geq 0$ for all $j$, and is no smaller than $\Delta |\theta_j|/2$ for $j \in S(\Delta/2)^c$.
	Thus, it is guaranteed that 
	\begin{equation}
		\Risk(\mytheta) - \Risk(\thetahat)	
			\geq
			\frac{\lambda\Delta}{2\sqrt{n}} \| \mytheta_{S(\Delta/2)^c} \|_1 + \frac1{2n}\|\bX(\mytheta - \thetahat)\|_2^2\,.
	\end{equation}

	Now choose $S_+\subset[p]$ satisfying Eq.~\eqref{eq:support-decomposition}.
	Condition \emph{(ii)} of Eq.~\eqref{eq:support-decomposition} implies
	\begin{align}
	\label{eq:growth-on-S+}
		\Risk(\mytheta) - \Risk(\thetahat) 
		\geq 
		c_1 \| \mytheta_{S_+^c}\|_2\,,
	\end{align}
	where $c_1>0$ depends on $\cuPmodel$.
	Next we prove that there exists $c' > 0$ such that for $\|\mytheta_{S_+} - \thetahat_{S_+}\|_2 < c'$,
	\begin{align}
	\label{eqn:convex-with-splus}
		\Risk(\mytheta) - \Risk(\thetahat) 
		\geq 
		c' \| \mytheta_{S_+} - \thetahat_{S_+}\|_2^2\;\;\; \text{holds true on event $\cuA$}.
	\end{align}
	In order to see this, if $\|\bX_{S_+}(\mytheta_{S_+} - \thetahat_{S_+})\|_2/2 \geq \|\bX_{S_+^c}\mytheta_{S_+^c}\|_2$, then 
	\begin{equation}
		\Risk(\mytheta) - \Risk(\thetahat) 
			\geq 
			\frac1{8n} \|\bX_{S_+}(\mytheta_{S_+} - \thetahat_{S_+})\|_2^2
			\geq
			\frac{\kappamin'^2}{8} \|\mytheta_{S_+} - \thetahat_{S_+}\|_2^2\,,
	\end{equation}
	as a consequence of Eq.~\eqref{eq:high-prob-sparse-eigenvalue}.
	Otherwise, if $\|\bX_{S_+}(\mytheta_{S_+} - \thetahat_{S_+})\|_2/2 < \|\bX_{S_+^c}\mytheta_{S_+^c} \|_2$, 
	then $\|\mytheta_{S_+} - \thetahat_{S_+}\|_2 \leq \kappamin'^{-1/2}\|\bX_{S_+}(\mytheta_{S_+} - \thetahat_{S_+})\|_2 / \sqrt{n} \leq 2\kappamin'^{-1/2}\|\bX_{S_+^c}\mytheta_{S_+^c}\|_2/\sqrt{n} \leq C\|\mytheta_{S_+^c}\|_2$.
	Thus
	\begin{equation}
		\Risk(\mytheta) - \Risk(\thetahat) 
			\geq 
			c_1 \| \mytheta_{S_+^c}\|_2
			\geq 
			c_1 \| \mytheta_{S_+} - \thetahat_{S_+}\|_2\,,
	\end{equation}
	where the value of $c_1$ changes between the last inequalities. 
	Combining the previous two displays, we have established inequality~\eqref{eqn:convex-with-splus}, 
	where again the value of $c'$ has changed from the previous displays.

	Combined with Eq.~\eqref{eq:growth-on-S+},
	we conclude there exists $c_1,c' > 0$ depending only on $\cuPmodel$
	such that for $\|\mytheta - \thetahat\|_2\leq c'$,
	\begin{equation}
		\Risk(\mytheta) - \Risk(\thetahat) 
		\geq 
		c' \|\mytheta - \thetahat \|_2^2\,.
	\end{equation}
	The proof is completed.

%%%%%%%%%%%%%%%%%%%%%%%%%%%%%%%%%%%%%%%%%%%%%%%%%%%%%%%%%%%%%%%%%%%%%%

\subsubsection{Proof of Corollary \ref{CorDBLassoCI}}
	To start, let us define 
	\begin{align}
		\phi_1(x,\Delta) := \min(1,x/\Delta - z_{1-q/2}/\Delta + 1)_+.
	\end{align}
	The function $\phi_1(x)$ equals to $0$ for $x \leq z_{1-q/2} - \Delta$ and $1$ for $x \geq z_{1-q/2}$, and linearly interpolates between these two regions elsewhere.
	Therefore, the false-coverage proportion $\FCP := \frac1p \sum_{j = 1}^p \indic{\theta^*_j \not \in \mathsf{CI}_j}\,$ can be controlled as 
	\begin{align}
		\FCP 
			&=
			\frac1p \sum_{j=1}^p \indic{}\left\{
				|\thetahatd_j - \theta^*_j| 
				> 
				\frac{\Sigma_{j|-j}^{-1/2}\|\by - \bX \thetahat \|_2 }{ n(1 - \|\thetahat\|_0/n) }z_{1-q/2}
			\right\} \nonumber \\
		&\leq 
			\frac1p \sum_{j=1}^p \phi_1\left(
				\frac{\Sigma_{j|-j}^{1/2}(1-\|\thetahat\|_0/n)\sqrt{n}|\thetahatd_j - \theta^*_j|}{\|\by - \bX \thetahat\|_2/\sqrt{n}},
				\Delta
			\right) \nonumber
		\\
		&\leq
			\frac1p \sum_{j=1}^p \phi_1\left(
				\Sigma_{j|-j}^{1/2}\sqrt{n}|\thetahatd_j - \theta^*_j|/\taustar,
				\Delta
			\right)
			+ 
			\frac1\Delta
			\left|
				\frac{1-\|\thetahat\|_0/n}{\|\by - \bX\thetahat\|_2/\sqrt{n}} 
				- 
				\frac1{\taustar}
			\right|	
			\left(
				\frac1p \sum_{j=1}^p \sqrt{n}\Sigma_{j|-j}^{1/2}|\thetahatd_j - \theta^*_j| 
			\right)\,.
	\end{align}
	We bound the terms on the right-hand side respectively.
	\begin{itemize}

		\item 
		By Theorems \ref{ThmLassoResidual} and Theorem \ref{ThmLassoSparsity},
		there exist $C,c,c' > 0$ depending only on $\cuPmodel$ such that for $\epsilon < c'$, 
		we have $\left|\frac{1-\|\thetahat\|_0/n}{\|\by - \bX\thetahat\|_2/\sqrt{n}} - \frac1{\taustar}\right| < \epsilon$ with probability at least $1 - \frac{C}{\epsilon^3}e^{-cn\epsilon^6}$.

		\item 
		Because $\Sigma_{j|-j}^{1/2} \leq \kappamax^{1/2}$ for all $j$,
		the quantity $\frac1p \sum_{j=1}^p \sqrt{n}\Sigma_{j|-j}^{1/2}|\thetahatd_j - \theta^*_j|$
		is $\sqrt{\kappamax n/p}$-Lipschitz in $\bthetahatd$.
		Moreover, when $\bthetahatd$ is replaced by $\thetastar + \taustar\mySigma^{-1/2}\bg / \sqrt{n}$,
		this equantity has expectation bounded by a constant depending only on $\cuPmodel$.
		By Theorem \ref{ThmDBLasso},
		there exist $C,C',c > 0$ depending only on $\cuPmodel$
		such that $\frac1p \sum_{j=1}^p \sqrt{n} \Sigma_{j|-j}^{1/2}|\thetahatd_j - \theta^*_j| < C' \sqrt{n/p}$ with probability
		at least $1-Ce^{-cn}$.

		\item 
		The quantity $\frac1p\sum_{j=1}^p \phi_1\left(\Sigma_{j|-j}^{1/2}\sqrt{n}|\thetahatd_j - \theta^*_j|/\taustar,\Delta\right)$ is $\frac{L\sqrt{n}}{\Delta\sqrt p}$-Lipschitz in $\bthetahatd$,
		where $L$ is a constant depending only on $\cuPmodel$.
		By Theorem \ref{ThmDBLasso},
		we conclude there exists $C,c,c' > 0$ depending only on $\cuPmodel$
		such that for $\epsilon < c'$, we have
		\begin{equation*}
		 	\frac1p\sum_{j=1}^p \phi_1\left(\Sigma_{j|-j}^{1/2}\sqrt{n}|\thetahatd_j - \theta^*_j|/\taustar,\Delta\right) < \frac1p\sum_{j=1}^p \E\left[\phi_1\left(\Sigma_{j|-j}^{1/2}|(\taustar \mySigma^{-1/2}\bg)_j|/\taustar,\Delta\right)\right] + \epsilon\sqrt{n/p}/\Delta,
		 \end{equation*} 
		with probability at least $1 - \frac{C}{\epsilon^3}e^{-cn\epsilon^6}$.

		\item 
		Using the fact that the standard Gaussian density is upper bounded by $(2\pi)^{-1/2}$, we obtain the bound $\E\left[\phi_1\left(\Sigma_{j|-j}^{1/2}|(\taustar \mySigma^{-1/2}\bg)_j|/\taustar,\Delta\right)\right] \leq q + \frac{2\Delta}{\sqrt{2\pi}}$.
	\end{itemize}

	Combining the previous bounds, we conclude there exist $C,C',c,c' > 0$ depending only on $\cuPmodel$ such that for all $\epsilon < c'$, 
	we have $\FCP \leq q + C'(\Delta + \epsilon\sqrt{n/p}/\Delta)$ with probability at least $1 - \frac{C}{\epsilon^3}e^{-cn\epsilon^6}$.
	Optimizing over $\Delta$, we conclude there exists $C,c,c' > 0$ depending only on $\cuPmodel$ and $\Deltamax $ such that for all $\epsilon < c'$, we have 
	$\FCP \leq q + \epsilon$ with probability at least $1 - \frac{C}{\epsilon^6}e^{-cn\epsilon^{12}}$.

	The lower bound holds similarly.

%%%%%%%%%%%%%%%%%%%%%%%%%%%%%%%%%%%%%%%%%%%%%%%%%%%%%%%%%%%%%%%%%%%%%%

% \input{CI_proof}
\subsection{More details on confidence interval for a single coordinate}
% : description of exact test and proof of Lemma~\ref{LemLOO-DB-Dist}}
\label{SecPfSingleCoordCI}

Because they may be of independent interest, 
we first describe in detail the construction of the exact tests outlined in the discussion in Section \ref{SecSingleCoordCI}
and state some results about the quantities involved in the construction (Lemma \ref{LemLOO-DB-Dist} and Theorem \ref{ThmExactTestPower} below).
The proof of Theorem \ref{ThmCILengthAndPower} uses a special case of Lemma \ref{LemLOO-DB-Dist}, whereas Theorem \ref{ThmExactTestPower} is independent of any future development, 
and is stated only for general interest. 

\subsubsection{Description of exact test}

We begin with the following lemma.
\begin{lem}
\label{LemLOO-DB-Dist}
  We have the following.
    \begin{enumerate}[label=(\alph*)]

      \item % a
      \emph{(Exact conditional normality of $\xi_j$ when $\widehat \theta_{j,\mathrm{init}} = \theta^*_j$).}
      If $\widehat \theta_{j,\mathrm{init}} = \theta^*_j$, then
      \begin{equation}\label{eq:exact-normality-under-null}
        \sqrt{n} (\xi_j - \theta_j^*) / \widehat \tau_{\mathrm{loo}}^j
        \sim \normal(0,\Sigma_{j|-j}^{-1})\,.
      \end{equation}

        \item % b
        \emph{(Approximate normality of $\xi_j$ in general).}
        Assume $p\geq 2$. 
        Let $\delta_{\mathrm{loo}} = n / (p-1)$.
        Assume Assumption \ref{assump:1} is satisfied with $\mySigma_{-j,-j}$ in place of $\mySigma$ and $\mytheta_{-j}^*$ in place of $\thetastar$.

      Then there exist constants $C,c,c' > 0$ depending only on $\cuPmodel$ such that the following occurs.
      There exist random variables $r_j,R_j,Z_j$ such that
      \begin{equation}\label{eq:approx-normality-under-alt}
        \sqrt{n}(\xi_j - \theta_j^*)/\widehat\tau_{\mathrm{loo}}^j = r_jZ_j + R_j\,,
      \end{equation}
      and for all $|\theta_j^* - \widehat \theta_{j,\mathrm{init}}| < c'$ 
      \begin{align}\label{EqnCIError}
          Z_j &\sim \normal(0,\Sigma_{j|-j}^{-1})\,,\;\;\;\;\;\;\;\;
          \mprob\left(|r_j-1| > C|\theta_j^*|\right) \leq  e^{-cn}\,,\;\;\;\;\;\;\;\;
          \mprob\left(|R_j| > \epsilon\right) \leq \frac{C{\theta_j^*}^2}{\epsilon^2}.
      \end{align}

    \end{enumerate}
\end{lem}
\noindent 
Lemma \ref{LemLOO-DB-Dist}(a) implies that the test which rejects when $|\xi_j| \geq \Sigma_{j|-j}^{-1/2} \widehat \tau^j_{\mathrm{loo}} z_{1-\alpha/2}$ is an exact level-$\alpha$ test of the null $\theta^*_j = \widehat \theta_{j,\mathrm{init}}$.
Lemma \ref{LemLOO-DB-Dist}(b) states that under the alternative $\xi_j$ is approximately normal with mean $\theta^*_j - \widehat \theta_{j,\mathrm{init}}$ 
and standard deviation $\Sigma_{j|-j}^{-1/2}\widehat\tau_{\mathrm{loo}}^j / \sqrt{n}$.
(The latter quantity is random but concentrates).
Thus, Lemma \ref{LemLOO-DB-Dist}(b) permits a power analysis of the exact test.

The next theorem is included because it may be of independent interest. 
No future development depend upon this theorem, and it can safely be skipped.
\begin{theo}[Insensitivity of fixed point parameter to $\widehat\theta_{j,\mathrm{init}}$]
\label{ThmExactTestPower}
  Let $\tau_{\mathrm{loo}}^*(\widehat \theta_{j,\mathrm{init}},\zeta_{\mathrm{loo}}^*(\widehat \theta_{j,\mathrm{init}})$ be the solution to the fixed point equations \eqref{EqnEqn1} and \eqref{EqnEqn2} in the leave-one-out model for the Lasso at regularization $\lambda$.

    There exists $c',L > 0$ depending only on $\cuPmodel$ such that for $|\theta_j^* - \widehat \theta_{j,\mathrm{init}}| \leq c'$,
    we have $|\tau^*_{\mathrm{loo}}(\theta_j^*) - \tau^*_{\mathrm{loo}}(\widehat \theta_{j,\mathrm{init}})| \leq L\sqrt{|\theta^*_j - \widehat \theta_{j,\mathrm{init}}|}$.

\end{theo}
\noindent Theorem \ref{ThmExactTestPower} says that the noise variance is effectively constant for all $\widehat \theta_{j,\mathrm{init}} - \theta_j^* = o(1)$.

\subsubsection{Proof of Lemma \ref{LemLOO-DB-Dist}, Theorem \ref{ThmCILengthAndPower}, and Theorem \ref{ThmExactTestPower}}
\label{sec:thm12}

\subsubsection*{Proof of Lemma \ref{LemLOO-DB-Dist}(a)}

	When $\theta^*_j = \widehat \theta_{j,\mathrm{init}}$, 
	the data $(\by_{\mathrm{init}},\bX_{-j})$ is independent of $\xper_j$. 
	Because $\thetahatloo$ is $\sigma(\by_{\mathrm{init}},\bX_{-j})$-measurable,
	\begin{equation}
	\sqrt{n} (\xi_j - \theta_j^*) / \widehat \tau_{\mathrm{loo}}^j = \frac{(\xper_j)^\top(\by_{\mathrm{init}} - \bX_{-j} \thetahatloo)}{\Sigma_{j|-j}\|\by_{\mathrm{init}} - \bX_{-j}\thetahatloo\|_2}\,.
	\end{equation}
	Because $\xper_j \sim \normal(0,\Sigma_{j|-j}\Ind_p)$ and is independent of $\by_{\mathrm{init}} - \bX_{-j} \thetahatloo$,
	conditionally on $\by_{\mathrm{init}},\bX_{-j}$ the quantity is distributed $\normal(0,\Sigma_{j|-j}^{-1})$.
	Thus, it is distributed $\normal(0,\Sigma_{j|-j}^{-1})$ unconditionally as well. 
	We have established \eqref{eq:exact-normality-under-null}.

\subsubsection*{Proof of Lemma \ref{LemLOO-DB-Dist}(b)}
  We may without loss of generality consider the case $\widehat \theta_{j,\mathrm{init}} = 0$.
  Indeed, the joint distribution of $(\by_{\mathrm{init}},\bX_{-j},\xper_j,\thetahatloo)$ under $\theta^*_j$ is equal to the joint distribution of $(\by_{\mathrm{init}},\bX_{-j},\xper_j,\thetahatloo)$ if the $j^\text{th}$ coordinate of the original model were instead $\theta^*_j - \widehat \theta_{j,\mathrm{init}}$, and the leave-one-out model and Lasso are taken with $\widehat \theta_{j,\mathrm{init}} = 0$.
  Under this transformation, the conditions of the Theorem are still met, possibly with $M'$ replaced by $2M'$.

	Thus, consider the case $\widehat \theta_{j,\mathrm{init}} = 0$.
	In this case, $\by_{\mathrm{init}} = \by$, and we will write the latter in place of the former in what follows.
	Define the quantity
	\begin{align}
		\tilde \xi_j
			&:= 
			(\xper_j)^\top ( \by - \bX_{-j} \thetahatloo ) - \theta_j^*\Sigma_{j|-j}(n - \| \thetahatloo \|_0 ) \,.
		\end{align}
	Direct calculations give 
	\begin{align}
		\tilde \xi_j  &= 
			(\xper_j)^\top ( \sigma \bz + \xper_j \theta_j^* + \bX_{-j} \thetastarloo - \bX_{-j} \thetahatloo ) 
			- 
			\theta_j^*\Sigma_{j|-j}(n - \| \thetahatloo \|_0 ) 
		\nonumber\\
		&= 
			\underbrace{(\xper_j)^\top ( \sigma \bz + \bX_{-j} \thetastarloo - \bX_{-j} \thetahatloo')}_{\Delta_1}
			+
			\underbrace{(\xper_j)^\top( \xper_j \theta_j^* +  \bX_{-j}( \thetahatloo' - \thetahatloo ) ) - \theta_j^*\Sigma_{j|-j}(n - \| \thetahatloo \|_0 ) }_{\Delta_2}\,,
	\end{align}
	where
	\begin{equation}
		\thetahatloo' \defn \arg\min_{\mytheta \in \reals^p} \left\{\frac1{2n} \| \sigma \bz + \bX_{-j} \thetastarloo - \bX_{-j} \mytheta \|_2^2 + \frac\lambda{\sqrt{n}} \|\mytheta\|_1\right\}\,.
	\end{equation}
	In particular, $\thetahatloo'$ is $\sigma(\bz,\bX_{-j})$-measurable, so is independent of $\xper_j$,
	whence
	\begin{equation}\label{EqDelta1Dist}
		\Delta_1\, | \, \bz,\bX_{-j} \sim \normal\Big( 0 , \Sigma_{ j | - j} \|\sigma \bz + \bX_{-j} \thetastarloo - \bX_{-j} \thetahatloo'\|_2^2 \Big)\,.
	\end{equation}
	The estimate $\thetahatloo$ is a function of $\bz$, $\bX_{-j}$, and $\xper_j$.
	We make this explicit by writing $\thetahatloo(\bz,\bX_{-j},\xper_j)$. Following this notation, 
	$\thetahatloo'$ defined above is equal to $\thetahatloo(\bz,\bX_{-j},\bzero)$. 
	
	Next consider the term $\Delta_2$. First define 
	\begin{align}
		F(\bz,\bX_{-j},\xper_j) \defn  \xper_j \theta_j^* +  \bX_{-j}( \thetahatloo(\bz,\bX_{-j},\bzero) - \thetahatloo(\bz,\bX_{-j},\xper_j)) \,.
	\end{align}
	Use $\nabla_{\xper_j}$ to denote the Jacobian with respect to $\xper_j$. 
	Almost surely, $\nabla_{\xper_j}F(\bz,\bX_{-j},\xper_j) = \theta_j^*(\Ind_n- \mathsf{P}_{{\bX_{-j}}_{\widehat S}})$, 
	where $\mathsf{P}_{{\bX_{-j}}_{\widehat S}}$ is the projector onto the span of $\{\brevebx_k \mid k \in \widehat S\}$ and $\widehat S$ is the support of $\thetahatloo(\bz,\bX_{-j},\xper_j)$ \cite{zou2007}.
	The function $F$ is $\theta_j^*$-Lipschitz in $\xper_j$ for fixed $\bz,\bX_{-j}$.
	Therefore we conclude that
	$\Delta_2 = (\xper_j)^\top F(\bz,\bX_{-j},\xper_j) - \Sigma_{j|-j}\, \mathrm{div}_{\xper_j}\; F(\bz,\bX_{-j},\xper_j) $. 
	Applying Stein's formula and the second-order Stein's formula \cite[Eq.~(2.1) and Theorem 2.1]{Bellec2018SecondOS}, we get
	\begin{align}
          \Exs[ \Delta_2 | \bz,\bX_{-j} ] &= 0  \qquad \text{and} \\
         \Var(\Delta_2 | \bz,\bX_{-j} ) &= \Sigma_{j|-j} \left(\E\Big[ \|F(\bz,\bX_{-j},\xper_j)\|_2^2  +
            \Sigma_{j|-j} \trace( \nabla_{\xper_j}F(\bz,\bX_{-j},\xper_j)^2 ) \bigm| \bz,\bX_{-j}\Big]\right) \,.
	\end{align}
	Note that, almost surely $\trace(\nabla_{\xper_j}F(\bz,\bX_{-j},\xper_j)^2) = {\theta_j^*}^2(n-\|  \thetahatloo\|_0)$.
	Further, $\|F(\bz,\bX_{-j},\xper_j)\|_2^2 \leq {\theta_j^*}^2 \|\xper_j\|_2^2$
	because $F$ is $\theta_j^*$-Lipschitz and $F(\bz,\bX_{-j},\bzero) = \bzero$.
	Thus,
	\begin{equation}
		\Var(\Delta_2 | \bz,\bX_{-j} ) 
			\leq 
			\Sigma_{j|-j}\E\Big[{\theta_j^*}^2\Big(\|\xper_j\|_2^2 + \Sigma_{j|-j}\Big(n- \|\thetahatloo\|_0 \Big)\Big)\Big| \bz,\bX_{-j} \Big] 
			\leq 
			2n\Sigma_{j|-j}^2{\theta_j^*}^2\;\; \text{almost surely.}
	\end{equation}
	Because the $\Exs[\Delta_2|\bz,\bX_{-j}] = 0$ almost surely, we have $\Var(\Delta_2) \leq 2n\Sigma_{j|-j}^2{\theta_j^*}^2$ as well. 
	Next observe that $\frac{\tilde \xi_j}{\Sigma_{j|-j}(n - \|\thetahatloo\|_0)} = \xi_j - \theta_j^*$.
	Thus,
	\begin{align}
		\frac{(n-\|\thetahatloo\|_0)(\xi_j - \theta_j^*)}{\|\by - \bX_{-j} \thetahatloo\|_2} 
			&=
			\frac{\|\sigma \bz + \bX_{-j} \thetastarloo - \bX_{-j} \thetahatloo'\|_2}{\|\by - \bX_{-j} \thetahatloo\|_2} 
			\frac{\Delta_1}{\Sigma_{j|-j}\|\sigma \bz + \bX_{-j} \thetastarloo - \bX_{-j} \thetahatloo'\|_2} 
		\nonumber\\
		&\quad\quad\quad\quad\quad\quad\quad\quad\quad\quad +
			\frac{\Delta_2}{\Sigma_{j|-j}\|\by - \bX_{-j} \thetahatloo\|_2} 
		\nonumber\\
		&=: r_jZ_j + R_j\,,
	\end{align}
	where $Z_j \defn \frac{\Delta_1}{\Sigma_{j|-j}\|\sigma \bz + \bX_{-j} \thetastarloo - \bX_{-j} \thetahatloo'\|_2} \sim \normal(0,\Sigma_{j|-j}^{-1})$ (and normality follows by the proof of Eq.~\eqref{eq:exact-normality-under-null}).

	The singular values of $\mySigma_{-j,-j}$ are bounded between the minimal and maximal singular values of $\mySigma$.
	Thus, the matrix $\mySigma_{-j,-j}$ satisfies Assumption \ref{assump:1} because $\mySigma$ does.
	In particular, the triple $\lambda$, $\mySigma_{-j,-j}$, and $\sigma_{\mathrm{loo}}^2$ satisfy Assumption \ref{assump:1} provided $|\theta_j^* | \leq 1$ (or any constant).
	Because $\thetastar_{-j}$ satisfies Assumption \ref{assump:1}$(d)$ for matrix $\mySigma_{-j,-j}$,
	and because $\|\theta^*_j\Sigma_{-j,-j}^{-1} \Sigma_{-j,j}\|_1 \leq |\theta_j^*|\|\Sigma_{-j,-j}^{-1} \Sigma_{-j,j}\|_2 \leq \sqrt{p}|\theta_j^*|\kappamin^{-1/2} \kappamax $,
	Assumption \ref{assump:1}(d) for the leave-one-out model if $M$ is replaced by $M + |\theta_j^*|\kappamin^{-1/2} \kappamax / \sqrt{\Deltamin}$.
	In particular, we may apply all of our results to this model when $|\theta_j^*| \leq 1$.

	Because $F$ is $\theta_j^*$-Lipschitz in $\xper_j$, 
	we have $|\|\by - \bX_{-j} \thetahatloo\|_2 - \|\sigma \bz + \bX_{-j} \thetastarloo - \bX_{-j} \thetahatloo'\|_2| < \theta_j^* \|\xper_j\|_2$.
	By Theorem \ref{ThmLassoResidual} and since ${\theta^*_j}^2\|\xper_j\|_2^2 \sim \Sigma_{j|-j}{\theta^*_j}^2\chi_n^2$, 
	there exist $C,c,c' > 0$ depending only on $\cuPmodel$ and $M'$ such that for $\epsilon < c'$, it is guaranteed that 
	\begin{align*}
		\mprob\left(|r_j-1|>\epsilon\right) 
		&=  
		\mprob\left(\left|\frac{\|\sigma \bz + \bX_{-j} \thetastarloo - \bX_{-j} \thetahatloo'\|_2}{\|\by - \bX_{-j} \thetahatloo\|_2} -1 \right|>\epsilon\right) \\
		&= \mprob\left(\left|\frac{\|\sigma \bz + \bX_{-j} \thetastarloo - \bX_{-j} \thetahatloo'\|_2 - \|\by - \bX_{-j} \thetahatloo\|_2}{\|\by - \bX_{-j} \thetahatloo\|_2}  \right|>\epsilon\right) \\
		&\leq \mprob\left(\frac{|\theta_j^*|\|\xper_j\|_2}{\|\by - \bX_{-j} \thetahatloo'\|_2 - |\theta_j^*|\|\xper_j\|_2} > \epsilon\right).
	\end{align*}
	Recall we have used that $\|\by - \bX_{-j} \thetahatloo'\|_2/\sqrt{n}$ concentrates on a quantity for which we have a lower bound by Theorem \ref{ThmLassoResidual}.
	Second, we have used that $\mprob(|\theta_j^*|\|\xper_j\|_2/\sqrt{n} > 2\Sigma_{j|-j}^{1/2} |\theta_j^*|) \leq C\exp(-cn)$.
	Thus, choosing $C,C',c,c' > 0$ depending only on $\cuPmodel$,
	we have for $|\theta_j^*| < c'$ and $\epsilon = C|\theta_j^*|$ that the previous display is buounded above by $Ce^{-cn}$

	Similarly, combining the concentration of $\Sigma_{j|-j}\|\by - \bX_{-j} \thetahatloo'\|_2/\sqrt{n}$ on a quantity for which we have a lower bound, the high probability upper bound on $|\theta_j^* \|\xper_j\|_2/\sqrt{n}$, and Chebyshev's inequality applied to $\Delta_2$,
	there exists $C,c' > 0$ depending only on $\cuPmodel$ such that for $\epsilon < c'$
	\begin{equation}
		\mprob\left(|R_j| > \epsilon\right) < \frac{C{\theta_j^*}^2}{\epsilon^2}\,.
	\end{equation}
	The proof of the lemma is complete.

\subsubsection*{Proof of Theorem \ref{ThmCILengthAndPower}(a)}
  The event $\theta \not \in \mathsf{CI}^{\mathrm{loo}}_j$ is equivalent to 
  \begin{equation}
    \frac{\Sigma_{j|-j}^{1/2}\sqrt{n}(\xi_j - \theta)}{\widehat \tau^j_{\mathrm{loo}}} \not\in [-z_{1-\alpha/2},z_{1-\alpha/2}]\,.
  \end{equation}
  %
  % %
  % \begin{equation}
  %   \frac{\Sigma_{j|-j}^{1/2}(n-\|\thetahatloo\|_0)|\xi_j - \theta|}{\|\by_{\mathrm{init}} - \bX_{-j} \thetahatloo\|_2} \geq z_{1-\alpha/2}\,.
  % \end{equation}
  % %
  With $r_j,R_j$ defined as in Theorem \ref{LemLOO-DB-Dist}, this is equivalent to 
  \begin{equation}
    A := \Sigma_{j|-j}^{1/2}(r_jZ_j + R_j) + \frac{\Sigma_{j|-j}^{1/2}\sqrt{n}(\theta_j^*-\theta)}{\widehat \tau^j_{\mathrm{loo}}} \not \in [-z_{1-\alpha/2},z_{1-\alpha/2}]\,.
  \end{equation}
  By Theorems \ref{ThmLassoResidual} and \ref{ThmLassoSparsity} on concentration of the Lasso residual and sparsity and Theorem \ref{LemLOO-DB-Dist} on the concentration of $r_j$ and $R_j$,  
  there exist $C,c,c' > 0$ depending only on $\cuPmodel$ such that for all $\epsilon_1,\epsilon_2 < c'$,
  \begin{equation}
    \mprob\left(\left|A - \Sigma_{j|-j}^{1/2}Z_j - \frac{\Sigma_{j|-j}^{1/2}\sqrt{n}(\theta_j^* - \theta)}{\tau^*_{\mathrm{loo}}}\right| > \epsilon_1  + \sqrt{n}|\theta_j^* - \theta| \epsilon_2 \right) 
    \leq 
    \frac{C}{\epsilon_2^3} e^{-cn\epsilon_2^6} + e^{-cn} + \frac{C(\theta_j^* - \widehat \theta_{j,\mathrm{init}})^2}{\epsilon_1^2}\,.
  \end{equation}
  Thus, by direct calculation (where $C$ may take different values between lines)
  \begin{equation}
  \begin{aligned}
    \mprob&\left(
        A \not \in [-z_{1-\alpha/2},z_{1+\alpha/2}]
      \right)
      \geq 
      \mprob\left(
        \Big|\Sigma_{j|-j}^{1/2}Z_j + \frac{\Sigma_{j|-j}^{1/2}\sqrt{n}(\theta_j^* - \theta)}{\tau^*_{\mathrm{loo}}}\Big| > z_{1-\alpha/2}+\epsilon_1+\sqrt{n}|\theta_j^* - \theta|\epsilon_2
      \right)
    \\
    &\quad\quad\quad\quad\quad\quad\quad\quad\quad\quad\quad-
      \mprob\left(
        \Big|A - \Sigma_{j|-j}^{1/2}Z_j - \frac{\Sigma_{j|-j}^{1/2}\sqrt{n}(\theta_j^* - \theta)}{\tau^*_{\mathrm{loo}}}\Big| > \epsilon_1+\sqrt{n}|\theta_j^* - \theta|\epsilon_2
      \right)
    \\
    &\quad\quad
    \geq \mprob\left(
        \Big|\Sigma_{j|-j}^{1/2}Z_j + \frac{\Sigma_{j|-j}^{1/2}\sqrt{n}(\theta_j^* - \theta)}{\tau^*_{\mathrm{loo}}}\Big| > z_{1-\alpha/2}+\epsilon_1+\sqrt{n}|\theta_j^* - \theta|\epsilon_2
      \right)
     \\
     &\quad\quad\quad\quad\quad\quad\quad\quad\quad\quad\quad\quad 
     - 
      C\left(\frac1{\epsilon_2^3}e^{-cn\epsilon_2^6} + e^{-cn} + \frac{(\theta_j^* - \widehat \theta_{j,\mathrm{init}})^2}{\epsilon_1^2}\right) 
    \\ 
    &\quad\quad\geq 
      \mprob\left(|\theta_j^* + \Sigma_{j|-j}^{-1/2}\tau^*_{\mathrm{loo}} G/\sqrt{n} - \theta| \geq \Sigma_{j|-j}^{-1/2}\tau^*_{\mathrm{loo}}  z_{1-\alpha/2}\right) 
    \\ 
    &\quad\quad\quad\quad\quad\quad\quad\quad 
      - C\left( \epsilon_1+ \sqrt{n} |\theta^*_j - \theta|\epsilon_2 + \frac1{\epsilon_2^3}e^{-cn\epsilon_2^6} + e^{-cn} + \frac{(\theta_j^*-\widehat \theta_{j,\mathrm{init}})^2}{\epsilon_1^2}\right)\,.
  \end{aligned}
  \end{equation}
  Taking $\epsilon_1 = |\theta_j^* - \widehat \theta_{j,\mathrm{init}}|^{2/3}$ and $\epsilon_2 = n^{-1/6 + \gamma}$,
  we get one side of the inequality.
  The reverse inequality is obtained similarly.

\subsubsection*{Proof of Theorem \ref{ThmCILengthAndPower}(b)}
  By definition, one has 
  \begin{equation}
    \frac{\htauloo^j}{\tau^*_{\mathrm{loo}}} 
    = 
    \frac{\|\by - \bX_{-j} \thetahatloo\|_2/\sqrt{n}}{(1 - \|\thetahatloo\|_0/n)\tau^*_{\mathrm{loo}}}\,.
  \end{equation}
  As argued in the proof of Theorem \ref{LemLOO-DB-Dist}(b),
  the leave-one-out model obeys Assumption \ref{assump:1} provided $\widehat \theta_j - \theta_j^* \leq 1$ (or some constant).
  Then Equation \eqref{EqnCILengthBound} follows from Theorems \ref{ThmLassoResidual} and \ref{ThmLassoSparsity} on the concentration results for the Lasso residual and the sparsity.

% \subsubsection*{Proof of Lemma \ref{ThmExactTestPower}(a)}
% 	The joint distribution of $(\by_{\mathrm{init}},\bX_{-j},\xper_j,\thetahatloo)$ under $\theta^*_j$ is equal to the joint distribution of $(\by_{\mathrm{init}},\bX_{-j},\xper_j,\thetahatloo)$ if the $j^\text{th}$ coordinate of the original model were instead $\theta^*_j - \widehat \theta_{j,\mathrm{init}}$,
% 	and the leave-one-out model and Lasso are taken with $\widehat \theta_{j,\mathrm{init}} = 0$.
% 	(We also used this fact in the proof of Theorem \ref{LemLOO-DB-Dist}(b)).
% 	Thus, 
% 	%
% 	\begin{equation}
% 	\begin{aligned}
% 		\mprob_{\theta_j^*,\widehat \theta_{j,\mathrm{init}}}(|\xi_j| \geq \Sigma_{j|-j}^{-1/2} \widehat \tau^{\omega,j}_{\mathrm{loo}} z_{1-\alpha/2})
% 			&=
% 			\mprob_{\theta_j^*}\left(\frac{\sqrt{n}(1-\|\thetahatloo^\omega\|_0/n)|\xi^\omega_j|}{\|\by^\omega - \bX_{-j} \thetahatloo^\omega\|_2} > \Sigma_{j|-j}^{-1/2}z_{1-\alpha/2}\right)
% 			\\
% 		%
% 		&= 
% 			\mprob_{\theta_j^* - \theta}\left(\frac{\sqrt{n}(1-\|\thetahatloo\|_0/n)|\xi_j|}{\|\by - \bX_{-j} \thetahatloo\|_2} > \Sigma_{j|-j}^{-1/2}z_{1-\alpha/2}\right)
% 			\\
% 		%
% 		&= \mprob_{\theta_j^* - \omega}\big(0 \not \in \mathsf{CI}^{\mathrm{loo}}_j\big)\,.
% 	\end{aligned}
% 	\end{equation}
% 	%
% 	Then Eq.~\eqref{EqnExactCIPower} follows from Eq.~\eqref{EqnCIPower}.

\subsubsection*{Proof of Theorem \ref{ThmExactTestPower}}
	As argued in the proof of Theorem \ref{LemLOO-DB-Dist}(b), 
	Assumption \ref{assump:1} is satisfied by the leave-one-out model provided $|\theta_j^* - \widehat \theta_{j,\mathrm{init}}| \leq 1$ (or any constant).

	To emphasize the dependence of $\tau^*_{\mathrm{loo}}$ on $\widehat \theta_{j,\mathrm{init}}$, 
	we write $\tau^*_{\mathrm{loo}}(\widehat \theta_{j,\mathrm{init}})$.
	Our goal is to bound $|\tau^*_{\mathrm{loo}}(\widehat \theta_{j,\mathrm{init}}) - \tau^*_{\mathrm{loo}}(\theta_j^*)|$
	To control the fixed point parameter $\tau^*_{\mathrm{loo}}$,
	we will study the functional objective $\cuE_0$ of Eq.~\eqref{EqFunctionalObj} for the leave-one-out model as we vary $\widehat \theta_{j,\mathrm{init}}$.
	For simplicity of notation, we will drop the subscript on $\cuE_0$.
	As we vary $\widehat \theta_{j,\mathrm{init}}$, 
  	the only parameters defining the leave-one-out model which change are the noise variance $\sigma_{\mathrm{loo}}^2 = \sigma_{\mathrm{loo}}^2(\widehat \theta_{j,\mathrm{init}})$ and $\thetastarloo = \thetastarloo(\widehat \theta_{j,\mathrm{init}})$.
  	We write
	\begin{equation}
	\begin{aligned}
		\cuE(\bv;\widehat \theta_{j,\mathrm{init}}) 
		&:=
			\frac12\Big(
			\sqrt{\|\bv\|_{L^2}^2+\sigma_{\mathrm{loo}}^2 }\,-\frac{\langle \bg , \bv \rangle_{L^2}}{\sqrt{n}}
			\Big)^2_+
			+
			\frac{\lambda}{\sqrt{n}}
			\E\Big\{
				\|\thetastarloo + (\mySigma_{-j,-j})^{-1/2}\bv\|_1-\|\thetastarloo\|_1\big)
			\Big\}\,.
	\end{aligned}
	\end{equation}
	%
	% The model \eqref{eq:y-omega-linear-model} for $\omega = \theta^*_j$ corresponds to the choice $\sigma_{\mathrm{loo}}^2 = \sigma^2$.
	Denote the unique minimizer of $\cuE$ by $\bv^* = \bv^*(\widehat \theta_{j,\mathrm{init}})$. Existence and uniqueness is guaranteed by Lemma \ref{LemAlphaFixedPtSoln}.
	Also by the proof of Lemma \ref{LemAlphaFixedPtSoln}, 
	\begin{equation}
	\label{eq:taustar-theta-via-E}
		\tau^*_{\mathrm{loo}} = \sqrt{\sigma_{\mathrm{loo}}^2 + \|\bv^*\|_{L^2}^2}.
	\end{equation}

	Under Assumption \ref{assump:1}, the objective $\cuE$ is $L$-Lipschitz in $\sigma_{\mathrm{loo}}^2$ on $\sigma_{\mathrm{loo}}^2 > \sigmamin^2$ and is $L$-Lipschitz in $\thetastarloo$ for some $L$ depending only on $\cuPmodel$.
	Recall that  $\sigma_{\mathrm{loo}}^2 - \sigma^2 = \Sigma_{j|-j}(\theta^*_j-\widehat \theta_{j,\mathrm{init}})^2 \leq C|\theta_j^* - \widehat \theta_{j,\mathrm{init}}|$ for $|\theta_j^* - \widehat \theta_{j,\mathrm{init}}| < c'$ 
	and $\| \thetastarloo(\theta_j^*) - \thetastarloo(\widehat \theta_{j,\mathrm{init}})\|_2 \leq C |\theta_j^* - \widehat \theta_{j,\mathrm{init}}|$,
	for appropriately chosen $C,c'$.
	By the proof of Lemma \ref{LemFixedPtContinuity}, 
	there exists $r,a > 0$ depending only on $\cuPmodel$ such that $\cuE(\cdot;\theta_j^*)$ is $a$-strongly convex in $\bv$ on $\|\bv - \bv^*(\theta_j^*)\|_2 \leq r$.
	Thus, for $\|\bv - \bv^*(\theta_j^*)\|_{L^2} \leq r$,
	\begin{equation}
	\begin{aligned}
		\cuE(\bv;\widehat \theta_{j,\mathrm{init}}) 
			&\geq 
			\cuE(\bv;\theta_j^*) - L|\theta_j^*-\widehat \theta_{j,\mathrm{init}}|
		\\
			&\geq 
			\cuE(\bv^*(\theta_j^*);\theta_j^*) + a\|\bv - \bv^*(\theta_j^*)\|_{L^2}^2 - L|\theta_j^* - \widehat \theta_{j,\mathrm{init}}|
		\\
		&\geq 
			\cuE(\bv^*(\theta_j^*);\widehat \theta_{j,\mathrm{init}}) 
				+ a\|\bv - \bv^*(\theta_j^*)\|_{L^2}^2 
				- 2L|\theta_j^* - \widehat \theta_{j,\mathrm{init}}|.
	\end{aligned}
	\end{equation}
	We conclude that if $\sqrt{2L|\theta_j^* - \widehat \theta_{j,\mathrm{init}}|/a} \leq r$, 
  then $\|\bv^*(\widehat \theta_{j,\mathrm{init}}) - \bv^*(\theta_j^*)\|_{L^2} \leq \sqrt{2L|\theta_j^* - \widehat \theta_{j,\mathrm{init}}|/a}$.
	By Eq.~\eqref{eq:taustar-theta-via-E},
	\begin{equation}
	\begin{aligned}
		|\tau^*_{\mathrm{loo}}(\widehat \theta_{j,\mathrm{init}}) - \tau^*_{\mathrm{loo}}(\theta_j^*)|
		&=
		\left|\sqrt{\sigma_{\mathrm{loo}}^2(\widehat \theta_{j,\mathrm{init}}) + \|\bv^*(\widehat \theta_{j,\mathrm{init}})\|_{L^2}^2} - \sqrt{\sigma^2 + \|\bv^*(\theta_j^*)\|_{L^2}^2}\right|\\
		\\
		&\leq L\sqrt{|\theta_j^*-\omega|}\,,
	\end{aligned}
	\end{equation}
	where the $L$ in the final line differs from the one in the preceding line and depends only on $\cuPmodel$.

	The proof is complete.

%%%%%%%%%%%%%%%%%%%%%%%%%%%%%%%%%%%%%%%%%%%%%%%%%%%%%%%%%%%%%%%%%%%%%%

\subsubsection{Proof of Lemma \ref{LemCostComparisonLambda}}
\label{sec:pfLemmaCostComparisonLambda}

\begin{lem}
\label{lemma:ControlL1Norm}
  Assume $n/p \geq \delta_{\mathrm{DT}}(\sign(\barthetastar),\mySigma) + \Deltamin$.
  Then there exist finite constants $a,c_0,c_1,C_0>0$ depending only on
  $\Deltamin,\kappamin,\kappamax$
  such that if $n \geq c_1$ the following happens with probability at least $1-C_0e^{-c_0n}$. 
  For any $\bw\in\reals^p$:
  \begin{align}
  \label{EqnMinSingular}
    \|\barthetastar+\bw\|_1-\|\barthetastar\|_1\le 0\;\; \Rightarrow \;\; \frac1{\sqrt{n}}\|\bX\bw\|_2\ge a\|\bw\|_2\, .
   \end{align}
\end{lem}

\begin{proof}[Proof of Lemma \ref{lemma:ControlL1Norm}]
The Gaussian width $\cuG(\bx,\mySigma)$ is an upper bound on the standard notion of Gaussian width $\cuG_{d}(\bx,\mySigma)$ defined in Eq.~\eqref{Eqn:std-GW}.
Thus,
\begin{equation}
\label{eq:standard-to-functional-GW}
	\sqrt{n/p - \Deltamin}
	\geq 
	\cuG_{d}(\bx,\mySigma) 
	=
	\frac1{\sqrt{p}} \E\Big[ 
		\max_{ 
			\substack{	\bv \in \cuK(\bx,\mySigma) \\
						\|\bv\|_2^2/p \leq 1
						} 
			} 
			\langle \bv , \bg \rangle 
		\Big]\,.
\end{equation}
The result then follows from standard results; see, for example, Corollary 3.3 of \cite{chandrasekaran2010} and its proof.
We repeat the proof here for convenience.

For simplicity of notation, we will denote $\cuK = \cuK(\bx,\mySigma)$.
Let $\bx = \sign(\barthetastar)$.
Note that
\begin{align*}
\|\barthetastar+\bw\|_1-\|\barthetastar\|_1 & \, = 
\|\bw_{S^c}\|_1 + \|(\barthetastar+\bw)_{S}\|_1 - \|\barthetastar_{S}\|_1  \geq 
\|\bw_{S^c}\|_1 + \sum_{j \in \supp(\bx)} x_j \bw_j\,,
\end{align*}
whence $\|\barthetastar+\bw\|_1-\|\barthetastar\|_1 \leq 0$ implies $\mySigma^{1/2}\bw \in \cuK$.
Thus, it suffices to show that with probability at least $1 - C_0e^{-c_0n}$, one has 
\begin{equation}
	\mySigma^{1/2} \bw \in \cuK \;\; \Rightarrow \;\; \frac1{\sqrt{n}}\|\bX\bw\|_2\ge a\|\bw\|_2\, .
\end{equation}
Define the minimum singular value over $\cuK$ as 
\begin{align}
	\kappa_{-} (\bX, \cuK) \defn \inf\, \Big\{ \frac1{\sqrt{n}}\ltwo{\bX \bw} \mid \mySigma^{1/2}\bw \in \cuK, \ltwo{\bw} = 1 \Big\}\,,
\end{align}
and define $\tilde\kappa_{-} (\bX, \cuK) \defn \inf\, \Big\{ \frac1{\sqrt{n}} \ltwo{\bX \bw} \mid \mySigma^{1/2}\bw \in \cuK, \ltwo{\mySigma^{1/2}\bw} = 1 \Big\}.$
Then, because $\cuK$ is a cone (and so is scale invariant),
\begin{align*}
	\kappa_{-}(\bX,\cuK) 
		\geq 
		\tilde\kappa_{-}(\bX,\cuK) \cdot \min_{\ltwo{\bw}=1}\ltwo{\mySigma^{1/2}\bw} 
		\geq 
		\tilde\kappa_{-}(\bX,\cuK) \kappamin^{1/2}\,.
\end{align*}
Thus, it suffices to show there exists $a>0$ depending on $\Deltamin,\kappamin,\kappamax$ such that with high-probability $\kappa_{-}(\bX,\cuK) \geq a$.

By definition,
\begin{align*}
	- \Exs [ \tilde\kappa_{-} (\bX, \cuK)] 
	= \Exs \Big[ \max_{\substack{\mySigma^{1/2}\bw \in \cuK \\ \ltwo{\mySigma^{1/2}\bw} = 1}} - \frac1{\sqrt{n}}\ltwo{\bX \bw}\Big] = 
	\Exs \Big[\max_{\substack{\mySigma^{1/2}\bw \in \cuK \\ \ltwo{\mySigma^{1/2}\bw} = 1}}\min_{\|\bu\|_2 = 1} 
	\frac1{\sqrt{n}}\bu^\top \bX \bw\Big]\,.
\end{align*}
Recall that the rows of $\bX$ are distributed iid from $\normal(\bzero,\mySigma)$.
By Gordon's lemma (Corollary G.1 of \cite{miolane2018distribution})
\begin{align*}
	- \Exs [ \tilde\kappa_{-} (\bX, \cuK)] 
		& \, 
		\leq 
		\E\Big[\max_{\substack{\mySigma^{1/2}\bw \in \cuK \\ \ltwo{\mySigma^{1/2}\bw} = 1}}\min_{\|\bu\|_2 = 1} 
			\frac1{\sqrt{n}} \|\mySigma\bw\|_2 \inprod{\bh}{\bu}
			+
			\frac1{\sqrt{n}} \|\bu\|_2 \inprod{\mySigma^{1/2}\bw}{\bg}
			\Big]
		\\
	&= 
		\frac1{\sqrt{n}}\E\Big[
			\max_{\substack{\mySigma^{1/2} \bw \in \cuK \\ \ltwo{\mySigma^{1/2} \bw} = 1}}
			- \|\bh\|_2 + \inprod{\mySigma^{1/2}\bw}{\bg} 
			\Big] 
		\leq   
		-\sqrt{\frac{n}{n+1}}
		+ 
		\sqrt{\frac{p}{n}} \cuG(\bx,\mySigma)
		\\
	&\leq 
		-\sqrt{\frac{n}{n+1}} + \sqrt{\frac{p}{n}} \wedge \sqrt{1 - \Deltamin \frac{p}{n}}
	\leq 
		-\sqrt{\frac{n}{n+1}} + \frac12 \vee \Big(1 - \frac{\Deltamin}4\Big)\,,
\end{align*}
where the second-to-last equality uses $\E[\|\bh\|_2] \geq \frac{n}{\sqrt{n+1}}$ and the definition of $ \cuG(\bx,\mySigma)$,
and the last inequality uses the upper bound on the Gaussian width \eqref{eq:standard-to-functional-GW} and considering cases $p/n \leq 1/4$ and $p/n \geq 1/4$.
For all $n \geq 2$ we have $\sqrt{n/(n+1)} \geq \sqrt{(n-1)/n} \geq 1 - \frac1{\sqrt{2}n}$.
Thus, for $n \geq \frac{1}{(\sqrt{2}/4)\wedge (\sqrt{2}\Deltamin/8))}$,
$\Exs [ \tilde\kappa_{-} (\bX, \cuK)] \geq \frac14 \vee  \frac{\Deltamin}{8}$.

The quantity $\tilde\kappa_{-} (\bX, \cuK)$ as a function of $\bX\mySigma^{-1/2}$ is $\frac{1}{\sqrt{n}}$-Lipschitz with respect to the Frobenius norm. 
Thus 
\begin{align*}
	\mathbb{P}(\tilde\kappa_{-} (\bX, \cuK)
	\leq \Exs [ \tilde\kappa_{-} (\bX, \cuK)] - t)  
	\leq e^{-n t^2/2}.
\end{align*}
Taking $t = (1/8) \vee (\Deltamin/16)$ completes the proof.
\end{proof}

\begin{lem}
\label{lem:L1Norm}
   Under Assumption \ref{assump:1},
   there exist constants $C,C_0,c_0 > 0$ depending only on
   $\cuPmodel$ such that
   \begin{align}
   \label{eq:high-prob-penality-at-soln-bound}
     \mprob\left(\forall \lambda\in[\lambda_{\min},\lambda_{\max}]: \; \frac{1}{\sqrt{n}}\big|\|\thetastar+\mySigma^{-1/2}\hbv^\lambda\big\|_1-\big\|\thetastar\big\|_1\big|\le C\right)\ge 1-C_0e^{-c_0n}\, .
   \end{align}
\end{lem}
\begin{proof}[Proof of Lemma \ref{lem:L1Norm}]
  The proof follows almost exactly that for \cite[Proposition C.4]{miolane2018distribution}, using Lemma \ref{lemma:ControlL1Norm}.
  The primary difference is the approximation of $\thetastar$ by $\barthetastar$.
  
  Let $\barthetastar$ be as in Assumption \ref{assump:1}(d).
	Note that
  \begin{equation}
  \begin{aligned}
      \frac\lambda{\sqrt{n}} \left(\|\thetastar + \mySigma^{-1/2}\hbv^\lambda \|_1 - \|\thetastar\|_1\right) 
      &\geq 
        \frac\lambda{\sqrt{n}} \left(\|\barthetastar + \mySigma^{-1/2}\hbv^\lambda \|_1 - \|\barthetastar\|_1 - 2\|\barthetastar - \thetastar\|_1\right)\\
      &\geq 
        \frac\lambda{\sqrt{n}} \left(\|\barthetastar + \mySigma^{-1/2}\hbv^\lambda \|_1 - \|\barthetastar\|_1\right) - \frac{2\lambda M}{\Deltamin}\,.
    \end{aligned}
  \end{equation}
  We show that the high probability event \eqref{eq:high-prob-penality-at-soln-bound} is implied by the event
  \begin{equation}
		\cuA:= \left\{\|\barthetastar+\bw\|_1-\|\barthetastar\|_1\le 0\;\; \Rightarrow \;\; \frac1{\sqrt{n}}\|\bX\bw\|_2\ge a\|\bw\|_2\right\} \bigcap \left\{\|\bz\|_2 \leq 2 \sqrt{n}\right\},
  \end{equation}
  where $a$ is as in Lemma \ref{lemma:ControlL1Norm}.
  On this event, $\cuC(\hbv^\lambda) \leq \cuC_\lambda(\bzero) = \sigma^2\|\bz\|_2^2/(2n) \leq 2\sigma^2$, whence 
  \begin{equation}
      \frac1{\sqrt{n}} \left(\|\barthetastar + \mySigma^{-1/2}\hbv^\lambda \|_1 - \|\barthetastar\|_1\right) 
      \leq 
      2\sigma^2/\lambdamin + 2 M/\Deltamin\,,
  \end{equation}
  which further implies
  \begin{equation}
  \label{eq:penalty-at-soln-UB}
      \frac1{\sqrt{n}} \left(\|\thetastar + \mySigma^{-1/2}\hbv^\lambda \|_1 - \|\thetastar\|_1\right) 
      \leq 
      2\sigma^2/\lambdamin + 4M/\Deltamin\,.
  \end{equation}
  Let $\hbw^\lambda = \mySigma^{-1/2} \hbv^\lambda$.
  On the event $\cuA$, we also have
  \begin{equation}
  \begin{aligned}
      2\sigma^2 
      &\geq 
      \cuC(\hbv^\lambda)
      \geq 
      \frac1{2n}\|\sigma \bz - \bX \hbw^\lambda \|_2^2 - \frac\lambda{\sqrt{n}} \|\hbw^\lambda\|_1 - \frac{2\lambda M}{\Deltamin} \\
      &\geq 
      \frac1{4n}\|\bX \hbw^\lambda \|_2^2 - \frac{\sigma^2}{2n}\|\bz\|_2^2 - \frac{\lambda}{\sqrt{n/p}} \|\hbw^\lambda\|_2 - \frac{2\lambda M}{\Deltamin} \\
      &\geq 
      \frac{a}{4}\|\hbw^\lambda \|_2^2 - 4\sigma^2 - \frac{\lambda}{\sqrt{\Deltamin}} \|\hbw^\lambda\|_2 - \frac{2\lambda M}{\Deltamin}\,.
    \end{aligned}
  \end{equation}
  We conclude that $      \|\hbw^\lambda\|_2 
      \leq 
      C,$
  for $C$ depending only on $\cuPmodel$.
  Because $\frac1{\sqrt{n}}\big|\| \thetastar + \hbw^\lambda \|_1 - \| \thetastar \|_1 \big| \leq \| \hbw \|_1 / \sqrt{n} \leq \| \hbw \|_2 / \sqrt{\Deltamin}$, the result follows.
  \end{proof}
Lemma \ref{LemCostComparisonLambda} follows from Lemma \ref{lem:L1Norm} by exactly the same argument in the proof of \cite[Lemma C.5]{miolane2018distribution}.

\subsubsection{Proof of Lemma \ref{LemPopCostComparisonLambda}}
\label{sec:pfPopCostComparisonLambda}

Recall from the proof of Lemma \ref{LemAlphaFixedPtSoln} (in particular, Eq.~\eqref{EqnTcuEtoProx})
that $\hbv^{f,\lambda}$, where the latter is viewed as a function of $\bg$ and hence an element of $L^2(\reals^p;\reals^p)$, is the unique minimizer of $\cuE^\lambda$.
By optimality,
\begin{equation}
\label{eq:expected-l1-ub}
	\frac{\sigma^2}2
	=
	\cuE^\lambda(\bzero)
	\geq
	\cuE^\lambda(\hbv^{f,\lambda})
	\geq 
	\frac\lambda {\sqrt{n}} \E\Big\{
				\|\thetastar+\mySigma^{-1/2}\hbv^{f,\lambda}\|_1-\|\thetastar\|_1\,,
			\Big\}\,.
\end{equation}
Take $\barthetastar$ as in Assumption \ref{assump:1}.
Note that
\begin{equation}
\begin{aligned}
  \frac\lambda{\sqrt{n}} 
  \E\left\{\|\thetastar + \mySigma^{-1/2}\hbv^\lambda \|_1 - \|\thetastar\|_1\right\}
  &\geq 
    \frac\lambda{\sqrt{n}} 
    \E\left\{\|\barthetastar + \mySigma^{-1/2}\hbv^\lambda \|_1 - \|\barthetastar\|_1 - 2\|\barthetastar - \thetastar\|_1\right\}\\
  &\geq 
    \frac\lambda{\sqrt{n}} 
    \E\left\{\|\barthetastar + \mySigma^{-1/2}\hbv^\lambda \|_1 - \|\barthetastar\|_1\right\} - \frac{2\lambda M}{\sqrt{\Deltamin}}\,.
\end{aligned}
\end{equation}
By the definition of Gaussian width (see Eq.~\eqref{EqnGaussianWidth}),
either
\begin{equation}
	\frac\lambda{\sqrt{n}} \E\left\{\|\barthetastar + \mySigma^{-1/2}\hbv^\lambda \|_1 - \|\barthetastar\|_1\right\}
	\geq 
	0\,,
\end{equation}
or 
\begin{equation}
\begin{aligned}
	\frac{\sigma^2}{2} 
		= \cuE^\lambda(\bzero) 
		\geq \cuE^\lambda(\hbv^{f,\lambda})
		&\geq 
		\frac12 \left(\|\hbv^{f,\lambda}\|_{L^2} - \frac{\cuG(\bx,\mySigma)\|\hbv^{f,\lambda}\|_{L^2}}{\sqrt{n/p}}\right)_+^2 - 
		\frac\lambda{\sqrt{n}} 
		\E\left\{\|\mySigma^{-1/2}\hbv^{f,\lambda}\|_1\right\} - \frac{2\lambda M}{\sqrt{\Deltamin}}
		\\
	&\geq 
		\frac12\|\hbv^{f,\lambda}\|_{L^2}^2 \Big(1 - \sqrt{\frac{p}{n}} \wedge \sqrt{1 - \frac{p}{n} \Deltamin} \Big)^2 
	\\
	&\qquad\qquad\qquad\qquad\qquad\qquad-
		\frac\lambda{\sqrt{\Deltamin}} 
		\E\left\{\|\mySigma^{-1/2}\hbv^{f,\lambda}\|_2\right\} - \frac{2\lambda M}{\sqrt{\Deltamin}}
		\\
	&\geq 
		\frac12\|\hbv^{f,\lambda}\|_{L^2}^2 \Big(\frac1{4} \wedge \Big(1 - \frac{\Deltamin}{4}\Big)^2 \Big) - \frac{\lambda\kappamin^{-1/2}}{\sqrt{\Deltamin}} \|\hbv^{f,\lambda}\|_{L^2} - \frac{2\lambda M}{\sqrt{\Deltamin}}\,,
\end{aligned}
\end{equation}
where the second inequality uses that $\cuG(\bx,\mySigma) \leq 1 \vee \sqrt{n/p - \Deltamin}$, 
and the last inequality considers separately cases $p/n \leq 1/4$ and $p/n \geq 1/4$.
The previous display implies $\| \hbv^{f,\lambda} \|_{L^2} \leq C$ for some $C$ depending on $\cuPmodel$,
in which case
% in which case, $\mc{stopped here}
% %
% \begin{equation}
% 	\frac{\|\hbv^{f,\lambda}\|_{L^2}}{\sqrt n}
% 	\leq 
% 	2\sqrt{2\lambda M + \frac{\sigma^2}{2} + \frac{\lambda^2 \kappamax^2}{2\delta\Deltamin^2}}\,.
% \end{equation}
% %
% Thus, in this case
% %
\begin{equation}
\begin{aligned}
\label{eq:expected-l1-lb}
	\frac1{\sqrt{n}} \E\left\{\|\thetastar + \mySigma^{-1/2}\hbv^{f,\lambda} \|_1 - \|\thetastar\|_1\right\}
		&\geq 
		- 
		\frac1{\sqrt{n}} \E\Big\{\| \mySigma^{-1/2} \hbv^{f,\lambda} \|_1\Big\}
		\geq 
		- 
		\frac{\kappamin^{-1/2}}{\sqrt{\Deltamin}} \| \hbv^{f,\lambda} \|_{L^2}
		\geq -C,
\end{aligned}
\end{equation}
for some possibly new value of $C$.
Combining Eqs.~\eqref{eq:expected-l1-ub} and \eqref{eq:expected-l1-lb},
there exists $C$ depending only on $\cuPmodel$
such that
\begin{equation}
	\left|\frac1{\sqrt{n}} \E\left\{\|\thetastar + \mySigma^{-1/2}\hbv^{f,\lambda} \|_1 - \|\thetastar\|_1\right\}\right|
	\leq 
	C\,.
\end{equation}

By Lemma \ref{LemTauZetaBounds},
the solutions to the fixed point equations \eqref{EqnEqn1} and \eqref{EqnEqn2} are bounded by $\cuPmodel$-dependent constants.
By the proof of Lemma \ref{LemFixedPtContinuity},
there exists $r,a > 0$ depending only on $\cuPmodel$ such that $\cuE^\lambda$ is $a$ strongly-convex in the neighborhood $\|\bv - \hbv^{f,\lambda}\|_{L^2} \leq r$ around its minimizer.
Thus, we conclude for any $\bv \in L^2$ we have
\begin{equation}
	\cuE^\lambda(\bv) \geq \cuE^\lambda(\hbv^{f,\lambda}) + h(\|\bv - \hbv^{f,\lambda}\|_{L^2})
\end{equation}
where $h(x) := \min\{ ax^2/2 , ar|x|/2 \}$.
It worth emphasizing that this bound holds for with the same $a,r$ for all $\lambda \in [\lambdamin,\lambdamax]$.
Then direct calculations further give 
\begin{equation}
\begin{aligned}
	\cuE^\lambda(\hbv^{f,\lambda'})
		&\geq
		\cuE^\lambda(\hbv^{f,\lambda}) + h(\|\hbv^{f,\lambda'} - \hbv^{f,\lambda}\|_{L^2}) 
		\geq 
		\cuE^{\lambda'}(\hbv^{f,\lambda}) + h(\|\hbv^{f,\lambda'} - \hbv^{f,\lambda}\|_{L^2})  - C|\lambda' - \lambda|
		\\
	&\geq 
		\cuE^{\lambda'}(\hbv^{f,\lambda'}) + 2h(\|\hbv^{f,\lambda'} - \hbv^{f,\lambda}\|_{L^2}) - C|\lambda' - \lambda| 
		\\
	&\geq 
		\cuE^{\lambda}(\hbv^{f,\lambda}) + 4h(\|\hbv^{f,\lambda'} - \hbv^{f,\lambda}\|_{L^2}) - 2C|\lambda' - \lambda| \,,
\end{aligned}
\end{equation}
where the last inequality holds by the same string of manipulations justifying the first three.
Take $c' = ar^2/C$.
If $|\lambda - \lambda'| < c'$, we have $h(\|\hbv^{f,\lambda'} - \hbv^{f,\lambda}\|_{L^2}) \leq ar^2/2$, whence in fact $h(\|\hbv^{f,\lambda'} - \hbv^{f,\lambda}\|_{L^2}) = \frac{a\|\hbv^{f,\lambda'} - \hbv^{f,\lambda}\|_{L^2}^2}{2}$.
We conclude $\|\hbv^{f,\lambda'} - \hbv^{f,\lambda}\|_{L^2} \leq K|\lambda - \lambda'|^{1/2}$ for appropriately chosen $K$.

\subsection{Control of the emprical distribution: proof of Corollary \ref{CorEmpDist}}\label{sec:proofOfCorEmpDist}

In the fixed design model, let $I$ be uniformly distributed on $[p]$ independently of $\bg$.
Let $\mu_*$ be the distribution of $(\theta_I^*,\widehat \theta^f_I)$.
For any $k$,
we have $\frac1p \sum_{i=1}^p \phi_k(\sqrt{n}\theta_i^*,\sqrt{n}\widehat \theta^f_i) $ is $\tau \kappamin^{-1/2}/\sqrt{p}$-Lipschitz in $\bg$,
so that by Gaussian concentration of Lipschitz functions,
\begin{equation*}
	\mprob\left(
			\Big|\frac1p \sum_{i=1}^p \phi_k(\sqrt{n}\theta_i^*,\sqrt{n}\widehat \theta^f_i) 
			- 
			\E\Big[\frac1p \sum_{i=1}^p \phi_k(\sqrt{n}\theta_i^*,\sqrt{n}\widehat \theta^f_i)\Big]\Big| > t
		\right) \leq 2 e^{-2p\kappamin t^2/\taumax^2},
\end{equation*}
whence $\E\left[\Big|\frac1p \sum_{i=1}^p \phi_k(\theta_i^*,\widehat \theta^f_i) - \E\Big[\frac1p \sum_{i=1}^p \phi_k(\theta_i^*,\widehat \theta^f_i)\Big]\Big|\right] \leq C / \sqrt{p}$, for some $C$ depending on $\cuPmodel$.
Summing the above inequality over $k=1,\ldots,\infty$, we obtain that 
$$
	\E\left[d_{w^*}\left(\frac1p \sum_{i=1}^p \delta_{\sqrt{n}\theta_i^*,\sqrt{n}\widehat \theta_i^f},\mu_*\right)\right] \leq C/\sqrt{p}.
$$
Note further that $d_{w^*}\left(\frac1p \sum_{i=1}^p \delta_{\sqrt{n}\theta_i^*,\sqrt{n}\widehat \theta_i^f},\mu_*\right)$ is $\taumax\kappamin^{-1/2}/\sqrt{p}$-Lipschitz in $\bg$.
Applying Gaussian Lipschitz concentration in the fixed design model, we conclude the second inequality in Corollary \ref{CorEmpDist}.
Because the function is $\taumax \kappamin^{-1/2} \sqrt{\Deltamax} $-Lipschitz in its argument, applying Theorem \ref{ThmControlLassoEst}, we conclude the first inequality in Corollary \ref{CorEmpDist}.

%%%%%%%%%%%%%%%%%%%%%%%%%%%%%%%%%%%%%%%%%%%%%%%%%%%%%%%%%%%%%%%%%%%%%%

\section{Auxiliary results and proofs}

\subsection{Unbounded risk below the Donoho-Tanner phase transition: Proof of Proposition \ref{prop:DT-necessary}}
\label{sec:DT-necessary}

\begin{proof}[Proof of Proposition \ref{prop:DT-necessary}]
	By Proposition \ref{prop:std-width},
	we can prove the result with $\cuG_d$ in place of $\cuG$,
	at the cost of changing constants.
	Recall $\cuC(\cdot)$ and $\cuL(\cdot)$ as defined in Eqs.~\eqref{eq:cuC-def} and \eqref{eq:def-cuL-main}.
	First note that 
	\begin{equation}
		\cuC(\bv) 
			\geq 
			\cuC'(\bv)
			:= 
			\frac1{2n} \| \sigma \bz - \bX \mySigma^{-1/2} \bv \|_2^2 - \frac{\lambda}{\sqrt{n}} \| \mySigma^{-1/2} \bv \|_1.
	\end{equation}
	Because $\cuC'(\cdot)$ is continuous almost surely,
	for all $r > 0$, there exists $T > 0$, depending on $r$ but not on $\thetastar$, such that 
	\begin{equation}
		\mprob\Big(\inf_{\|\bv\|_2 \leq r} \cuC(\bv) \geq -T/2\Big) 
		\geq 
		\mprob\Big(\inf_{\| \bv \|_2 \leq r} \cuC'(\bv) \geq -T/2\Big) \geq 1 - Ce^{-cp\epsilon^2}.
	\end{equation}
	To complete the proof, it suffices to show that for appropriately chosen $\thetastar$ such that $\sign(\thetastar) = \sign(\bx)$,
	the minimum value of $\cuC(\bv)$ is smaller than $-T$ with high probability.

	By Gaussian concentration of Lipschitz functions, 
	with probability at least $1 - 2e^{-p\epsilon^2/8}$,
	\begin{equation}
		\sup_{\substack{\| \bv \|_2 \leq 1\\ F(\bv;\bx,\mySigma) \leq 0} } 
			\frac1{\sqrt{p}} \< \bv , \bg \> 
			>
			\sqrt{\frac{n}{p}} + \frac{\epsilon}{2}.
	\end{equation}
	Thus, let $\bv = \bv(\bg)$ be a random variable such that $\| \bv \|_2 = 1$ and $F(\bv;\bx,\mySigma) \leq 0$ always,
	and  $\< \bv , \bg \> / \sqrt{p} > \sqrt{n/p} + \epsilon/2$ with probability at least $1 - 2e^{-p\epsilon^2/8}$.
	For some $0 < \epsilon_1 \leq 1$, let
	\begin{equation}
	\label{eq:DT-necc-good-event}
		\bv_1 = \bv - \frac{\epsilon_1}{\sqrt{p}} \mySigma^{1/2} \bx.
	\end{equation}
	Note that $\| \bv_1 \|_2 \leq 1 + \kappamax^{1/2} \epsilon_1$,
	$F(\bv_1;\bx,\mySigma) = - \frac{\epsilon_1}{\sqrt{p}} \| \bx \|_2^2$, 
	and on the event that $\< \bv , \bg \> / \sqrt{p} > \sqrt{n/p} + \epsilon/2$ and $\< \mySigma^{1/2} \bx , \bg \> / p < \epsilon/4$,
	which has probability at least $1 - Ce^{-c\epsilon^2}$,
	we have
	\begin{equation}
	\label{}
		\frac1{\sqrt{p}} \< \bv_1 , \bg \> > \sqrt{\frac{n}{p}} + \frac{\epsilon}{4}.
	\end{equation}
	Note that for $\thetastar = t\kappamin^{-1/2}(1 + \kappamax^{1/2} \epsilon_1) \bx $ and $S = \supp(\bx)$, we have $\sign\big((\thetastar + t\mySigma^{-1/2} \bv_1  )_S\big) = \sign(\bx_S)$,
	whence $\| \thetastar + t \mySigma^{-1/2} \bv_1 \|_1 - \| \thetastar \|_1 = tF(\bv_1;\bx,\mySigma) \leq - t \epsilon_1 \| \bx \|_2^2 / \sqrt{p}$.
	Denoting $\cuL(\bv; \thetastar)$ the objective \eqref{eq:def-cuL-main} as a function also of $\thetastar$, 
	and choosing $\thetastar$ as above,
	we have
	\begin{equation}
		\cuL(t\bv_1; \thetastar)
			\leq 
			\frac12 
			\left(
				\sqrt{\sigma^2 + t^2(1 + \kappamax^{1/2}\epsilon_1)^2} 
				- 
				t\Big(
					1 + \sqrt{\frac{p}{n}} \frac{\epsilon}{4}
				\Big)
			\right)_+^2
			- \frac{t\epsilon_1\| \bx \|_2^2}{\sqrt{p}},
	\end{equation}
	on the event \eqref{eq:DT-necc-good-event}.
	Now we see that if we take $\epsilon_1$ small enough that $ 1 + \kappamax^{1/2}\epsilon_1 < 1 + \sqrt{p/n}\,\epsilon/4$, 
	then the right-hand side above goes to $-\infty$ as $t \rightarrow \infty$.
	Take $\epsilon_1$ this small,
	and take $t$ large enough so that the right-hand side above is smaller than $-T$. 
	Thus, by Gordon's lemma (Lemma \ref{LemGordonMain}), 
	and choosing $\thetastar$ as above for this value of $t$ and $\epsilon_1$,
	we conclude that
	\begin{equation}
		\mprob\Big(\inf_{\bv \in \reals} \cuC(\bv) \leq -T\Big) \geq Ce^{-cp\epsilon^2}.
	\end{equation}
	This completes the proof.	
\end{proof}

\subsection{Gaussian width under correlated designs: proof of Proposition  \ref{ClaimBallToApproxSparse}}
\label{Sec:PFclaimBall}
\begin{proof}[Proof of Proposition \ref{ClaimBallToApproxSparse}]
	Parts (a) and (b) of Proposition \ref{ClaimBallToApproxSparse} now follow from the following constructions.
	\begin{enumerate}[label=(\alph*)]

		\item % a 
		We may take $\barthetastar = \thetastar$. 
		This is then an immediate consequence of Lemma \ref{lem:width-under-corr}.

		\item % b
		If $\|\thetastar\|_q^q / p \leq \nu^q$ for some $\nu > 0$ and $q > 0$,
		take $\barthetastar$ to be supported on the largest (in absolute value) $s$ coordinates of $\thetastar$, and take
		$\bar \theta^*_j = \theta^*_j$ for $j$ on this support.
		The Gaussian width bound for $\barthetastar$ of Assumption \ref{assump:1}(d) holds by part (a).
		Because $\| \thetastar \|_q \leq \nu$, the $s^\text{th}$-largest coordinate of $\thetastar$ in absolute value is no larger than $\nu/s^{1/q}$.
		Thus, $\| \thetastar - \barthetastar \|_1 \leq (p-s)\nu/s^{1/q}$, whence the $\ell_1$-approximation of  \ref{assump:1}(d) holds with $M = \sqrt{n}\nu(1-s/p)/p^{1/q}$.
		
		\item % c
		As in part (b), take $\barthetastar$ to be supported on the largest (in absolute value) $s$ coordinates of $\thetastar$, and take
		$\bar \theta^*_j = \theta^*_j$ for $j$ on this support.
		Call this support $S \subset [p]$.
		As above, the Gaussian width bound for $\barthetastar$ of Assumption \ref{assump:1}(d) holds by part (a).
		Because $\sum_{j=1}^n \min(1,\sqrt{n}|\theta_j^*|/\lambda)\leq s$, 
		there are at most $s$ coordinates of $\thetastar$ with $\sqrt{n}|\theta_j^*|/\lambda \geq 1$.
		In particular, $\sqrt{n}|\theta_j^*|/\lambda < 1 $ for all $j \in S^c$.
		Thus, $\| \thetastar - \barthetastar \|_1 = \| \thetastar_{S^c} - \barthetastar_{S^c} \|_1 = \sum_{j=1}^p \min(1,\sqrt{n}|\theta_j^*|/\lambda) \leq s$.
		Thus, the $\ell_1$-approximation of  \ref{assump:1}(d) holds with $M = \lambda s /p$.

	\end{enumerate}
	The proof of Proposition \ref{ClaimBallToApproxSparse} is complete.
\end{proof}

\subsection{Properties of the design matrix}
\label{SecPropX}

Given every integer $ j \in \{1,\ldots, p\}$, each row of our design matrix is sampled independently from a multivariate Gaussian distribution, namely 
\begin{align*}
\text{for $i = 1,\ldots,n$ }
~~~(X_{i,j}, \bm{X}_{i.-j}) \sim \normal \left(0, \frac{1}{\numobs} \mySigma\right) \qquad
\mySigma = 
	\begin{pmatrix}
		\Sigma_{j,j} & \mySigma_{j, -j}\\
		\mySigma_{-j,j} & \mySigma_{-j,-j}
	\end{pmatrix},
\end{align*}
where the $1/\numobs$ factor is due to the normalization of the design matrix. 
Here $\bm X_{.,j}$ stands for the covariate corresponding to the $j$-th coordinate of $\mytheta$ 
and $\bm X_{.,-j} \in \reals^{(\usedim-1)}$ stands for covariates corresponding to 
rest of $\mytheta.$

Let us further define $X_j^\perp \defn X_j - \mySigma_{j,-j}\mySigma_{-j,-j}^{-1}\bm X_{-j}$ for every 
$j \in \{1,\ldots, p\}$ and the sampled version 
$\xper \defn \bm x_j - \bX_{-j} \mySigma_{-j,-j}^{-1}\mySigma_{-j,j} \in \reals^{\numobs}$.
Then the linear model can be written as 
\begin{align}
	\by = \xper \theta^*_j + \bX_{-j} (\mytheta^*_{-j}+\theta^*_j \mySigma_{-j,-j}^{-1}\mySigma_{-j,j}) + \sigma \bz.
\end{align}

In addition, we state without proof the following straightforward properties.
\begin{itemize}
	\item $X_j | \bm X_{-j} \sim \normal(\mySigma_{j,-j}\mySigma_{-j,-j}^{-1} \bm X_{-j},~
	\frac{1}{\numobs}(\Sigma_{j,j} - \mySigma_{j,-j}\mySigma_{-j,-j}^{-1}\mySigma_{-j,j}))$.

	\item $X_j^\perp | \bm X_{-j} \sim \normal(0,~
	\frac{1}{\numobs}(\Sigma_{j,j} - \mySigma_{j,-j}\mySigma_{-j,-j}^{-1}\mySigma_{-j,j}))$.

	\item $X_j^\perp \sim \normal(0,~
	\frac{1}{\numobs}(\Sigma_{j,j} - \mySigma_{j,-j}\mySigma_{-j,-j}^{-1}\mySigma_{-j,j}))$.

	\item The entries of $\xper$ are i.i.d with distribution 
	$\normal(0,~\frac{1}{\numobs}(\Sigma_{j,j} - \mySigma_{j,-j}\mySigma_{-j,-j}^{-1}\mySigma_{-j,j}))$.
\end{itemize}

\subsection{Proof of Propsition \ref{prop:RE-and-DT}: restricted eigenvalues and the DT phase transition}
\label{sec:RE-to-DT}

\begin{proof}[Proof of Proposition \ref{prop:RE-and-DT}]
	Because $\cuC_{\mathrm{RE}}$ is increasing in $c$, the restriced eigenvalue $\theta(S,c)$ is decreasing in $c$.
	Thus, it suffices to show the result for $c = 1$.
	Note that for any $\bx$ with $\bx_S \in \{-1,1\}^s$ and $\bx_{S^c} = 0$, $F(\bw;\bx,\Ind_p) = \<\bx,\bw\> + \| \bw_{S^c} \|_1 \leq 0$ implies $\| \bw_S \|_1 \geq \| \bw_{S^c} \|_1$,
	whence $\bw \in \cuC_{\mathrm{RE}}(S,1)$.
	Thus, for any such $\bx$,
	\begin{equation}
		\RE(S,1)
			= 
			\inf_{\substack{\mytheta : F(\mytheta; \bx,\Ind_p) \leq 0\\ \| \mytheta \|_2 = 1}}
			\frac1{\sqrt{n}}\| \bX \mytheta \|_2.
	\end{equation}
	Thus, our proof strategy is to find a $\bx$ such that $\inf_{\substack{\mytheta : F(\mytheta; \bx,\Ind_p) \leq 0\\ \| \mytheta \|_2 = 1}}\frac1{\sqrt{n}}\| \bX \mytheta \|_2$ is 0 on some finite interval of sampling rates above $\delta_{\mathrm{DT}}(\bx^*,\mySigma) := \cuG(\bx^*,\mySigma)^2$.

	% \begin{equation}
	% 	\cuC_{\mathrm{RE}}(S,1)
	% 		= 
	% 		\bigcup_{\substack{\bx_S \in \{-1,1\}^S\\ \bx_{S^c} =0}}
	% 		\Big\{
	% 			\bw :
	% 			F(\bw; \bx , \Ind_p ) \leq 0
	% 		\Big\}.
	% \end{equation} 
	% Indeed, for any such $\bx$, $F(\bw;\bx,\Ind_p) = \<\bx,\bw\> + \| \bw_{S^c} \|_1 \leq 0$ implies $\| \bw_S \|_1 \geq \| \bw_{S^c} \|_1$, 
	% whence $\bw \in \cuC_{\mathrm{RE}}(S,1)$.
	% Conversely,
	% if $\bw \in \cuC_{\mathrm{RE}}(S,1)$,
	% then taking $\bx_S = -\sign(\bw_S)$ and $\bx_{S^c} = 0$,
	% we have $F(\bw;\bx,\Ind_p) \leq 0$. 
	% Therefore, 
	% \begin{equation}
	% 	\theta(S,1)
	% 		= 
	% 		\min_{\substack{\bx_S \in \{-1,1\}^S\\ \bx_{S^c} =0}}
	% 		\inf_{\substack{\mytheta : F(\mytheta; \bx,\Ind_p) \leq 0\\ \| \mytheta \|_2 = 1}}
	% 		\frac1{\sqrt{n}}\| \bX \mytheta \|_2.
	% \end{equation}
	% By Proposition \ref{prop:std-width}, $\cuG_{\mathrm{std}}(\bx,\mySigma)^2 \rightarrow \delta_{\mathrm{DT}}$ as $p \rightarrow \infty$.
	% By \cite[Proposition 2.4, Fact 2.8, and Theorem II]{amelunxen2014living} (which locate the Donoho-Tanner phase transition based on the standard Gaussian width $\cuG_{\mathrm{std}}(\bx,\mySigma)$),
	% for $\delta < \cuG_{\mathrm{std}}(\bx,\mySigma)$, 
	% we have that $\inf_{\substack{\mytheta : F(\mytheta; \bx,\Ind_p) \leq 0\\ \| \mytheta \|_2 = 1}}
	% 		\frac1{\sqrt{n}}\| \bX \mytheta \|_2 = 0$ with probability going to 1.

	We will simplify the optimization Eq.~\eqref{EqnGaussianWidth} defining the Gaussian width for any $\bx$ whose first $s$ coordinates consist of $s/2$ copies of some $\bx^{(1)} \in \{-1,1\}^2$ and whose remaining coordinates are 0.
	First consider any $\bv \in L^2$ (recall that this is a random variable in $\reals^p$).
	For $j = 1,\ldots,s/2$,
	let $\Pi_j$ be the marginal joint distribution of $(\bv_{[2j-1,2j]},\bg_{[2j-1,2j]})$, and let $\Pi^{(1)} = \frac1{s/2} \sum_{j=1}^{s/2} \Pi_j$ be the mixture of the joint distributions.
	For $j = s+1,\ldots,p$, let $\tilde \Pi_j$ be the joint distribution of $(v_j,g_j)$, and let $\Pi^{(2)} = \frac1{p-s} \sum_{j=s+1}^p \tilde \Pi_j$ be the mixture of these joint distributions.
	Construct a new random vector $\bv' \in L^2$ such that $(\bv',\bg)$ has joint distribution $(\Pi^{(1)})^{\otimes s/2} \otimes (\Pi^{(2)})^{\otimes (p-s)}$.
	The value of the objective and constraints in Eq.~\eqref{EqnGaussianWidth} are not changed by this replacement, so that we can compute the Gaussian width by optimizing over the distributions $\Pi^{(1)}$ and $\Pi^{(2)}$.
	Writing in this way, we have 
	\begin{equation}
	\begin{aligned}
		&\cuG(\bx,\mySigma) =
		\sup\Big\{
			\E\big[(\eps/2) \< \bg^{(1)}, \bv^{(1)} \> + (1-\eps)GV\big]
			:
		\\
		&\qquad\qquad
			\E\big[(\eps/2) \| \bv^{(1)} \|_2^2 + (1-\eps)V^2\big] \leq 1,
			\;\;
			\E\big[(\eps/2) \<\bK^{-1/2}\bx^{(1)},\bv^{(1)}\> + (1-\eps)|V|\big] \leq 0
		\Big\},
	\end{aligned}
	\end{equation}
	where $s/p=\epsilon$,  and $\bg^{(1)} \sim \normal(0,\Ind_2)$, $G \sim \normal(0,1)$, and $\bv^{(1)}$ and $V$ are random variables in $\reals^2$ and $\reals$, respectively.
	Because $\bx^*$ is of the form under consideration, we have shown that $\delta_{\mathrm{DT}}(\bx^*,\mySigma)$ does not depend on $p$.
	Let $\tilde \bg^{(1)}$ and $\tilde \bv^{(1)}$ be the vectors $\bv^{(1)}$ and $\bg^{(1)}$ written in a basis consisting of the unit vector $\bK^{-1/2}\bx^{(1)}/\|\bK^{-1/2}\bx^{(1)}\|_2$ and the unit vector orthogonal to it.
	In this basis, the optimization above is 
	\begin{equation}
	\begin{aligned}
		&\sup\Big\{
			\E\big[(\eps/2) \< \tilde \bg^{(1)}, \tilde \bv^{(1)} \> + (1-\eps)GV\big]
			:
		\\
		&\qquad\qquad
			\E\big[(\eps/2) \| \tilde \bv^{(1)} \|_2^2 + (1-\eps)V^2\big] \leq 1,
			\;\;
			\E\big[(\eps/2) \xi \tilde v^{(1)}_1 + (1-\eps)|V|\big] \leq 0
		\Big\},
	\end{aligned}
	\end{equation}
	where $\xi = \| \bK^{-1/2}\bx^{(1)} \|_2$.
	Because $\tilde \bg^{(1)} \sim \normal(0,\Ind_2)$,
	the value of the optimization problem depends on $\bK$ and $\bx^{(1)}$ only via $\xi$.

	Moreover,
	the value of the optimization is strictly increasing in $\xi$,
	as we now show.
	First, observe that $(\tilde \bv^{(1)},V) = (\bzero_2,0)$ is not a solution to the optimization problem. 
	Indeed, a better solution which satisfies the constraints is $(\tilde \bv^{(1)},V) = (\tilde \bg^{(1)},0)$.
	Moreover, the second constraint must bind at the solution.
	Indeed, the unique solution to the optimization when the second constraint is removed is $(\tilde \bv^{(1)},V) = (\tilde \bg^{(1)},G)$ almost surely, which violates the second constraint.
	Now, consider $\xi < \xi'$. 
	Because $(\tilde \bv^{(1)},V)$ is not identically 0, the second constraint implies that the solution $(\tilde \bv^{(1)},V)$ at $\xi$ must satisfy $\E[\tilde v^{(1)}_1] < 0$.
	The second constraint under $\xi'$ is strictly feasible when evaluated at this solution.
	Because the objective and the first constraint do not depend on $\xi$, 
	and the second constraint must bind at the solution,
	the value of the optimization at $\xi'$ must be strictly larger than at $\xi$,
	as claimed. 

	For $\rho > 0$ and taking $\bx^{*(1)} = (1,1)^\top$ and $\bx^{(1)} = (1,-1)^\top$  
	by the argument above, both $\delta_{\mathrm{DT}}(\bx^{*(1)},\bK) := \delta_{\mathrm{DT}}(\bx^*,\mySigma) = \cuG(\bx^*,\mySigma)^2$ and $\delta_{\mathrm{DT}}(\bx^{(1)},\bK) := \delta_{\mathrm{DT}}(\bx,\mySigma) = \cuG(\bx,\mySigma)^2$ do not depend on $p$.
	Direct evaluation yields $\| \bK^{-1/2} \bx^{*(1)} \|_2 < \| \bK^{-1/2} \bx^{(1)} \|_2$, which implies 
	$\delta_{\mathrm{DT}}(\bx^{*(1)},\bK) < \delta_{\mathrm{DT}}(\bx^{(1)},\bK). $
	By Proposition \ref{prop:std-width}, $\cuG_{\mathrm{std}}(\bx,\mySigma)^2 \rightarrow \delta_{\mathrm{DT}}(\bx^{(1)},\bK)$ as $p \rightarrow \infty$.
	By \cite[Proposition 2.4, Fact 2.8, and Theorem II]{amelunxen2014living} (which locate the Donoho-Tanner phase transition based on the standard Gaussian width $\cuG_{\mathrm{std}}(\bx,\mySigma)$),
	for $\delta < \cuG_{\mathrm{std}}(\bx^{(1)},\bK)$, 
	we have that $\inf_{\substack{\mytheta : F(\mytheta; \bx,\Ind_p) \leq 0\\ \| \mytheta \|_2 = 1}} \frac1{\sqrt{n}}\| \bX \mytheta \|_2 = 0$ with probability going to 1.
	Thus, the Proposition holds by taking $\Delta = \delta_{\mathrm{DT}}(\bx^{(1)},\bK) - \delta_{\mathrm{DT}}(\bx^{*(1)},\bK)$.
\end{proof}

\section{Additional Simulations}
\label{SecAdditionalSims}

\subsection{Hidden Markov model specification}

In the hidden Markov model, the covariates $x_{ij}$ are conditionally independent given latent states $s_{ij}$ which are generated according to a Markov chain. 
In particular, the distribution satisfies $\mprob(s_{i(j+1)} ~|~ \{s_{i\ell}\}_{\ell \leq j}) = \mprob(s_{i(j+1)} ~|~ s_{ij})$. The latent states (values for $s_{ij})$ and observed values (values of $x_{ij}$) in the hidden Markov model that we consider here, take values in $\{1,2,3,4,5\}$.
Both the transition and emission probabilities are given by a symmetric random walk with reflection at the boundary; that is, 
$$
\mprob(s_{i(j+1)} = a | s_{ij} = b)
	= 
	\mprob(x_{ij} = a | s_{ij} = b) 
	= 
	\begin{cases}
		1/2 \quad &  b \in \{2,3,4\} \text{ and } |a-b| = 1,\\
		1 \quad  & b \in\{1,2\} \text{ and } |a-b| = 1,\\
		0 \quad & \text{otherwise.}
	\end{cases}
$$
We initialize this Markov chain (i.e., $s_{i1}$) from its stationary distribution.
In this case, the covariance of $x_{ij}$ and $x_{ij'}$ is only a function of $|j-j'|$, as plotted below.
We see that covariates which are within approximately distance $10$ of each other have non-trivial correlation.
\medskip
\centerline{\includegraphics[width=.5\textwidth]{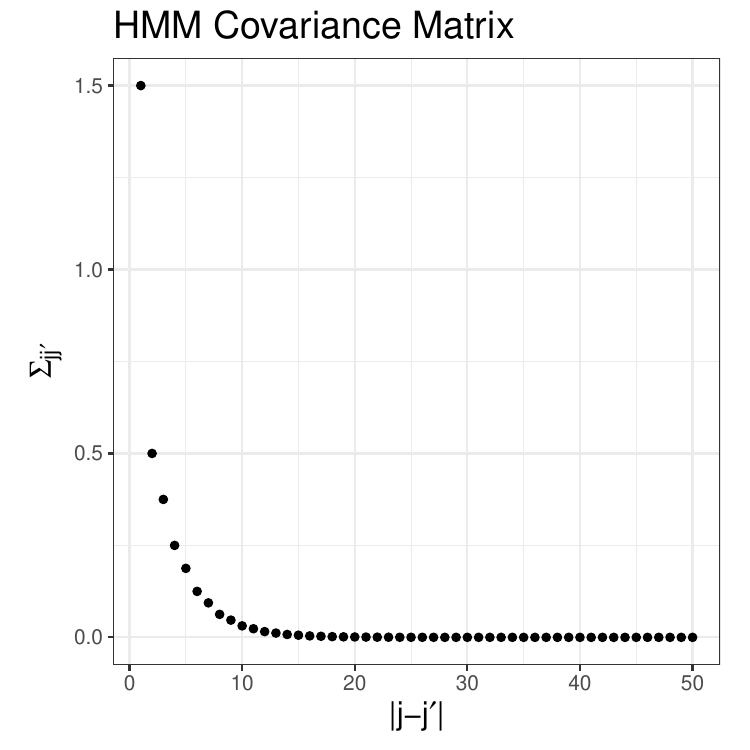}}

\subsection{Debiasing under Gaussian AR1 models}
\label{sec:sim-db-AR}
Here we collect simulations which repeat those in Figure \ref{FigQQplots_and_histograms} at different model parameters.
These simulations demonstrate the success of debiasing at many settings of the model parameters.
In particular, we run the simulations varying the correlation parameter $\rho = 0,.5,.8$, the sample size $n = 500, 750$, and the sparsity $s = 20,100,200$.
We show the legend for the first two plots. The legend for the remaining plots is the same.

%%%% rho = 0
% \centerline{\includegraphics[width=.78\textwidth]{plots/qqplot_p100_s020_rho0_n25_lam4_sig1_mu25.jpg}}
% \centerline{\includegraphics[width=.8\textwidth]{plots/hist_p100_s020_rho0_n25_lam4_sig1_mu25.jpg}}

\centerline{\includegraphics[width=.78\textwidth]{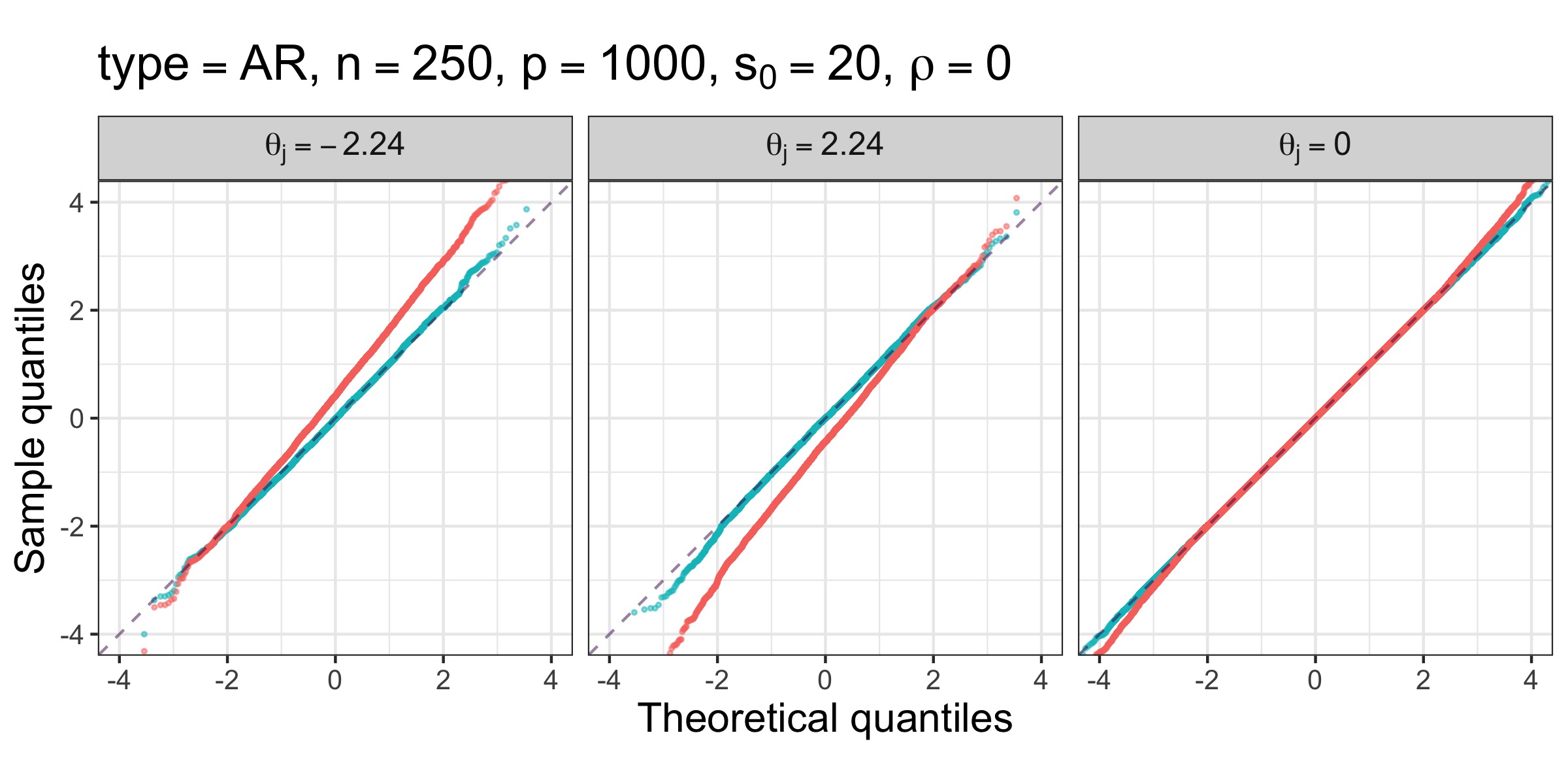}}
\centerline{\includegraphics[width=.8\textwidth]{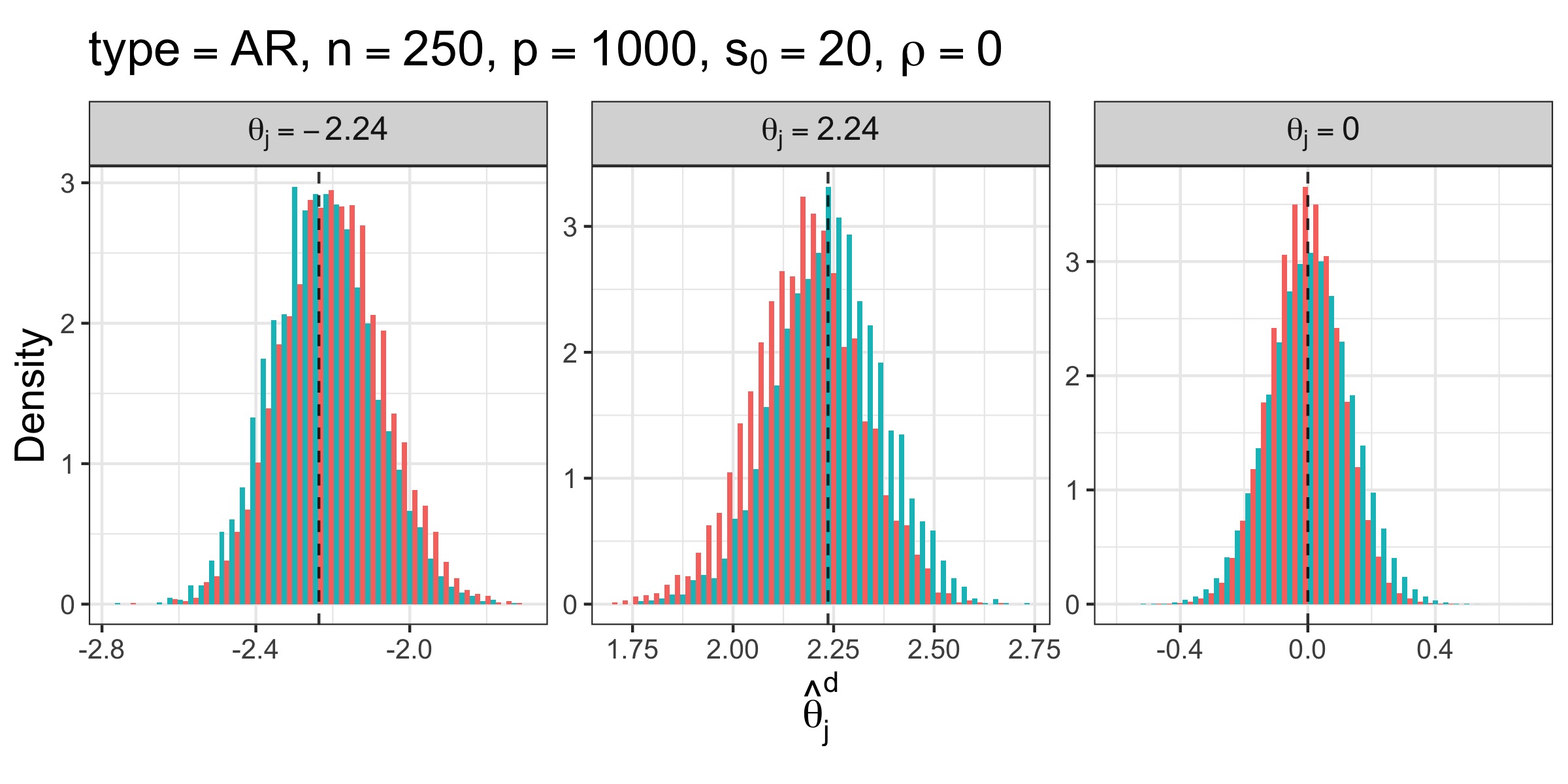}}

\centerline{\includegraphics[width=.78\textwidth]{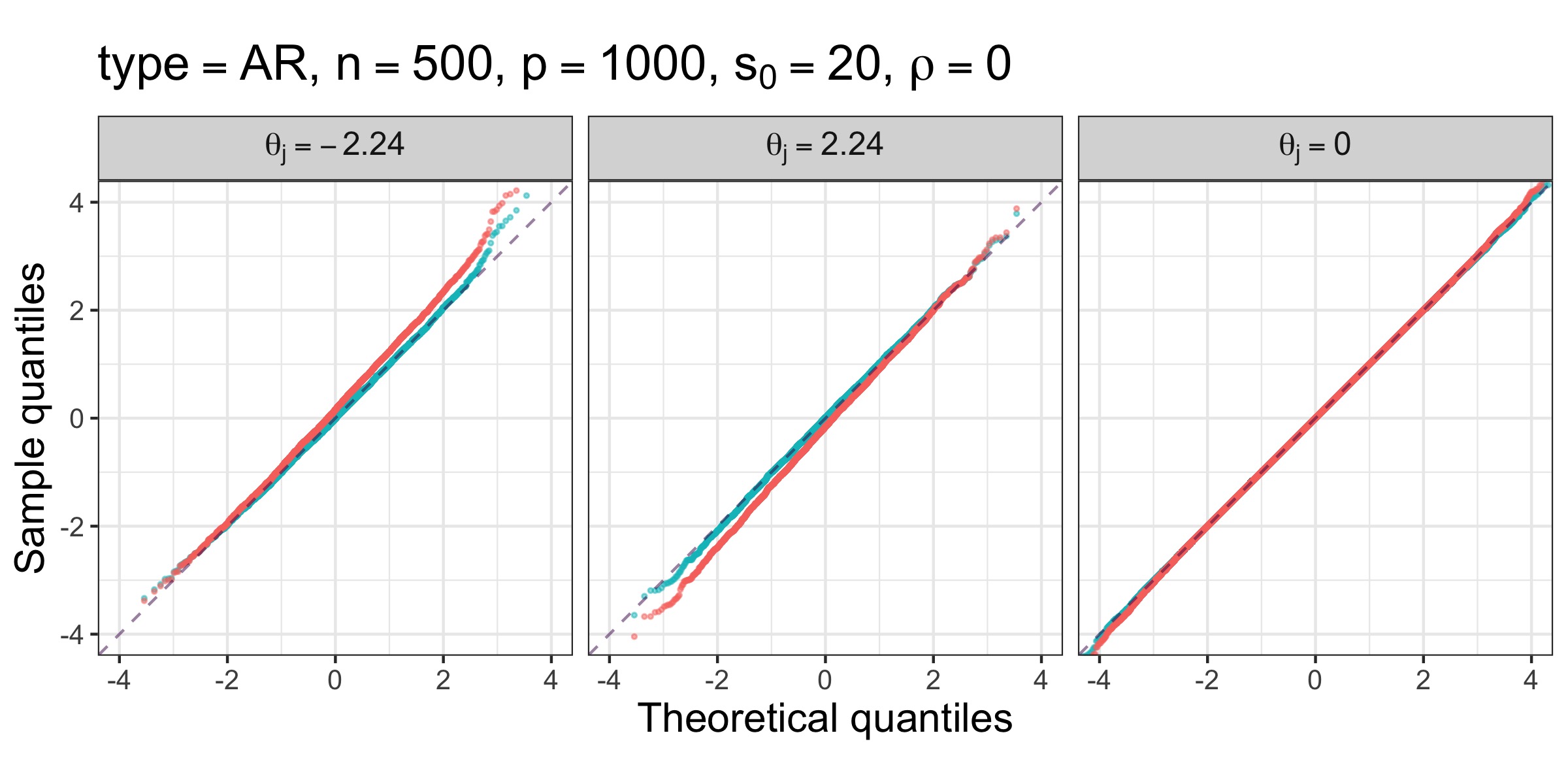}}
\centerline{\includegraphics[width=.78\textwidth]{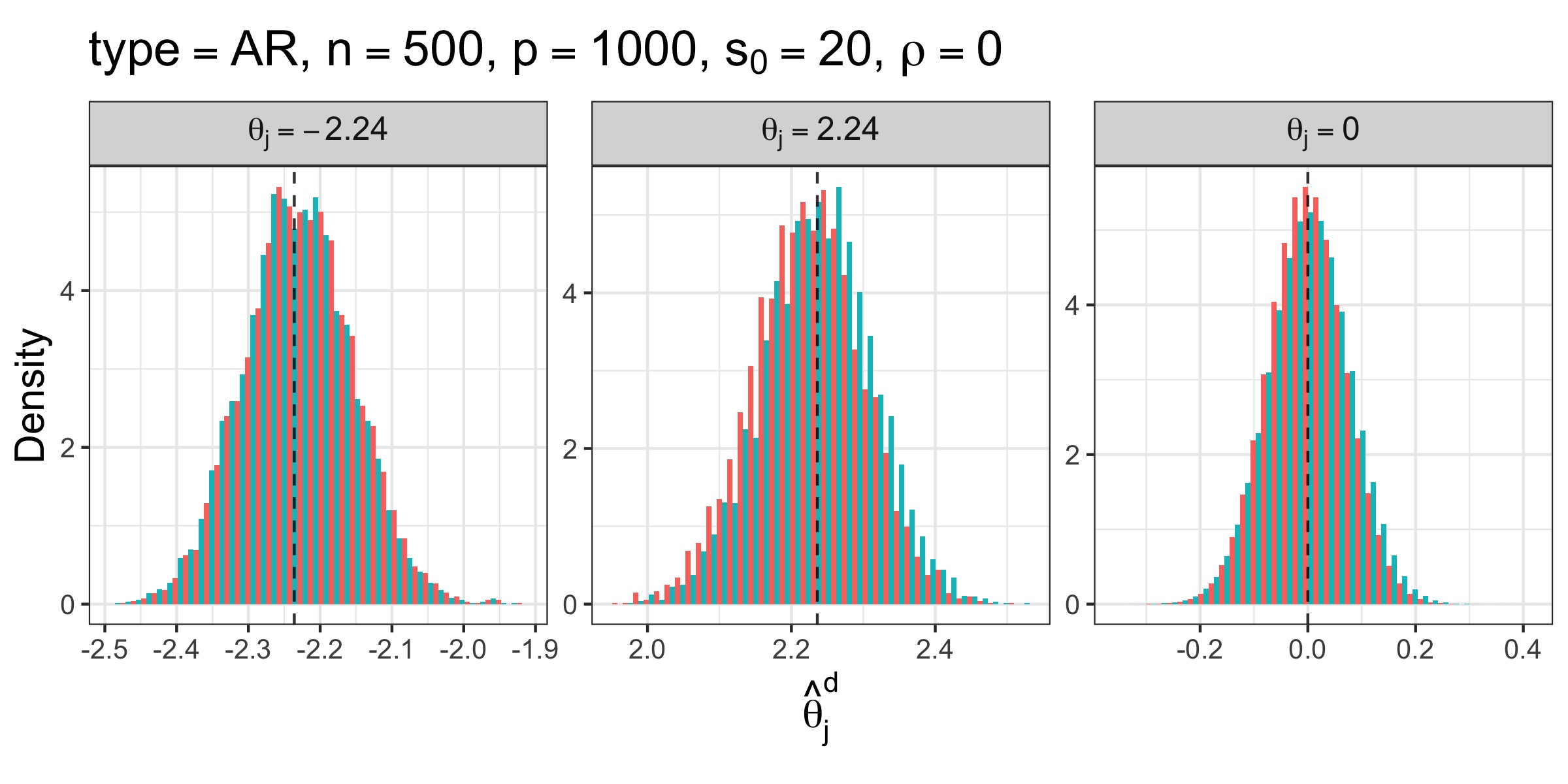}}

\centerline{\includegraphics[width=.78\textwidth]{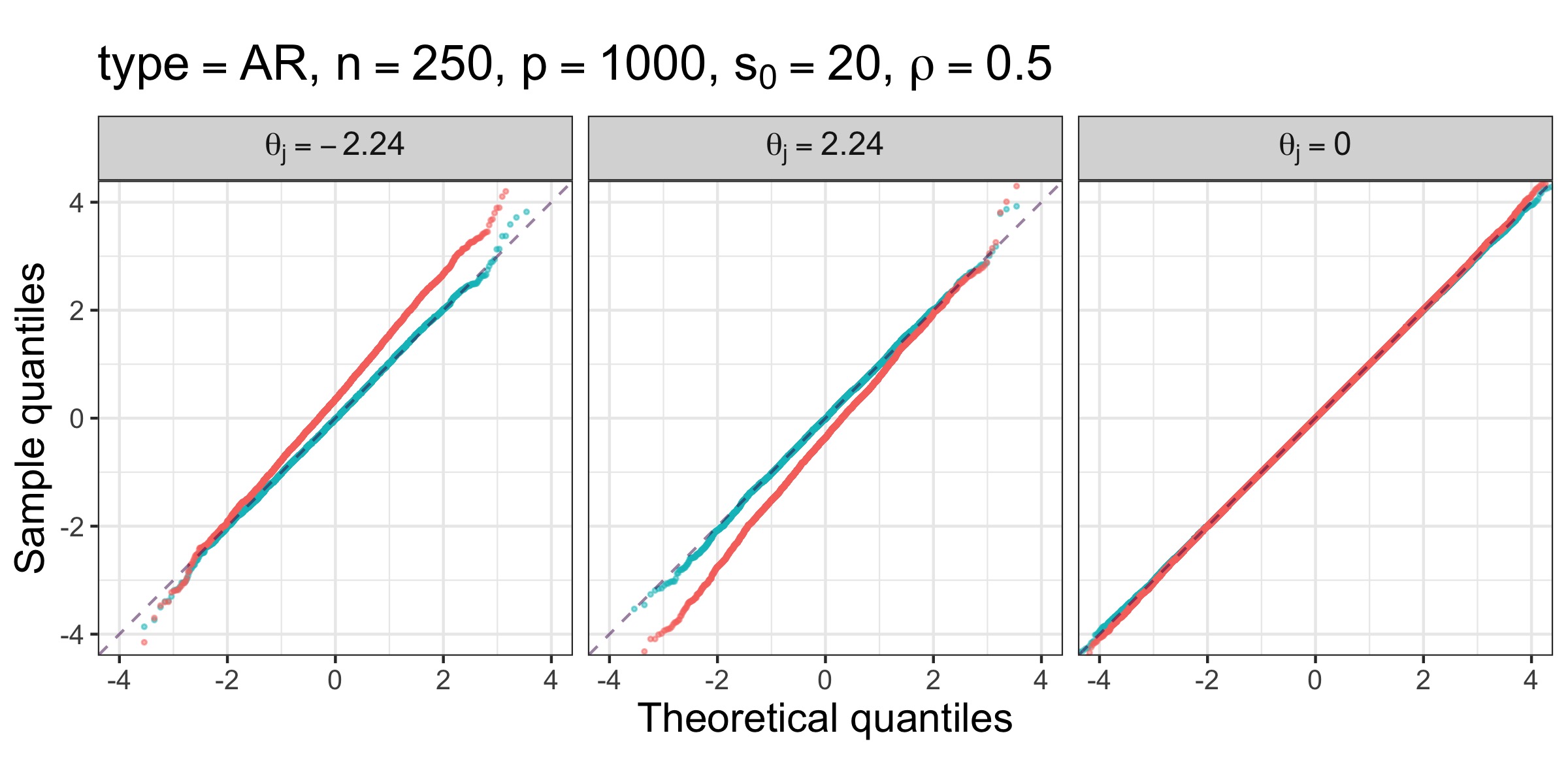}}
\centerline{\includegraphics[width=.78\textwidth]{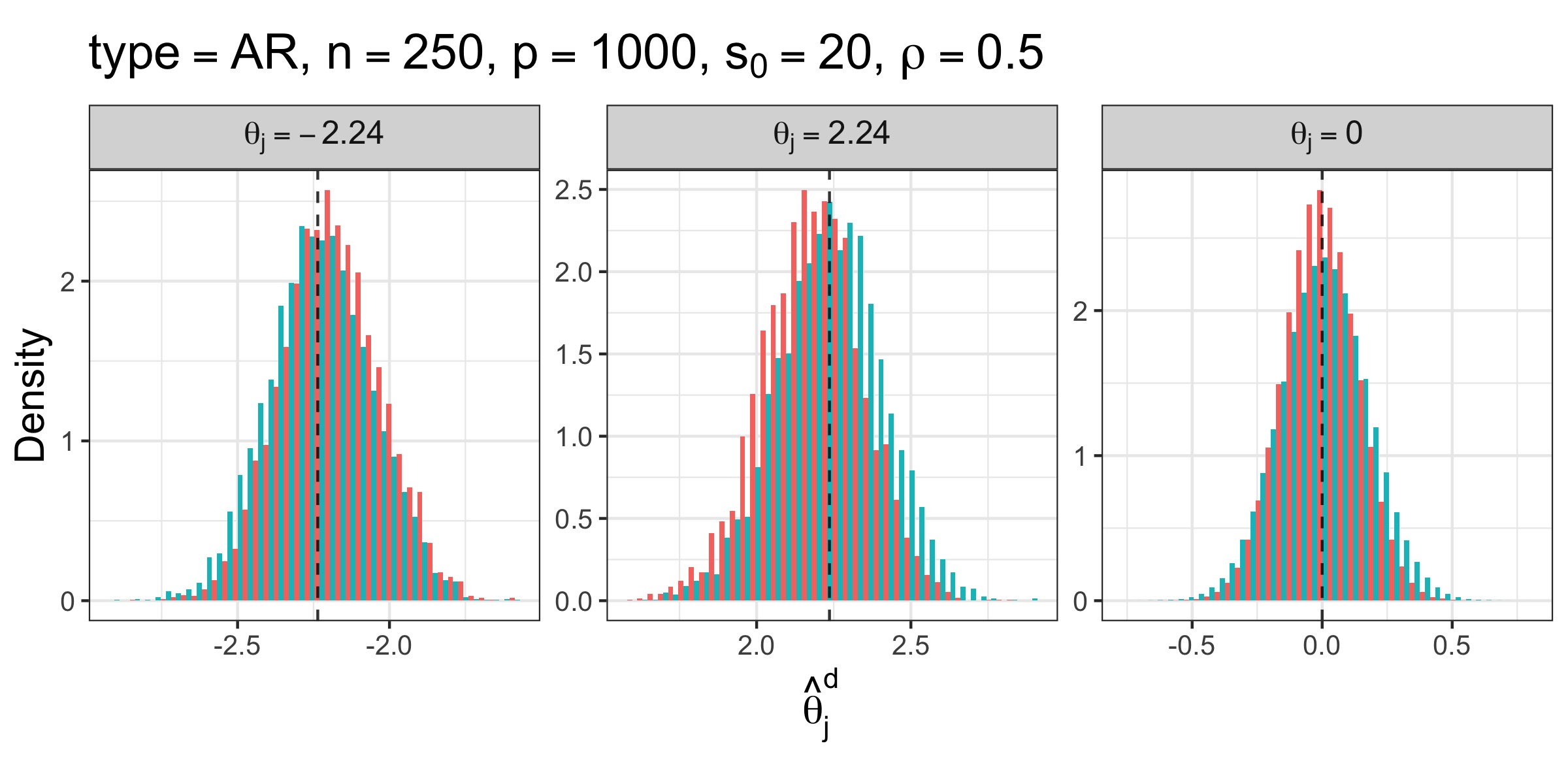}}

\centerline{\includegraphics[width=.78\textwidth]{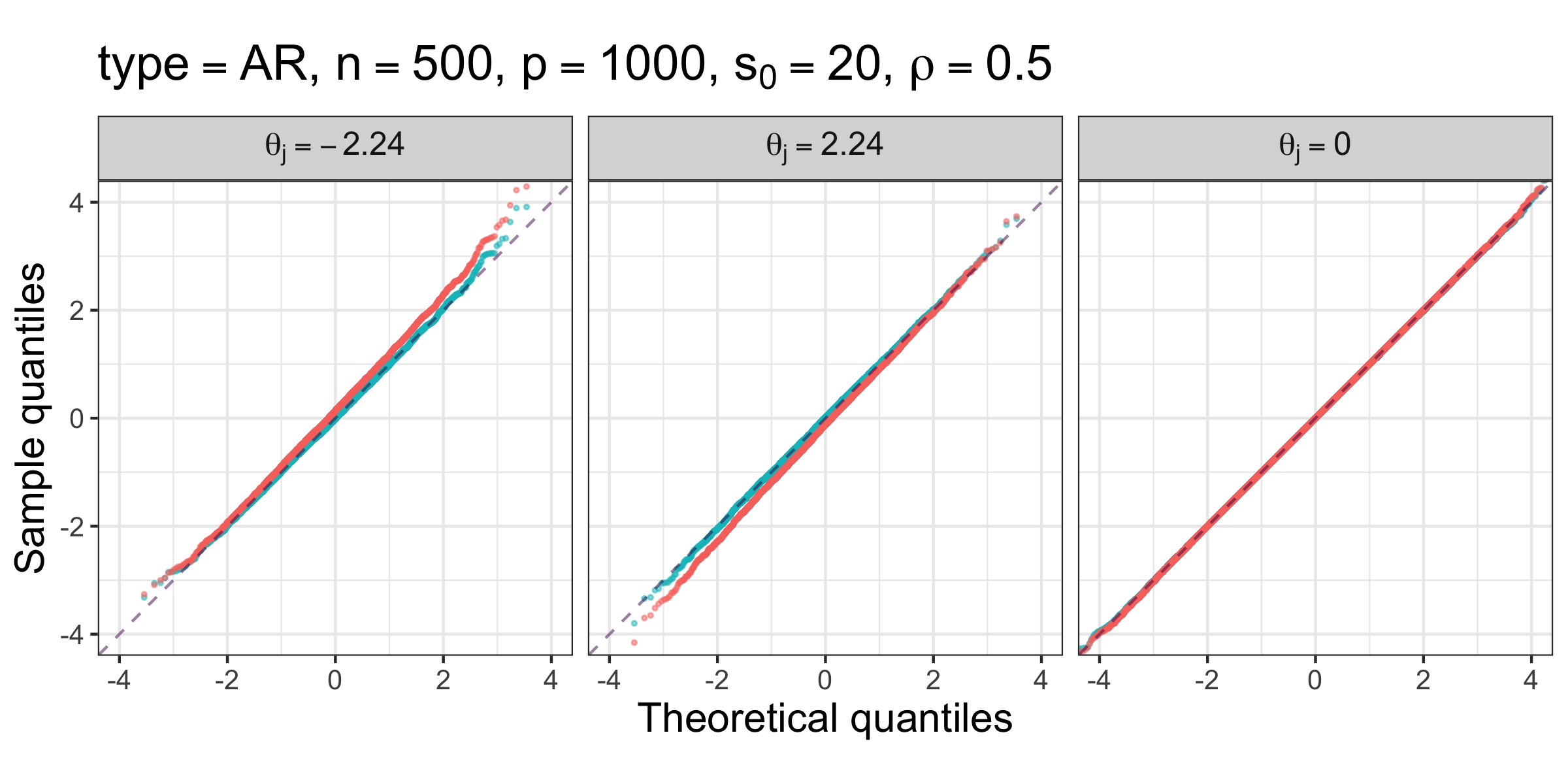}}
\centerline{\includegraphics[width=.78\textwidth]{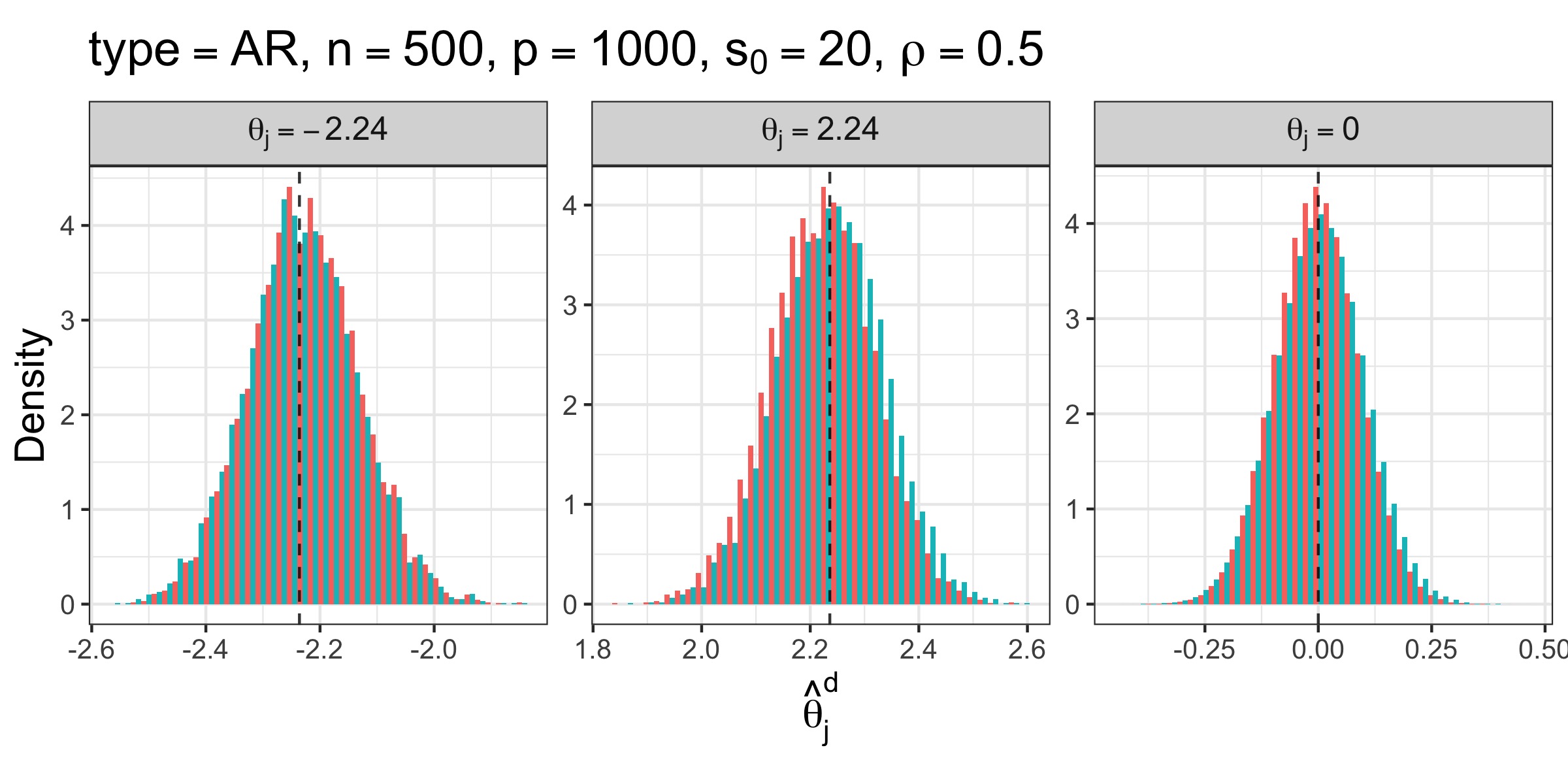}}

\centerline{\includegraphics[width=.78\textwidth]{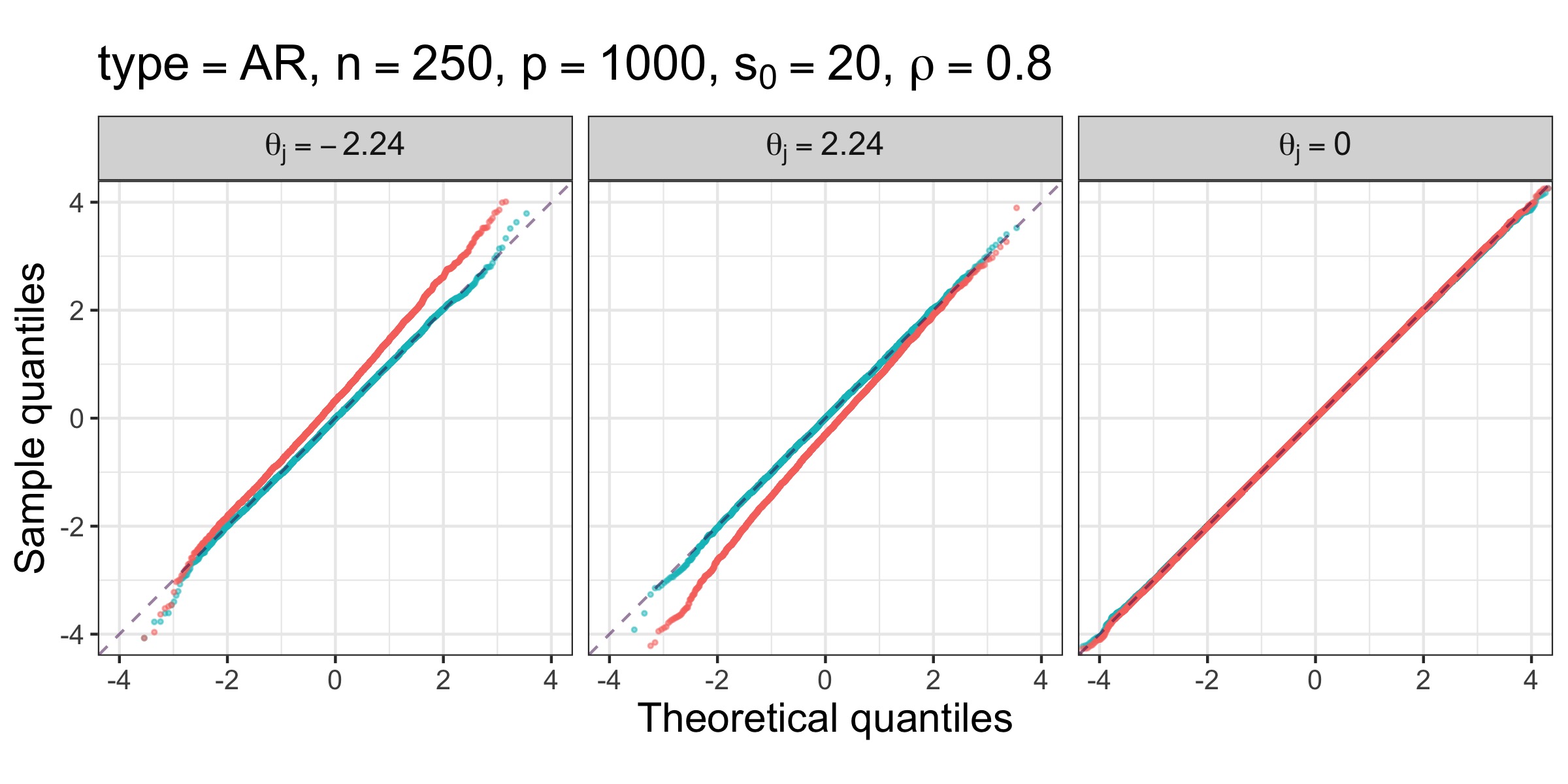}}
\centerline{\includegraphics[width=.78\textwidth]{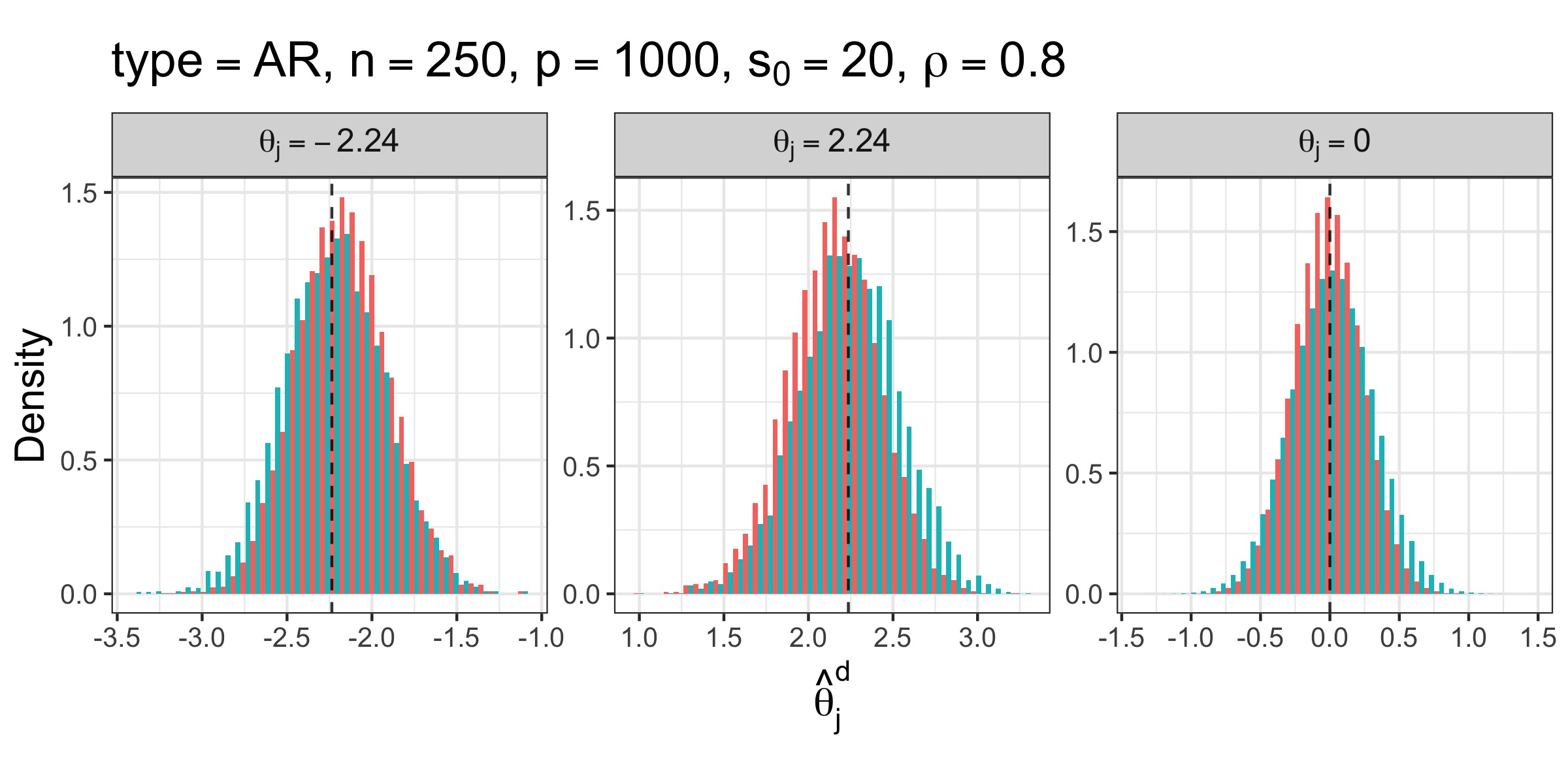}}

\centerline{\includegraphics[width=.78\textwidth]{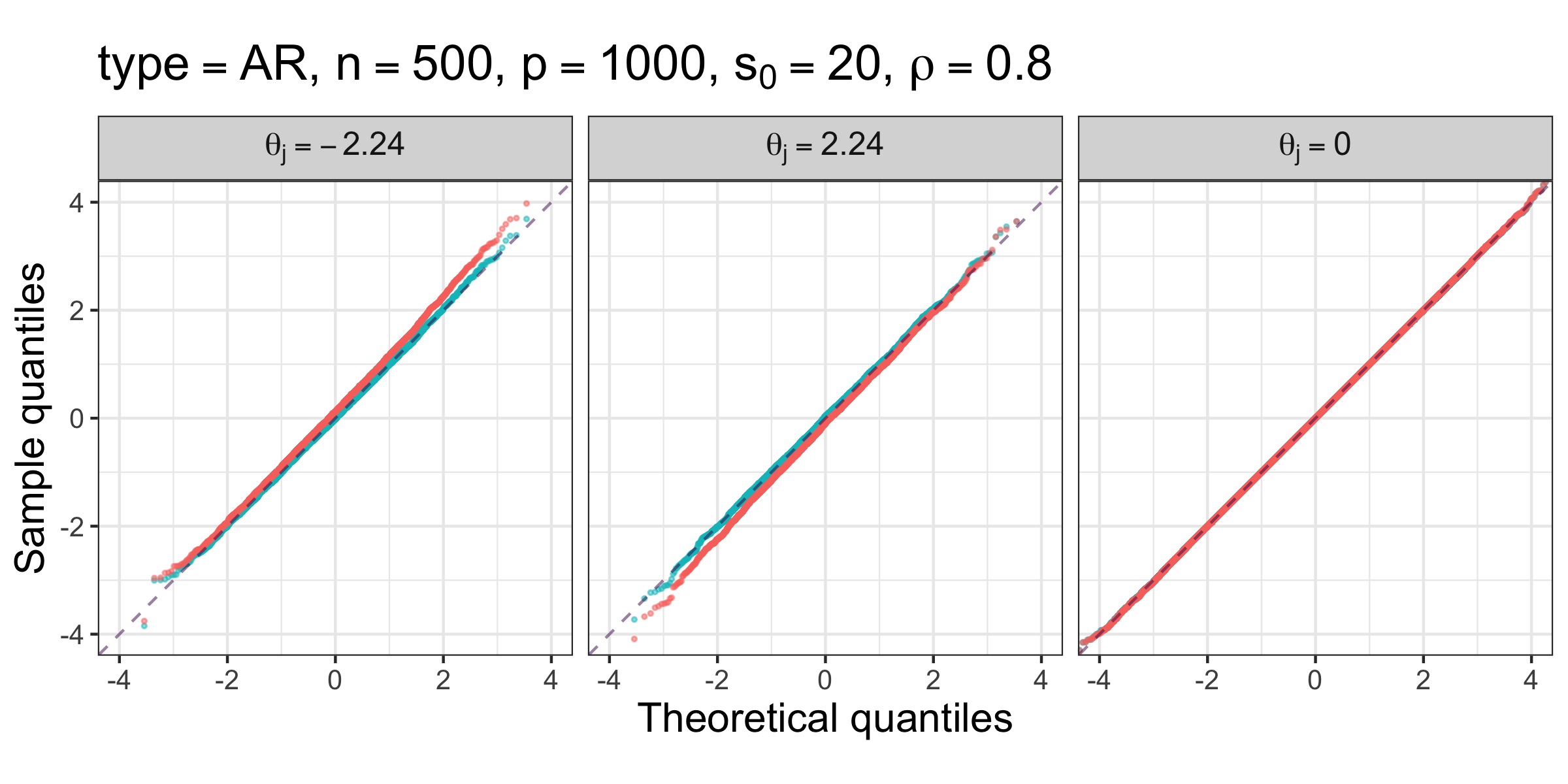}}
\centerline{\includegraphics[width=.78\textwidth]{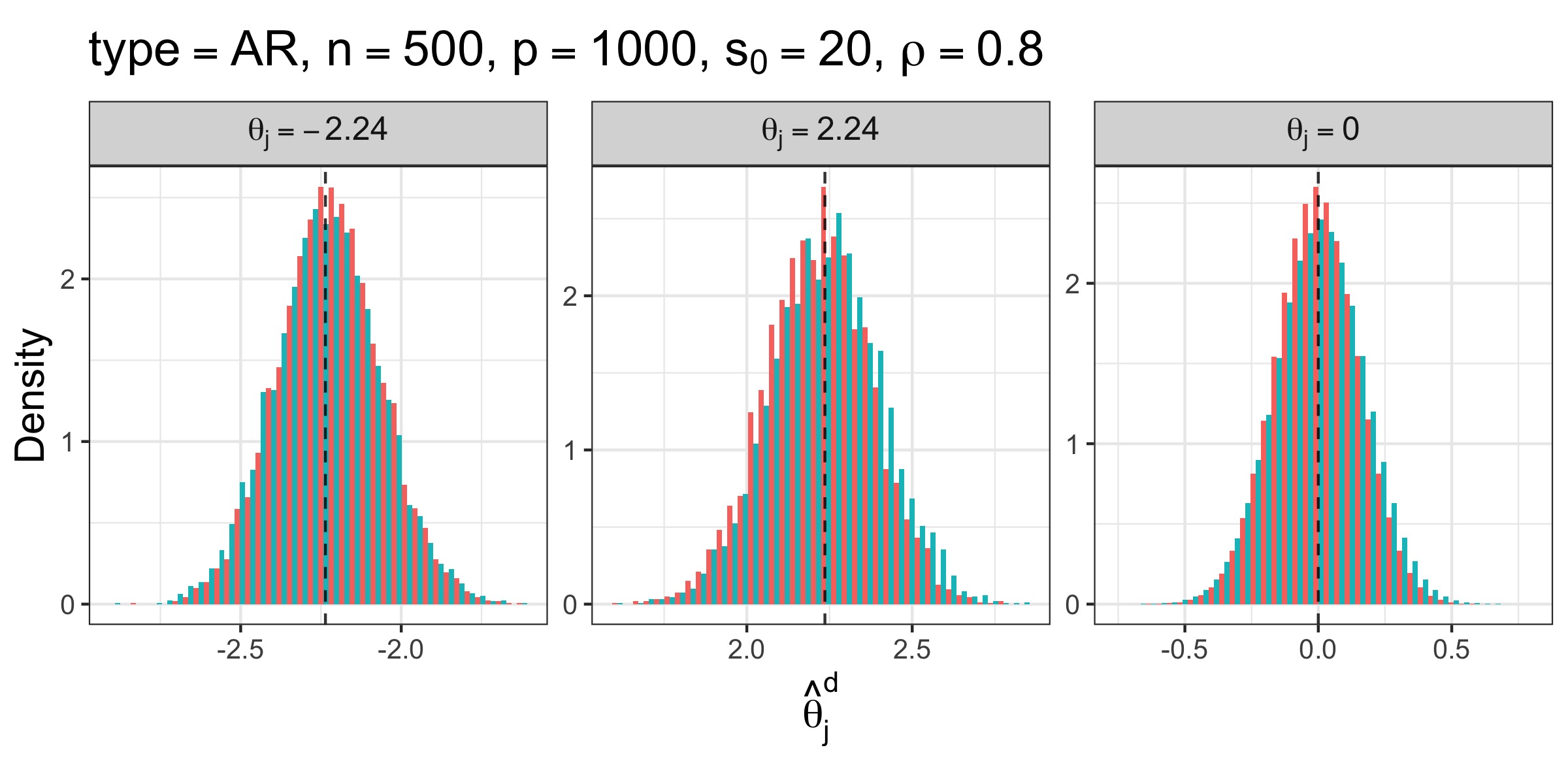}}

% ----- 

\centerline{\includegraphics[width=.78\textwidth]{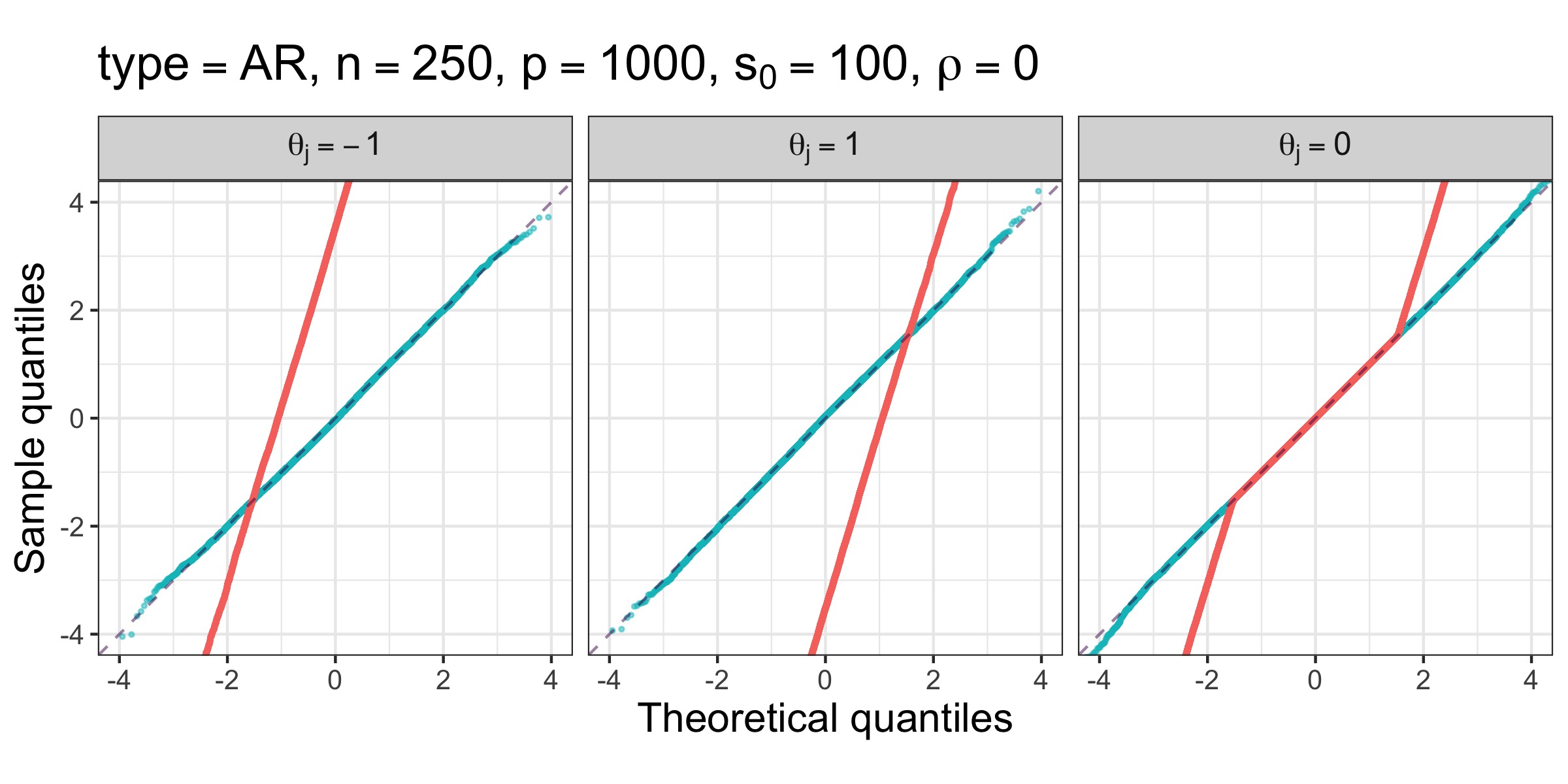}}
\centerline{\includegraphics[width=.78\textwidth]{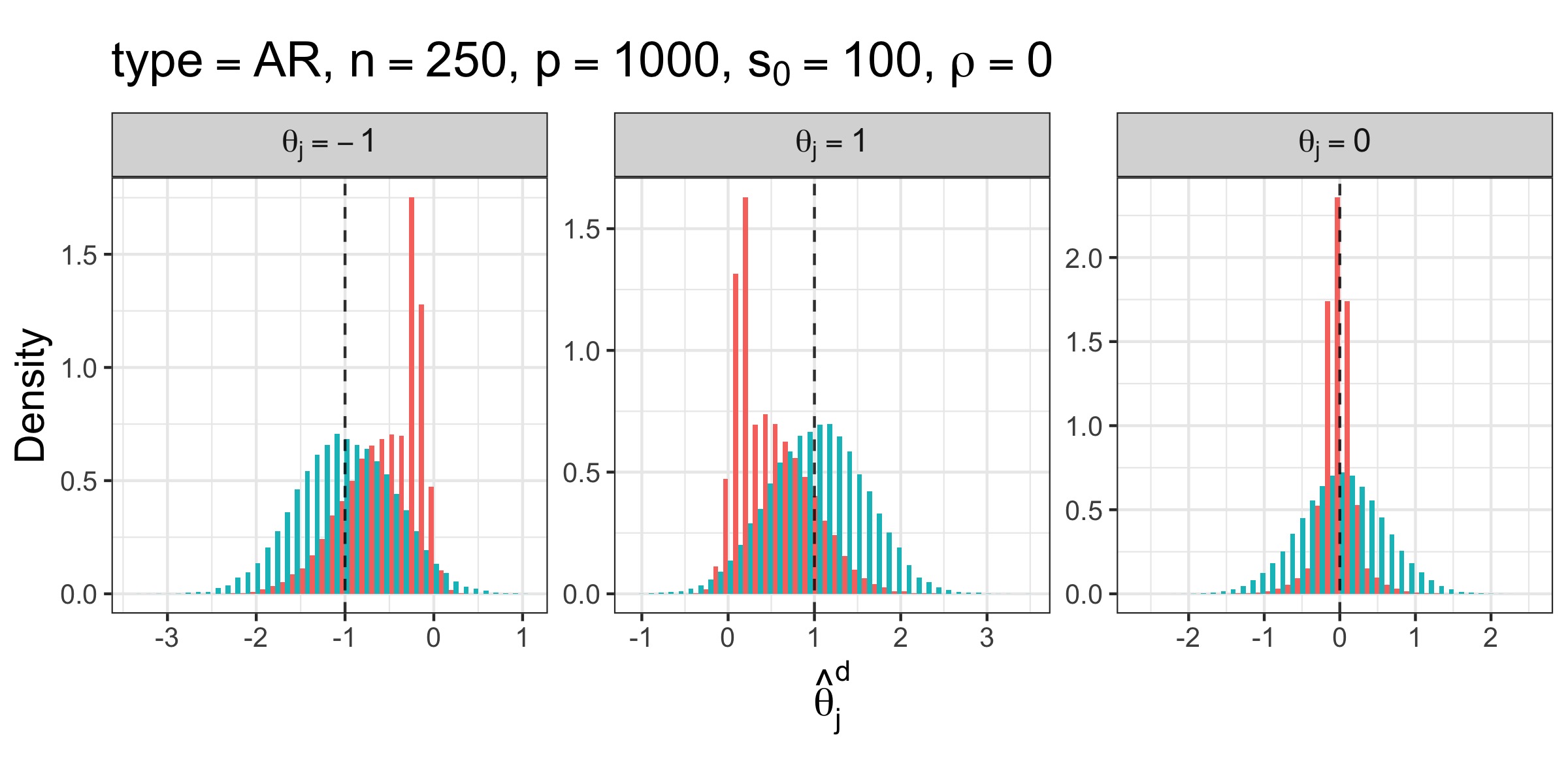}}

\centerline{\includegraphics[width=.78\textwidth]{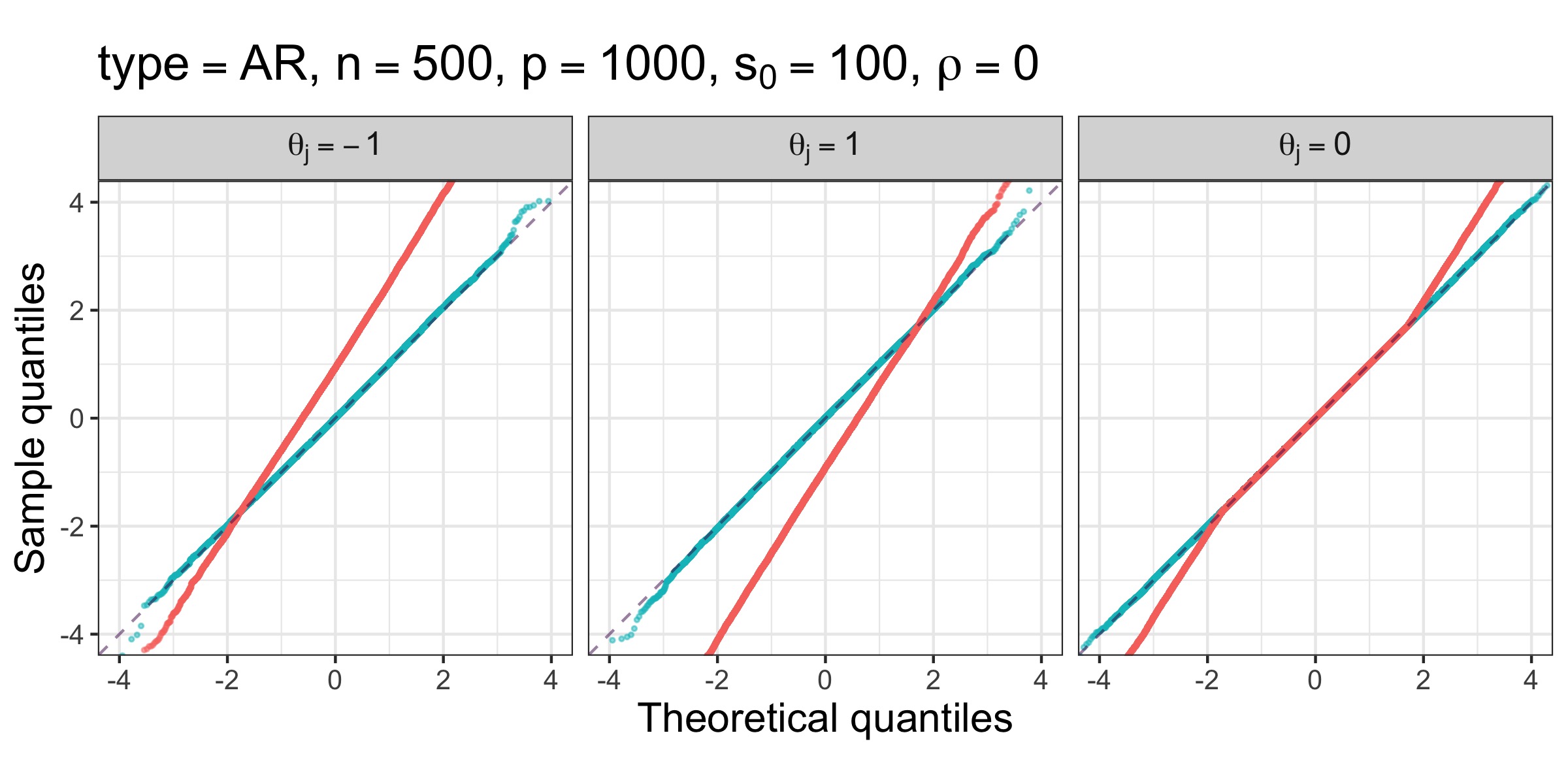}}
\centerline{\includegraphics[width=.78\textwidth]{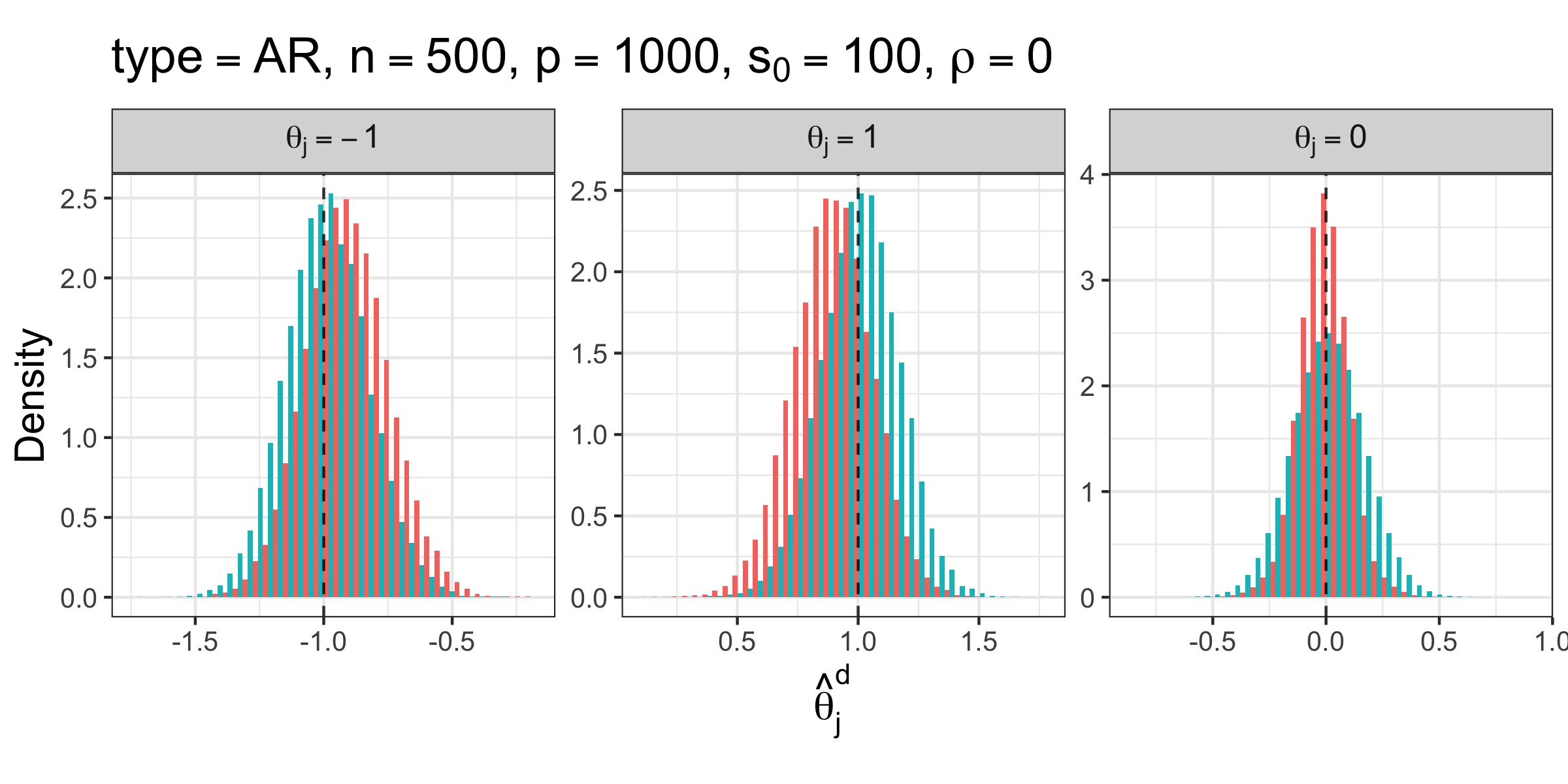}}

\centerline{\includegraphics[width=.78\textwidth]{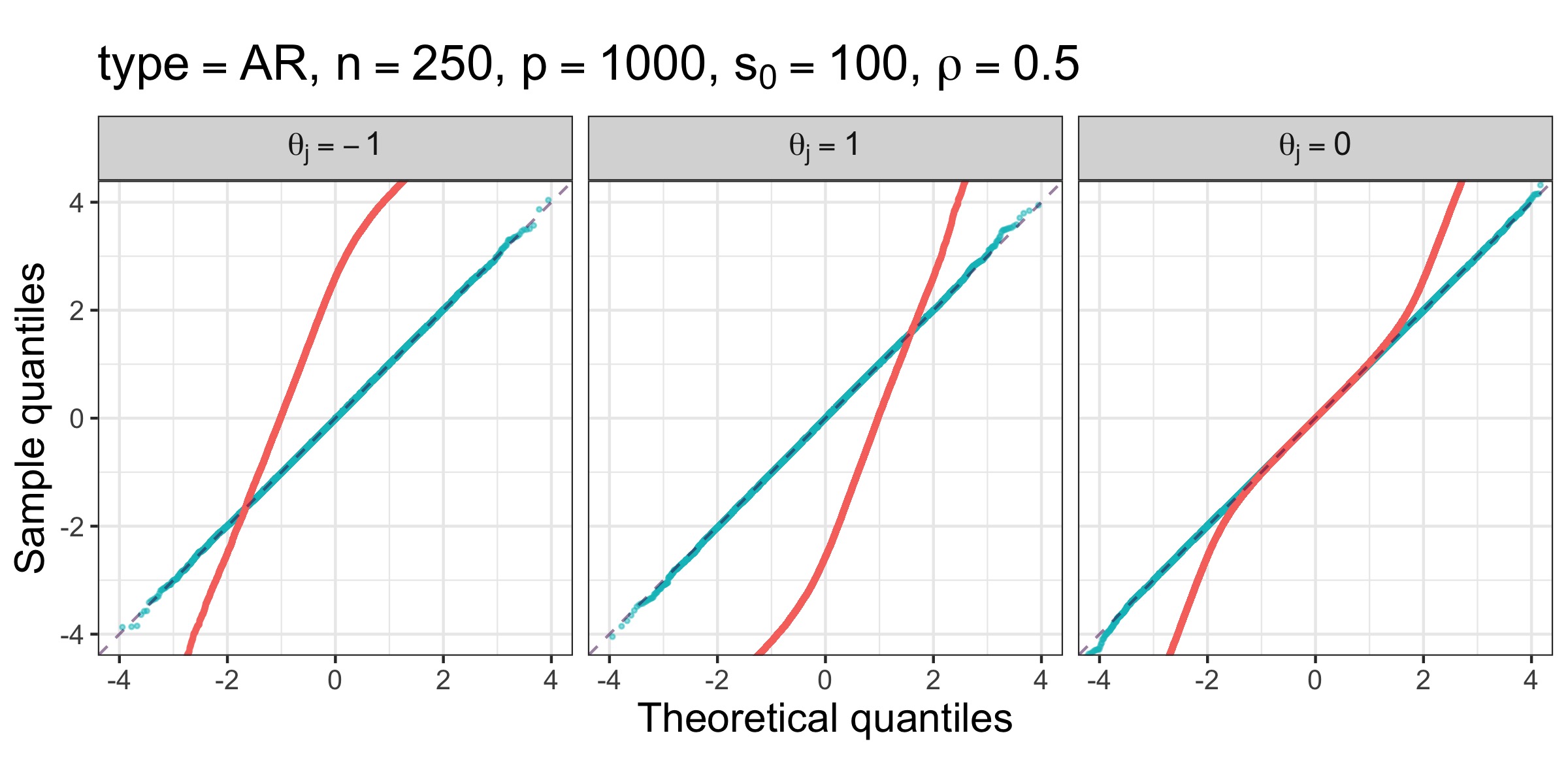}}
\centerline{\includegraphics[width=.78\textwidth]{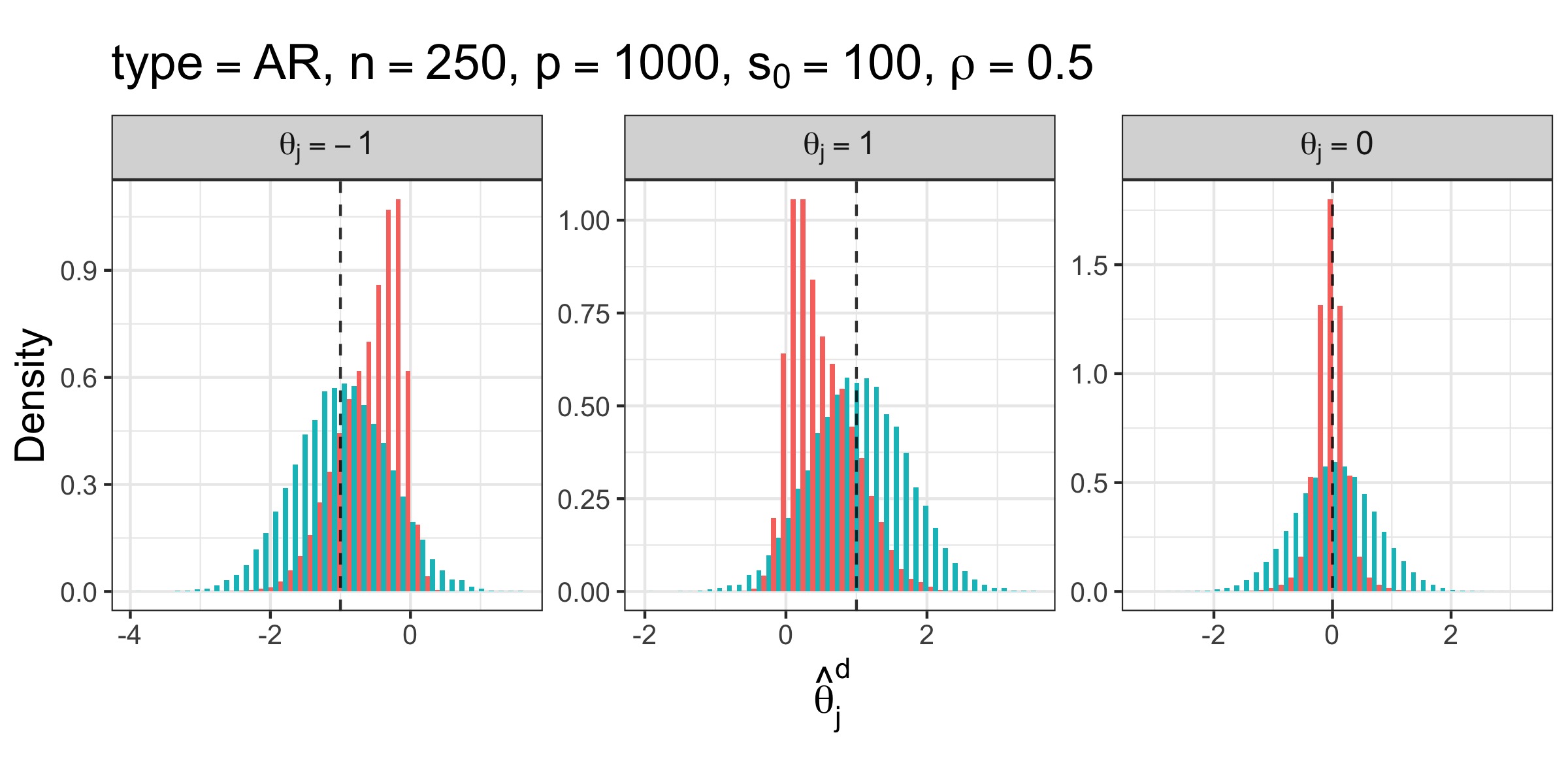}}

\centerline{\includegraphics[width=.78\textwidth]{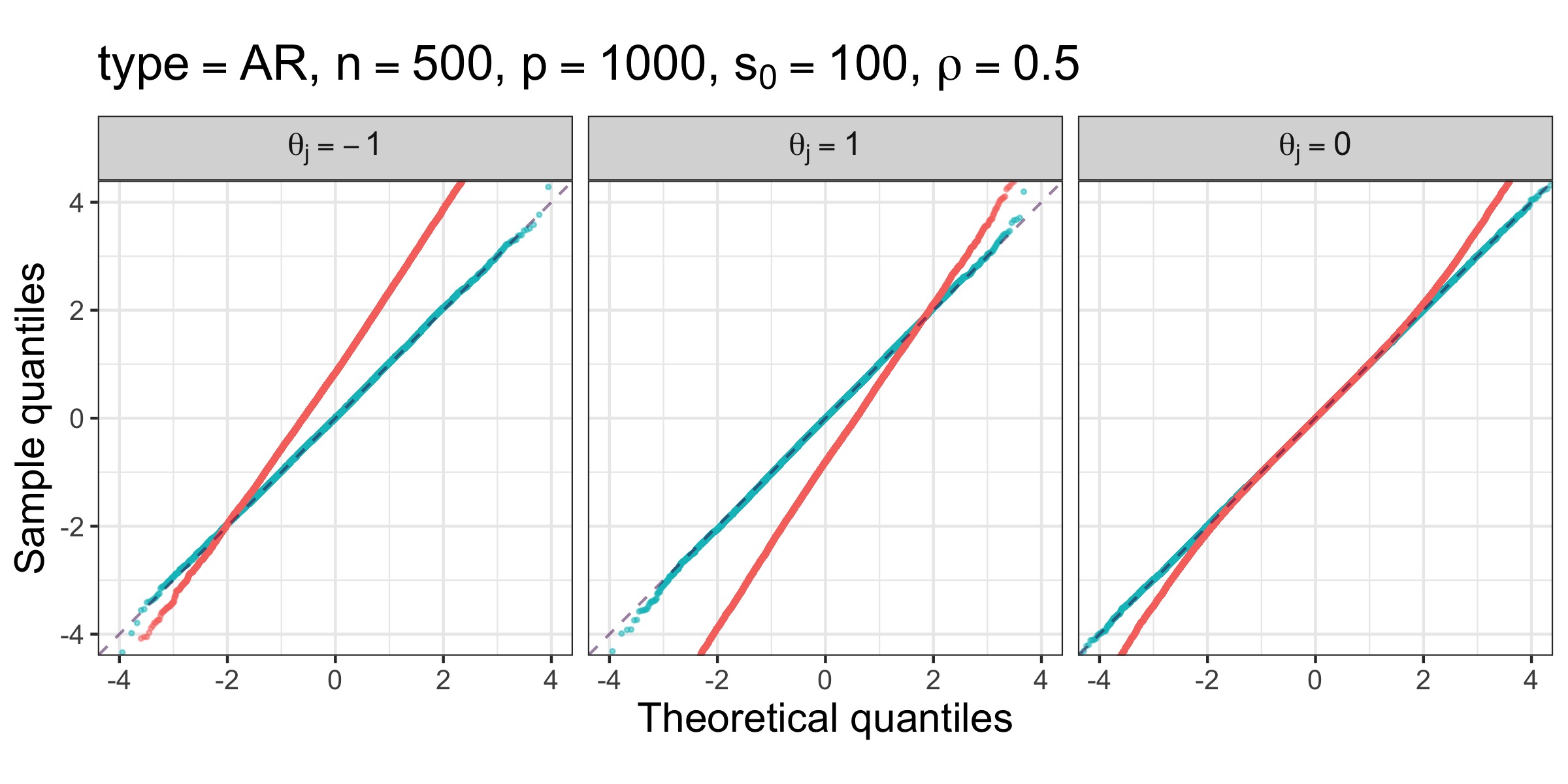}}
\centerline{\includegraphics[width=.78\textwidth]{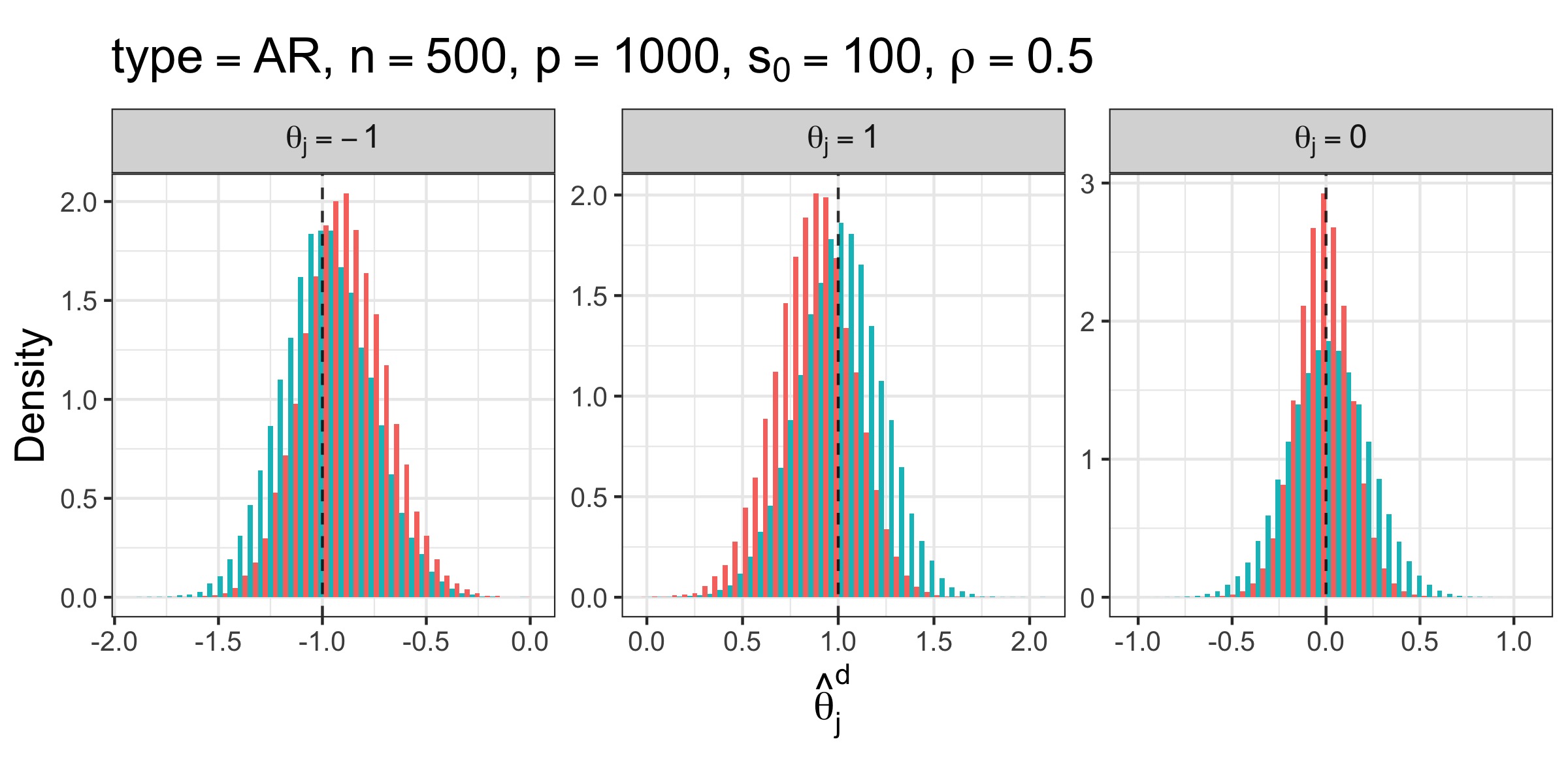}}

\centerline{\includegraphics[width=.78\textwidth]{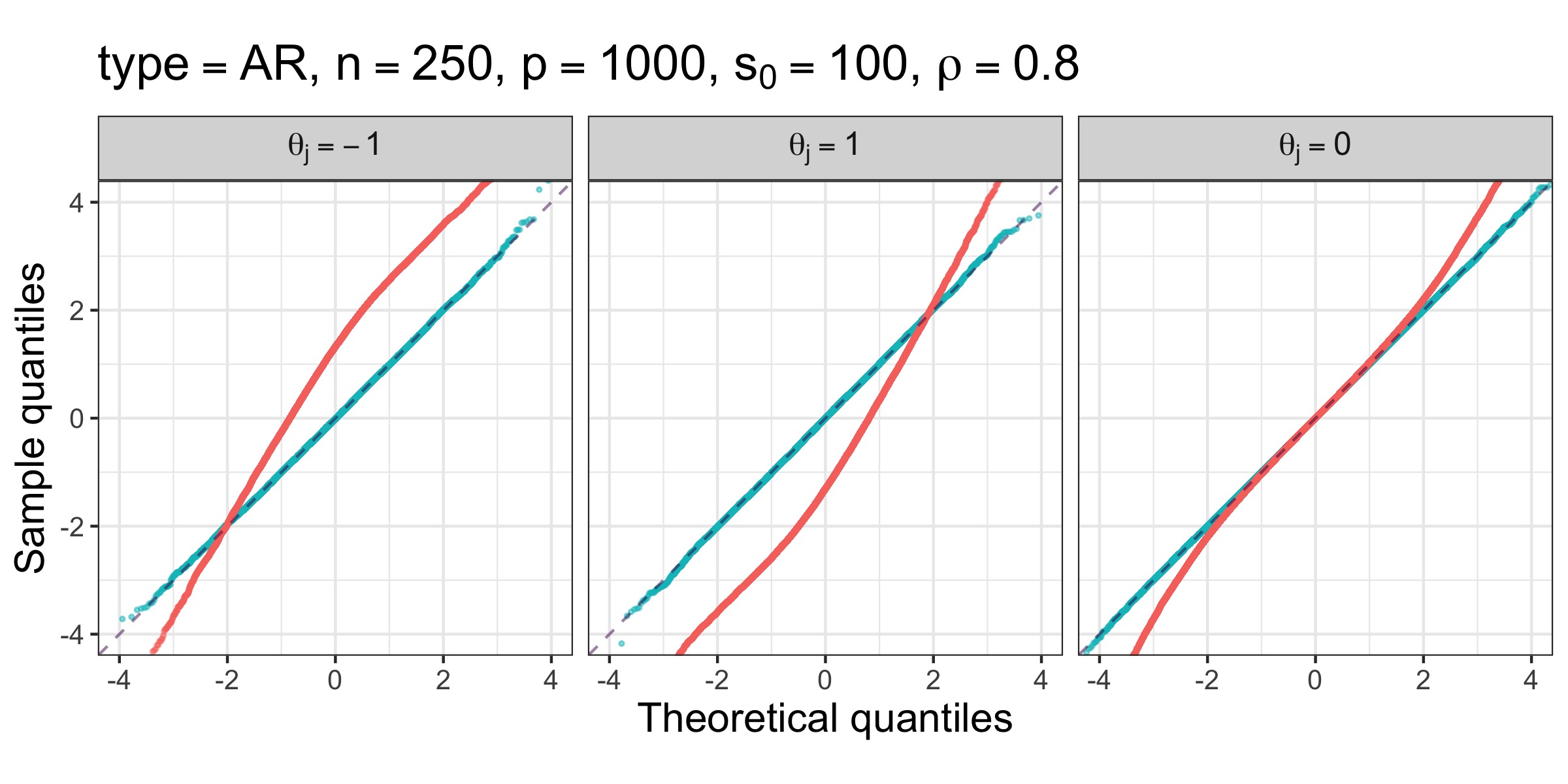}}
\centerline{\includegraphics[width=.78\textwidth]{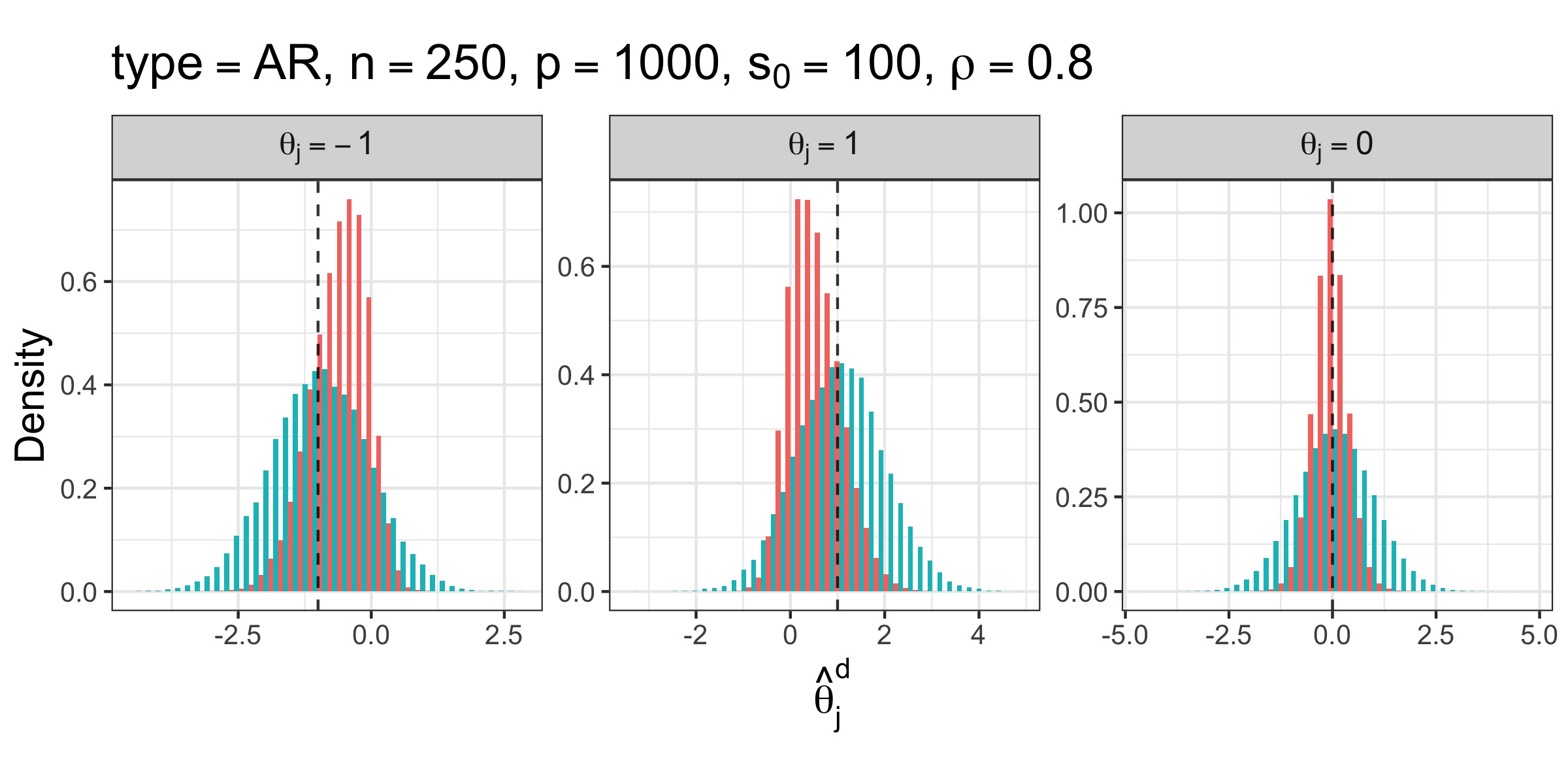}}

\centerline{\includegraphics[width=.78\textwidth]{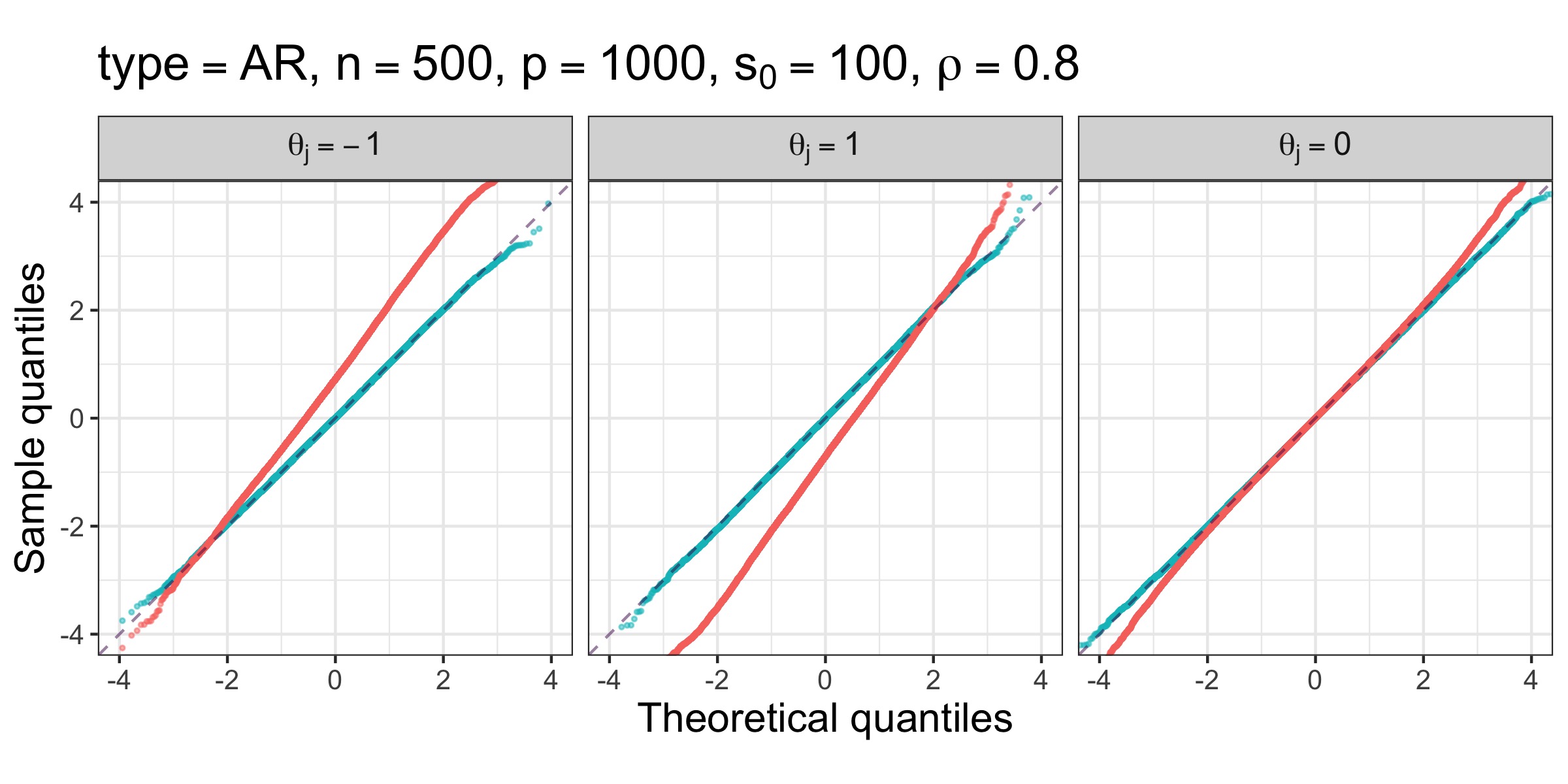}}
\centerline{\includegraphics[width=.78\textwidth]{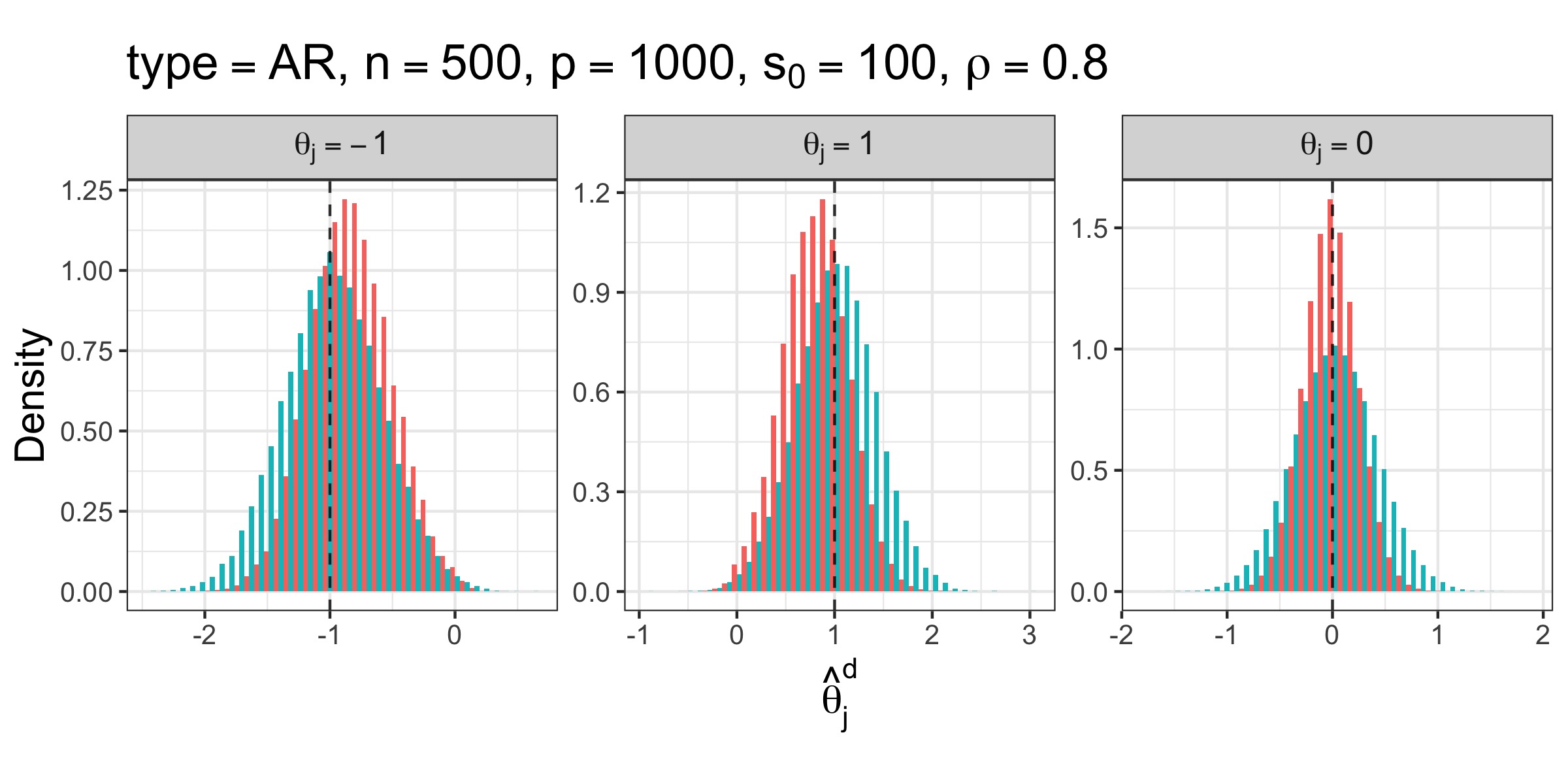}}

% ----- 

\centerline{\includegraphics[width=.78\textwidth]{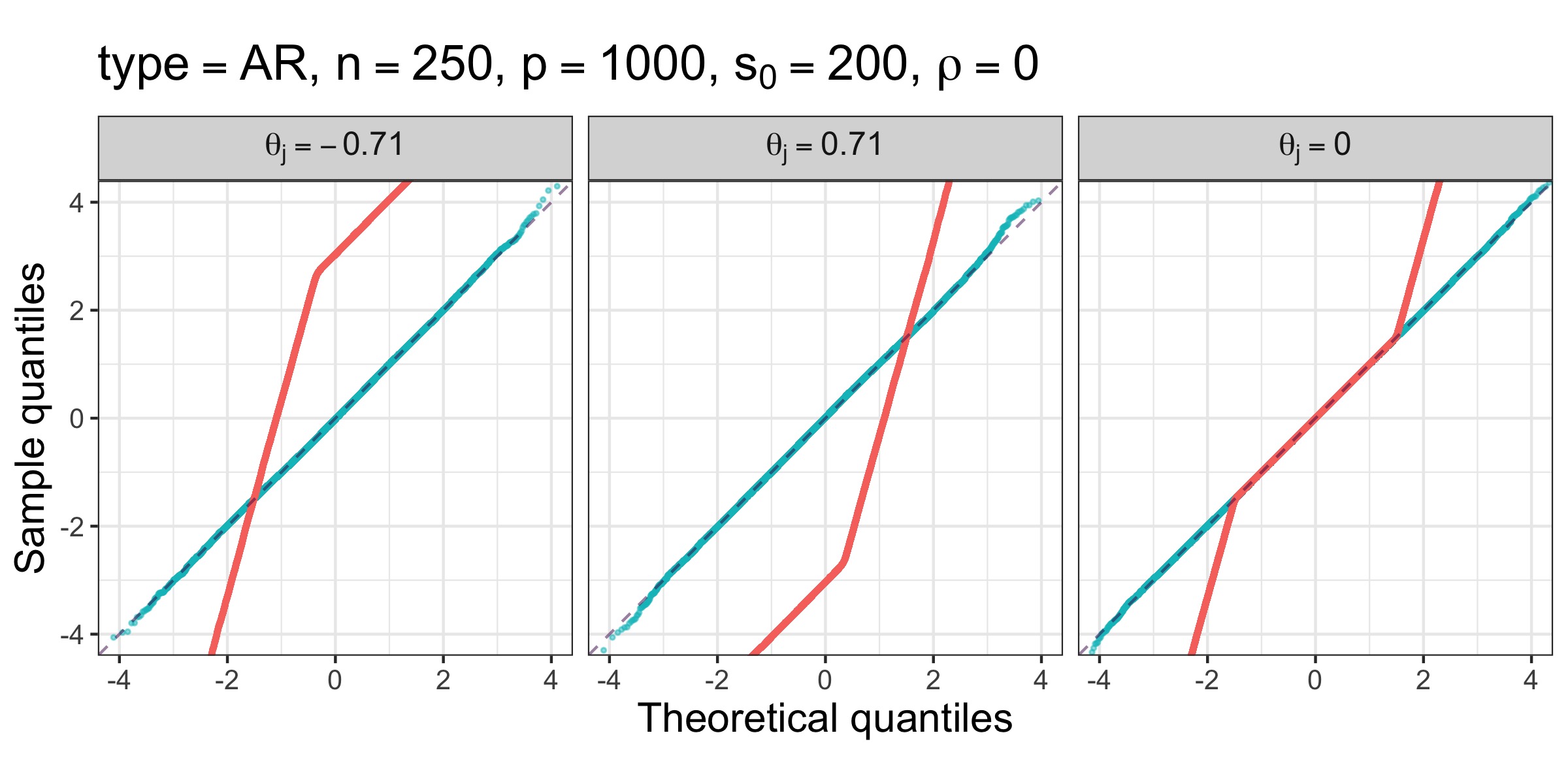}}
\centerline{\includegraphics[width=.78\textwidth]{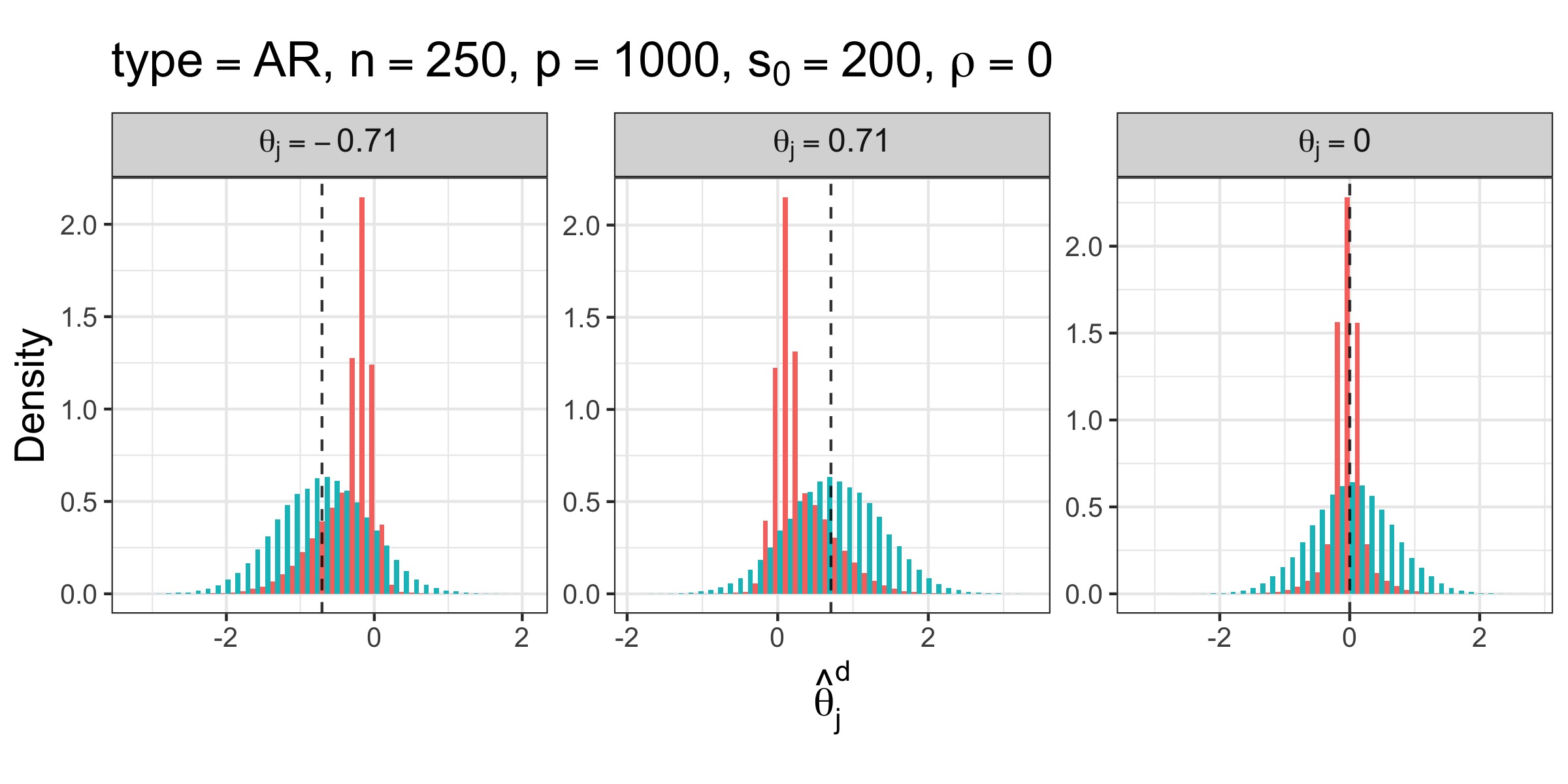}}

\centerline{\includegraphics[width=.78\textwidth]{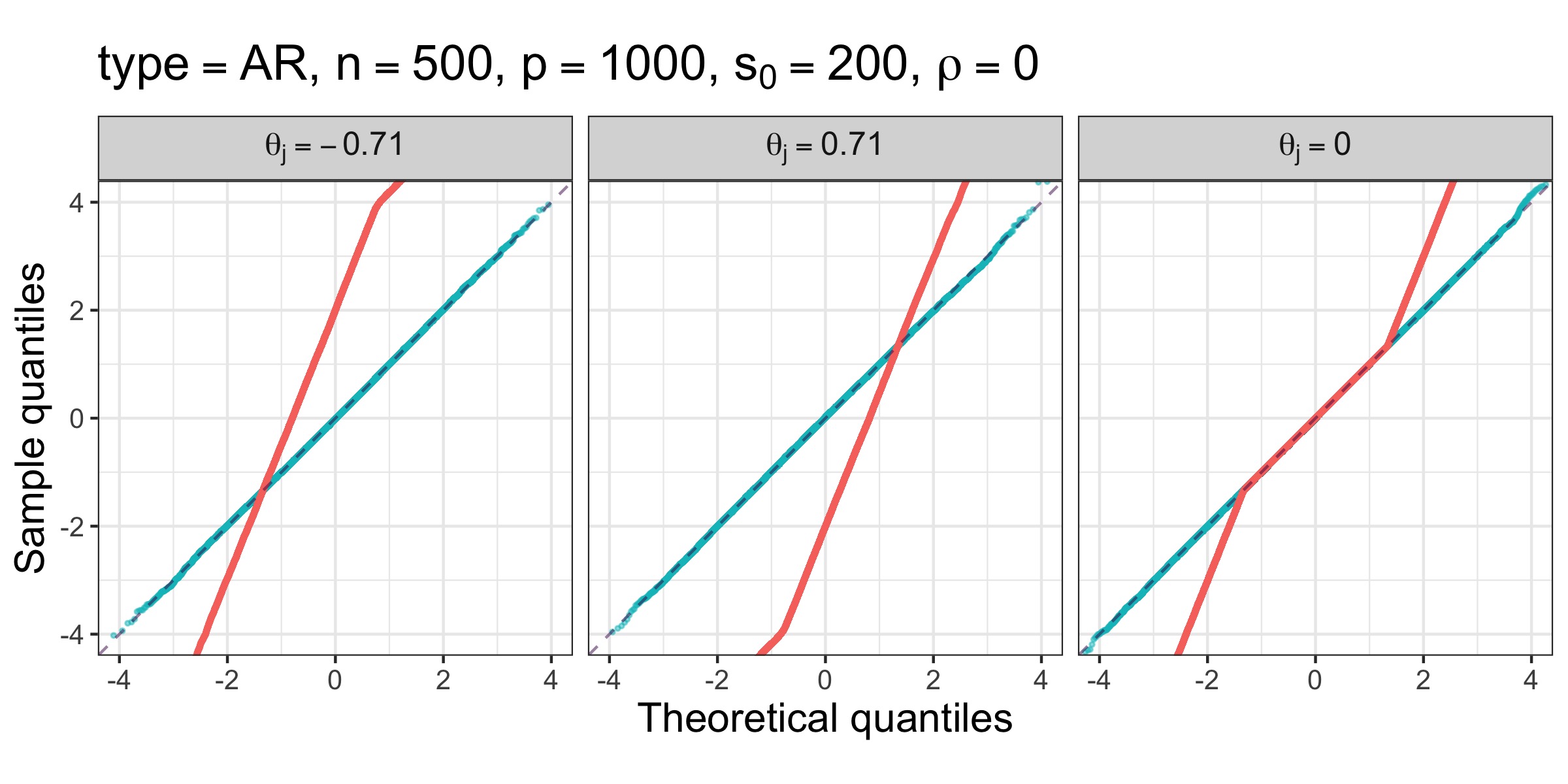}}
\centerline{\includegraphics[width=.78\textwidth]{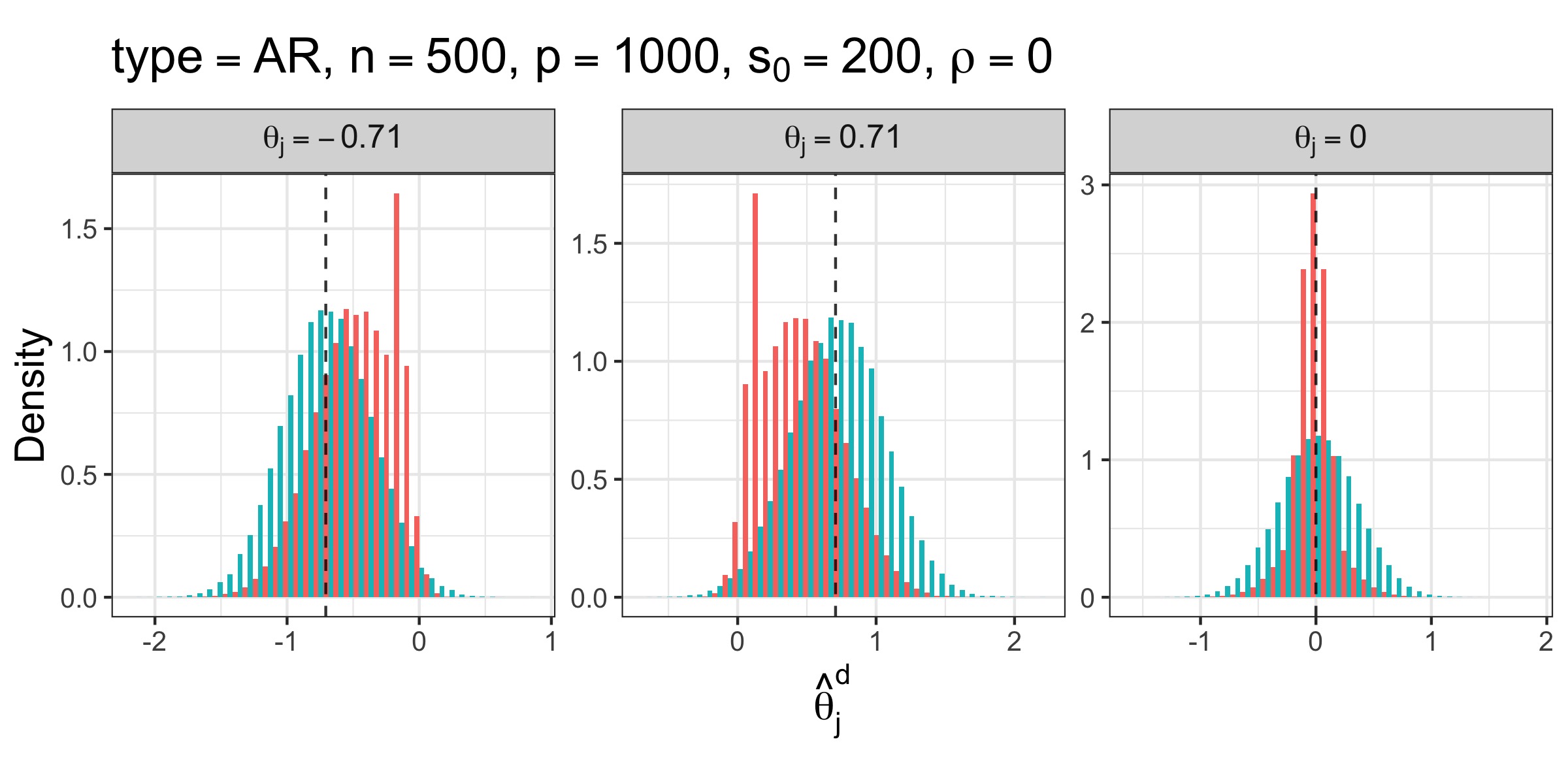}}

\centerline{\includegraphics[width=.78\textwidth]{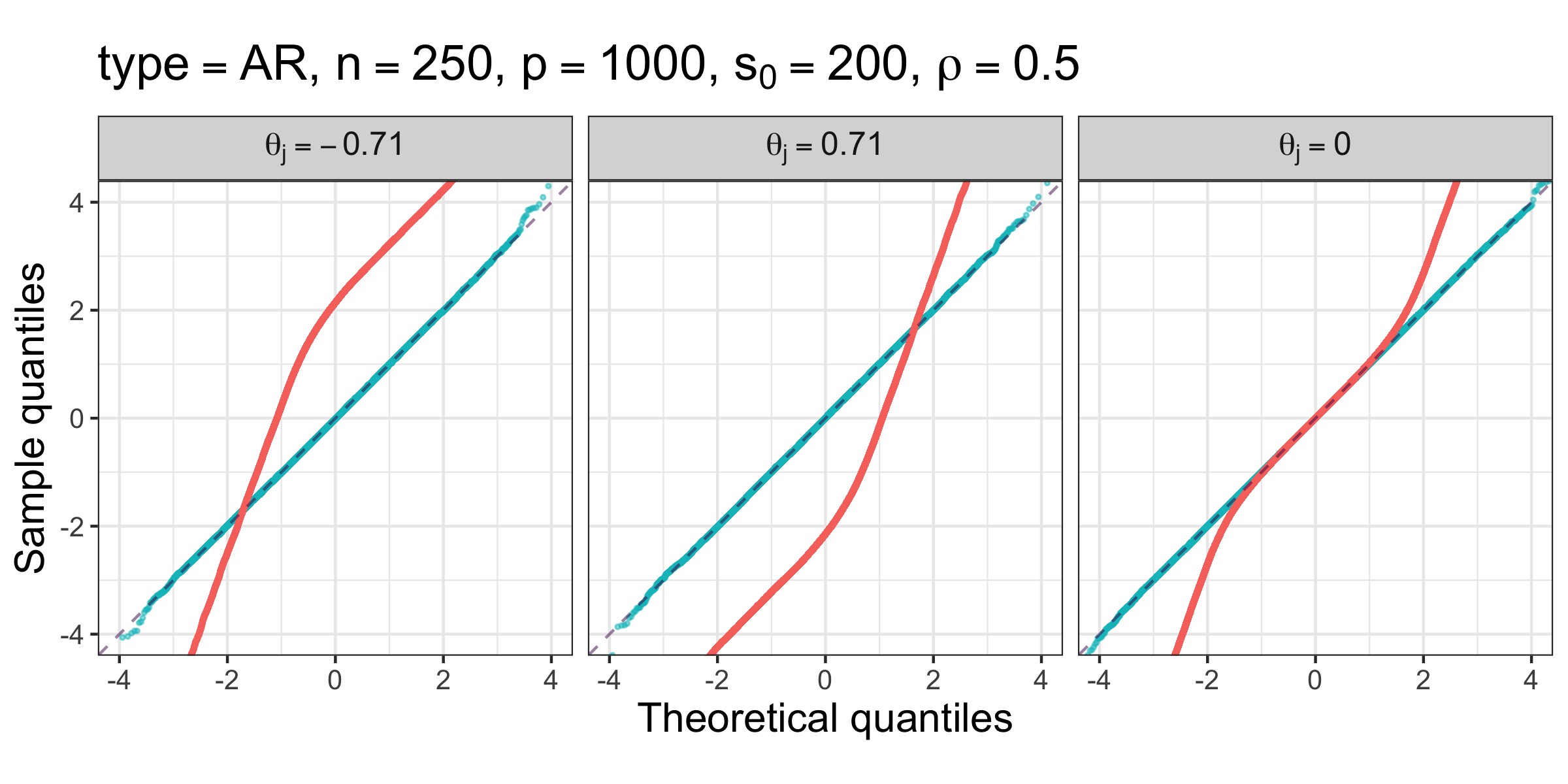}}
\centerline{\includegraphics[width=.78\textwidth]{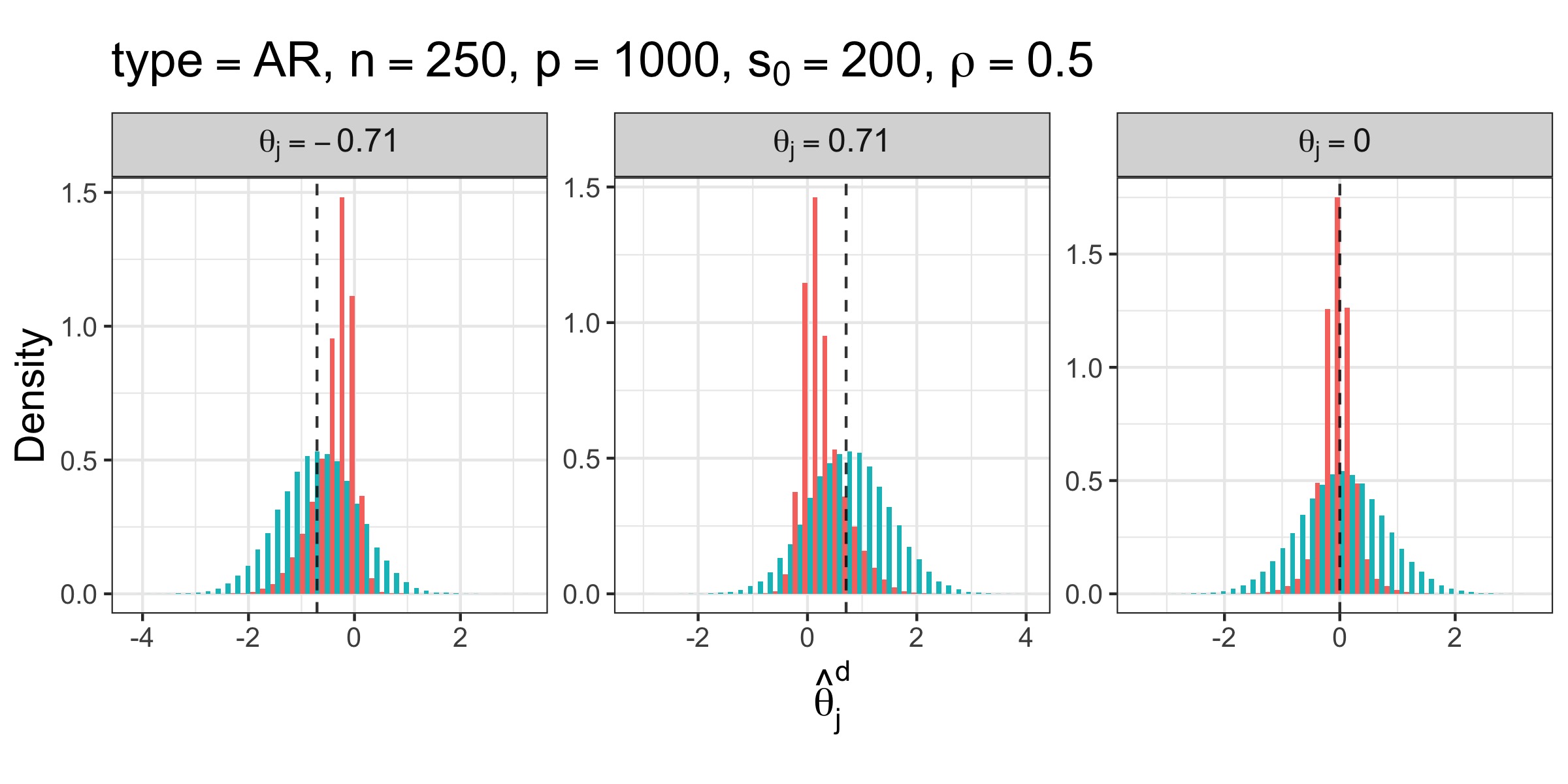}}

\centerline{\includegraphics[width=.78\textwidth]{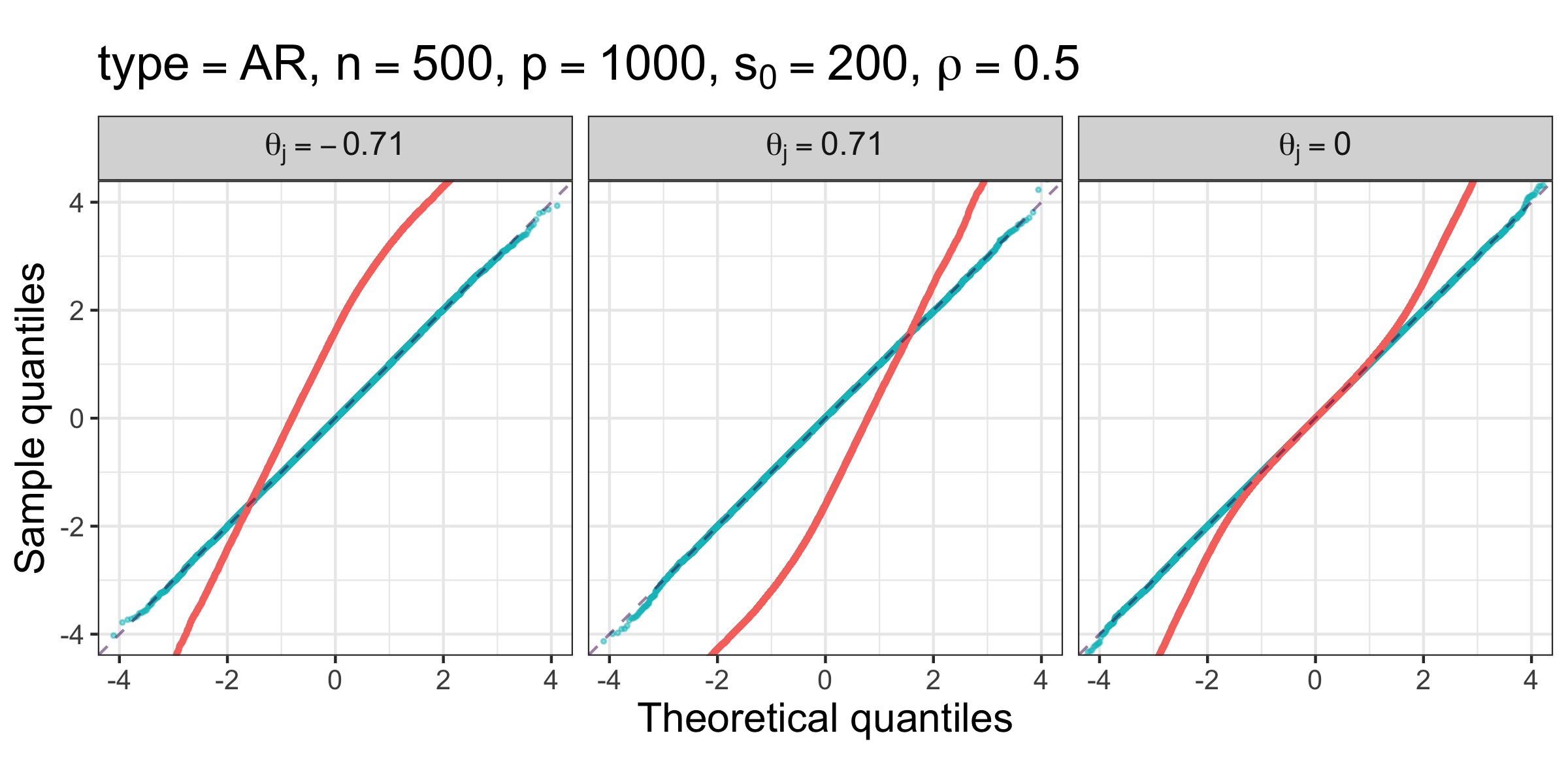}}
\centerline{\includegraphics[width=.78\textwidth]{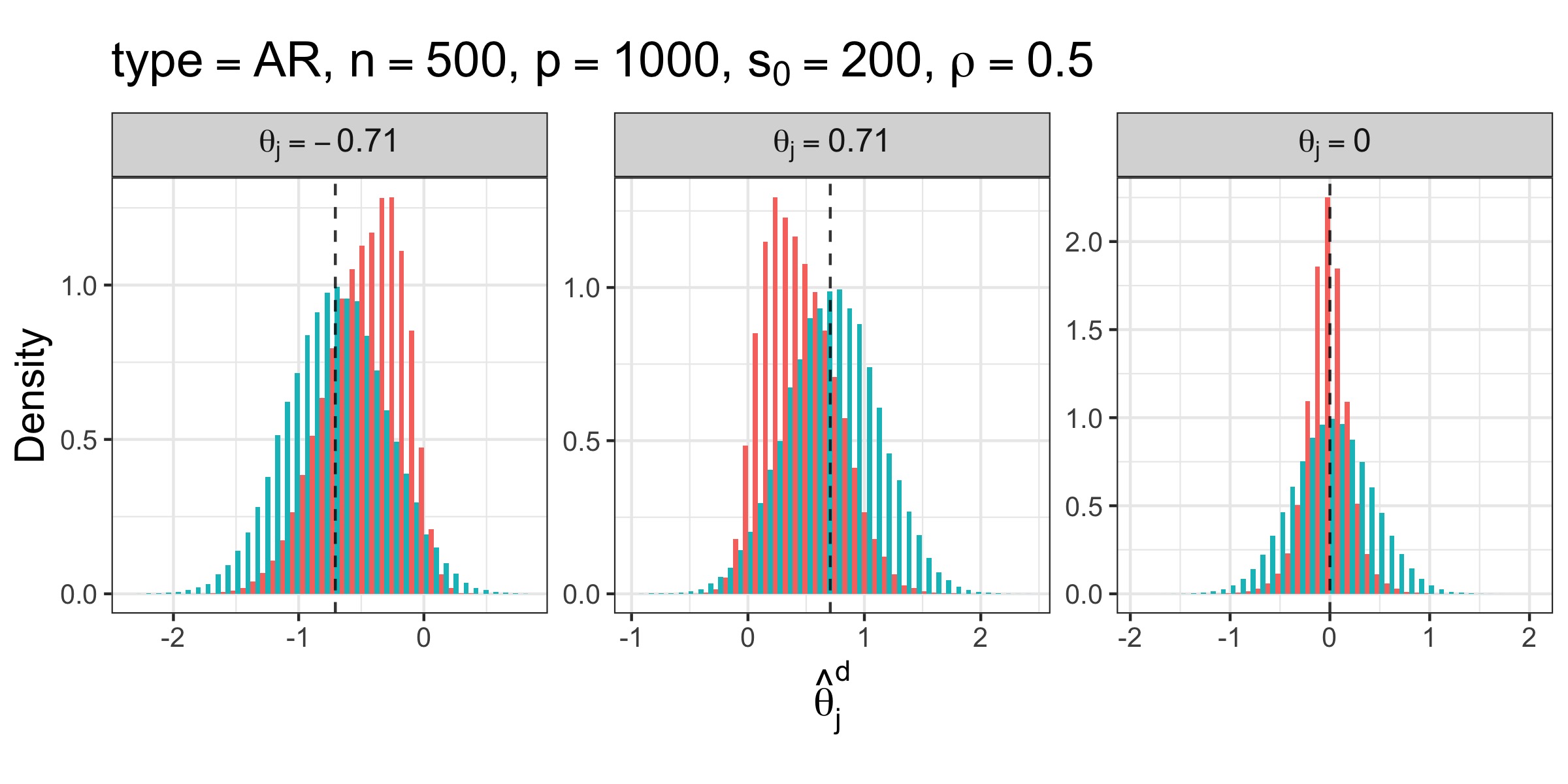}}

\centerline{\includegraphics[width=.78\textwidth]{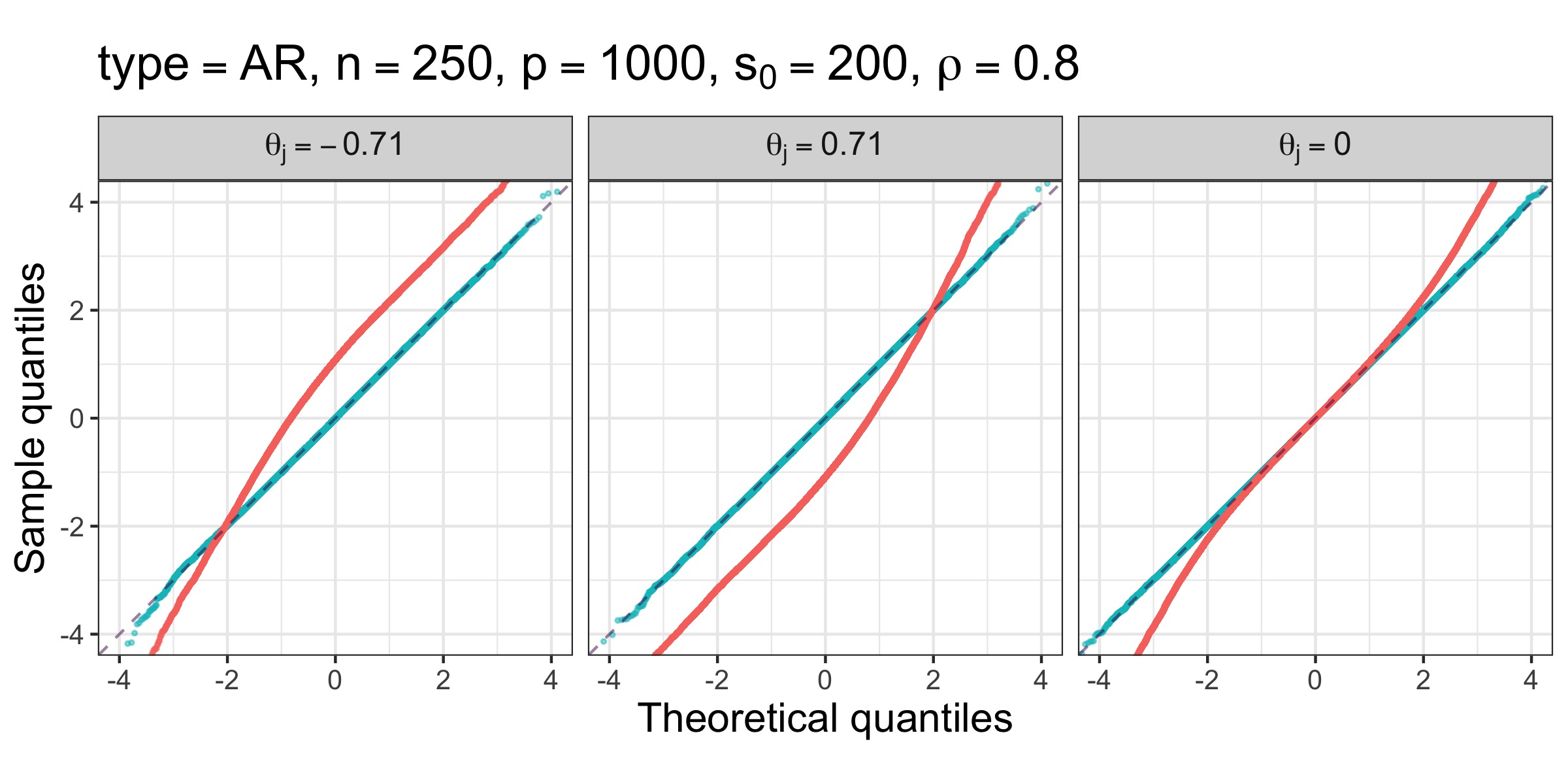}}
\centerline{\includegraphics[width=.78\textwidth]{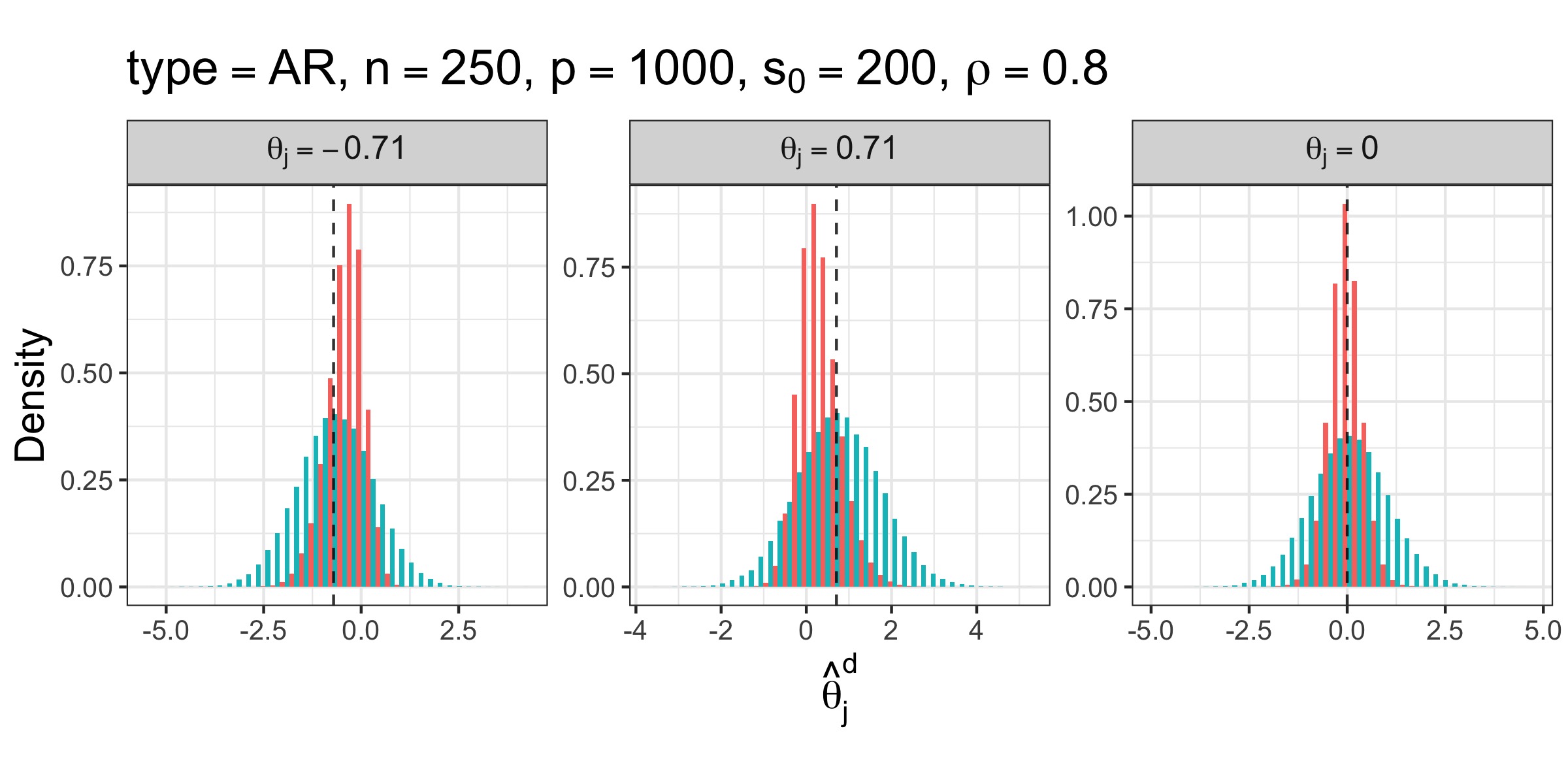}}

\centerline{\includegraphics[width=.78\textwidth]{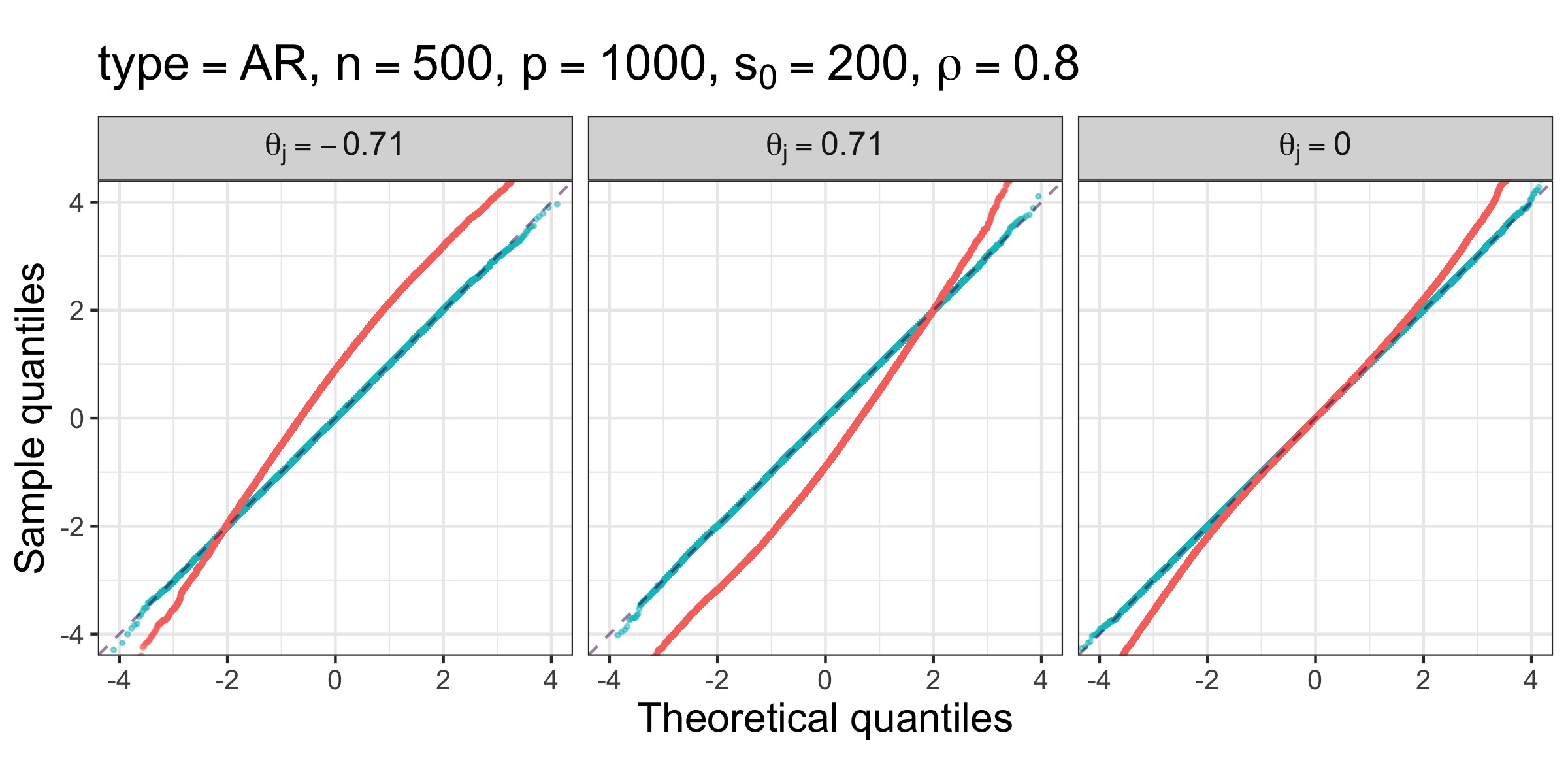}}
\centerline{\includegraphics[width=.78\textwidth]{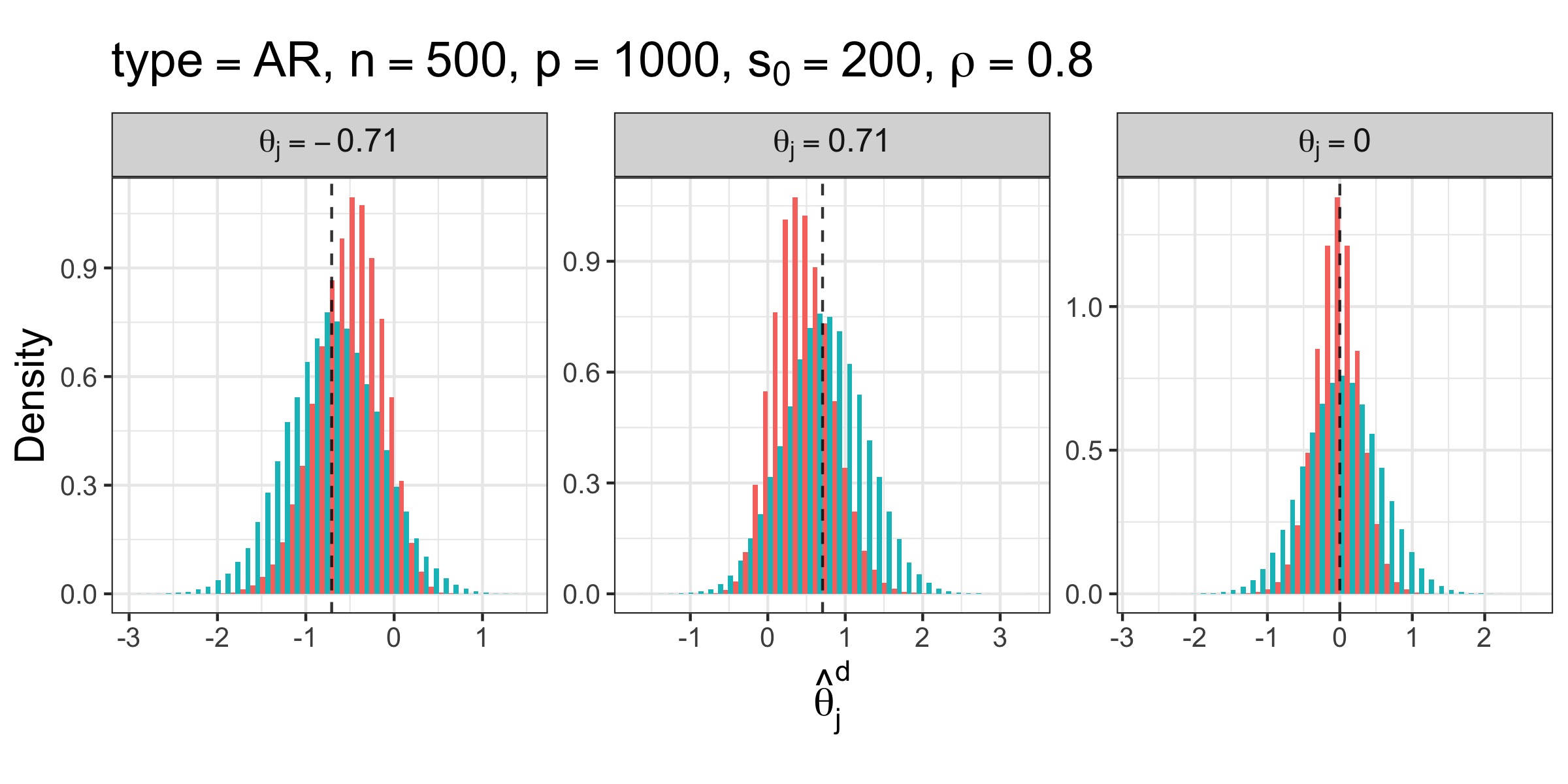}}

\subsection{Debiasing under Gaussian Equicorrelated models}
\label{sec:sim-db-equi}
Here we collect simulations which repeat those in Figure \ref{FigQQplots_and_histograms} but for equicorrelated models: $\Sigma_{ij} = \delta_{ij} + \rho(1-\delta_{ij})$.
We consider correlation parameter $\rho = .5$, the sample size $n = 250, 500$, and the sparsity $s = 20,100,200$.
The legends are the same as in the previous section, so are not shown.

\centerline{\includegraphics[width=.78\textwidth]{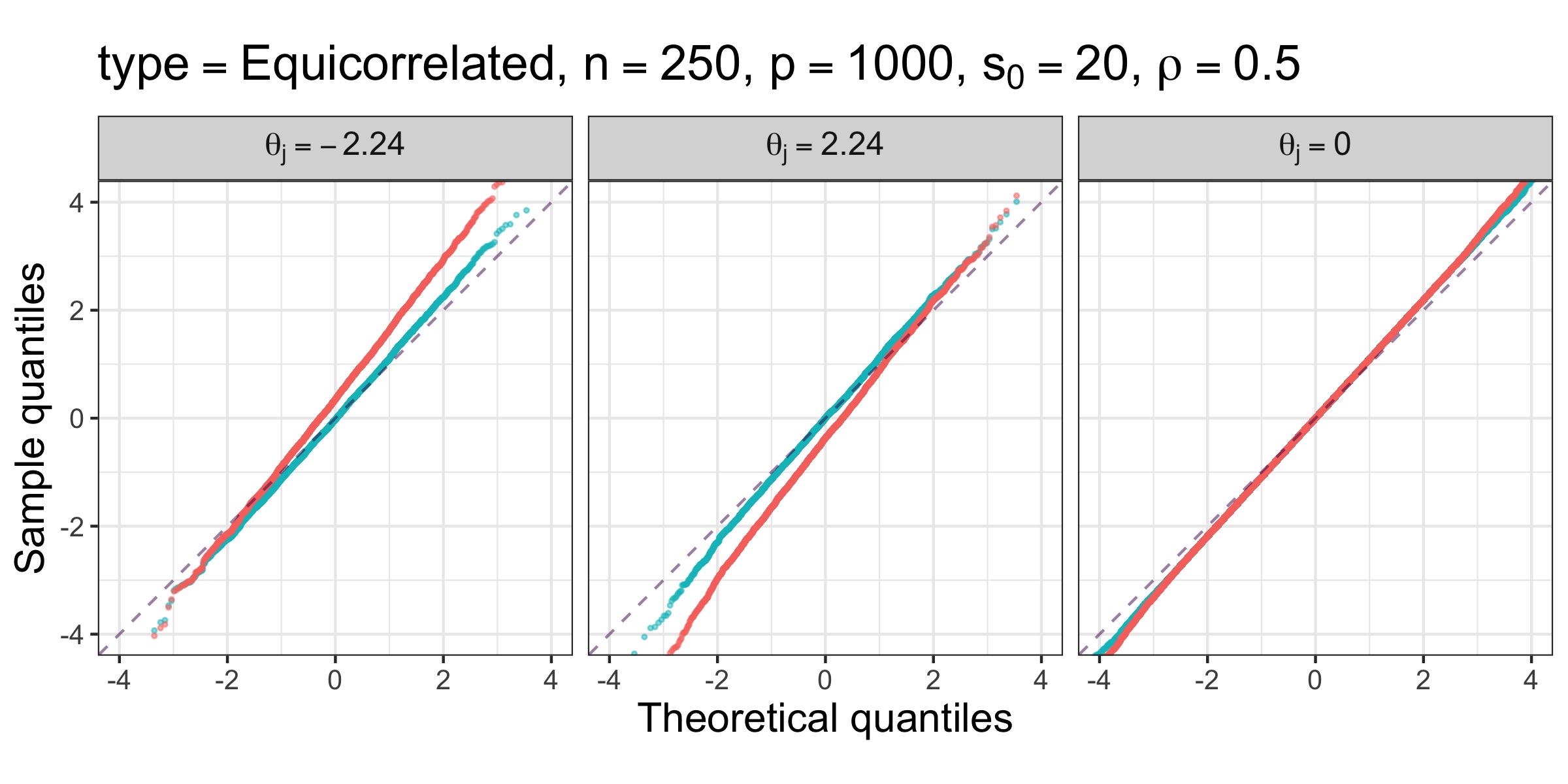}}
\centerline{\includegraphics[width=.78\textwidth]{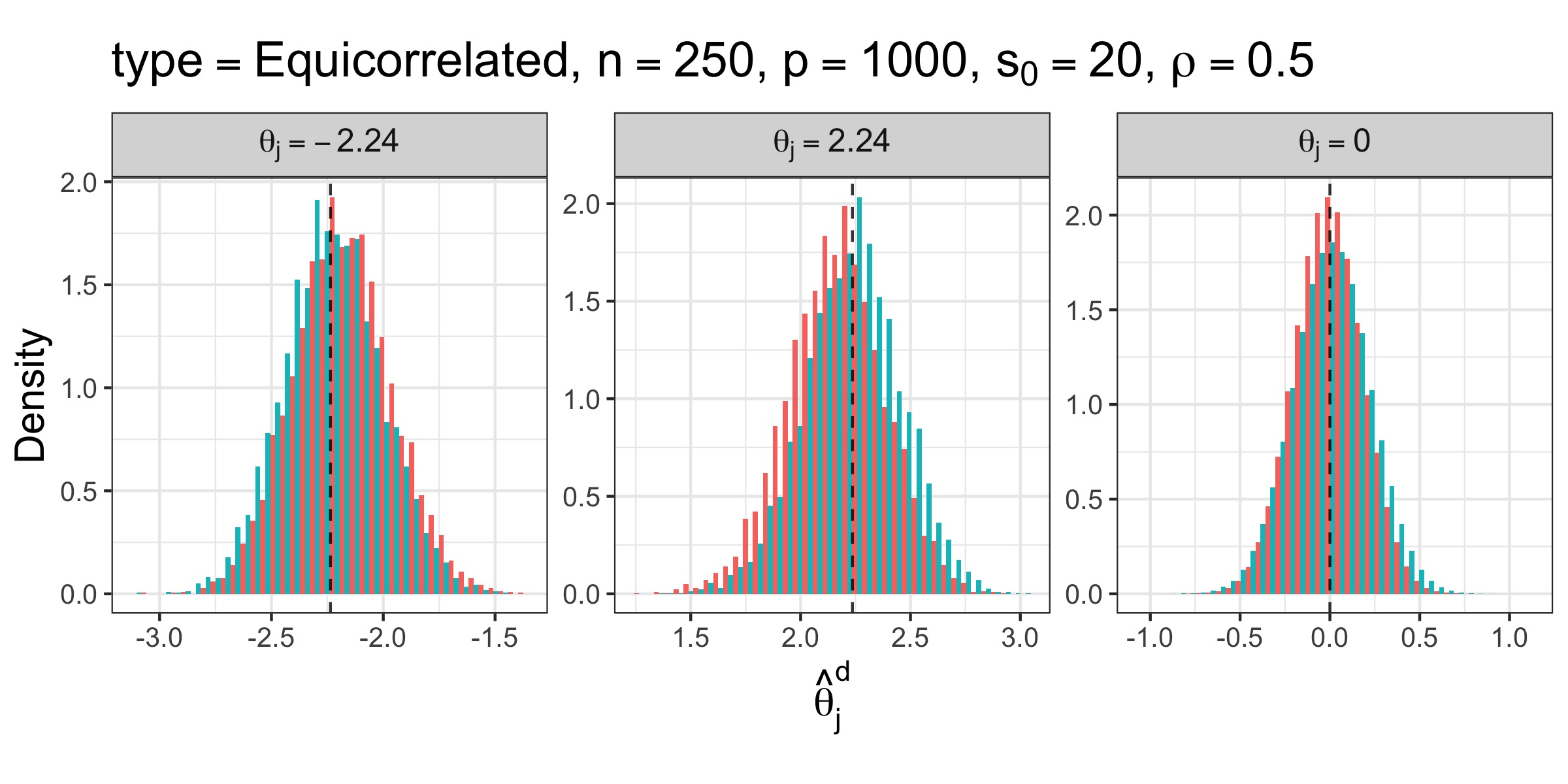}}

\centerline{\includegraphics[width=.78\textwidth]{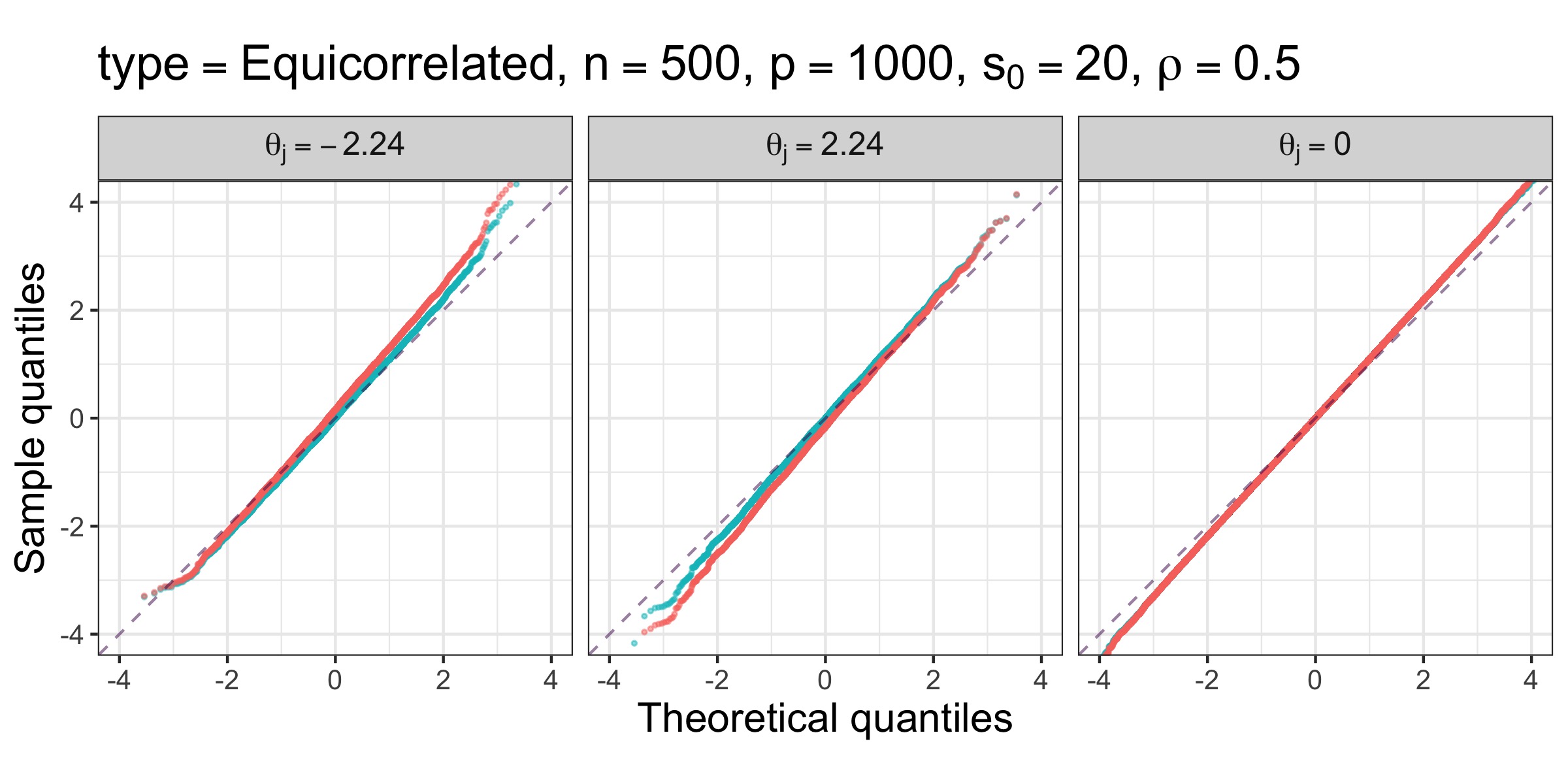}}
\centerline{\includegraphics[width=.78\textwidth]{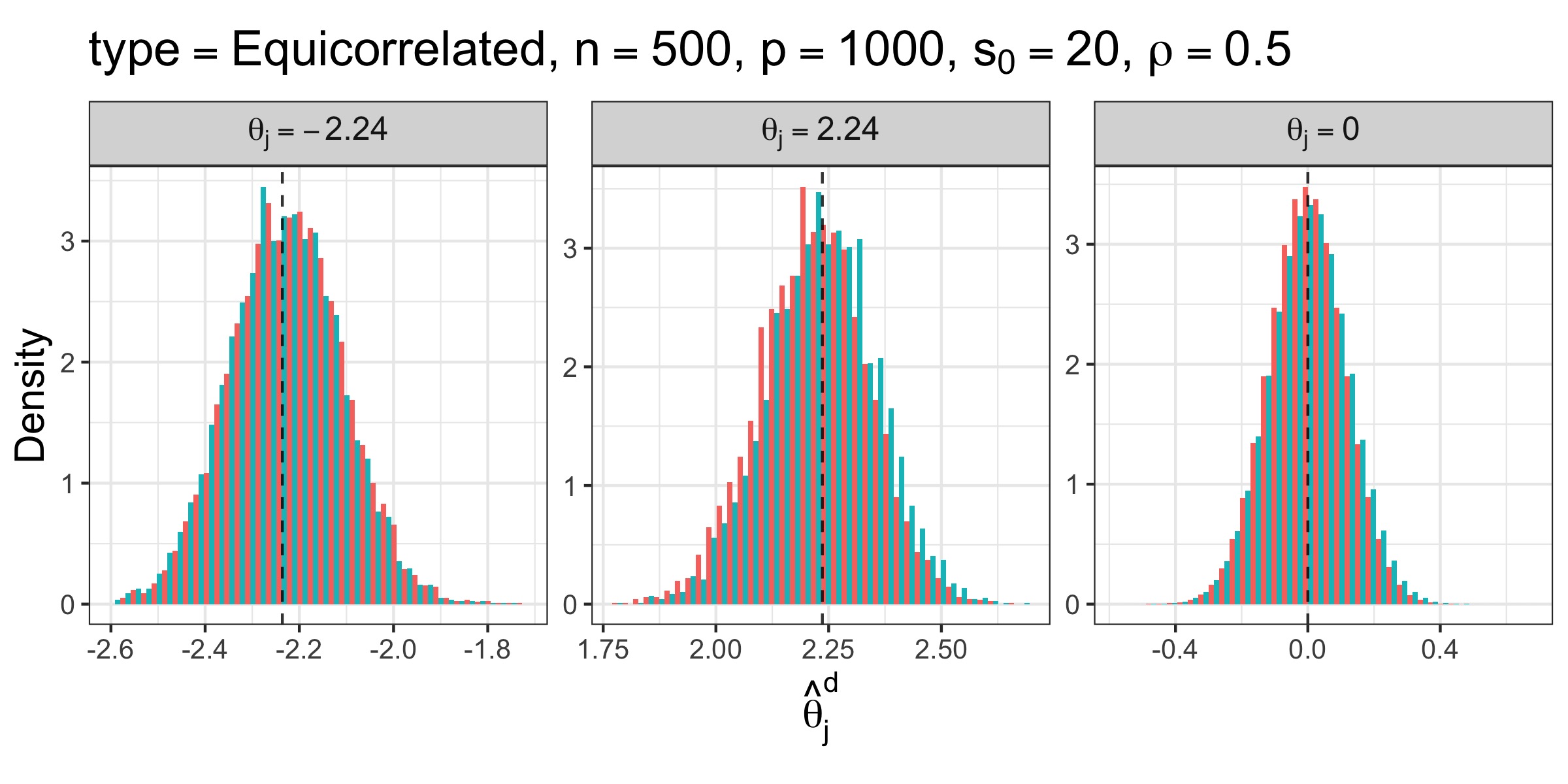}}

\centerline{\includegraphics[width=.78\textwidth]{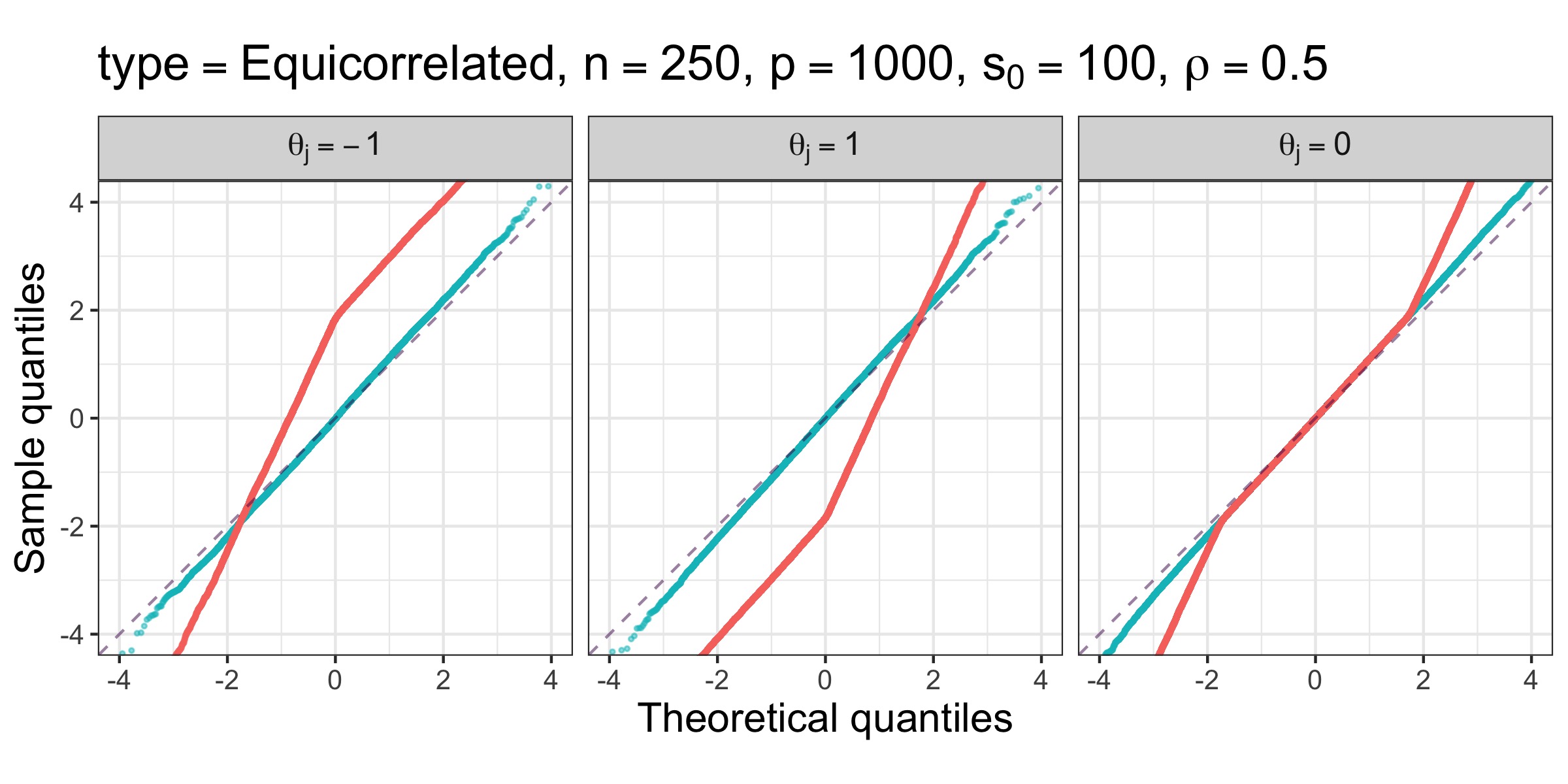}}
\centerline{\includegraphics[width=.78\textwidth]{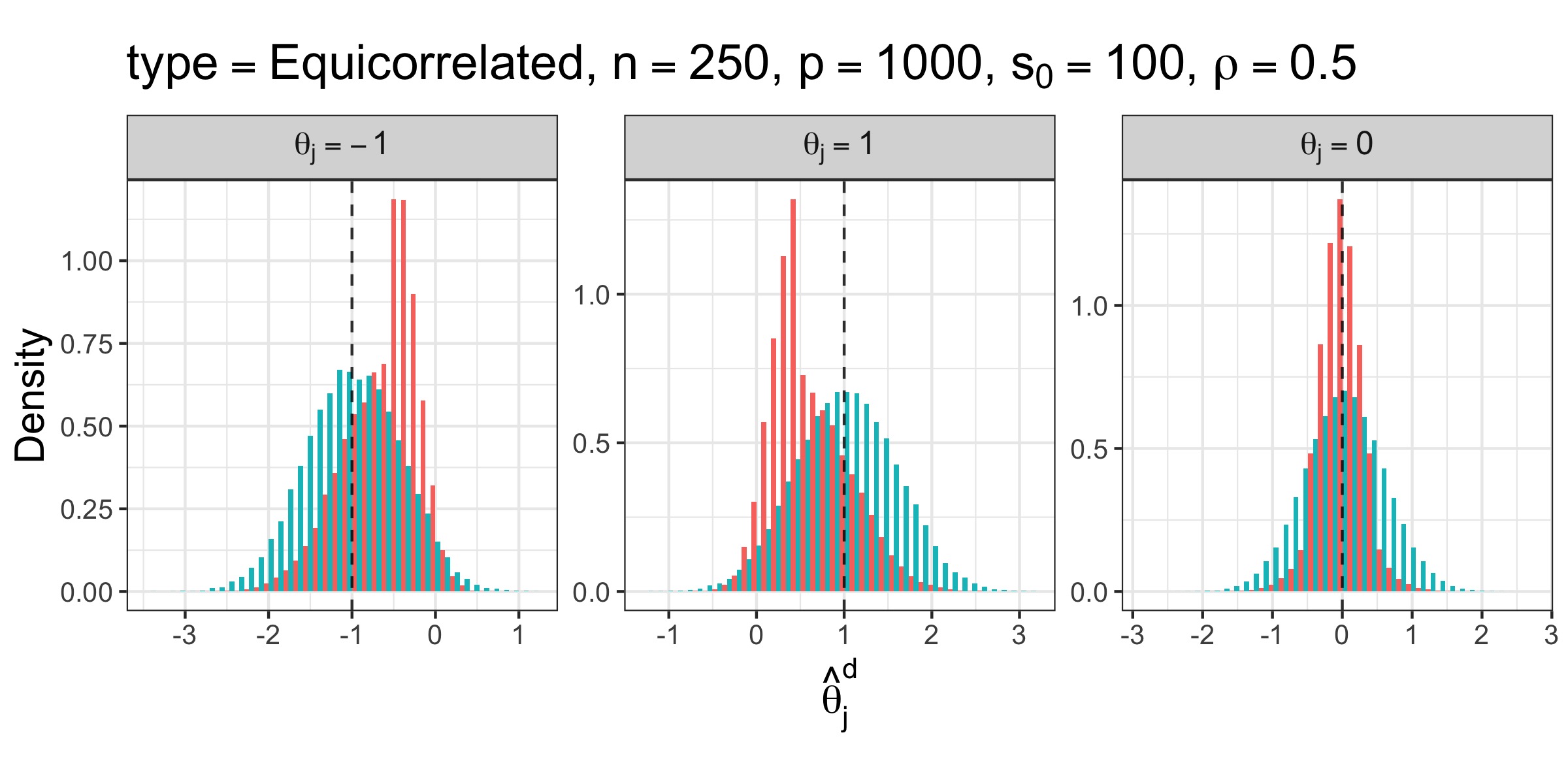}}

\centerline{\includegraphics[width=.78\textwidth]{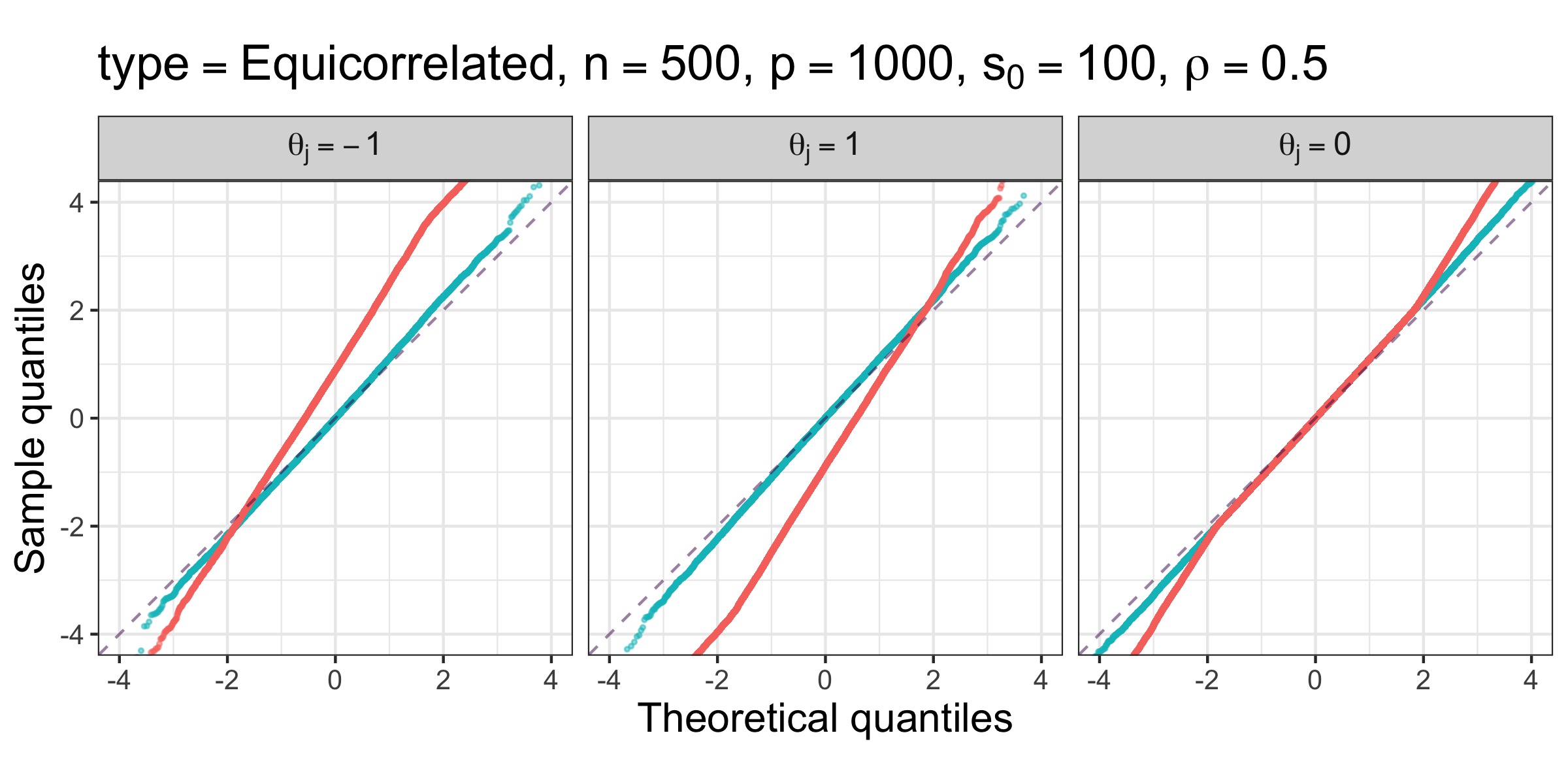}}
\centerline{\includegraphics[width=.78\textwidth]{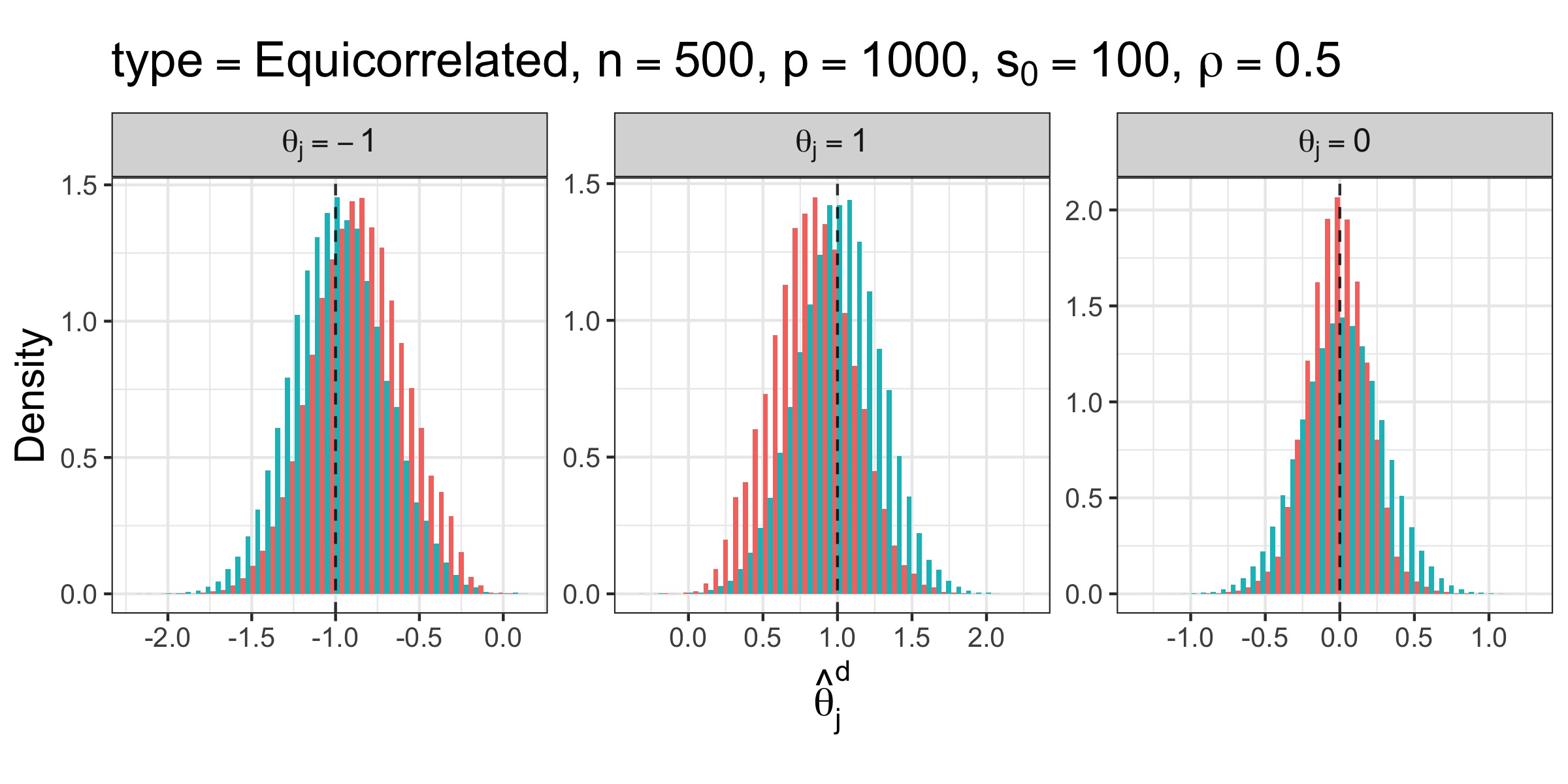}}

\centerline{\includegraphics[width=.78\textwidth]{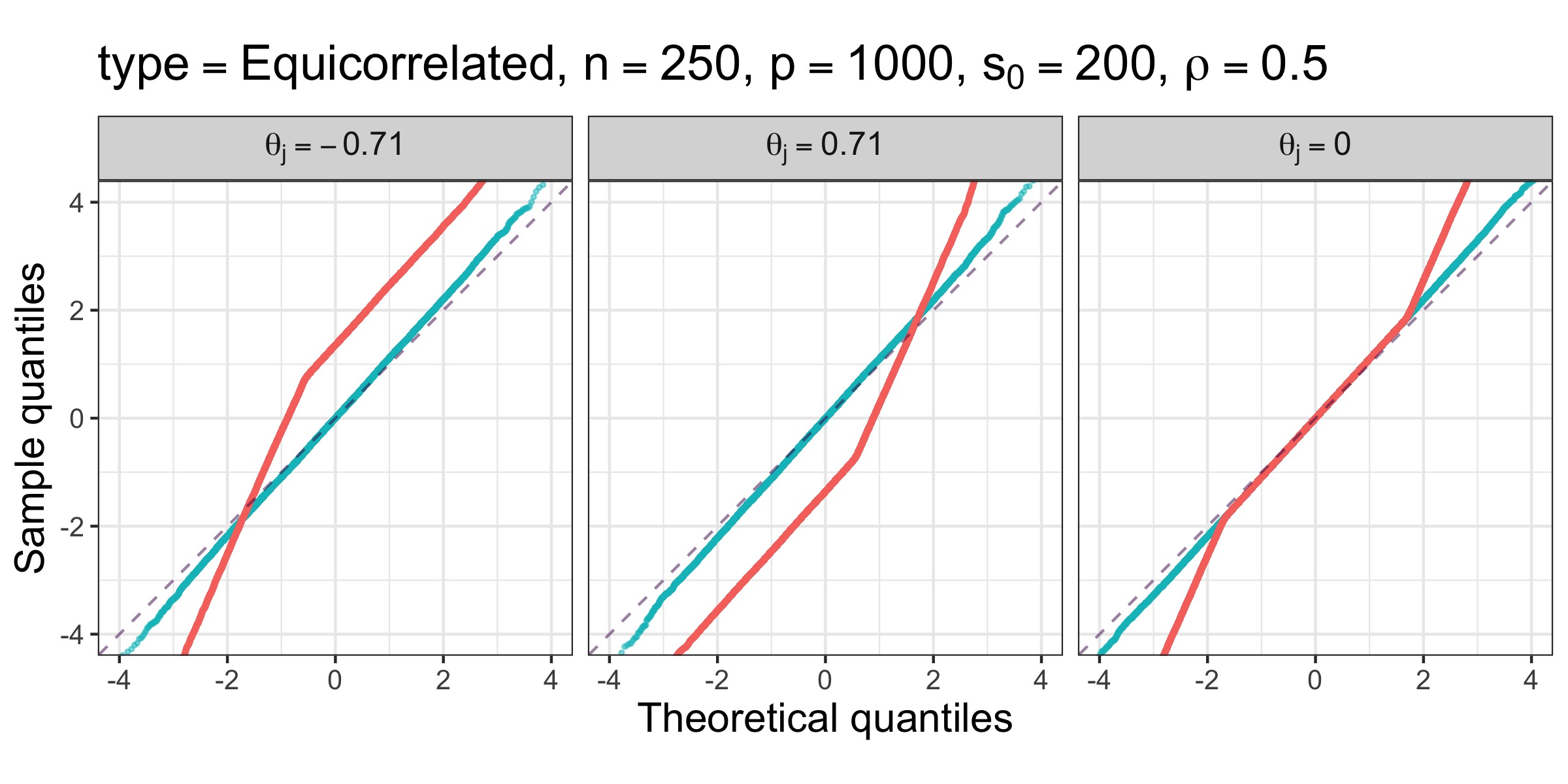}}
\centerline{\includegraphics[width=.78\textwidth]{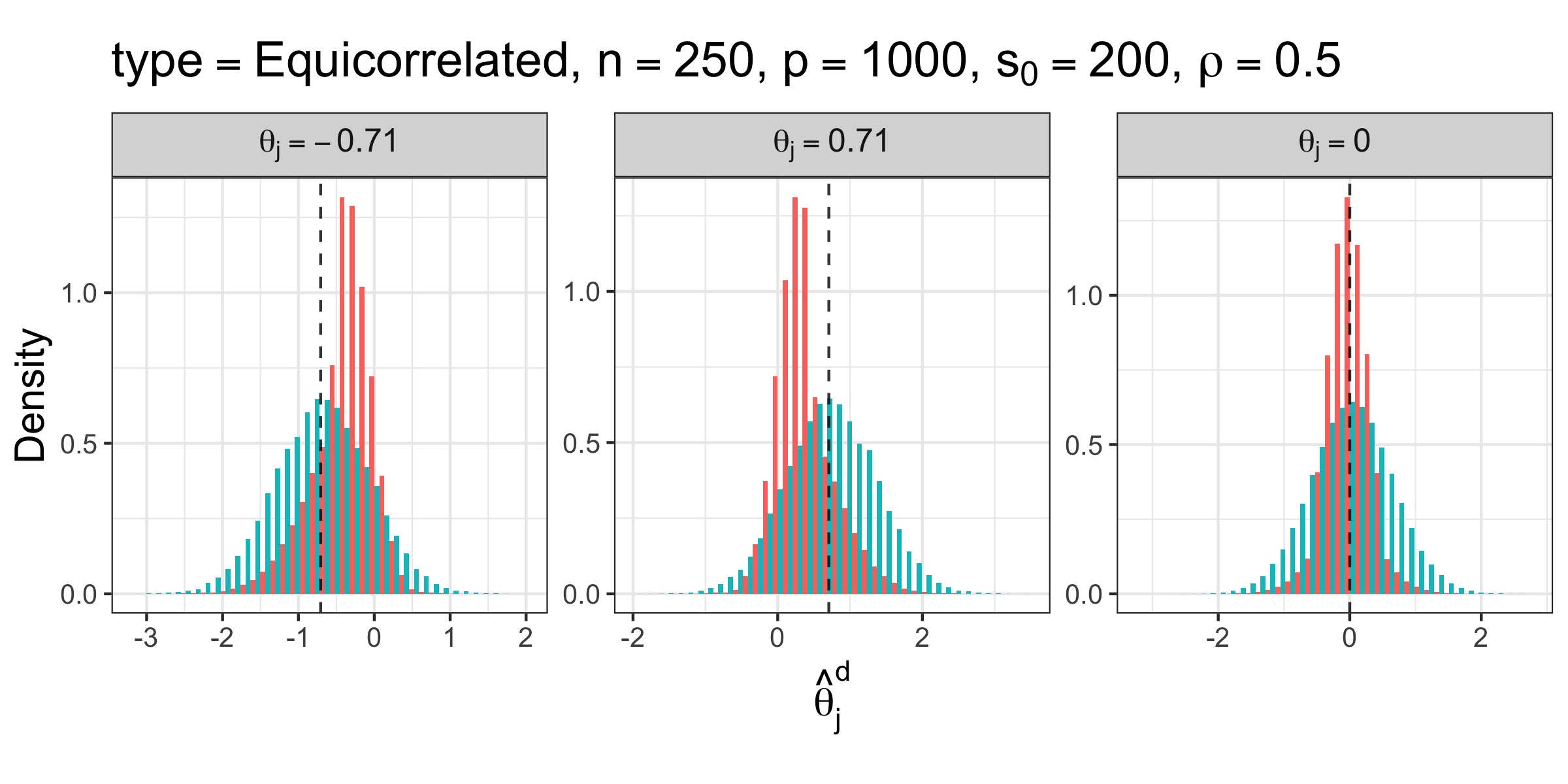}}

\centerline{\includegraphics[width=.78\textwidth]{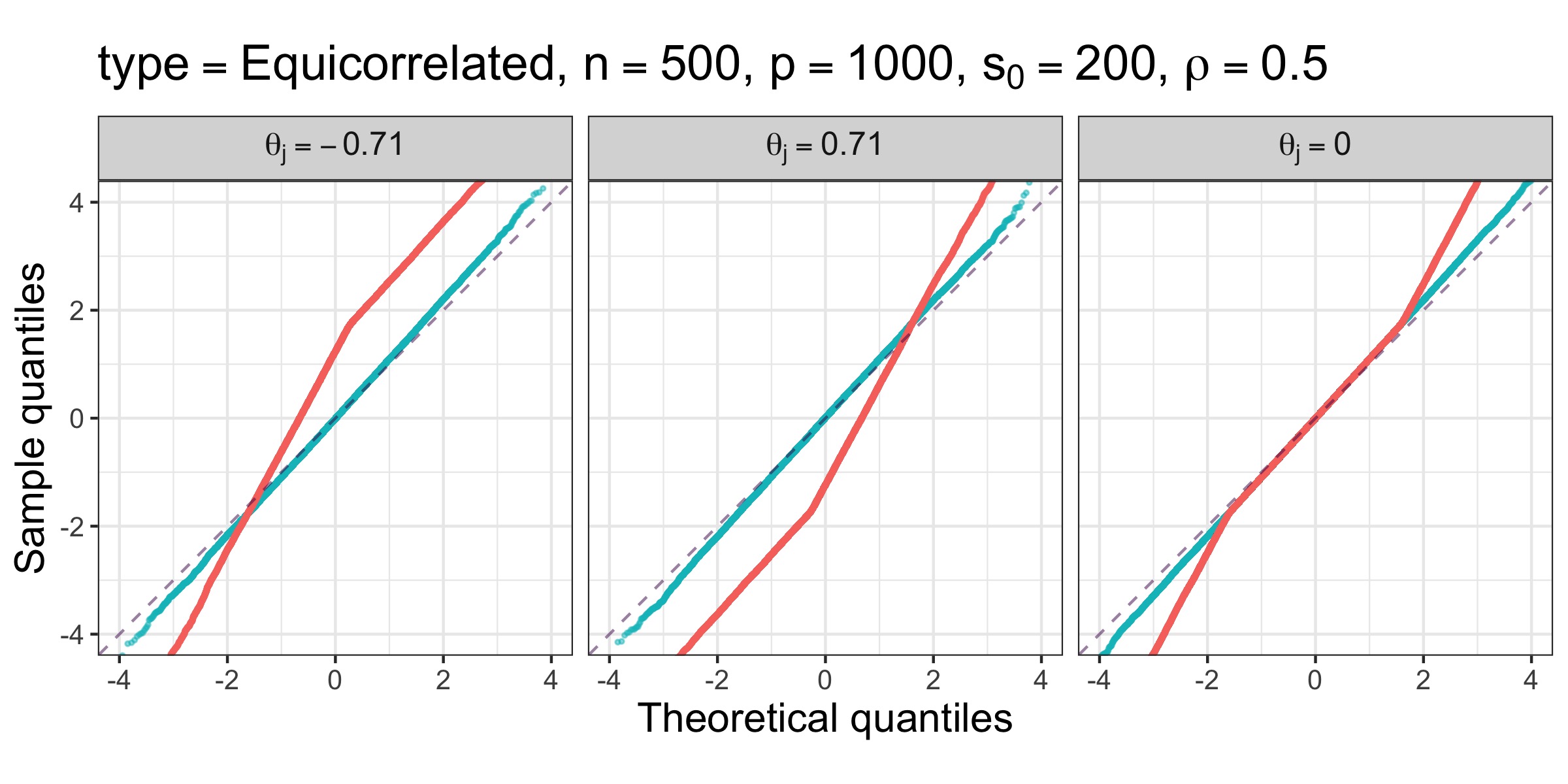}}
\centerline{\includegraphics[width=.78\textwidth]{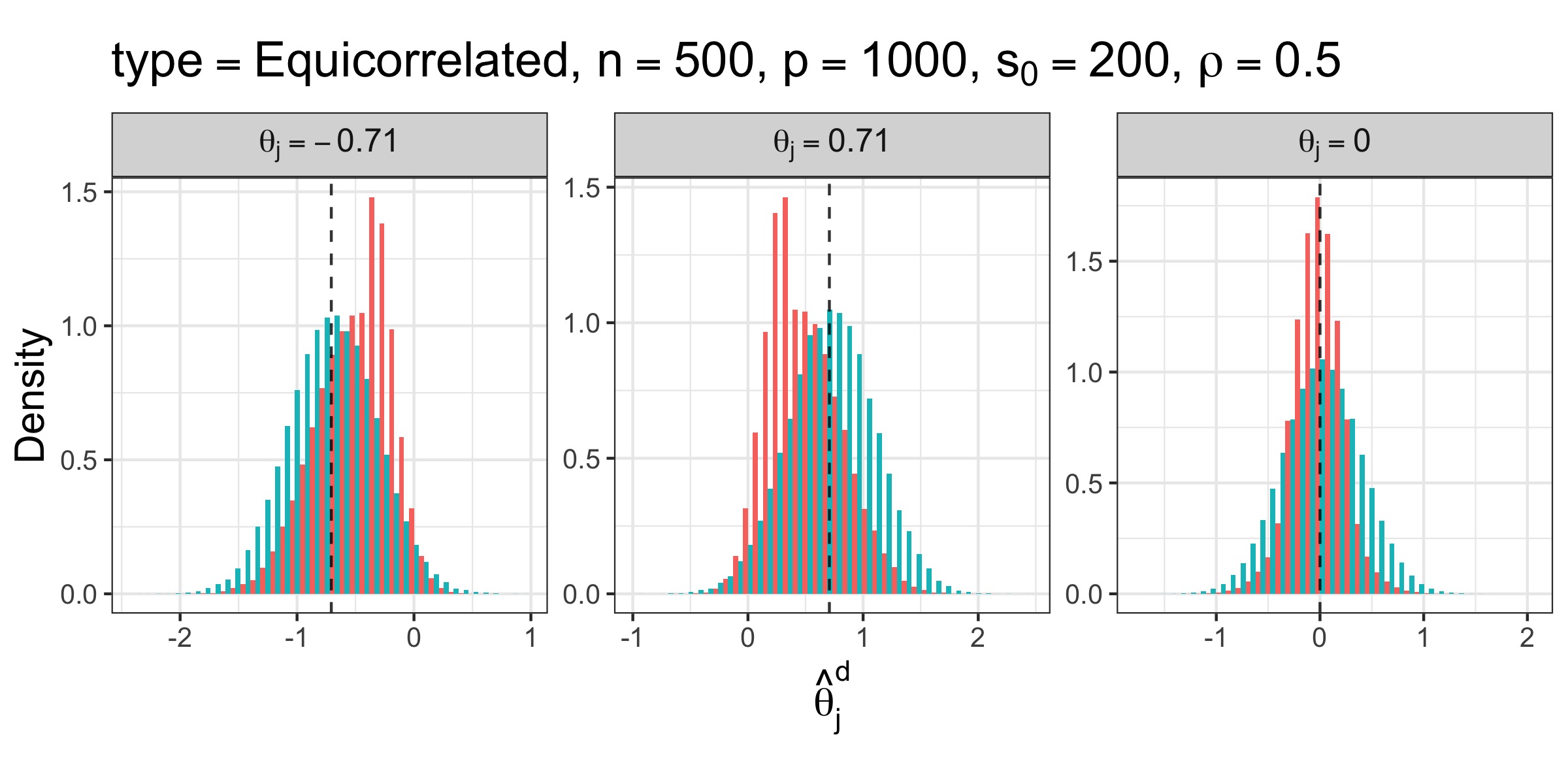}}

\end{document}